\documentclass[12pt,reqno]{amsart}

\usepackage{amsmath, amssymb, amsfonts, amscd, latexsym, amsthm, mathrsfs, verbatim, textcomp}
\usepackage{keytheorems}
\usepackage{cite}
\usepackage{hypbmsec}
\usepackage{enumerate}
\usepackage{bm}
\usepackage{tikz}
\usepackage{mathtools}
\usepackage{xfrac, mathabx}
\theoremstyle{plain}
\usepackage{manfnt}
\usepackage{tikz}
\usetikzlibrary{matrix,arrows,decorations.pathmorphing}
\usepackage[margin=1.1in]{geometry}
\usepackage{csquotes}

\usepackage[colorlinks, linkcolor=black, citecolor=black, hypertexnames=false]{hyperref}
\usepackage[capitalize, nameinlink]{cleveref}

\usepackage[alphabetic, nobysame]{amsrefs}

\AtBeginDocument{\def\MR#1{}}

\newcommand*{\bbone}{\text{\usefont{U}{bbold}{m}{n}1}}

\newkeytheorem{theorem}[name=Theorem, numberwithin=section, refname={Theorem,Theorems}, Refname={Theorem,Theorems}]

\newkeytheorem{question}[name=Question, sibling=theorem, refname={Question,Questions}, Refname={Question,Questions}]
\newkeytheorem{lemma}[name=Lemma, sibling=theorem, refname={Lemma,Lemmas}, Refname={Lemma,Lemmas}]
\newkeytheorem{prop}[name=Proposition, sibling=theorem, refname={Proposition,Propositions}, Refname={Proposition,Propositions}]
\newkeytheorem{corollary}[name=Corollary, sibling=theorem, refname={Corollary,Corollaries}, Refname={Corollary,Corollaries}]
\newkeytheorem{definition}[name=Definition, sibling=theorem, refname={Definition,Definitions}, Refname={Definition,Definitions}]
\newkeytheorem{property}[name=Property, sibling=theorem, refname={Property,Properties}, Refname={Property,Properties}]
\newkeytheorem{remark}[name=Remark, sibling=theorem, refname={Remark,Remarks}, Refname={Remark,Remarks}]
\newkeytheorem{conjecture}[name=Conjecture, sibling=theorem, refname={Conjecture,Conjectures}, Refname={Conjecture,Conjectures}]
\newkeytheorem{claim}[name=Claim, sibling=theorem, refname={Claim,Claims}, Refname={Claim,Claims}]
\newkeytheorem{convention}[name=Convention, sibling=theorem, refname={Convention,Conventions}, Refname={Convention,Conventions}]
\newkeytheorem{example}[name=Example, sibling=theorem, refname={Example,Examples}, Refname={Example,Examples}]

\newkeytheorem{claim*}[name=Claim, numbered=false]

\crefname{section}{Section}{Sections}
\Crefname{section}{Section}{Sections}
\crefname{appendix}{Appendix}{Appendices}
\Crefname{appendix}{Appendix}{Appendices}

\numberwithin{equation}{section}

\renewcommand{\Im}{\operatorname{Im}}
\renewcommand{\Re}{\operatorname{Re}}

\newcommand{\Z}{\mathbb{Z}}
\newcommand{\Q}{\mathbb{Q}}
\newcommand{\R}{\mathbb{R}}
\newcommand{\C}{\mathbb{C}}

\newcommand{\sumtwo}{\operatorname*{\sum\sum}}
\newcommand{\sumthree}{\operatorname*{\sum\sum\sum}}

\renewcommand{\pmod}[1]{\ (\mathrm{mod}\ #1)}

\usepackage{etoolbox}
\patchcmd{\section}{\scshape}{\bfseries}{}{}
\makeatletter
\renewcommand{\@secnumfont}{\bfseries}
\makeatother

\numberwithin{theorem}{section}
\numberwithin{equation}{section}

\newcommand{\pfrac}[2]{\left(\frac{#1}{#2}\right)}

\makeatletter
\@namedef{subjclassname@2020}{\textup{2020} Mathematics Subject Classification}
\makeatother

\newcommand\cube{\begin{tikzpicture}[scale=2.3]
    \coordinate (A1) at (0, 0);
    \coordinate (A2) at (0, 0.1);
    \coordinate (A3) at (0.1, 0.1);
    \coordinate (A4) at (0.1, 0);
    \coordinate (B1) at (0.03, 0.03);
    \coordinate (B2) at (0.03, 0.13);
    \coordinate (B3) at (0.13, 0.13);
    \coordinate (B4) at (0.13, 0.03);

    \draw (A1) -- (A2);
    \draw (A2) -- (A3);
    \draw (A3) -- (A4);
    \draw (A4) -- (A1);
    \draw[densely dotted] (A1) -- (B1);
    \draw[densely dotted] (B1) -- (B2);
    \draw (A2) -- (B2);
    \draw (B2) -- (B3);
    \draw (A3) -- (B3);
    \draw (A4) -- (B4);
    \draw (B4) -- (B3);
    \draw[densely dotted] (B1) -- (B4);
\end{tikzpicture}}

\newcommand\tinycube{\begin{tikzpicture}[scale=1.3]
    \coordinate (A1) at (0, 0);
    \coordinate (A2) at (0, 0.1);
    \coordinate (A3) at (0.1, 0.1);
    \coordinate (A4) at (0.1, 0);
    \coordinate (B1) at (0.03, 0.03);
    \coordinate (B2) at (0.03, 0.13);
    \coordinate (B3) at (0.13, 0.13);
    \coordinate (B4) at (0.13, 0.03);

    \draw (A1) -- (A2);
    \draw (A2) -- (A3);
    \draw (A3) -- (A4);
    \draw (A4) -- (A1);
    \draw[densely dotted] (A1) -- (B1);
    \draw[densely dotted] (B1) -- (B2);
    \draw (A2) -- (B2);
    \draw (B2) -- (B3);
    \draw (A3) -- (B3);
    \draw (A4) -- (B4);
    \draw (B4) -- (B3);
    \draw[densely dotted] (B1) -- (B4);
\end{tikzpicture}}

\begin{document}

\author{Chantal David}
\address{Department of Mathematics and Statistics, Concordia University, 1455 de Maisonneuve West, Montr\'{e}al, Qu\'{e}bec, Canada H3G 1M8}
\email{chantal.david@concordia.ca}

\author{Alexandre de Faveri}
\address{EPFL SB MATH, Station 10, 1015 Lausanne, Switzerland}
\email{alexandre.defaveri@epfl.ch}

\author{Alexander Dunn}
\address{School of Mathematics, Georgia Institute of Technology, Atlanta, USA}
\email{adunn61@gatech.edu}

\author{Joshua Stucky}
\address{School of Mathematics, Georgia Institute of Technology, Atlanta, USA}
\email{jstucky3@gatech.edu}


\title{Non-vanishing for cubic Hecke $L$-functions}


\begin{abstract}
Let $\omega$ be a primitive cubic root of unity. 
We study the non-vanishing problem for the family of Hecke $L$-functions associated to primitive cubic characters defined over the Eisenstein quadratic number field $\mathbb{Q}(\omega)$. We prove unconditionally that a positive proportion of Hecke $L$-functions associated to the cubic residue symbols $\chi_q$ with $q \in \mathbb{Z}[\omega]$ squarefree and $q \equiv 1 \pmod{9}$ do not vanish at the central point.  
    
Our proof goes through the method of first and second mollified moments. The principal new contribution of this paper is the asymptotic evaluation of the mollified second moment with power saving error term. No asymptotic formula for the mollified second moment of a cubic family was known (even over function fields) prior to the writing of this paper. Our new approach makes crucial use of Patterson's evaluation of the Fourier coefficients of the cubic metaplectic theta function, Heath-Brown's cubic large sieve, and a Lindel\"{o}f-on-average upper bound for the second moment of cubic Dirichlet series that we establish.
    
The significance of our result is that the (unitary) family considered does \emph{not} satisfy a perfectly orthogonal large sieve bound. This is quite unlike other families of Dirichlet $L$-functions in the literature for which unconditional results are known: the symplectic family of quadratic characters and the unitary family of all Dirichlet characters $\chi \pmod{q}$. Consequently, our proof has fundamentally different features from the corresponding works of Soundararajan and of Iwaniec and Sarnak.
\end{abstract}


\maketitle
\tableofcontents


\section{Introduction}

\subsection{Statement of non-vanishing result}

A famous (open) folklore conjecture of Chowla predicts that $L(1/2,\chi) \neq 0$ for all 
primitive Dirichlet characters $\chi \pmod q$. The conjecture appears to have first been made in the case
when $\chi$ is a primitive quadratic Dirichlet character (over $\mathbb{Q}$) \cite{Chow}. The quadratic 
case has been extensively studied.
\"{O}zl\"{u}k and Snyder \cite{OzSny} 
 showed under the GRH that $L\big(1/2,\big(\frac{d}{\cdot} \big) \big) \neq 0$ for at least $15/16$ of the fundamental 
 discriminants $|d| \leq X$. Their proof made use of the
 one-level density for low-lying zeros of the family.
 The conjectures of Katz and Sarnak \cite{KaSa}
 imply that $L \big(1/2, \big( \frac{d}{\cdot} \big) \big ) \neq 0$ for almost 
 all fundamental discriminants $d$.
  In a breakthrough
paper, Soundararajan \cite{Sou} proved unconditionally 
 that  $L\big(1/2,\big(\frac{8d}{\cdot} \big) \big) \neq 0$ for at least $7/8$ of the fundamental 
 discriminants $8d$ with $d$ odd and $0<8d \leq X$. The method of \cite{Sou}
 made use of both the first and the second \emph{mollified moments} -- it is well known that the first two un-mollified moments are insufficient to
 obtain a positive proportion of non-vanishing, see \cite[Conjecture~1.5.3]{CFKRS}
 and \cite{GH,Jut,VT}. In the same paper, Soundararajan also established an asymptotic formula 
for the third moment of quadratic Dirichlet $L$-functions. Independent to \cite{Sou},
Diaconu, Goldfeld, and Hoffstein \cite{DGH} used multiple Dirichlet series methods
to establish an asymptotic formula for the third moment, and also conjectured
the presence of a second order main term. Their conjecture was later established by Diaconu and Whitehead \cite{DW}.
Following the ideas of the breakthrough work of Li \cite{XLi} on the second moment for 
twists of modular $L$-functions,
the fourth author and Shen \cite{SS} have unconditionally established an asymptotic formula 
for the fourth moment of primitive quadratic Dirichlet $L$-functions. This improved upon an earlier conditional result (under GRH) of Shen \cite{Shen}. 

Let $\mathbb{Q}(\omega)$ be the Eisenstein quadratic number field,
where $\omega$ is identified with $e^{2\pi i/3} \in \mathbb{C}$.
This number field has ring of integers $\mathbb{Z}[\omega]$,
discriminant $-3$,
finite unit group $\langle - \omega \rangle $,
and class number $1$.
Let 
$N(x):=N_{\mathbb{Q}(\omega)/\mathbb{Q}}(x)=|x|^2$
denote the norm form on $\mathbb{Q}(\omega)/\mathbb{Q}$. 
 The subject of this paper is the non-vanishing problem for cubic Hecke $L$-functions (over $\mathbb{Q}(\omega)$) at the central point.
When compared to the quadratic case, the non-vanishing results for cubic and higher order characters (in various different families)
up until this paper have been more restrictive because of the presence of Gauss sums 
(as the values of Gauss sums are equidistributed on the unit circle when their order is strictly greater than two).
Previous results on cubic and higher order characters (over number fields) have fallen into one of two categories: 
\begin{itemize}
\item  Unconditional and prove that the $L$-value at the central point does not vanish for a density zero (infinite) subset of the family \cite{BY,BGL,Luo}. 
\item  Conditional on GRH and prove that the $L$-value at the central point does not vanish for a positive proportion of the family \cite{DG,GY,GZ2}.
\end{itemize}
In this paper we prove unconditionally that a positive proportion of Hecke $L$-functions $L(1/2,\chi_q)$ associated to 
the cubic residue symbols $\chi_q$ with $q \in \mathbb{Z}[\omega]$ squarefree and $q \equiv 1 \pmod{9}$
do not vanish.

\begin{theorem} \label{thm:main}
    For at least $14 \%$ of $q \in \mathbb{Z}[\omega]$ squarefree with $q \equiv 1 \pmod{9}$, we have $L(1/2,\chi_q) \neq 0$. More precisely, for any fixed $\varepsilon>0$ and sufficiently large $X>0$,
    \begin{equation*}
        \sum_{\substack{q \in \mathbb{Z}[\omega] \\ q \equiv 1 \pmod{9} \\ N(q) \leq X  \\ L(1/2,\chi_q) \neq 0}} \mu^2(q)  \geq \Big(\frac{1}{7} - \varepsilon\Big) \sum_{\substack{q \in \mathbb{Z}[\omega] \\ q \equiv 1 \pmod{9} \\ N(q) \leq X}} \mu^2(q).
    \end{equation*}
\end{theorem}

In order to enable further discussion, we introduce 
some notation. 
The unique ramified prime in $\mathbb{Z}[\omega]$ is $\lambda:=1+2 \omega$.
For $a \in \Z[\omega]$ and $\pi \equiv 1 \pmod{3}$ prime, the cubic Jacobi symbol
is defined by 
\begin{equation} \label{jacobidefine}
\Big ( \frac{a}{\pi} \Big )_{3} \equiv a^{\frac{N(\pi) - 1}{3}} \pmod{\pi}
\end{equation}
and the condition it take values in $\{0,1,\omega,\omega^2 \}$.
The cubic symbol is clearly multiplicative in $a$ and can be extended multiplicatively to all $b \equiv 1 \pmod{3}$ by setting
\begin{equation} \label{multdefine}
\Big ( \frac{a}{b} \Big )_{3} := \prod_{i} \Big ( \frac{a}{\pi_i} \Big )_{3}.
\end{equation}
Here $b = \prod_{i} \pi_i$, with prime factors $\pi_i$ counted with multiplicity and chosen so that $\pi_i \equiv 1 \pmod{3}$; see \cref{sec:prelim}.
Such a factorization exists because the residual unit is congruent to $1 \pmod{3}$, hence equals $1$.

For $q \in \mathbb{Z}[\omega]$ with $q \equiv 1 \pmod{3}$, the cubic Dirichlet character 
\begin{equation} \label{chiqdef}
\chi_q(\alpha):=\Big( \frac{\alpha}{q} \Big)_3, \quad \alpha \in \mathbb{Z}[\omega],
\end{equation}
on $\mathbb{Z}[\omega]/q \mathbb{Z}[\omega]$
is a (cubic) Hecke character if $\chi_q(\omega)=1$ (i.e.\ trivial on units
of $\mathbb{Z}[\omega]$).  A given cubic Dirichlet character $\chi$ (as in \eqref{chiqdef}) is primitive 
if and only if $\chi=\chi_q$ for some $q \in \mathcal{C}_3$, 
where 
\begin{align} \label{C3def}
\mathcal{C}_3:= \big \{ q_1 q_2^2 : q_1, q_2 \in \mathbb{Z}[\omega], &\ q_1,q_2 \equiv 1 \pmod{3}, \nonumber \\ 
&\ \mu^2(q_1 q_2)=1, \text{ and } N(q_1 q_2^2) \equiv 1  \pmod{9} \big \}.
\end{align}
Each $\chi=\chi_{q_1 q^2_2}$ with $q_1 q_2^2 \in \mathcal{C}_3$ has conductor $\mathfrak{c_{\chi}}:=\text{cond}(\chi)=q_1 q_2 \mathbb{Z}[\omega]$. 
It is natural and technically convenient (cf.\ \cref{supprem}) to restrict to a family of Hecke characters 
$\chi=\chi_{q_1 q^2_2} $ such that $q_1 q^2_2 \in \mathcal{C}_3$ belongs to a fixed arithmetic progression modulo $9$, namely
\begin{align}  \label{F3def}
\mathcal{F}_3:= \big \{ 1 \neq q_1 q^2_2 \in \mathcal{C}_3 : &\ q_1,q_2 \equiv 1 \pmod{3}, \nonumber  \\ 
&\ \mu^2(q_1 q_2)=1, \text{ and } q_1 q_2^2 \equiv 1   \pmod{9} \big \},
\end{align}
and to a subfamily
\begin{equation} \label{F3primedef}
\mathcal{F}^{\prime}_3:= \big \{ q \in \mathcal{F}_3 : \mu^2(q)=1 \text{ and } q \equiv 1 \pmod{9} \big \} \subset \mathcal{F}_3.
\end{equation}
\begin{remark} \label{supprem}
Note that $\chi_q(\omega)=\chi_q(\lambda)=1$ for each $q \in \mathcal{F}_3$ by \eqref{cubesupp}.
\end{remark}
For $X \geq 1/2$, let
\begin{equation} \label{eq:F_3_def}
\mathcal{F}_3(X):= \big \{ q_1 q^2_2 \in \mathcal{F}_3: N(\text{cond}(\chi_{q_1 q^2_2})) \leq X \big \}
\end{equation}
and
\begin{equation*}
\mathcal{F}^{\prime}_3(X):= \big \{ q \in \mathcal{F}^{\prime}_3: N(\text{cond}(\chi_{q})) \leq X \big \}.
\end{equation*}

For primitive cubic characters over $\mathbb{Q}$, the relevant Gauss sums are not defined 
over the ground field. For primitive cubic characters of $\mathbb{Q}(\omega)$, the full family $\mathcal{F}_3$ allows both factors $q_1,q_2$ in \eqref{F3def} to vary, and satisfies $|\mathcal{F}_3(X)| \sim C_1 X \log X$ for some $C_1>0$.
Obtaining an asymptotic formula with a power saving error term for its second moment presents additional difficulties, discussed in \cref{sec:results}.
We restrict to the thin subfamily $\mathcal{F}^{\prime}_3$ given by $q_2=1$, for which $|\mathcal{F}^{\prime}_3(X)| \sim C_2 X$ for some $C_2>0$.
This subfamily has received considerable attention in the literature and is the subject of \cref{thm:main}.
Luo \cite{Luo} established an asymptotic formula for the first moment, as well as
an upper bound for the second moment, and proved that
$L(1/2,\chi_q) \neq 0$ for $q$ belonging to a density zero (infinite) subset of $\mathcal{F}_3^{\prime}$.

Friedberg, Hoffstein, and Lieman \cite{FHL} used multiple Dirichlet series to compute the first moment
(with certain arithmetic weights) of Hecke $L$-functions 
attached to the family of $n$-th order characters 
over a number field 
containing $\mathbb{Q}(\zeta_n)$.
Diaconu \cite{Dia} used the Rankin--Selberg convolution of metaplectic Eisenstein series
on the $n$-fold cover of $\operatorname{GL}_2$ to construct multiple Dirichlet series whose
Dirichlet coefficients are the absolute value squared of twists of Hecke $L$-functions.
Diaconu used this construction
to unconditionally establish the second moment (with certain arithmetic weights) 
of Hecke $L$-functions
attached to the family of $n$-th order characters 
over a number field containing $\mathbb{Q}(\zeta_n)$. We point out that these moments 
(similar to \cite{FHL}) are \emph{not} sieved down to 
squarefree ideals, unlike our \cref{secondmomentcor}.
Interestingly, Diaconu conjectured \cite[Conjecture~4.5]{Dia} the existence of  a second order main term
$X^{1/2+1/n} Q(\log X)$ in the asymptotics for the second moment
over squarefree ideals, where $Q(x)$ is a linear polynomial depending on $n$ and the number field.

These results were further developed into density zero (infinite) non-vanishing results
by Blomer, Goldmakher and Louvel \cite{BGL}, where they proved 
a large sieve inequality for the thin family of $n$-th order Hecke characters.
Using one-level density, the first author and G\"{u}lo\u{g}lu \cite{DG} proved under GRH that
$L(1/2, \chi_q) \neq 0$ for a positive proportion (namely 2/13) of $q \in \mathcal{F}^{\prime}_3$.

The full family $\mathcal{F}_3$ has been studied by G\"{u}lo\u{g}lu and Yesilyurt in \cite{GY}.
They proved an unconditional asymptotic formula for the first mollified moment, upper bounds 
for all mollified higher moments under GRH, and consequently established a
conditional positive proportion (around $\exp(-e^{101})$) of non-vanishing for Hecke $L$-functions in the family $\mathcal{F}_3$. 
In a follow-up work,
G\"{u}lo\u{g}lu \cite{Gu} proved unconditionally that 
$L(1/2,\chi_q) \neq 0$ for $q$ belonging to a density zero (infinite) subset of $\mathcal{F}_3$.

Baier and Young \cite{BY} computed the first moment of Dirichlet $L$-functions attached to
primitive cubic characters over $\mathbb{Q}$, and as a consequence showed non-vanishing of this family for a density zero (infinite) subset. Under GRH, Gao and Zhao \cite{GZ2} computed the order of magnitude of all mollified $k$-th moments for the same family of cubic characters over $\Q$, and (conditionally) obtained a positive proportion of non-vanishing.

Much more is known unconditionally about moments and non-vanishing for $L$-functions over function fields, since the corresponding Riemann Hypothesis was proved by Deligne \cite{De74, De80}.
Florea computed the first four moments of quadratic Dirichlet $L$-functions over function fields
\cite{Flo1,Flo2,Flo3}, exhibiting a secondary term below square-root cancellation for the first moment. Using the one-level density, Bui and Florea \cite{BF} proved
that $94 \%$ of quadratic Dirichlet $L$-functions over function fields do not vanish at the central point.
More recently, Koymans, Pagano, and Shusterman \cite[Corollary~1.2]{KPS} proved that, for every fixed prime power $q \equiv 1 \pmod{8}$, $100\%$ of the imaginary quadratic Dirichlet $L$-functions over $\mathbb{F}_q(T)$ do not vanish at the central point, as the genus tends to infinity.
The first author together with Florea and Lal\'{i}n \cite{DFL} proved that a 
positive proportion of cubic Dirichlet $L$-functions over function fields do not
vanish at the central point in the non-Kummer case (when the ground field does not contain a third root of unity).
Ellenberg, Li, and Shusterman \cite{ELS} used algebraic geometry techniques to show that 
Dirichlet $L$-functions associated with $\ell$-th order twists 
do not vanish for a density zero (infinite) subset of the family (this extended and improved an earlier paper of 
the first author, Florea, and Lal\'{i}n \cite{DFL2} that concerned the cubic case). 
Subsequent to the writing of this paper,
the first author together with Florea and Lal\'{i}n proved that a positive proportion of 
Dirichlet $L$-functions over function fields attached to $\ell$-th order twists do not vanish at the central point
\cite{DFL25}. Also, subsequent to the writing of this paper, Goel and Ray \cite[Theorem~A]{GR} asymptotically computed the second moment
of cubic Dirichlet $L$-functions over the rational function field $\mathbb{F}_q(T)$, where $q$ is an odd power of a prime
satisfying $q \equiv 2 \pmod{3}$. Their asymptotic has no power saving in the conductor and does not directly yield an asymptotic formula for the mollified second moment.
In a recent preprint that uses a novel approach from algebraic topology, Bergstr\"{o}m, Diaconu, Petersen, 
and Westerland \cite{BDPW} have established unconditional asymptotic formulae for all moments of 
quadratic Dirichlet $L$-functions over function fields at the central point. This paper makes striking progress toward a conjecture 
of Conrey, Farmer, Keating, Rubinstein, and Snaith \cite{CFKRS}. We also point out that the regime for the asymptotics in
\cite{BDPW} is when the size of the ground field 
is sufficiently large in terms of the exponent in the moment. 

Upper bounds for all moments of $\zeta(1/2+it)$ were established under GRH 
in a seminal paper of Soundararajan \cite{Sou2}, and later refined by Harper \cite{Har}.
Lester and Radziwi{\l\l} further developed these ideas to provide 
upper bounds on all mollified moments for $L$-functions attached to quadratic twists of modular forms.
We would like to highlight that a substantial part of the recent literature on mollified moments and non-vanishing 
results for $L$-functions associated 
with cubic characters, in both the number field case \cite{GY,GZ2} and the function field case \cite{DFL},
are owed to the circle of ideas in \cite{Sou2,Har,LR}. In this paper we adopt a new line of attack,
see \cref{hlsketch} for more details.

In a different direction,
 Balasubramanian and Murty \cite{BalMu} proved that 
 $L(1/2,\chi) \neq 0 $ for a small proportion of primitive Dirichlet characters $\chi \pmod{q}$ (a unitary family),
 with $q \in \mathbb{Z}$ a large prime modulus.
 Iwaniec and Sarnak \cite{IS} improved this proportion to $1/3$ using the method of mollified moments.
 The proportion of non-vanishing for this family was later improved by Khan, Mili\'{c}evi\'{c}, and Ngo \cite{KMN}. 
Heath-Brown \cite{HB3} was the first to prove an asymptotic for the fourth moment (valid for a density one set of integers $q$),
later extended by Soundararajan \cite{Sou3} to all integers. 
 In an important work, Young \cite{You} in 2011 further sharpened the fourth moment asymptotic to have an error term
 that was a power saving over the main term for a prime modulus $q$.
 The power saving error term established by Young was later improved by Blomer et al \cite{BFKMM}.
 The asymptotic for this family was later extended to a general modulus
 $q$ by Wu \cite{Wu}.

\subsection{Statement of results on moments}\label{sec:results}
We now outline our new results on moments.
We follow the notation and set-up in \cite[\S 1]{Sou}, suitably adapted, and with different normalizations.
 
For $q \in \mathbb{Z}[\omega]$ with $q \equiv 1 \pmod{9}$,
consider the mollifier 
\begin{equation} \label{mollifier}
\mathcal{M}(q):=\sum_{\substack{0 \neq \mathfrak{b} \unlhd \Z[\omega] \\ N(\mathfrak{b}) \leq M}} \lambda(\mathfrak{b}) \sqrt{N(\mathfrak{b})} \Big( \frac{\mathfrak{b}}{q} \Big)_3,
\end{equation}
for a $\C$-valued sequence $\boldsymbol{\lambda}:=(\lambda(\mathfrak{b}))_{0 \neq \mathfrak{b} \unlhd \mathbb{Z}[\omega]}$
to be chosen later, and supported only on squarefree ideals coprime with $3$. We also assume that $\lambda(\mathfrak{b}) \ll_{\varepsilon} N(\mathfrak{b})^{-1+\varepsilon}$.
The notation $\big (\frac{\mathfrak{b}}{q} \big)_3$ is well-defined since $q \equiv 1 \pmod{9}$, see \cref{supprem}.

Let $\boldsymbol{\beta}$ be any $\mathbb{C}$-valued sequence and $f$ denote a Schwartz function compactly supported in the interval $(1,2)$.  
Define 
\begin{equation}
\mathcal{S}(\boldsymbol{\beta};f) = \mathcal{S}_X(\boldsymbol{\beta};f):=\sum_{\substack{q \in \mathbb{Z}[\omega] \\ q \equiv 1 \pmod{9} }} 
\mu^2(q) \beta_q f \Big( \frac{N(q)}{X} \Big).
\end{equation}
Let $Y>1$ be a parameter to be chosen later, and write
\begin{equation} \label{mudecomp}
\mu^2(q)=M_Y(q)+R_Y(q),
\end{equation}
where 
\begin{equation} \label{MYRYdef}
M_Y(q):=\sum_{\substack{ \mathfrak{l}^2 \mid q \mathbb{Z}[\omega] \\ N(\mathfrak{l}) \leq Y  }} \mu(\mathfrak{l}) \quad \text{and} 
\quad R_Y(q):=\sum_{\substack{ \mathfrak{l}^2 \mid q \mathbb{Z}[\omega] \\ N(\mathfrak{l}) > Y  }} \mu(\mathfrak{l}).
\end{equation}
Define 
\begin{align} \label{SMdef}
\mathcal{S}_M(\boldsymbol{\beta};f)=\mathcal{S}_{M,X,Y}(\boldsymbol{\beta};f)
:=\sum_{\substack{q \in \mathbb{Z}[\omega] \\ q \equiv 1 \pmod{9}}} M_Y(q) \beta_q f \Big( \frac{N(q)}{X} \Big),
\end{align}
and 
\begin{equation} \label{SRdef}
\mathcal{S}_R(\boldsymbol{\beta};f)=\mathcal{S}_{R,X,Y}(\boldsymbol{\beta};f)
:=\sum_{\substack{q \in \mathbb{Z}[\omega] \\ q \equiv 1 \pmod{9}}}  |R_Y(q) \beta_q| f \Big( \frac{N(q)}{X} \Big).
\end{equation}
Hence, if $f$ is non-negative, then
\begin{equation} \label{nonnegproperty}
\mathcal{S}(\boldsymbol{\beta};f)=\mathcal{S}_{M}(\boldsymbol{\beta};f)+
O(\mathcal{S}_R(\boldsymbol{\beta};f)).
\end{equation}

In this notation, we seek to evaluate the mollified moments 
\begin{equation*}
\mathcal{S}(L(1/2, \chi_q) \mathcal{M}(q); F) \qquad \text{and} \qquad \mathcal{S}(|L(1/2,\chi_q) \mathcal{M}(q)|^2; F),
\end{equation*}
where $F$ is a non-negative Schwartz function compactly supported on $(1,2)$ and satisfying $0 \leq F(t) \leq 1$ for all $t \in \R$. In view of the 
approximate functional equations in \cref{afecenter}, we need asymptotic formulae for 
\begin{equation} \label{SMstate}
\mathcal{S}_M(\mathcal{M}(q)[A_1(q)+\widetilde{g}_3(q) \cdot \overline{A_1(q)}]; F) \quad \text{and} \quad \mathcal{S}_M( |\mathcal{M}(q)|^2 A_2(q); F),
\end{equation}
and (since $|\widetilde{g}_3(q)| \leq 1$) good estimates for 
\begin{equation*}
\mathcal{S}_R(|\mathcal{M}(q)A_1(q)|; F) \quad \text{and} \quad \mathcal{S}_R( |\mathcal{M}(q)^2 A_2(q)|; F),
\end{equation*}
where $A_1(q)$ and $A_2(q)$ are given in \eqref{A1def} and \eqref{A2def} respectively, and $\widetilde{g}_3(q)$
is the normalized cubic Gauss sum (see \eqref{generalgauss} and \eqref{normalized}).

\begin{prop} \label{SRestimate1}
    Let $\varepsilon > 0$ and $1 \leq Y \leq X^{100}$. Assume that $1 \leq M \leq X^{100}$ in \eqref{mollifier}. Then 
    \begin{align*}
        \mathcal{S}_R(|\mathcal{M}(q) A_1(q)|; F) \ll_{\varepsilon} X^\varepsilon \Big( \frac{X}{Y^{5/6}} + X^{13/18} + \frac{X^{5/6}M^{1/3}}{Y^{1/2}} + X^{2/3} M^{1/3} + X^{7/12} M^{1/2} \Big).
    \end{align*}
\end{prop}

\begin{prop} \label{SRestimate2}
    Let $\varepsilon > 0$ and $1 \leq Y \leq X^{100}$. Assume that $1\leq M \leq X^{100}$ in \eqref{mollifier}. Then 
    \begin{align*}
        \mathcal{S}_R(|\mathcal{M}(q)^2 A_2(q)|; F) \ll_{\varepsilon} X^{\varepsilon} \Big( X \Big(\frac{M}{Y}\Big)^{2/3} + X^{5/6} M \Big).
    \end{align*}
\end{prop}

The analogous result in \cite{Sou} to our \cref{SRestimate2} 
makes use of the Cauchy--Schwarz inequality, and an optimal 
bound for the fourth moment of quadratic Dirichlet $L$-functions.
The latter bound is owed to a perfectly orthogonal large sieve
bound for primitive quadratic characters due to Heath-Brown \cite{HB2}.
The third author and Radziwi{\l\l} \cite{DR} proved (under GRH)
that the cubic large sieve is \emph{not} perfectly orthogonal,
so an optimal fourth moment bound is not available in the cubic case.
Subsequent to the writing of this paper, the second and third authors, together with Hoffstein, established this non-orthogonality unconditionally \cite[Theorem~1.1]{DDH}.
Instead, we opt to take the supremum norm over the mollifier,
and use an upper bound for the second moment of Hecke $L$-functions indexed by $\mathcal{F}_3$ 
(not the thin starting family $\mathcal{F}^{\prime}_3$).
These arguments appear in
Sections \ref{second_moment_section} and \ref{SRproofsection}.

We now turn our attention to the quantities in \eqref{SMstate}.
In order to compute them, consider
\begin{align} \label{SMexpand}
& \mathcal{S}_M(\mathcal{M}(q) [A_1(q)+\widetilde{g}_3(q) \cdot \overline{A_1(q)}];F) \nonumber \\
& =\sum_{\substack{ 0 \neq  \mathfrak{b} \unlhd \mathbb{Z}[\omega] \\ N(\mathfrak{b}) \leq M  \\ 
(\mathfrak{b},3)=1}} \lambda(\mathfrak{b}) \sqrt{N(\mathfrak{b} )}
\mathcal{S}_M \Big(\Big( \frac{\mathfrak{b}}{q} \Big)_3 \Big[A_1(q)+\widetilde{g}_3(q) \cdot \overline{A_1(q)}\Big]; F \Big),
\end{align}
and
\begin{align} \label{SMexpand2}
& \mathcal{S}_M(|\mathcal{M}(q)|^2 A_2(q);F) \nonumber \\
&=\mathop{\sum \sum}_{\substack{ 0 \neq  \mathfrak{b}_1,\mathfrak{b}_2 \unlhd \mathbb{Z}[\omega] \\ N(\mathfrak{b}_1), 
N(\mathfrak{b}_2) \leq M  \\ (\mathfrak{b}_1 \mathfrak{b}_2,3)=1}} \lambda(\mathfrak{b}_1) \overline{ \lambda(\mathfrak{b}_2)} \sqrt{N(\mathfrak{b}_1 \mathfrak{b}_2 )}
\mathcal{S}_M \Big( \Big( \frac{\mathfrak{b}_1}{q} \Big)_3 \overline{\Big( \frac{\mathfrak{b}_2}{q} \Big)_3} A_2(q); F \Big).
\end{align}
Thus it suffices to compute respectively
\begin{equation} \label{general}
\mathcal{S}_M \Big( \Big( \frac{\mathfrak{b}}{q} \Big)_3 \Big[A_1(q)+\widetilde{g}_3(q) \cdot \overline{A_1(q)}\Big] ; F \Big),
\quad \text{ for } 0 \neq \mathfrak{b} \unlhd \mathbb{Z}[\omega],
\end{equation}
and
\begin{equation} \label{general2}
\mathcal{S}_M \Big( \Big( \frac{\mathfrak{b}_1}{q} \Big)_3 \overline{\Big( \frac{\mathfrak{b}_2}{q} \Big)_3} A_2(q); F \Big),
\quad \text{ for } 0 \neq \mathfrak{b}_1,\mathfrak{b}_2 \unlhd \mathbb{Z}[\omega],
\end{equation}
where $\mathfrak{b}$, $\mathfrak{b}_1$, and $\mathfrak{b}_2$ are squarefree and coprime with $3$.

In order to state our main asymptotic formulas, denote
\begin{equation*}
    \check{F}(w) := \int_0^\infty F(t) t^w dt.
\end{equation*}
Our work leads to the following result on the first moment.

\begin{prop} \label{SMestimate1}
    Let $0 \neq \mathfrak{b} \unlhd \Z[\omega]$ be squarefree and coprime with $3$. Assume that $1\leq Y \leq X^{100}$ and $N(\mathfrak{b}) Y^2 \leq X^{1/2-\nu}$ for some fixed $\nu>0$. Then
    \begin{align}
        \mathcal{S}_M \Big( \Big( \frac{\mathfrak{b}}{q} \Big)_3 \Big[A_1(q) + \widetilde{g}_3(q) \cdot \overline{A_1(q)}\Big]; F \Big) = &\ C X \check{F}(0) \frac{r(\mathfrak{b})}{N(\mathfrak{b})} + O \Big( \frac{X}{Y N(\mathfrak{b})}\Big) \nonumber \\
        &\ + O_{F, \nu, \varepsilon} \Big( X^\varepsilon \Big(X^{3/4} N(\mathfrak{b})^{1/2} + \frac{X^{5/6}}{N(\mathfrak{b})^{1/2}} + \frac{X^{11/12}}{N(\mathfrak{b})^{2/3}} \Big) \Big)\nonumber
    \end{align}
    for any $\varepsilon>0$, where 
    \begin{align}
        C := \frac{\pi}{36(\sqrt{3}-1) \cdot \zeta_{\Q(\omega)}(2)} \prod_{\substack{\mathfrak{p} \text{ prime} \\ (\mathfrak{p}, 3) = 1 \\ q\; :=\; N(\mathfrak{p})}} \Big( 1 + \frac{q}{(q+1)(q^{3/2}-1)} \Big) \label{Cdef}
    \end{align}
    and $r$ is the multiplicative function given, for $\mathfrak{p}$ prime, $k\geq 1$, and $q := N(\mathfrak{p})$, by 
    \begin{equation}
        r(\mathfrak{p}^k) := \frac{q^{5/2}}{q^{5/2} + q^{3/2} - 1} = 1 + O\Big(\frac{1}{N(\mathfrak{p})}\Big). \label{r_prime_def}
    \end{equation}
\end{prop}

The proof of \cref{SMestimate1} is contained in \cref{first_moment_asymptotics_section},
and the result is used in \cref{mollifier_section} to obtain the first mollified moment. 

Furthermore, we have the following result for the second moment.

\begin{prop} \label{SMestimate2}
    Let $0 \neq \mathfrak{b}_1, \mathfrak{b}_2 \unlhd \Z[\omega]$ be squarefree and coprime with $3$. Assume that $1\leq Y, N(\mathfrak{b}_1\mathfrak{b}_2) \leq X^{100}$ and $\varepsilon>0$. Then denoting $\mathfrak{b} = (\mathfrak{b}_1, \mathfrak{b}_2)$ and $\mathfrak{b}_1 \mathfrak{b}_2 = \mathfrak{a}\mathfrak{b}^2$, we have
    \begin{align*}
        \mathcal{S}_M \Big( \Big( \frac{\mathfrak{b}_1}{q} \Big)_3 \overline{\Big( \frac{\mathfrak{b}_2}{q} \Big)_3} A_2(q); F \Big) = D \check{F}(0) X \frac{h(\mathfrak{a}) g(\mathfrak{b})}{\sqrt{N(\mathfrak{a})}} \Big[ \log{\Big(\frac{X}{N(\mathfrak{a})}\Big)} + \mathcal{O}(\mathfrak{b}_1, \mathfrak{b}_2) \Big] & \\
        +\ O_\varepsilon\Big(\frac{X^{1+\varepsilon}}{Y} + \frac{X^{5/6+\varepsilon}}{N(\mathfrak{a})^{1/3}}\Big) + \mathcal{R}(\mathfrak{b}_1, \mathfrak{b}_2)&.
    \end{align*}
    Here
    \begin{align}
        D :=  &\ \frac{\pi^2}{648(2-\sqrt{3}) \cdot \zeta_{\Q(\omega)}(2)}  \prod_{\substack{\mathfrak{p} \text{ prime} \\ (\mathfrak{p}, 3) = 1 \\ q\; :=\; N(\mathfrak{p})}} \Big( 1 - \frac{1}{q(q+1)} + \frac{2q}{(q+1) (q^{3/2}-1)} \Big), \label{Ddef}
    \end{align}
    and the multiplicative functions $g$ and $h$ are defined, for $\mathfrak{p}$ prime, $k\geq 1$, and $q := N(\mathfrak{p})$, by
    \begin{align}
        g(\mathfrak{p}^k) := 1 - \frac{(q^{3/2} - 1)(q - 1)}{q^{7/2} + q^{5/2} + q^{2} - q^{3/2} - q + 1} = 1 + O\Big(\frac{1}{N(\mathfrak{p})}\Big) \label{g_prime_def}
    \end{align}
    and
    \begin{align}
        h(\mathfrak{p}^k) := 1 + \frac{(q^2 - q^{3/2} + 1)(q - 1)}{q^{7/2} + q^{5/2} + q^{2} - q^{3/2} - q + 1} = 1 + O\Big(\frac{1}{\sqrt{N(\mathfrak{p})}}\Big). \label{h_prime_def}
    \end{align}
    Moreover
    \begin{equation*}
        \mathcal{O}(b_1, b_2) := C_0 + \sum_{\substack{\mathfrak{p} \emph{ prime} \\  \mathfrak{p} \mid \mathfrak{b}}} D_{1}(\mathfrak{p}) \frac{\log{N(\mathfrak{p})}}{N(\mathfrak{p})} + \sum_{\substack{\mathfrak{p} \emph{ prime} \\ \mathfrak{p} \mid \mathfrak{a}}} D_{2}(\mathfrak{p}) \frac{\log{N(\mathfrak{p})}}{\sqrt{N(\mathfrak{p})}},
    \end{equation*}
    where $C_0 = C_1 + C_2 \frac{\check{F}'}{\check{F}}(0)$ for some absolute constants $C_1$ and $C_2$, and $D_i(\mathfrak{p}) \ll 1$ for $i \in\{1, 2\}$. Furthermore, the error term $\mathcal{R}(\mathfrak{b}_1, \mathfrak{b}_2)$ satisfies
    \begin{align}\label{R_average_bound}
        \mathop{\sum\sum}_{\substack{0 \neq \mathfrak{b}_1, \mathfrak{b}_2 \unlhd \Z[\omega] \\ N(\mathfrak{b}_1) \sim B_1,\ N(\mathfrak{b}_2) \sim B_2 \\ (\mathfrak{b}_1 \mathfrak{b}_2, 3)=1}} \mu^2(\mathfrak{b}_1)&\mu^2(\mathfrak{b}_2)|\mathcal{R}(\mathfrak{b}_1, \mathfrak{b}_2)| \nonumber\\
        &\ll_{F,\varepsilon} X^{1/2+\varepsilon} (B_1B_2)^{1/2} (X^{1/6} Y^{1/2}B^{3/2} + Y B^2 + X^{1/3}B ) 
    \end{align}
    for any $\frac{1}{2}\leq B_1, B_2 \leq X^{100}$ and $B:= \max\{B_1, B_2\}$.
\end{prop}

The proof of \cref{SMestimate2} is contained in \cref{second_main_sum_section}, and the result is used in \cref{mollifier_section} to obtain the second mollified moment. 

We also immediately obtain an asymptotic for the smoothed un-mollified second moment with error term $O_{F, \varepsilon}(X^{5/6+\varepsilon})$.

\begin{corollary} \label{secondmomentcor}
    Let $\varepsilon > 0$ and let $F$ be a Schwartz function with compact support on $(1,2)$ satisfying $0 \leq F(t) \leq 1$. Then
    \begin{equation*}
        \sum_{\substack{ q \in \mathbb{Z}[\omega] \\ q \equiv 1 \pmod{9} }} \mu^2(q) |L(1/2,\chi_q)|^2 F \Big( \frac{N(q)}{X} \Big) = 2 D \check{F}(0) X \Big (\log X + C_1 + C_2 \frac{\check{F}'}{\check{F}}(0) \Big) + O_{F, \varepsilon}(X^{5/6+\varepsilon})
    \end{equation*}
    as $X \rightarrow \infty$, where $D$ is the constant given in \eqref{Ddef}, and $C_1$ and $C_2$ are the absolute constants in \cref{SMestimate2}.
\end{corollary}

\begin{proof}
By \cref{afecenter}, $|L(1/2, \chi_q)|^2 = 2 A_2(q)$. Using \eqref{nonnegproperty} with $\beta_q = 2 A_2(q)$, the result follows directly from taking $M=1$, $\lambda(1)=1$ (cf.\ \eqref{mollifier}), and $Y=X^{1/3}$ in \cref{SRestimate2}, and then taking $\mathfrak{b}_1=\mathfrak{b}_2=1$, $B_1 = B_2 = B = \frac{1}{2}$, and $Y=X^{1/3}$ in \cref{SMestimate2}.
\end{proof}

Similarly, we also obtain an asymptotic for the smoothed un-mollified first moment.

\begin{corollary} \label{firstmomentcor}
    Let $\varepsilon > 0$ and let $F$ be a Schwartz function with compact support on $(1,2)$ satisfying $0 \leq F(t) \leq 1$. Then
    \begin{equation*}
        \sum_{\substack{ q \in \mathbb{Z}[\omega] \\ q \equiv 1 \pmod{9} }} \mu^2(q) L(1/2,\chi_q) F \Big( \frac{N(q)}{X} \Big) = C \check{F}(0) X+O_{F, \varepsilon}(X^{11/12+\varepsilon})
    \end{equation*}
    as $X \rightarrow \infty$, where $C$ is the constant given in \eqref{Cdef}.
    \end{corollary}

\begin{proof}
By \cref{afecenter}, $L(1/2, \chi_q) = A_1(q)+ \widetilde{g}_3(q) \cdot \overline{A_1(q)}$. Using \eqref{nonnegproperty} with 
$\beta_q =A_1(q)+ \widetilde{g}_3(q) \cdot \overline{A_1(q)}$, the result follows directly from taking $M=1$, $\lambda(1)=1$ (cf.\ \eqref{mollifier}), and $Y=X^{1/4-1/100}$ (say) in \cref{SRestimate1}, and then taking $\mathfrak{b}=1$ and $Y=X^{1/4-1/100}$ in \cref{SMestimate1}.
\end{proof}
      
With more effort we could make our work effective in the test function (similar to \cite{Sou}), and consequently 
obtain a version of \cref{secondmomentcor} with sharp cut-offs at the expense of having a worse power saving.
We refrain from this additional work.

Subsequent to the writing of this paper, Hamdar \cite{Ham} established an asymptotic for the first moment (with error term $O_\varepsilon(X^{4/5+\varepsilon})$) that captures a second order main term of size $X^{5/6}$, using an unbalanced approximate functional equation. 
An improvement in the error term for our first moment does not improve our proportion of non-vanishing, so we 
refrain from the additional work of using an unbalanced approximate functional equation.

More recently, the second author \cite[Theorem~1.3]{dF} obtained an asymptotic expansion for the smoothed second moment over the squarefree cubic family,
with a secondary term of the form $X^{5/6} Q(\log X)$ conjectured by Diaconu \cite[Conjecture~4.5]{Dia} and error term $O_\varepsilon(X^{3/4+\varepsilon})$,
where $Q$ is a polynomial of degree at most one. This result is formulated for $L$-functions with finitely many Euler factors removed.

Another subsequent development to the present work is that Castillo and the second and third authors \cite[Theorems~1.1 and~1.2]{CDF} established a positive proportion of non-vanishing for the squarefree family of quartic Hecke $L$-functions over $\mathbb{Q}(i)$, and used it to
show that the elliptic curves $y^2=x^3-qx$ have Mordell--Weil rank zero over $\mathbb{Q}(i)$ for a positive proportion of squarefree $q\in\mathbb{Z}[i]$ ordered by norm. 

Afterwards, Diaconu, Ion, Pa\c{s}ol, and Popa \cite[Theorem~C]{DIPP} proved a positive proportion of non-vanishing for families of Hecke characters of any order $r$ over number fields containing the $2r$-th roots of unity. They combine the large sieve machinery of the present paper with techniques from multiple Dirichlet series.

For the full cubic family $\mathcal{F}_3$ described in \eqref{F3def}, one must distinguish ordering by the norm $N(q_1q_2^2)$ from ordering by the conductor norm $N(q_1q_2)$ as in \eqref{eq:F_3_def}.
Subsequent to the writing of this paper, Diaconu, Ion, Pa\c{s}ol, and Popa \cite[Theorem~D]{DIPP} obtained a power saving asymptotic formula,
including a secondary term, for the second moment over the family $\mathcal{F}_3$ ordered by norm, with an additional weight equal to the number of divisors of $q_2$ and finitely many Euler factors removed.

The family $\mathcal{F}_3(X)$ ordered by conductor seems more challenging, since even a tight upper bound for its second moment requires going beyond the cubic large sieve. Very recently, the second author \cite[Theorem~1.2]{dF} obtained an $O_\varepsilon(X^{1+\varepsilon})$ upper bound for this second moment.
Obtaining a power saving asymptotic by this approach would require finer analysis of the off-diagonal contributions.

\begin{remark} \label{mol2}
    We now consider the maximal mollifier length $M$ for which we can prove an asymptotic formula for the second mollified moment $\mathcal{S}(|\mathcal{M}(q)^2 A_2(q)|; F)$. Applying \cref{SMestimate2} to the right side of \eqref{SMexpand2}, and recalling that $\lambda(\mathfrak{b}) \ll N(\mathfrak{b})^{-1+\varepsilon}$, we see that contribution from the error terms (we ignore the main terms) in \cref{SMestimate2} to $\mathcal{S}_M \big( |\mathcal{M}(q)|^2 A_2(q); F \big)$ is
    \begin{equation} \label{SM2error}
        \ll_{F, \varepsilon} X^{\varepsilon} \Big( \frac{MX}{Y}+ X^{1/2} (X^{1/6} Y^{1/2}M^{3/2} + Y M^2 + X^{1/3}M ) \Big).
    \end{equation} 
    \cref{SRestimate2} also gives 
    \begin{equation} \label{SRerror2}
        \mathcal{S}_R(|\mathcal{M}(q)^2 A_2(q)|; F) \ll_{F, \varepsilon} X^{\varepsilon} \Big( X \Big(\frac{M}{Y}\Big)^{2/3} + X^{5/6} M \Big).
    \end{equation}
    To obtain an asymptotic for $\mathcal{S}(|\mathcal{M}(q)^2 A_2(q)|; F)$, observe from \eqref{SRerror2} that we must have $Y \gg MX^{\delta_1}$ and $M \ll X^{1/6-\delta_2}$, for small fixed $\delta_1,\delta_2>0$. Taking $Y=MX^{\delta}$ and $M = X^{1/6-\delta}$ for some small fixed $\delta>0$ ensures that \eqref{SM2error} and \eqref{SRerror2} are genuine error terms, and one indeed has an asymptotic formula. The mollifier length $X^{1/6-\delta}$ is shorter than Soundararajan's $X^{1/2-\delta}$ in the quadratic case \emph{\cite{Sou}}.
\end{remark}

\begin{remark} \label{mol1}
Arguing analogously to \cref{mol2}, instead with \cref{SMestimate1} applied to the right side 
of \eqref{SMexpand}, and also considering \cref{SRestimate1}, one can check that the largest allowable mollifier length for the first moment 
$\mathcal{S}_M(\mathcal{M}(q) [A_1(q)+\widetilde{g}_3(q) \cdot \overline{A_1(q)}];F)$ is $M \ll X^{1/4-\delta}$ for $\delta>0$ small and fixed, and 
$Y = X^{\varepsilon}$. 

Thus we are ultimately limited to $M \ll X^{1/6-\delta}$ from the second moment in \cref{mol2}, and this is reflected in 
\cref{mollifier_section}.
\end{remark}

 
\section{High level sketch} \label{hlsketch}

As ``proof of concept'', we give a heuristic sketch of the argument giving the error term $O_{\varepsilon}(X^{5/6+\varepsilon})$ for the smoothed un-mollified second moment in \cref{secondmomentcor}. We focus on the core part of the proof: the error term in the asymptotic formula for $\mathcal{S}_M(A_2(q);F)$. For simplicity, in this sketch we assume coprimality of all relevant variables, suppress smooth functions, ignore units and powers of the ramified prime $\lambda$ in $\mathbb{Z}[\omega]$, and ignore congruence conditions with fixed modulus.


\subsection{Evaluation of $\mathcal{S}_M(A_2(q);F)$}
We first use the approximate functional equation for $|L(1/2,\chi_q)|^2$, and then remove the squarefree condition on $q$. After
applying Poisson summation on the sum over $q$ and removing the main term (corresponding to the frequency $k=0$), it suffices to estimate
\begin{align} \label{start}
    \frac{X}{N_1 N_2} \frac{1}{L^2}  \mathop{ {\sum \sum}}_{\substack{ \ell,k \in \mathbb{Z}[\omega] \\ k \neq 0  \\ \ell \equiv 1 \pmod{3} \\ N(\ell) \sim L \\ N(k) \ll Z}}  \mu(\ell) \Big( \sum_{\substack{n_1 \in \mathbb{Z}[\omega]  \\ n_1 \equiv 1 \pmod{3} \\ N(n_1) \sim N_1 }}  \widetilde{g}_3(k\ell , n_1) \Big) \Big( \sum_{\substack{ n_2 \in \mathbb{Z}[\omega] \\ n_2 \equiv 1 \pmod{3} \\ N(n_2) \sim N_2 }}  \overline{\widetilde{g}_3(k\ell , n_2)} \Big),
\end{align}
where $\widetilde{g}_3(\mu,c)$ denotes the normalized cubic Gauss sum (see \eqref{generalgauss} and \eqref{normalized}), which typically has absolute value $1$, and $L, N_1, N_2$ run over powers of two and satisfy 
\begin{equation} \label{LNZ}
    1 \ll L \ll Y, \quad 1 \ll N_1 N_2 \ll X, \quad \text{and} \quad Z:=N_1 N_2 L^2/X.
\end{equation}
We have two different approaches depending on whether $N_1$ and $N_2$ are \emph{balanced} or \emph{unbalanced}. Without loss of generality (by symmetry) we may assume that $N_1 \geq N_2$.
\newline 

In both of the approaches, we first need to understand a sum of cubic Gauss sums. After performing Perron summation and a contour shift to the critical line, we pass over (at most) a simple pole at $s= \frac{5}{6}$. We obtain
\begin{align}
    \sum_{\substack{n \in \mathbb{Z}[\omega]  \\ n \equiv 1 \pmod{3} \\ N(n) \sim N }}  \widetilde{g}_3(k\ell , n) = \Upsilon \frac{\tau_3(k \ell)}{N(k \ell)^{1/6}} N^{5/6} + \int_{\mathcal{C}_{\varepsilon}} \widetilde{\psi}(k \ell,s)  N^{s} \frac{ds}{s} =: \mathcal{P}+\mathcal{I}, \label{asymptotic}
\end{align}
where $\Upsilon$ is an absolute constant, $\tau_3(u) \in \mathbb{C}$ (for $u \in \mathbb{Z}[\omega]$) denote the Fourier coefficients of Patterson's cubic theta function \cite{Pat1}, $\mathcal{C}_{\varepsilon}$ denotes the line segment $\Re(s) = \frac{1}{2} + \varepsilon$ and $|\Im(s)| \leq X^\varepsilon$, and $\widetilde{\psi}(a,s)$ is the Dirichlet series 
\begin{equation*}
    \widetilde{\psi}(a,s):=\sum_{\substack{c \in \mathbb{Z}[\omega] \\ c \equiv 1 \pmod{3} }} \frac{\widetilde{g}_3(a,c)}{N(c)^s}, \quad \Re(s)>1, \quad 0 \neq  a \in \mathbb{Z}[\omega].
\end{equation*}
The salient point is that $\widetilde{\psi}(a,s)$ has a meromorphic continuation to all $\mathbb{C}$, satisfies a functional equation, and also satisfies a $\operatorname{GL}_1$ convexity bound in the $a$-aspect, i.e.\ of the form $\widetilde{\psi}(a, \frac{1}{2}+\varepsilon + it) \ll N(a)^{1/4+\varepsilon} (1+|t|)^{100}$. Patterson \cite{Pat1} used a tour de force Hecke converse argument to show that the Fourier coefficient $\tau_3(\mu)$ is essentially the cubic Gauss sum $\overline{\widetilde{g}_3(\mu)}$ (see \eqref{generalgauss} and \eqref{normalized}). Nothing is lost by assuming that $k \equiv 1 \pmod 3$ and $(k,\ell)=1$ in this heuristic. Since $\mu^2(\ell)=1$ and $(k,\ell)=1$, Patterson's result \cite[Proposition~8.1]{Pat1} tells us that $\tau_3(k \ell)=0$ unless $k=c d^3$, where $\mu^2(c)=1$ and $c, d \equiv 1 \pmod{3}$. In this case we have that
\begin{equation} \label{tau3eval}
    \tau_3( \ell c d^3)= 3^3 \cdot  \overline{\widetilde{g}_3(\ell c)} N(d)^{1/2}.
\end{equation}
Before continuing with our sketch for the second moment, we highlight that \eqref{asymptotic} and \eqref{tau3eval} are also important ingredients in our proof of the first moment asymptotic in \cref{SMestimate1}.

We use \eqref{asymptotic} and \eqref{tau3eval} to evaluate the sum over $n_1$ in \eqref{start}, since it is longer than that over $n_2$. The contribution to \eqref{start} from the polar term $\mathcal{P}_1$ given in \eqref{asymptotic} is equal (up to an absolute constant factor) to
\begin{align} \label{main}
    & \frac{X}{N_1^{1/6}N_2} \frac{1}{L^2} \mathop{ {\sum \sum \sum}}_{\substack{ d,\ell,c \in \mathbb{Z}[\omega] \\ d,\ell,c \equiv 1 \pmod{3} \\ N(\ell) \sim L \\ N(c d^3) \ll Z}} \mu(\ell) \frac{\overline{\widetilde{g}_3(\ell c)}}{N(\ell c)^{1/6}} \sum_{\substack{ n_2 \in \mathbb{Z}[\omega] \\ n_2 \equiv 1 \pmod{3} \\ N(n_2) \sim N_2 }}  \overline{\widetilde{g}_3(cd^3\ell , n_2)}.
\end{align}
To handle the contribution to \eqref{start} from the integral term $\mathcal{I}_1$ given in \eqref{asymptotic}, we also evaluate the sum over $n_2$ using \eqref{asymptotic} and \eqref{tau3eval} to obtain a cross term $|\mathcal{I}_1||\mathcal{P}_2|$
that is
\begin{align} \label{cross}
    & \ll \frac{X}{N_1^{1/2} N_2^{1/6}} \frac{1}{L^2} \int_{\mathcal{C}_{\varepsilon}}  \mathop{\sum \sum \sum}_{\substack{ d,\ell,c \in \mathbb{Z}[\omega] \\ d,\ell,c \equiv 1 \pmod{3} \\ N(\ell) \sim L \\ N(c d^3) \ll Z}}  \frac{\mu^2(\ell) \mu^2(c)  }{N(\ell c)^{1/6}}  \big|\widetilde{\psi}(c \ell d^3,s)\big| |ds|,
\end{align}
and a pure integral term $|\mathcal{I}_1||\mathcal{I}_2|$
that is
\begin{align} \label{diagerror}
    & \ll \frac{X}{(N_1 N_2)^{1/2}} \frac{1}{L^2} \int_{\mathcal{C}_{\varepsilon}} \int_{\mathcal{C}_{\varepsilon}} \mathop{ {\sum \sum}}_{\substack{ \ell,k \in \mathbb{Z}[\omega] \\ k \neq 0  \\ \ell \equiv 1 \pmod{3} \\ N(\ell) \sim L \\ N(k) \ll Z}} \mu^2(\ell) \big|\widetilde{\psi}(k \ell ,s_1)\big| \big|\widetilde{\psi}(k \ell,s_2)\big|  |ds_1 |  |ds_2 |. 
\end{align}
In order to estimate \eqref{start}, it suffices to estimate \eqref{main}, \eqref{cross}, and \eqref{diagerror}. Our treatment of \eqref{cross} and \eqref{diagerror} will be the same for all sizes of $N_1$ and $N_2$, however our treatment of \eqref{main} will depend on the relative sizes of $N_1$ and $N_2$.

Let us first dispense with \eqref{cross}. We have $\widetilde{\psi}(c \ell d^3,s) \approx \widetilde{\psi}(c \ell,s)$. Then clump together $c \ell$ as one variable at the expense of $X^{\varepsilon}$ coming from the divisor function, and extend the summation over that variable to all Eisenstein integers with norm $\ll LZ/N(d)^3$ by positivity. We use Cauchy--Schwarz and a Lindel\"{o}f-on-average bound (in the $a$-aspect) for the second moment of the cubic Dirichlet series $\widetilde{\psi}(a, s)$. The bound hinges on the cubic large sieve and a $\operatorname{GL}_1$ (in the $a$-aspect) approximate functional equation for $\widetilde{\psi}(a,s)$. See \cref{Gauss_sums_section}, and in particular \cref{second_moment_lindelof_lemma}, for details. 

We obtain that \eqref{cross} is
\begin{align} \label{t2}
    \ll X^{1/6+\varepsilon} (N_1 N_2)^{1/3}  L^{1/2}  N_2^{1/3} \ll  X^{1/2+\varepsilon} L^{1/2} N_2^{1/3},
\end{align}
for all $N_1$ and $N_2$ satisfying $1 \ll N_1 N_2 \ll X$ and $N_1 \geq N_2$. The treatment of \eqref{diagerror} follows from a similar argument using Cauchy--Schwartz over $\ell$ and $k$. We deduce that \eqref{diagerror} is
\begin{equation} \label{t3}
    \ll X^{\varepsilon} (N_1 N_2)^{1/2} L  \ll X^{1/2+\varepsilon} L
\end{equation}
for all ranges of $N_1$ and $N_2$ specified above.


\subsubsection{The polar contribution \eqref{main} in the unbalanced regime: $N_1$ large and $N_2$ small}  \label{unbalancedsection}

If $N_2$ is small, there is not much use in evaluating the short $n_2$ sum in \eqref{main} using \eqref{asymptotic}, as the integral on the critical line it too large. Instead we seek cancellation from the $c$ and $\ell$ sums. In this sketch it is safe to assume the coprimality conditions $(cd, \ell n_2) = (\ell,n_2)=1$, in which case \cref{cubic-GS-lemma1}(\ref{changevar}) and (\ref{twistmult}) give
\begin{equation*} 
    \overline{\widetilde{g}_3(\ell c)} \cdot \overline{\widetilde{g}_3(cd^3 \ell,n_2 )} = \chi_c(\ell n_2) \cdot \overline{\widetilde{g}_3(c)} \cdot \overline{\widetilde{g}_3(\ell n_2) }.
\end{equation*}
Substituting this into \eqref{main} and applying Cauchy--Schwarz in $c$, we conclude that \eqref{main} is
\begin{align*}  
    & \ll \frac{X^{2/3} N_1^{1/6}}{N_2^{2/3} L^{3/2}} \sum_{\substack{ d \in \mathbb{Z}[\omega] \\ d \equiv 1 \pmod{3} \\ N(d^3) \ll Z}} \frac{1}{N(d)} \Bigg(\sum_{\substack{c \in \mathbb{Z}[\omega] \\ c \equiv 1 \pmod{3} \\ N(c) \ll Z/N(d^3)}} \mu^2(c) \Bigg| \sum_{\substack{\ell, n_2 \in \mathbb{Z}[\omega] \\ \ell, n_2\equiv 1 \pmod{3} \\ N(\ell) \sim L \\ N(n_2) \sim N_2}} \chi_c(\ell n_2) \overline{\widetilde{g}_3(\ell n_2)} \Bigg|^2 \Bigg)^{1/2}.
\end{align*}
The key point is that we may now clump the variables $\ell$ and $n_2$ together. Applying the cubic large sieve, we see that the display above, and hence \eqref{main}, is
\begin{align} 
    & \ll \frac{X^{2/3} N_1^{1/6}}{N_2^{2/3} L^{3/2}} \sum_{\substack{ d \in \mathbb{Z}[\omega] \\ d \equiv 1 \pmod{3} \\ N(d^3) \ll Z}} \frac{1}{N(d)}  \left( \frac{Z}{N(d)^3} + LN_2 + \left(\frac{LN_2Z}{N(d^3)}\right)^{2/3}\right)^{1/2} (LN_2)^{1/2} \nonumber \\
    & \ll X^{\varepsilon} \left(X^{1/6} (N_1 N_2)^{1/2} \left(\frac{N_1}{N_2}\right)^{1/6} + \frac{X^{2/3}(N_1N_2)^{1/6}N_2^{1/6}}{L^{1/2}} + X^{1/3} (N_1N_2)^{1/2}\right) \nonumber \\
    & \ll X^{5/6+\varepsilon} + \frac{X^{5/6+\varepsilon}N_2^{1/6}}{L^{1/2}}, \label{t1}
\end{align}
where the last two displays follow from \eqref{LNZ}. Observe that the bound \eqref{t1} for \eqref{main} performs well when $N_2$ is small. We call this the unbalanced regime because $N_1 \geq N_2$ and $1 \ll N_1 N_2 \ll X$, so $N_1$ could potentially be very large.


\subsubsection{The polar contribution \eqref{main} in the balanced regime: $N_1$ and $N_2$ of moderate size}

If $N_2$ is not too small, we also evaluate the sum over $n_2$ in \eqref{main} using \eqref{asymptotic}. We see that \eqref{main} is majorized by the sum of 
\begin{align*} 
    \frac{X}{(N_1 N_2)^{1/6}} \frac{1}{L^2} \mathop{ {\sum \sum \sum}}_{\substack{ d,\ell,c \in \mathbb{Z}[\omega] \\ d,\ell,c \equiv 1 \pmod{3} \\ N(\ell) \sim L \\ N(c d^3) \ll Z}} \frac{\mu^2(\ell) \mu^2(c)  }{N(\ell c)^{1/3}} \ll X^{1/3} (N_1 N_2)^{1/2} \ll X^{5/6+\varepsilon},
\end{align*}
and
\begin{align*} 
    & \frac{X}{N_1^{1/6} N_2^{1/2}} \frac{1}{L^2} \int_{\mathcal{C}_{\varepsilon}}  \mathop{\sum \sum \sum}_{\substack{ d,\ell,c \in \mathbb{Z}[\omega] \\ d,\ell,c \equiv 1 \pmod{3} \\ N(\ell) \sim L \\ N(c d^3) \ll Z}}  \frac{\mu^2(\ell) \mu^2(c)  }{N(\ell c)^{1/6}}  |\widetilde{\psi}(c \ell d^3,s)| |ds| \ll  X^{1/2+\varepsilon} L^{1/2} N_1^{1/3} \ll \frac{X^{5/6+\varepsilon} L^{1/2}}{N_2^{1/3}},
\end{align*}
where the inequalities follow from the same argument that established \eqref{t2} above (with $N_1$ and $N_2$ interchanged) and from $N_1 N_2 \ll X$. Thus \eqref{main} is 
\begin{equation} \label{t1alternative}
    \ll X^{5/6+\varepsilon}+\frac{X^{5/6+\varepsilon} L^{1/2}}{N_2^{1/3}}.
\end{equation}
Observe that the bound \eqref{t1alternative} performs well for $N_2$ is large. We call this the balanced regime because $N_1 \geq N_2$ and $1 \ll N_1 N_2 \ll X$, so both are forced to be of moderate size.


\subsubsection{Endgame for $\mathcal{S}_M(A_2(q);F)$ and the error term in \cref{secondmomentcor}}

Combining \eqref{t2}, \eqref{t3}, and the minimum of \eqref{t1} and \eqref{t1alternative}, we deduce that \eqref{start} is
\begin{equation} \label{totalbound} 
    \ll  X^{\varepsilon} \Big( X^{5/6} + X^{1/2}L + X^{1/2} L^{1/2} N_2^{1/3} + \min \Big(\frac{X^{5/6}N_2^{1/6}}{L^{1/2}}, \frac{X^{5/6}L^{1/2}}{N_2^{1/3}} \Big) \Big).
\end{equation}
When $L \geq N_2^{1/2}$ we use the first term in the minimum, and when $L \leq N_2^{1/2}$ we use the second term in the minimum. Combining this with $L \ll Y$ and $N_2 \ll X^{1/2}$ (as $N_1 \geq N_2$ and $N_1 N_2 \ll X$), we see that \eqref{totalbound}, and hence \eqref{start}, is 
\begin{equation*} 
    \ll  X^{\varepsilon} ( X^{5/6} + X^{1/2}Y + X^{2/3} Y^{1/2} ),
\end{equation*}
for all possibilities for $L,N_1,N_2$ satisfying \eqref{LNZ} and $N_1 \geq N_2$. Choosing $Y=X^{1/3}$ shows that
\begin{equation} \label{smbound2}
    \mathcal{S}_M(A_2(q);F) = \text{main term} + O_\varepsilon(X^{5/6+\varepsilon}),
\end{equation}
while \cref{SRestimate2} with $Y=X^{1/3}$ yields the estimate 
\begin{equation} \label{srbound2}
    \mathcal{S}_R(A_2(q);F) \ll  X^{5/6+\varepsilon}.
\end{equation}
Combining \eqref{smbound2} and \eqref{srbound2} yields an overall error term of $O_{\varepsilon}(X^{5/6+\varepsilon})$ for \cref{secondmomentcor}.


\subsection{Conventions} \label{conventions}
For $n \in \mathbb{N}$ and $N>0$, we use $n \sim N$ to mean $N<n \leq 2N$,
and $n \asymp N$ to mean that there exist constants $c_1,c_2>0$ such that 
$c_1 N \leq n \leq c_2 N$.

Dependence of implied constants on parameters will be indicated in statements of results,
but suppressed throughout the body of the paper (i.e.\ in proofs).
Implied constants in the body of the paper 
are allowed to depend on $\varepsilon>0$ (which is possibly different in each instance) and on the implicit constants in $\asymp$ or $\ll$ notation.

Every ideal $0 \neq \mathfrak{n} \unlhd \mathbb{Z}[\omega]$ can be written
as $\mathfrak{n}=\lambda^k c \mathbb{Z}[\omega]$ with $k \in \mathbb{Z}_{\geq 0}$ and $c \equiv 1 \pmod{3}$.
We pass between ideals and their generators freely in this paper.

Given $0 \neq \mathfrak{d},\mathfrak{n} \unlhd \mathbb{Z}[\omega]$, the notation
$\mathfrak{d} \mid \mathfrak{n}$ means there exists $\mathfrak{a} \unlhd \mathbb{Z}[\omega]$
such that $\mathfrak{n}=\mathfrak{a} \mathfrak{d}$. Similarly, given $0 \neq d, n \in \Z[\omega]$, the notation $d \mid n$ means $(d) \mid (n)$. For $a, b \equiv 1 \pmod{3}$, the notation $a \mid b^{\infty}$ means that if $\pi \equiv 1 \pmod{3}$ is prime and $\pi \mid a$, then $\pi \mid b$.


\subsection*{Acknowledgements}
The second author thanks Kannan Soundararajan, and the third author thanks Maksym Radziwi{\l\l} for discussions and encouragement.
The authors also thank the referees for their meticulous comments on this paper.


\section{Preliminaries}
\subsection{Eisenstein quadratic field and the cubic symbol}\label{sec:prelim}
Recall that $\mathbb{Q}(\omega)$ is the Eisenstein quadratic field
and $\mathbb{Z}[\omega]$ is its ring of integers.
It is well known that any non-zero element of $\mathbb{Z}[\omega]$ can be uniquely 
written as $\zeta \lambda^{k} c$ with $\zeta \in \langle -\omega \rangle$ a unit (i.e.\ $\zeta^6=1$), 
$\lambda:=\sqrt{-3} = 1 + 2 \omega$ the unique ramified prime in $\mathbb{Z}[\omega]$, $k
\in \mathbb{Z}_{\geq 0}$, and $c \in \mathbb{Z}[\omega]$ with $c \equiv 1 \pmod{3}$. 
If $p \equiv 1 \pmod{3}$ is a rational prime, then 
$p=\pi \overline{\pi}$ in $\mathbb{Z}[\omega]$
with $N(\pi)=p$ and $\pi$ a prime in $\mathbb{Z}[\omega]$.
If $p \equiv 2 \pmod{3}$ is a rational prime, then $p=\pi$ is inert in
$\mathbb{Z}[\omega]$, and $N(\pi)=p^2$. Thus we have $N(\pi) \equiv 1 \pmod{3}$
for all primes $\pi$ with $(\pi) \neq (\lambda)$. 
Let $\check{e}(z):=e^{2 \pi i \text{Tr}_{\mathbb{C}/\mathbb{R}}(z)}
=e^{2 \pi i(z+\overline{z})}$ for $z \in \mathbb{C}$.
The dual of $\mathbb{Z}[\omega]$ is
$\mathbb{Z}[\omega]^{*}:=
\{z \in \mathbb{C}: \check{e}(zz^{\prime})=1 \text{ for all } z^{\prime} \in \mathbb{Z}[\omega] \} =\lambda^{-1} \mathbb{Z}[\omega]$.

Recall that the cubic Jacobi symbol is defined in \eqref{jacobidefine} (and the sentence following it),
and can be multiplicatively extended by \eqref{multdefine}.
The cubic symbol obeys cubic reciprocity: given $a,b \equiv 1 \pmod{3}$ 
we have 
\begin{equation} \label{cuberep}
\Big( \frac{a}{b} \Big)_3=\Big( \frac{b}{a} \Big)_3.
\end{equation}
See Ireland and Rosen \cite[Chapter~9, Theorem~1 and Exercise~20]{IR} for the prime and general versions, respectively\footnote{Their normalization takes primary elements to be congruent to $2 \pmod{3}$; applying their statement to $-a,-b$ gives \eqref{cuberep}, since changing the sign of either argument leaves the cubic symbol unchanged.}.

There are also supplementary laws for units and the ramified prime.
Given 
\begin{equation*}
d \equiv 1 + \alpha_2 \lambda^2 + \alpha_3 \lambda^3 \pmod{9} \quad \text{with} \quad \alpha_2, \alpha_3 \in \{-1,0,1\},
\end{equation*}
then
\begin{equation} \label{cubesupp}
\Big ( \frac{\omega}{d} \Big )_{3} = \omega^{\alpha_2} \quad \text{and} \quad \quad \Big ( \frac{\lambda}{d} \Big )_{3}
 = \omega^{-\alpha_3}.
\end{equation}
We follow the standard convention for an empty product,
\begin{equation*}
\Big( \frac{a}{1} \Big)_3=1 \quad \text{for all} \quad a \in \mathbb{Z}[\omega].
\end{equation*}

For $0 \neq \mathfrak{n} \unlhd \mathbb{Z}[\omega]$, let
\begin{equation}
d(\mathfrak{n}):=\sum_{\mathfrak{d} \mid \mathfrak{n}} 1
\end{equation}
be the divisor function on ideals.
For a given $\varepsilon>0$,
\begin{equation} \label{divis}
d(\mathfrak{n}) \ll_{\varepsilon} N(\mathfrak{n})^{\varepsilon} \quad \text{for all} \quad 0 \neq \mathfrak{n} \unlhd \mathbb{Z}[\omega].
\end{equation}
If $0 \neq n \in \mathbb{Z}[\omega]$, we define $d(n):=d(\mathfrak{n})$, where $\mathfrak{n} = (n)$.
Let $\mu(n)$ denote the M{\"o}bius function on $\mathbb{Z}[\omega]$,
and for $n \equiv 1 \pmod{3}$ let
\begin{equation*}
	\mathrm{rad}(n) := \prod_{\substack{\pi \text{ prime},\ \pi \mid n \\ \pi \equiv 1 \pmod{3}}} \pi.
\end{equation*}


\subsection{Cubic Hecke characters}

For $q \in \mathbb{Z}[\omega]$ with $q \equiv 1 \pmod{3}$, recall the cubic Dirichlet character 
given in \eqref{chiqdef}, and that it is 
a (cubic) Hecke character if $\chi_q(\omega)=1$ (i.e.\ trivial on units
of $\mathbb{Z}[\omega]$). 
Writing $q = ab^2c^3d^3$ for $a, b, c, d \in \Z[\omega]$ satisfying $a, b, c, d \equiv 1 \pmod{3}$, $\mu^2(abc) = 1$, and $d \mid (abc)^\infty$, note that in fact $\chi_q = \chi_a \overline{\chi_b} \mathbf{1}_c$, where $\mathbf{1}_c$ denotes the trivial character modulo $c \Z[\omega]$. Thus the modulus of $\chi_q$ is $abc \Z[\omega] = \mathrm{rad}(q) \Z[\omega]$, and $\chi_q$ is primitive exactly when $c=1$, or equivalently when $\chi_q$ is a product of characters of 
distinct prime conductors 
(i.e.\ either $\chi_{\pi}$ or $\overline{\chi_{\pi}}=\chi^2_{\pi}=\chi_{\pi^2}$).
We conclude that a given Dirichlet character is a primitive 
cubic Hecke character provided that $\chi=\chi_q$ for $q \in \mathcal{C}_3$, 
where $\mathcal{C}_3$ is given in \eqref{C3def}.
Each $\chi=\chi_{q_1 q^2_2}$ with $q_1 q_2^2 \in \mathcal{C}_3$ has conductor $\mathfrak{c_{\chi}}:=\text{cond}(\chi)=q_1 q_2 \mathbb{Z}[\omega]$. 
Sometimes we may
abuse terminology and refer to
$N(\mathfrak{c}_{\chi})=N(q_1 q_2)$ as the ``conductor" when referencing the lengths of various sums occurring in the Fourier analysis. 

\subsection{Cubic Gauss sums and variants}

Recall that $\check{e}(z):=e^{2 \pi i \text{Tr}_{\mathbb{C}/\mathbb{R}}(z)}
=e^{2 \pi i(z+\overline{z})}$ for $z \in \mathbb{C}$.
For $\mu \in \mathbb{Z}[\omega]$
and $c \in \mathbb{Z}[\omega]$ with
with $c \equiv 1 \pmod{3}$, the cubic Gauss sum (with shift $\mu$) is defined by
\begin{equation} \label{generalgauss}
g_3(\mu,c):=\sum_{d \pmod{c}} \chi_c(d) \check{e} \Big( \frac{\mu d}{c} \Big).
\end{equation}
We write $g_3(c):=g_3(1,c)$ for short. For $0 \neq b \in \Z[\omega]$ also denote 
\begin{equation*}
    \varphi(b) := \sum_{\substack{a\pmod{b} \\ (a, b)=1}} 1 = N(b) \prod_{\substack{\mathfrak{p} \text{ prime} \\ \mathfrak{p} \mid (b)}} \Big(1-\frac{1}{N(\mathfrak{p})}\Big),
\end{equation*}
where the product is over distinct prime ideals (i.e.\ over $\mathfrak{p} = (\pi)$ for primes $\pi = \lambda$ or $\pi \equiv 1\pmod{3}$). The following lemma records standard properties for cubic Gauss sums.
\begin{lemma} \label{cubic-GS-lemma1} 
Let $c, c_1, c_2 \equiv 1 \pmod 3$ and $\mu, \nu \in \mathbb{Z}[\omega]$.
\begin{enumerate}[(a)]
\item \label{changevar}  If $(\nu, c)=1$, then
\begin{equation*}
g_3(\mu \nu, c) = \overline{\chi_{c}(\nu)} g_3(\mu,c).
\end{equation*}
\item \label{twistmult}  If $(c_1, c_2)=1$, then 
\begin{equation*}
g_3(\mu, c_1 c_2) = \overline{\chi_{c_2}(c_1)} g_3(\mu,c_1) g_3(\mu, c_2) = g_3(\mu, c_1) g_3(\mu c_1, c_2).
\end{equation*}
\end{enumerate}
\end{lemma}

In order to compute $g_3(\mu,c)$ for general parameters $\mu,c$,
it suffices to 
compute $g_3(\pi^{k},\pi^{\ell})$ for $\pi \equiv 1 \pmod{3}$ prime
and $k,\ell \in \mathbb{Z}_{\geq 0}$ by \cref{cubic-GS-lemma1}.

\begin{lemma} \label{localcomp}
Let $k,\ell \in \mathbb{Z}_{\geq 0}$ and 
$\pi \in \mathbb{Z}[\omega]$ be prime and satisfy $\pi \equiv 1 \pmod{3}$.
Then we have
\begin{equation*}
g_3(\pi^{k},\pi^{\ell})=\begin{cases}
1 & \text{if } \ell=0, \\
\varphi(\pi^{\ell}) & \text{if } 1 \leq \ell \leq k, \quad \ell \equiv 0 \pmod{3} \\ 
-N(\pi)^{k} & \text{if } \ell=k+1, \quad \ell \equiv 0 \pmod{3} \\
N(\pi)^{k} g_3(\pi) & \text{if } \ell=k+1, \quad \ell \equiv 1 \pmod{3} \\
N(\pi)^{k} \overline{g_3(\pi)} & \text{if } \ell=k+1, \quad \ell \equiv 2 \pmod{3} \\
0 & \emph{otherwise}
\end{cases}.
\end{equation*}
\end{lemma}

\begin{proof}
A specialization of \cite[property (h), pg.~7]{Pro} yields the result.

\end{proof}

For $\pi \equiv 1 \pmod{3}$ prime we have the formula for the cube \cite[pp. 443--445]{H},
\begin{equation} \label{cuberel}
g_3(\pi)^3=-\pi^2 \overline{\pi}.
\end{equation}
Note that \cref{cubic-GS-lemma1}, \cref{localcomp}, and \eqref{cuberel}
imply that
\begin{equation} \label{sqrootcancel}
|g_3(c)|=\mu^2(c) N(c)^{1/2}
\end{equation}
for $c \equiv 1 \pmod{3}$.
We denote the normalized cubic Gauss sum (with shift $\mu \in \mathbb{Z}[\omega]$) by 
\begin{equation} \label{normalized}
\widetilde{g}_3(\mu,c):=N(c)^{-1/2} g_3(\mu,c).
\end{equation}

We need to consider slightly more general exponential sums that are the finite Fourier transforms of 
cubic Hecke characters (not necessarily primitive).
Let $c=c_1 c^2_2 \in \mathbb{Z}[\omega]$ where $c_1,c_2 \in \mathbb{Z}[\omega]$, $c_1,c_2 \equiv 1 \pmod{3}$, and $\mu^2(c_1)=1$. For $\mu \in \mathbb{Z}[\omega]$, let 

\begin{equation} \label{hdef}
\widetilde{h}_3(\mu,\chi_c):=\frac{1}{N(c_1 c_2)^{1/2}} \sum_{\substack{ x \pmod{c_1 c_2} \\ (x,c_1 c_2)=1}} \chi_c( x) \check{e} \Big( \frac{\mu x}{c_1 c_2} \Big). 
\end{equation}

\begin{lemma} \label{h3fourier}
Let $c=c_1 c^2_2 \in \mathbb{Z}[\omega]$ such that $c_1,c_2 \in \mathbb{Z}[\omega]$, $c_1,c_2 \equiv 1 \pmod{3}$, and $\mu^2(c_1)=1$. 
Let $\mu,\nu \in \mathbb{Z}[\omega]$, and $\widetilde{h}_3(\mu,\chi_c)$ be as in \eqref{hdef}.
\begin{enumerate}[(a)]
\item \label{mult} If $(c_1,c_2)=1$, then
\begin{equation*} 
\widetilde{h}_3(\mu,\chi_c)=\widetilde{g}_3(\mu,c_1) \: \overline{\widetilde{g}_3(\mu,c_2)}.
\end{equation*}
\item \label{h3prop} If $(\nu,c)=1$, then 
\begin{equation*}
\widetilde{h}_3(\mu \nu,\chi_c)=\overline{\chi_c(\nu)} \: \widetilde{h}_3(\mu,\chi_c).
\end{equation*}
\end{enumerate}
\end{lemma}

\begin{proof}
The claim (\ref{mult}) follows from a short computation with \eqref{hdef} using that
 $\chi_c=\chi_{c_1}  \overline{\chi_{c_2}}$, the Chinese remainder theorem to write  
$x=c_2 x_1+c_1 x_2$ with 
$x_i $ running modulo $c_i$ for $ i \in \{1,2\}$, since $(c_1,c_2)=1$, 
and cubic reciprocity $\big( \frac{c_2}{c_1} \big)_3 \overline{\big( \frac{c_1}{c_2} \big)_3}=1$.
The claim (\ref{h3prop}) follows immediately from (\ref{mult}) and \cref{cubic-GS-lemma1}(\ref{changevar}).

\end{proof}


\subsection{Hecke $L$-functions over $\mathbb{Q}(\omega)$}

Let $\mathfrak{m} \unlhd \Z[\omega]$ and let $\psi \pmod{\mathfrak{m}}$ be a Hecke character of $\Q(\omega)$ of trivial infinite type. 
The Hecke $L$-function attached to $\psi$ is given by 
\begin{equation} \label{Lspsi}
L(s,\psi):=\sum_{0 \neq \mathfrak{n} \unlhd \mathbb{Z}[\omega]} \frac{\psi(\mathfrak{n})}{N(\mathfrak{n})^{s}}, \qquad \Re(s)>1.
\end{equation}
Note that we put $\psi(\mathfrak{n})=0$ whenever $\mathfrak{n}$ and $\mathfrak{m}$ are not coprime. Let $\mathfrak{c}_{\psi} \unlhd \mathbb{Z}[\omega]$ denote the conductor of $\psi$. The completed Hecke $L$-function of $\psi$ is defined by 
\begin{equation} \label{completed}
\Lambda(s,\psi):=(|d_{\mathbb{Q}(\omega)}| N(\mathfrak{c}_\psi))^{s/2} (2 \pi)^{-s} \Gamma(s) L(s,\psi), \qquad s \in \mathbb{C},
\end{equation}
where $d_{\mathbb{Q}(\omega)}=-3$ is the discriminant of $\mathbb{Q}(\omega)$. 

\begin{prop} \emph{\cite[VII Cor.~8.6]{Neu}}  \label{funceq}
The completed $L$-function $\Lambda(s,\psi)$ is entire, provided that $\psi$ is primitive and 
$\mathfrak{c}_{\psi}=m \mathbb{Z}[\omega] \neq \mathbb{Z}[\omega]$. Furthermore, it satisfies the functional equation 
\begin{equation*}
\Lambda(s,\psi)=\frac{W(\psi)}{N(\mathfrak{c}_{\psi})^{1/2}} \Lambda(1-s,\overline{\psi}),
\end{equation*}
where
\begin{equation} \label{Wpsi}
W(\psi):=\sum_{\substack{ x \pmod{\mathfrak{c}_{\psi}} \\ (x,\mathfrak{c}_{\psi})=1 }} \psi(x) \check{e} \Big(\frac{x}{\lambda m}  \Big).
\end{equation}
\end{prop}

\begin{remark} \label{rootnumber}
Suppose that $q=q_1 q_2^2 \in \mathcal{F}_3$ and let $\psi=\chi_q$. Then $\mathfrak{c}_{\chi_q}=q_1 q_2 \mathbb{Z}[\omega]$,
and
\begin{equation*}
\frac{W(\chi_q)}{N(q_1 q_2)^{1/2}}=\chi_{q}(\lambda) \widetilde{h}_3(1,\chi_q)=\widetilde{g}_3(q_1) \overline{\widetilde{g}_3(q_2)},
\end{equation*}
where the last equality follows from \cref{h3fourier}(\ref{mult}) and \cref{supprem}.
\end{remark}

For $d \in \mathbb{Z}[\omega]$ with $d \equiv 1 \pmod{3}$, let 
\begin{equation} \label{A1def}
A_1(d):= \sum_{0 \neq \mathfrak{n} \unlhd \Z[\omega]} \frac{\chi_d(\mathfrak{n})}{ N( \mathfrak{n})^{1/2}} 
\Phi_1 \Big( \frac{N(\mathfrak{n})}{\sqrt{3 N(d)}} \Big),
\end{equation}
\begin{equation} \label{A2def} 
A_2(d):= \sumtwo_{0 \neq \mathfrak{n}_1, \mathfrak{n}_2 \unlhd \Z[\omega]} \frac{\chi_d(\mathfrak{n}_1) \overline{\chi_d( \mathfrak{n}_2)}}{ N( \mathfrak{n}_1 \mathfrak{n}_2)^{1/2}} 
\Phi_2 \Big( \frac{N(\mathfrak{n}_1 \mathfrak{n}_2)}{3 N(d)} \Big),
\end{equation}
and
\begin{equation} \label{def-Phi}
\Phi_j(y) = \frac{1}{2 \pi i} \int_{2-i \infty}^{2+i \infty} (2 \pi)^{-jw} y^{-w} \frac{\Gamma(1/2+w)^j}{\Gamma(1/2)^j}  \frac{dw}{w}, \quad j=1,2.
\end{equation}
We have the bound
\begin{equation}\label{Phi_bound}
y^k \Phi^{(k)}_j(y) \ll_{A, k} (1+y)^{-A}
\end{equation}
for all $k, A \in \mathbb{Z}_{\geq 0}$. 
The lemma below records the approximate functional equations
for cubic Hecke $L$-functions in our thin family $\mathcal{F}^{\prime}_3$ at the central point. It follows in a straightforward manner from
\cite[Theorem~5.3]{IwK}, \cref{funceq},
and \cref{rootnumber}.

\begin{lemma} \label{afecenter}
Let $\mathcal{F}^{\prime}_3$ be as in \eqref{F3primedef}, and $q \in \mathcal{F}^{\prime}_3$.
Then
\begin{equation*}
L(1/2, \chi_q) =A_1(q)+ \widetilde{g}_3(q) \cdot \overline{A_1(q)},
\end{equation*}
and
\begin{equation*}
|L(1/2, \chi_q)|^2 = 2 A_2(q),
\end{equation*}
where $A_1(\cdot)$ and $A_2(\cdot)$ are given in \eqref{A1def} and \eqref{A2def} respectively. 
\end{lemma}

We will need also a slightly more general approximate functional equation 
at more general points $s \in \mathbb{C}$ in order to derive a second moment estimate
in \cref{second_moment_section}.

\begin{lemma} \emph{\cite[Theorem~5.3]{IwK}} \label{afe}
Let  $\mathcal{F}_3$ be as in \eqref{F3def}, $Y>0$, and $G$ be any even function
that is holomorphic and bounded in $|\Re(u)|<4$, and that also satisfies $G(0)=1$.
Then for any $q=q_1 q_2^2 \in \mathcal{F}_3$ and $s \in \mathbb{C}$ with $0 \leq \Re(s) \leq 1$ we have that
\begin{align*}
L(s,\chi_q)&=\sum_{0 \neq \mathfrak{n} \unlhd \Z[\omega]} \frac{\chi_q(\mathfrak{n})}{N(\mathfrak{n})^s} V_s \Big( \frac{N(\mathfrak{n})}{Y \sqrt{3 N (q_1 q_2)}}   \Big) \\
&+ (3 N(q_1 q_2))^{1/2-s} (2 \pi)^{2s-1} \frac{\Gamma(1-s)}{\Gamma(s)} \widetilde{g}_3(q_1) \overline{\widetilde{g}_3(q_2)}
\sum_{0 \neq \mathfrak{n} \unlhd \Z[\omega]} \frac{\overline{\chi_q(\mathfrak{n})}}{N(\mathfrak{n})^{1-s}}
V_{1-s} \Big( \frac{Y N(\mathfrak{n})}{\sqrt{3 N (q_1 q_2)}} \Big),
\end{align*}
where 
\begin{equation} \label{Vsdef}
V_s(y):=\frac{1}{2 \pi i} \int_{2-i \infty}^{2+i \infty} (2 \pi)^{-w} y^{-w}  G(w) \frac{\Gamma(s+w)}{\Gamma(s)} \frac{dw}{w}.
\end{equation} 
\end{lemma}

We give the asymptotic properties $V_s(y)$ in the degree $2$ setting of \eqref{completed}.

\begin{lemma} \emph{\cite[Proposition~5.4]{IwK}} \label{decaylem}
Given  $A \in \mathbb{Z}_{>0}$,
let $G(u):=(\cos ( \frac{\pi u}{4 A}))^{-8A}$ for $u \in \mathbb{C}$ in \cref{afe}.
Suppose that $\Re(s) \geq 3 \alpha>0$, and that $b \in \mathbb{Z}_{\geq 0}$.
Then 
\begin{align*}
y^{b} V^{(b)}_s(y) \ll_{A,\alpha,b} \Big(1+\frac{y}{1+|s|} \Big)^{-A} \qquad \text{and} \qquad y^{b} V^{(b)}_s(y) =\delta_b +O_{A,\alpha,b} \Big( \Big(\frac{y}{1+|s|} \Big)^{\alpha} \Big)
\end{align*}
for all $y>0$, where $\delta_0=1$ and $\delta_b=0$ if $b>0$.
\end{lemma}

We also include here a straightforward but useful device.
\begin{lemma} \label{device}
Let $\mathcal{F}_3$ be as in \eqref{F3def}, $q \in \mathcal{F}_3$,
$U>0$, and $s=\sigma+it \in \mathbb{C}$ with $\sigma \in (0,1]$. Then 
\begin{align*}
L(s,\chi_q)^2 
 =\sum_{0 \neq \mathfrak{n} \unlhd \Z[\omega]} d(\mathfrak{n}) \chi_q(\mathfrak{n}) N(\mathfrak{n})^{-s} e^{-N(\mathfrak{n})/U} 
-\frac{1}{2 \pi i} \int_{\alpha-i \infty}^{\alpha+i \infty} L(w,\chi_q)^2 \Gamma(w-s) U^{w-s} dw,
 \end{align*}
valid for any $0 \leq \alpha < \sigma=\Re(s) \leq 1$.
 \end{lemma}
 \begin{proof}
 Mellin inversion gives 
 \begin{equation} \label{mellin1}
\sum_{0 \neq \mathfrak{n} \unlhd \Z[\omega]} d(\mathfrak{n}) \chi_q(\mathfrak{n}) N(\mathfrak{n})^{-s} e^{-N(\mathfrak{n})/U}=
\frac{1}{2 \pi i} \int_{2-i \infty}^{2+ i \infty} L(s+w,\chi_q)^2 \Gamma(w) U^{w} dw.
 \end{equation}
 After moving the contour to $\Re(w)=\alpha-\sigma$, we collect the residue $L(s,\chi_q)^2$ from the simple pole at $w=0$,
and see that \eqref{mellin1} is equal to
\begin{equation*}
L(s,\chi_q)^2+ \frac{1}{2 \pi i} \int_{\alpha-\sigma-i \infty}^{\alpha-\sigma+ i \infty} L(s+w,\chi_q)^2 \Gamma(w) U^{w} dw.
\end{equation*}
The claim follows after making the change of variable $w \mapsto w-s$.

\end{proof}


\subsection{Poisson summation}

\begin{lemma}\label{radialpois}
Let $q, c\in \Z[\omega]$ with $q \equiv 1 \pmod{3}$. Let $\Psi: \Z[\omega] \rightarrow \C$ be a $q$-periodic function and $V: \R \rightarrow \C$ be a Schwartz function. Then for any $M > 0$, we have
\begin{align*}
\sum_{\substack{m \in \Z[\omega] \\ m \equiv c \pmod{9}}} \Psi(m) V \Big( \frac{N(m)}{M} \Big) = \frac{4 \pi M}{3^{9/2}  N(q)}  \sum_{k \in \Z[\omega]} \ddot{\Psi}(k) \check{e} \Big({ -  \frac{ k c q^2}{9\lambda }} \Big) \ddot{V} \Big(  \frac{k \sqrt{M}}{q} \Big)
\end{align*}
where 
\begin{align} \label{ddotpsi}
\ddot{\Psi}(k) =
\sum_{\substack{b \pmod{q}}} \Psi(9\lambda b) \check{e} \Big( {- \frac{k b}{q}} \Big)
\end{align}
and $\ddot{V}: \C \rightarrow \C$ is defined by
\begin{align} \label{Vddot}
\ddot{V}(u) = \int_{0}^\infty  t V(t^2) J_0 \Big( \frac{4 \pi t |u|}{ 9\sqrt{3}} \Big) \; dt.
\end{align}
\end{lemma}

\begin{proof}
This is a slight modification of \cite[Lemma~4.3]{DR}. Apply \cite[Lemma 4.2]{DR} to obtain
\begin{align*}
\sum_{\substack{m \in \Z[\omega] \\ m \equiv c \pmod{9}}} \Psi(m) V \Big( \frac{N(m)}{M} \Big) = \frac{2}{\sqrt{3} N(9 q)} 
\sum_{k \in \Z[\omega]} \dot{\Psi}(k) \; \int_{\R^2} V \Big( \frac{x^2 + y^2}{M} \Big) \check{e} \Big( \frac{k (x+iy)}{9\lambda q} \Big) \, dx dy, \end{align*}
where 
\begin{align*}
\dot{\Psi}(k) = \sum_{\substack{t \pmod{9 q} \\ t \equiv c \pmod{9}}} \Psi(t) \check{e} \Big( {- \frac{kt}{9\lambda q}} \Big).
\end{align*}
As in the treatment of \cite[(4.6)]{DR}, we can write
\begin{align*}
\int_{\R^2} V \Big( \frac{x^2 + y^2}{M} \Big) \check{e} \Big( \frac{k (x+iy)}{9\lambda q } \Big) \, dx dy &= M \int_{0}^\infty r V(r^2) \int_0^{2 \pi} \exp \Big(  \frac{4 \pi i r \cos{\theta} |k| \sqrt{M} }{ 9\sqrt{3} |q| }\Big) \; d \theta d r \\
&= 2 \pi M \int_{0}^\infty r V(r^2) J_0 \Big( \frac{4 \pi r |k| \sqrt{M}}{9\sqrt{3}|q|} \Big) \; dr.
\end{align*}

It remains to compute $\dot{\Psi}(k)$. Applying the Chinese remainder theorem, we write $t \pmod{9 q}$ as $t = a q + 9b$, with $a \pmod{9}$ and $b \pmod{q}$. Thus $a \equiv t \overline{q} \equiv c \overline{q} \equiv cq^2\pmod{9}$, since $q \equiv 1 \pmod{3}$ implies that $q^3 \equiv 1 \pmod{9}$. We conclude that 
\begin{align*}
    \dot{\Psi}(k) &=  \Big( \sum_{\substack{a \pmod{9} \\ a \equiv c q^2 \pmod{9}}} \check{e} \Big( {-  \frac{ k a}{9\lambda }} \Big) \Big) \Big(  \sum_{\substack{b \pmod{q}}} \Psi(9b) \check{e} \Big( {- \frac{k b}{q\lambda}} \Big) \Big)\\
    &= \check{e} \Big( {-  \frac{ k c q^2}{9\lambda }} \Big) \sum_{\substack{b \pmod{q}}} \Psi(9 \lambda b) \check{e} \Big( {- \frac{k b}{q}} \Big),
\end{align*}
and the result follows.

\end{proof}

\begin{lemma} \emph{\cite[Lemma~4.4]{DR}} \label{ddotdecay}
    Let $V: \mathbb{R} \rightarrow \mathbb{C}$ be a smooth and compactly supported function. Then for any integer $A \geq 0$ and $u \in \mathbb{C}$,
    \begin{equation*}
    \ddot{V}(u) \ll_{A,V} (1+|u|)^{-A}.
    \end{equation*}
\end{lemma}


\subsection{Various incarnations of Heath-Brown's cubic large sieve}
\begin{theorem} \emph{\cite[Theorem~2]{HB}}  \label{cubic-large-sieve}
Let $\varepsilon>0$, $A, B \geq 1/2$, $\boldsymbol{\lambda}=(\lambda_b)$ be a $\mathbb{C}$-valued sequence
supported on $b \in \mathbb{Z}[\omega]$ with $b \equiv 1 \pmod{3}$, and 
$\chi_a(\cdot)$ denote the cubic residue symbol $\big( \frac{\cdot}{a} \big)_3$ for $a \in \mathbb{Z}[\omega]$ with $a \equiv 1 \pmod 3$.
Then
\begin{align*} 
\sum_{\substack{ a \in \mathbb{Z}[\omega] \\ a \equiv 1 \pmod{3} \\ N(a) \leq A  }} \mu^2(a) \Big | \sum_{\substack{ b \in \mathbb{Z}[\omega] \\ b \equiv 1 \pmod{3} \\ N(b) \leq B }}
\mu^2(b) \lambda_b \chi_a(b)  \Big |^2 
\ll_{\varepsilon} (AB)^\varepsilon ( A + B + (AB)^{2/3} ) \| \boldsymbol{\mu}^2 \boldsymbol{\lambda} \|_2^2.
\end{align*}
 \end{theorem}
 
Implicit in the proof of \cite[Theorem~2]{HB} are mean square estimates
where one of the variables is not required to be squarefree nor congruent to $1$ modulo $3$. 
We extract the relevant results using the duality principle for the large sieve \cite[(7.9)--(7.11)]{IwK}.

\begin{prop}[Cubic large sieve without $\mu^2$] \label{HBcubicdual} 
Let $\varepsilon>0$, $A,B \geq 1/2$, $\boldsymbol{\lambda}=(\lambda_b)$ be
a $\mathbb{C}$-valued sequence supported on $b \in \mathbb{Z}[\omega]$, and
$\chi_a(\cdot)$ denote the cubic residue symbol $\big( \frac{\cdot}{a} \big)_3$ for $a \in \mathbb{Z}[\omega]$ with $a \equiv 1 \pmod 3$.
Then
\begin{equation} \label{dualbd}
\sum_{\substack{ a \in \mathbb{Z}[\omega] \\ a \equiv 1 \pmod{3} \\ N(a) \leq  A  }} \mu^2(a) \bigg | \sum_{\substack{b \in \mathbb{Z}[\omega] \\ N(b) \leq B }} 
\lambda_b \chi_a(b) \bigg |^2
\ll_{\varepsilon} (AB)^{\varepsilon} (AB^{1/3} + B + (AB)^{2/3} ) \| \boldsymbol{\lambda} \|_2^2. 
\end{equation}
\end{prop}

\begin{remark}
    In general the term $AB^{1/3}$ cannot be improved, as can be seen from the contribution of the cubes for the dual problem with constant coefficients.
\end{remark}

\begin{proof}
We note that \cite[(22)]{HB} reads (with different notation)
\begin{equation} \label{B2def}
\mathcal{B}_2(M,N):=\sup  \{ \Sigma_2(M,N,\boldsymbol{c}) : \boldsymbol{c}=(c_n)_{n \in \mathbb{Z}[\omega]} \subset \mathbb{C} 
\text{ with } \| \boldsymbol{c} \|^2_2=1 \},
\end{equation}
where
\begin{equation*}  
\Sigma_2(M,N,\boldsymbol{c}):=\sum_{\substack{m \in \mathbb{Z}[\omega] \\ N(m) \sim M}} \Big | \sum_{\substack{n \in \mathbb{Z}[\omega] \\ n \equiv 1 \pmod{3} \\ N(n) \sim N  }} 
c_n \mu^2(n) \chi_m(n) \Big |^2.
\end{equation*}
Let us first prove that
\begin{equation} \label{B2bound}
\mathcal{B}_2(M,N) \ll (MN)^{\varepsilon} (M+M^{1/3}N + (MN)^{2/3} ).
\end{equation}
Indeed, combining \cite[Lemma 6]{HB} and \cite[Theorem 2]{HB} shows that there are $X, Y \gg 1$ satisfying $XY^2 \ll M$ such that 
\begin{equation*}
	\mathcal{B}_2(M, N) \ll (MN)^\varepsilon \Big(\frac{M}{XY^2}\Big)^{1/3} \min \{X f(Y, N), Y f(X, N) \},
\end{equation*}
where $f(Z, N) := Z + N + (ZN)^{2/3}$. Note that we always have $X^2Y \ll X^{3/2} M^{1/2} \ll M^2$.

If $X\geq Y$, we use the bound $Y f(X, N)$ to get
\begin{align*}
	\mathcal{B}_2(M, N) &\ll (MN)^\varepsilon M^{1/3} \Big(\frac{Y}{X}\Big)^{1/3} ( X + N + (XN)^{2/3} )\\
	&\ll (MN)^\varepsilon M^{1/3} \Big( (X^2Y)^{1/3} + \Big(\frac{Y}{X}\Big)^{1/3}N + (XY)^{1/3}N^{2/3} \Big)\\
	&\ll (MN)^\varepsilon M^{1/3} ( M^{2/3} + N + M^{1/3}N^{2/3})
\end{align*}
since $XY \ll XY^2 \ll  M$, which gives \eqref{B2bound}.

If instead $X\leq Y$, we use the bound $X f(Y, N)$ to get
\begin{align*}
	\mathcal{B}_2(M, N) &\ll (MN)^\varepsilon M^{1/3} \Big(\frac{X}{Y}\Big)^{2/3} ( Y + N + (YN)^{2/3} )\\
	&\ll (MN)^\varepsilon M^{1/3} \Big( (X^2Y)^{1/3} + \Big(\frac{X}{Y}\Big)^{2/3}N + (XN)^{2/3} \Big)\\
	&\ll (MN)^\varepsilon M^{1/3} ( M^{2/3} + N + M^{2/9}N^{2/3} )
\end{align*}
since in this case $X^3 \ll XY^2 \ll  M$, and this also gives \eqref{B2bound}.

Thus \eqref{B2def} and \eqref{B2bound} imply that 
\begin{equation} \label{Sigma2bound}
\Sigma_2(M,N,\boldsymbol{\beta}) \ll (MN)^{\varepsilon} (M+M^{1/3}N + (MN)^{2/3} ) \| \boldsymbol{\beta} \|^2_2,
\end{equation}
for any $\mathbb{C}$-valued sequence $\boldsymbol{\beta}=(\beta_n)$ supported on $n \in \mathbb{Z}[\omega]$.

Set
\begin{equation*}
 \phi(m,n):=\delta_{N(m) \sim M} \cdot \delta_{\substack{ n \equiv 1  \pmod{3} \\ N(n) \sim N }} \cdot \mu^2(n) \chi_m(n).
 \end{equation*}
Applying the 
duality principle \cite[(7.9)--(7.11)]{IwK} with the kernel $\phi(m,n)$ above and then cubic reciprocity \eqref{cuberep} yields
\begin{equation} \label{alphadual}
\sum_{\substack{n \in \mathbb{Z}[\omega] \\ n \equiv 1 \pmod{3} \\ N(n) \sim N }}  \mu^2(n) \Big | \sum_{\substack{ m \in \mathbb{Z}[\omega] \\ N(m) \sim M }} 
\alpha_m \chi_n(m) \Big |^2 \ll
(MN)^{\varepsilon} (M+M^{1/3}N + (MN)^{2/3} ) \| \boldsymbol{\alpha} \|^2_2,
\end{equation}
for any $\mathbb{C}$-valued sequence $\boldsymbol{\alpha}=(\alpha_m)$ supported on $m \in \mathbb{Z}[\omega]$.
The result follows by dyadically partitioning the variables $a$ and $b$ on the left side of \eqref{dualbd} (i.e.\ $N(a) \sim N$, $N(b) \sim M$), applying the Cauchy--Schwarz inequality to
the dyadic scales $M$, and then using \eqref{alphadual}.

\end{proof}


\section{Second moment bounds} \label{second_moment_section}

Recall the definition of $\mathcal{F}_3$ in \eqref{F3def}.
For $Q_1,Q_2 \geq 1/2$, consider 
\begin{equation} \label{FQ1Q2def}
\mathcal{F}_3(Q_1,Q_2):= \big \{q_1 q^2_2 \in \mathcal{F}_3: N(q_1) \asymp Q_1 \text{ and } N(q_2) \asymp Q_2 \big \}.
\end{equation}

\begin{prop}[Cubic large sieve with decaying coefficients] \label{cubic_Dirichlet_prop}
    Let $\mathcal{F}_3(Q_1,Q_2)$ be as in \eqref{FQ1Q2def} and $Q^{\star}:=\min \{Q_1,Q_2\}$. For $q \in \mathcal{F}_3$ let
    \begin{equation} \label{Rqdef}
        \mathcal{R}(q) := \sum_{\substack{0 \neq \mathfrak{b} \unlhd \Z[\omega] \\ N(\mathfrak{b}) \leq R}} \frac{a(\mathfrak{b}) \chi_q(\mathfrak{b})}{N(\mathfrak{b})^{1/2}},
    \end{equation}
    where $|a(\mathfrak{b})| \ll N(\mathfrak{b})^\varepsilon$. Then 
    \begin{equation} \label{R_cubic_large_sieve}
    \sum_{q \in \mathcal{F}_3(Q_1,Q_2)} |\mathcal{R}(q)|^2 \ll_{\varepsilon} (R Q_1 Q_2)^{\varepsilon} (Q_1Q_2 + Q^\star R + (Q^{\star})^{1/3} (Q_1Q_2R)^{2/3} ).
    \end{equation}
\end{prop}

\begin{proof}
    Write $\mathfrak{b}=\lambda^k c$, where $k \in \mathbb{Z}_{\geq 0}$ and $c \in \mathbb{Z}[\omega]$ satisfies $c \equiv 1 \pmod{3}$. Uniquely factorise $c=b_1 b_2^2$ with $b_1,b_2 \equiv 1 \pmod{3}$ and $b_1$ squarefree. We extend the definition of $a(\cdot)$ in an obvious way to elements of $\Z[\omega]$ (setting it to be zero for elements of norm exceeding $R$). Then dyadically partition the summation variables $N(b_1) \sim B_1$ and $N(b_2) \sim B_2$, obtaining
    \begin{align} \label{R_dyadic}
        \mathcal{R}(q) = \sum_{\substack{B_1, B_2 \text{ dyadic} \\ B_1,B_2\gg 1 \\ B_1 B_2^2 \ll R}} \ \sum_{k \in \mathbb{Z}_{\geq 0} } \mathop{\sum \sum}_{\substack{ b_1,b_2 \in \mathbb{Z}[\omega] \\ b_1, b_2 \equiv 1 \pmod{3} \\ N(b_1) \sim B_1,\ N(b_2) \sim B_2}}  \mu^2(b_1) \frac{a(\lambda^k b_1 b_2^2) \chi_q(\lambda^k b_1 b_2^2)}{N(\lambda^k b_1 b_2^2)^{1/2}}.
    \end{align}

    After substitution of \eqref{R_dyadic} into the left side of \eqref{R_cubic_large_sieve}, apply the Cauchy--Schwarz inequality on the sums over $B_1$, $B_2$, $k$, and $b_2$, to conclude that 
    \begin{align} \label{R_intermed}
        \sum_{q \in \mathcal{F}_3(Q_1,Q_2)} |\mathcal{R}(q)|^2 \ll &\ R^{\varepsilon} \sum_{\substack{B_1,B_2 \text{ dyadic} \\ B_1,B_2\gg 1 \\ B_1 B_2^2 \ll R}} \ \sum_{k \in \mathbb{Z}_{\geq 0}} \frac{1}{3^{k/2}} \sum_{\substack{ b_2 \equiv 1 \pmod{3} \\ N(b_2) \sim B_2}} \frac{1}{B_2} \nonumber\\
        & \times \sum_{q \in \mathcal{F}_3(Q_1,Q_2)} \Big | \sum_{\substack{ b_1 \equiv 1 \pmod{3} \\ N(b_1) \sim B_1 }} \frac{\mu^2(b_1) a(\lambda^k b_1b_2^2) \chi_q(b_1)}{N(b_1)^{1/2}} \Big |^2 .
    \end{align}
    We discard the conditions $(q_1,q_2)=1$ and $1\neq q_1q^2_2 \equiv 1 \pmod{9}$ (coming from \eqref{F3def}) by positivity. Thus \eqref{R_intermed} is 
    \begin{align} \label{R_intermed2}
        & \ll R^{\varepsilon} \sum_{\substack{B_1,B_2 \text{ dyadic} \\ B_1,B_2\gg 1 \\ B_1 B_2^2 \ll R}} \ \sum_{k \in \mathbb{Z}_{\geq 0}} \frac{1}{3^{k/2}} \sum_{\substack{ b_2 \equiv 1 \pmod{3} \\ N(b_2) \sim B_2}} \frac{1}{B_2} \nonumber\\
        & \times \mathop{\sum \sum}_{\substack{ q_1,q_2 \equiv 1 \pmod{3} \\  N(q_1) \sim Q_1,\ N(q_2) \sim Q_2}} \mu^{2}(q_1) \mu^{2}(q_2) \Big | \sum_{\substack{ b_1 \equiv 1 \pmod{3} \\ N(b_1) \sim B_1 }} \frac{\mu^2(b_1) a(\lambda^k b_1b_2^2) \chi_{q_1}(b_1) \overline{\chi_{q_2}(b_1)}}{N(b_1)^{1/2}} \Big|^2.
    \end{align}
    Given $i=1$ or $i=2$, we fix $q_i$ and then apply the cubic large sieve (\cref{cubic-large-sieve}). Thus \eqref{R_intermed2} is 
    \begin{align*}
        & \ll (R Q_1 Q_2)^{\varepsilon}  \sum_{\substack{B_1,B_2 \text{ dyadic} \\ B_1,B_2\gg 1 \\ B_1 B_2^2 \ll R}} \min \Big\{ Q_2(Q_1+B_1+(Q_1 B_1)^{2/3}), Q_1(Q_2+B_1+(Q_2 B_1)^{2/3}) \Big\}  \\
        & \ll (R Q_1 Q_2)^{\varepsilon} \Big(Q_1Q_2 + \min \Big\{ Q_2R + Q_2 Q_1^{2/3} R^{2/3}, Q_1 R + Q_1 Q_2^{2/3} R^{2/3} \Big\} \Big)  \\
        & \ll (R Q_1 Q_2)^{\varepsilon} (Q_1Q_2 + Q^\star R + (Q^{\star})^{1/3} (Q_1Q_2R)^{2/3} ),
    \end{align*}
    which establishes \eqref{R_cubic_large_sieve}.
    \end{proof}

\begin{prop}[Bound for twisted second moment]  \label{secondmoment}
    Let $\mathcal{F}_3(Q_1,Q_2)$ be as in \eqref{FQ1Q2def} and $Q^{\star}:=\min \{Q_1,Q_2\} \geq \frac{1}{2}$. Consider $h \in \mathcal{F}_3$ with $H := N(\mathfrak{c}_{\chi_h})$, and for $q \in \mathcal{F}_3$, let $\mathcal{R}(q)$ be as in \eqref{Rqdef} such that $|a(\mathfrak{b})| \ll N(\mathfrak{b})^\varepsilon$ and $\frac{1}{2} \leq R \leq (Q_1Q_2H)^{100}$. Then for any $s=\sigma+it$ with $t\in \R$ and $\sigma \in [1/2,1]$, we have
    \begin{align} 
        \sum_{\substack{q \in \mathcal{F}_3(Q_1,Q_2) \\ (q, h) = 1}} &|L(s,\chi_{qh}) \cdot \mathcal{R}(qh)|^2 \nonumber\\
        &\qquad \ll_{\varepsilon,\sigma} (Q_1 Q_2 H T)^\varepsilon (Q^\star (Q_1Q_2H)^{1/2} TR + (Q^\star H)^{1/3} Q_1 Q_2 (T R)^{2/3}), \label{second}
    \end{align}
    where for ease of notation we denote $T:=|t|+1$.
\end{prop}

\begin{proof} 
We adapt the argument appearing in \cite[pg.~1149--1150]{BGL}.
Set \begin{equation*}
\mathcal{T}^\mathcal{R}_2(Q_1,Q_2,h,s):=\sum_{\substack{q \in \mathcal{F}_3(Q_1,Q_2) \\ (q, h) = 1}} |L(s,\chi_{qh}) \cdot \mathcal{R}(qh)|^2 .
\end{equation*}
We first consider the critical line, i.e.\ $\sigma=1/2$. Writing $h = h_1h_2^2$ with $\mu^2(h_1h_2)=1$, so $N(h_1h_2) = H$, \cref{afe} (with $Y=1$)
and \cref{rootnumber} imply that
\begin{align} 
&\mathcal{T}^\mathcal{R}_2(Q_1,Q_2,h,1/2+it) \nonumber \\
&\qquad \ll \sum_{\pm} \sum_{\substack{q \in \mathcal{F}_3(Q_1,Q_2) \\ (q, h) = 1}} \Big | \sum_{0 \neq \mathfrak{n} \unlhd \Z[\omega]} \frac{\chi_{hq}(\mathfrak{n})}{N(\mathfrak{n})^{1/2 \pm it }}
V_{1/2 \pm it} \Big(  \frac{N(\mathfrak{n})}{ \sqrt{3N(q_1 q_2h_1h_2)} } \Big ) \Big |^2 \cdot |\mathcal{R}(qh)|^2. \label{AFEabsbound}
\end{align}

Using the decay of $V_{1/2 \pm it}$ as in \cref{decaylem}, since $R \leq (Q_1Q_2H)^{100}$ the ideals with $N(\mathfrak{n}) \gg Z$ in \eqref{AFEabsbound} can be removed at the cost of an error term $O_A(Z^{-A})$ for any $A$, where 
\begin{equation} \label{Zdef}
Z:=(Q_1 Q_2 H T)^{\varepsilon} (Q_1 Q_2 H)^{1/2} T,
\end{equation}
where $T:=|t|+1$.
We then open each $V_{1/2 \pm it}(\cdot)$ using the integral representation \eqref{Vsdef},
separate variables, move the contour to $\Re(w)=\varepsilon$,
and truncate the $w$-integral to $[\varepsilon-i Z^{\varepsilon},\varepsilon+i Z^{\varepsilon}]$
up to negligible error $O_A(Z^{-A})$ for all sufficiently large $Z>0$. We apply the Cauchy--Schwarz inequality on the $w$-integral to conclude that the right side of \eqref{AFEabsbound} is 
\begin{align} \label{intermed}
\ll &\ Z^{-1000}+ Z^{\varepsilon}
\sum_{\pm} \int_{\varepsilon-i Z^{\varepsilon}}^{\varepsilon+i Z^{\varepsilon}} \sum_{\substack{q \in \mathcal{F}_3(Q_1,Q_2) \\ (q, h) = 1}} 
\Big | \sum_{\substack{0 \neq \mathfrak{n} \unlhd \Z[\omega] \\ N(\mathfrak{n}) \ll Z }} \frac{\chi_{qh}(\mathfrak{n})}{N(\mathfrak{n})^{1/2 \pm it+w}} \Big |^2\cdot |\mathcal{R}(qh)|^2\ dw.
\end{align}
Applying \cref{cubic_Dirichlet_prop}, where the cubic Dirichlet polynomial inside absolute values has length $\ll ZR$, we conclude that
\begin{align*}
    & \mathcal{T}^\mathcal{R}_2(Q_1, Q_2,h, 1/2+it) \\
     &\ll (ZR)^\varepsilon ( Q_1Q_2 + Q^\star ZR + (Q^\star)^{1/3}(Q_1 Q_2 Z R)^{2/3} ) \\
    &\ll (Q_1 Q_2 H T)^\varepsilon ( Q_1Q_2 + Q^\star (Q_1Q_2H)^{1/2} TR + (Q^\star H)^{1/3} Q_1 Q_2 (T R)^{2/3} ),
\end{align*}
which establishes \eqref{second} when $\sigma=1/2$.

Now consider \eqref{second} in the case $1/2< \sigma \leq 1$.  This follows from \cref{device} (with $\alpha=1/2$ and $U=1$),  
the estimate 
 \begin{equation} \label{gammabd}
\Gamma(x+iy) \ll_x e^{-|y|},
 \end{equation}
and the established case of \eqref{second} when $\sigma=1/2$.

\end{proof}

 
\section{Proof of Propositions \ref{SRestimate1} and \ref{SRestimate2}} \label{SRproofsection}

We start by relating $A_j(q)$ back to the absolute value squared of a cubic Dirichlet polynomial of the correct length via $|L(1/2, \chi_q)|^j$, which while straightforward requires some technicalities since we must cover any $q \equiv 1 \pmod 9$. 

Recalling \eqref{SRdef}, for $j\in \{1, 2\}$ we have
\begin{equation}\label{SRexp_new}
    \mathcal{S}_{R}(|\mathcal{M}(q)^j A_j(q)|; F) =\sum_{\substack{ q \in \Z[\omega] \\ q \equiv 1 \pmod{9}}} |R_Y(q)| |\mathcal{M}(q)|^j |A_j(q)| F \Big( \frac{N(q)}{X} \Big).
\end{equation}

Note that $R_Y(q) \ll N(q)^{\varepsilon}$. By \eqref{mollifier}, \eqref{A1def}, and \eqref{A2def} we also have the trivial bounds $\mathcal{M}(q) \ll M^{1/2+\varepsilon}$ and $A_j(q) \ll N(q)^{j/4+\varepsilon}$, so the contribution of the terms where $q = a^3$ is a cube to \eqref{SRexp_new} is
\begin{equation*}
    \ll M^{\frac{j}{2}+\varepsilon} \sum_{\substack{a \in \Z[\omega] \\ a \equiv 1 \pmod{3}}} F\Big(\frac{N(a)^3}{X}\Big) N(a)^{\frac{3j}{4}+\varepsilon} \ll M^{\frac{j}{2}+\varepsilon} X^{\frac{j}{4} + \frac{1}{3}+\varepsilon}.
\end{equation*}

Also recalling \eqref{MYRYdef}, we observe that $R_Y(q)=0$ unless $q \mathbb{Z}[\omega]=\mathfrak{l}^2 \mathfrak{m}$, where $\mathfrak{m}$ is squarefree, $N(\mathfrak{l})>Y$, and $(\mathfrak{l} \mathfrak{m},\lambda)=1$ (note that $\mathfrak{l}$ here is not to be confused with the $\mathfrak{l}$ in \eqref{MYRYdef}). At the level of elements of $\mathbb{Z}[\omega]$, this is equivalent to $q=\ell^2 m$, where $N(\ell)>Y$, $m$ is squarefree, $\ell, m \equiv 1 \pmod{3}$, and $\ell^2 m \equiv 1 \pmod{9}$. Thus 
\begin{equation} \label{SRexpansion_new}
    \mathcal{S}_{R}(|\mathcal{M}(q)^j A_j(q)|; F) \ll M^{\frac{j}{2}} X^{\frac{3j+4}{12}+\varepsilon} + X^{\varepsilon} \sum_{\substack{ Y<N(\ell) \leq \sqrt{2X} \\ \ell \equiv 1 \pmod{3}  } } \sum_{\substack{ N(m) \sim  X/N(\ell)^2 \\  \ell^2 m \equiv 1 \pmod{9} \\ \ell^2 m \neq\tinycube}} \mu^2(m) |\mathcal{M}(\ell^2 m )^j A_j(\ell^2 m)|.
\end{equation}

Let $d:= (m, \ell)$.  Thus $m = d e$ and $\ell = d \ell'$ with $e, \ell^{\prime} \in \mathbb{Z}[\omega]$, $d, e, \ell^{\prime} \equiv 1 \pmod{3}$, and $(e,\ell^{\prime})=1$. Furthermore, since $m$ is squarefree, we have $d,e$ both squarefree, and $(e,d \ell^{\prime} )=1$. We further uniquely factorize $\ell' = ab^2c^3$ with $a,b,c \in \mathbb{Z}[\omega]$ such that $a,b,c \equiv 1 \pmod{3}$ and $\mu^2(ab)=1$. The condition $(e,d \ell^{\prime} )=1$ and the factorization above guarantees that $(e,abcd)=1$. Thus 
\begin{equation} \label{chareqn_new} 
    \chi_{\ell^2 m}=\chi_{e a^2 b^4 c^6 d^3}=\chi_{e b a^2} \textbf{1}_{cd}=\chi_{e b} \overline{\chi_{a}} \textbf{1}_{cd},
\end{equation}
where $\textbf{1}_{r}$ denotes the principal character to modulus $r$. Since $\mu^2(eab)=1$, the character $\chi_{e b a^2}=\chi_{e b} \overline{\chi_{a}}$ is a primitive cubic Dirichlet character to modulus $e ab$. Moreover, the condition $\ell^2 m \equiv 1 \pmod{9}$ is equivalent to $e b a^2  (b c^2 d)^3 \equiv 1 \pmod{9}$, which in turn implies that $e b a^2 \equiv 1 \pmod{9}$, since $b c^2 d \equiv 1 \pmod{3}$ and therefore $(b c^2 d)^3 \equiv 1 \pmod{9}$.
Finally, the condition $\ell^2m \neq\cube$ is equivalent to $eba^2 \neq 1$.

Putting all this together we have $e b a^2 \in \mathcal{F}_3$. Using \eqref{A1def}, \eqref{A2def}, and \eqref{chareqn_new},
\begin{equation} \label{A1comp_new}
    A_1(\ell^2 m) = \sum_{0 \neq \mathfrak{n} \unlhd \Z[\omega]} \frac{\chi_{e ba^2}\boldsymbol{1}_{cd}(\mathfrak{n})}{ N( \mathfrak{n})^{1/2}} \Phi_1 \Big( \frac{N(\mathfrak{n})}{\sqrt{3 N(e a^2 b^4 c^6 d^3)}} \Big)
\end{equation}
and
\begin{equation}  \label{A2comp_new}
    A_2(\ell^2 m) = \sum_{0 \neq \mathfrak{n}_1,\mathfrak{n}_2 \unlhd \Z[\omega]} \frac{\chi_{e ba^2} \boldsymbol{1}_{cd} (\mathfrak{n}_1) \overline{\chi_{e ba^2} \boldsymbol{1}_{cd}(\mathfrak{n}_2)}}{N(\mathfrak{n}_1 \mathfrak{n}_2)^{1/2}} \Phi_2 \Big( \frac{N(\mathfrak{n}_1 \mathfrak{n}_2)}{3 N(e a^2 b^4 c^6 d^3)} \Big).
\end{equation}

Opening $\Phi_j$ using \eqref{def-Phi} and interchanging the absolutely convergent summations and integration yields 
\begin{align} \label{A1interchange_new}
    A_1(\ell^2 m)&=\frac{1}{2 \pi i} \int_{2-i \infty}^{2 +i \infty} (2 \pi)^{-s} (3N(e a^2 b^4 c^6 d^3))^{s/2} \frac{\Gamma(1/2+s)}{\Gamma(1/2)} L(1/2+s,\chi_{eb a^2} \mathbf{1}_{cd} )\frac{ds}{s}
\end{align}
and
\begin{align} \label{A2interchange_new}
    A_2(\ell^2 m)=\frac{1}{2 \pi i} \int_{2-i \infty}^{2 +i \infty} & (2 \pi)^{-2s} (3N(e a^2 b^4 c^6 d^3))^s \frac{\Gamma(1/2+s)^2}{\Gamma(1/2)^2} \nonumber \\
    \times &\ L(1/2+s,\chi_{eb a^2} \mathbf{1}_{cd} ) \overline{L(1/2+\overline{s},\chi_{eb a^2} \mathbf{1}_{cd} ) } \frac{ds}{s}.
\end{align}
Then 
\begin{equation} \label{eulerprod_new}
    L(1/2+s,\chi_{e b a^2} \mathbf{1}_{cd} )=\mathcal{E}(s,e b a^2,cd) L(1/2+s,\chi_{eb a^2} )
\end{equation}
for
\begin{equation*}
    \mathcal{E}(s,e ba^2,cd):=\prod_{\substack{\pi \text{ prime} \\ \pi \equiv 1 \pmod{3} \\ \pi \mid cd,\ \pi \nmid e ab} } \big(1- \chi_{eba^2}(\pi) N(\pi)^{-1/2-s} \big). 
\end{equation*}

Observe that $\chi_{eb a^2}$ is a non-principal character since $eba^2 \neq 1$, so the right side of \eqref{eulerprod_new} is holomorphic for all $s \in \mathbb{C}$. We apply \eqref{eulerprod_new} to \eqref{A1interchange_new} and \eqref{A2interchange_new}, and then shift both contours to $\Re(s)=\varepsilon$. The triangle inequality yields 
\begin{align}\label{A_1_integral_bound}
    |A_1(\ell^2 m)|& \ll X^{\varepsilon} \int_{\varepsilon-i \infty}^{\varepsilon+i \infty}  |\Gamma(1/2+s)| |L(1/2+s,\chi_{eb a^2} )| |\mathcal{E}(s,e b a^2,cd)| \frac{|ds|}{|s|}
\end{align}
and
\begin{align} \label{A_2_integral_bound}
    |A_2(\ell^2 m)|& \ll X^{\varepsilon} \int_{\varepsilon-i \infty}^{\varepsilon+i \infty}  |\Gamma(1/2+s)|^2 |L(1/2+s,\chi_{eb a^2} )| |L(1/2+\overline{s},\chi_{eb a^2} ) | \nonumber \\
    & \qquad \qquad \qquad \qquad \times |\mathcal{E}(s,e b a^2,cd)|  |\mathcal{E}(\overline{s},e ba^2,cd)|   \frac{|ds|}{|s|}.
\end{align}
Note the uniform trivial bound
\begin{equation*}
    |\mathcal{E}(s,e b a^2,cd)| \ll N(cd)^{\varepsilon} \ll X^{\varepsilon} \quad \text{for} \quad \Re(s)=\varepsilon.
\end{equation*}
Using this combined with Stirling's bound and a trivial bound for the $L$-function, we may truncate the integrals at $X^\varepsilon$ in \eqref{A_1_integral_bound} and \eqref{A_2_integral_bound} to obtain
\begin{align*}
    |A_j(\ell^2m)| \ll X^{\varepsilon} \int_{\varepsilon-i X^\varepsilon}^{\varepsilon+i X^\varepsilon}  |L(1/2+s,\chi_{e b a^2} )|^j |ds| + O_B(X^{-B}),
\end{align*}
for all $X>0$ sufficiently large.

Thus
\begin{align}
    &\sum_{\substack{ Y<N(\ell) \leq \sqrt{2X} \\ \ell \equiv 1 \pmod{3}  } } \sum_{\substack{ N(m) \sim  X/N(\ell)^2 \\  \ell^2 m \equiv 1 \pmod{9} \\ \ell^2 m \neq\tinycube}} \mu^2(m) |\mathcal{M}(\ell^2 m )^j A_j(\ell^2 m)| \ll_B X^{-B} \nonumber \\
    & + X^\varepsilon \int_{\varepsilon-i X^\varepsilon}^{\varepsilon+i X^\varepsilon} \sum_{\substack{ a,b,c,d,e \in \mathbb{Z}[\omega] \\ a,b,c,d,e \equiv 1 \pmod{3} \\ eba^2 \equiv 1 \pmod{9} \\ eba^2 \neq 1,\ (c,e)=1 \\ N(a^2 b^4 c^6 d^3 e) \sim X \\ Y <N(ab^2 c^3 d) \leq \sqrt{2X}}} \mu^2(abe) \mu^2(de)  |\mathcal{M}(a^2 b^4 c^6 d^3 e) L(1/2+s,\chi_{eb a^2} )|^j |ds|. \label{ell_m_second_mom}
\end{align}
Insertion of a dyadic partition in the variables $a,b,c,d,e$ shows that \eqref{ell_m_second_mom} is $O_B(X^{-B})$ plus a term
\begin{align} \label{boundintermed_new}
    \ll X^\varepsilon \sum_{\substack{ A,B,C,D,E \text{ dyadic} \\ A^2 B^4 C^6 D^3 E \asymp X \\ Y\ll AB^2C^3D \ll \sqrt{X}}} \mathop{ \sum \sum}_{\substack{ c,d \in \mathbb{Z}[\omega] \\ c,d \equiv 1 \pmod{3} \\ N(c) \sim C,\ N(d) \sim D }} \int_{\varepsilon-i X^\varepsilon}^{\varepsilon+i X^\varepsilon} \sum_{ q \in \mathcal{F}_3(EB,A) }   |\widetilde{\mathcal{M}}(q) L(1/2+s,\chi_q )|^j |ds|,
\end{align}
where $\widetilde{\mathcal{M}}(q) = \widetilde{\mathcal{M}}_{c, d}(q) := \mathcal{M}(q (cd)^3)$. Indeed, when $q=q_1 q_2^2 \in \mathcal{F}_3$ with $q_1=eb$ and $q_2=a$ we have $\mathcal{M}(q (cd)^3) = \mathcal{M}(a^2 b^4 c^6 d^3 e)$ by \eqref{mollifier}. The contribution from the divisor bound incurred from clumping $e b$ as one variable was absorbed into $X^{\varepsilon}$.
 
\begin{proof}[Proof of \cref{SRestimate1}]
    Applying the Cauchy--Schwarz inequality, we see that the sum over $q$ in \eqref{boundintermed_new} for $j=1$ is 
    \begin{align} \label{qcauchy}
        \ll \Big(\sum_{ q \in \mathcal{F}_3(EB,A) }   |\widetilde{\mathcal{M}}(q)|^2  \Big)^{1/2} \Big( \sum_{ q \in \mathcal{F}_3(EB,A) }   |L(1/2+s,\chi_q )|^2 \Big)^{1/2}. 
    \end{align}
    
Consider the first term in the display above. We have
    \begin{align} 
        \sum_{ q \in \mathcal{F}_3(EB,A) }   |\widetilde{\mathcal{M}}(q)|^2 =  \sum_{ q \in \mathcal{F}_3(EB,A) }  \Big | \sum_{\substack{0 \neq \mathfrak{n} \unlhd \Z[\omega] \\ N(\mathfrak{n}) \leq M}} \lambda(\mathfrak{n}) \chi_{c^3 d^3} (\mathfrak{n}) \sqrt{N(\mathfrak{n})} \chi_{q}(\mathfrak{n})  \Big |^2. \nonumber 
    \end{align}
Applying \cref{cubic_Dirichlet_prop}, since $\lambda(\mathfrak{n}) \chi_{c^3 d^3} (\mathfrak{n}) N(\mathfrak{n}) \ll_\varepsilon N(\mathfrak{n})^\varepsilon$ by hypothesis, we obtain
    \begin{align}
         \sum_{ q \in \mathcal{F}_3(EB,A) }   |\widetilde{\mathcal{M}}(q)|^2 \ll X^\varepsilon \Big(ABE + \min(EB, A) M + \min(EB, A)^{1/3}(ABE M)^{2/3}\Big). \label{mollif_second_mom_bound1}
    \end{align}

 Similarly, by \cref{secondmoment} we have
    \begin{align}
        \sum_{ q \in \mathcal{F}_3(EB,A) }   |L(1/2+s,\chi_q )|^2 \ll X^\varepsilon \Big(\min(EB, A)(ABE)^{1/2} + \min(EB, A)^{1/3} ABE\Big). \label{L_second_mom_bound1}
    \end{align}
    Inserting \eqref{mollif_second_mom_bound1} and \eqref{L_second_mom_bound1} into \eqref{qcauchy}, and denoting $L := AB^2C^3D$, we conclude from \eqref{ell_m_second_mom} and \eqref{boundintermed_new} (for $j=1$) that  
    \begin{align}
        &\sum_{\substack{ Y<N(\ell) \leq \sqrt{2X} \\ \ell \equiv 1 \pmod{3}  } } \sum_{\substack{ N(m) \sim  X/N(\ell)^2 \\  \ell^2 m \equiv 1 \pmod{9} \\ \ell^2 m \neq\tinycube}} \mu^2(m) |\mathcal{M}(\ell^2 m ) A_1(\ell^2 m)| \label{A1_first_mom} \\
        & \ll X^\varepsilon \sum_{\substack{A,B,C,D,E \text{ dyadic} \\ A^2 B^4 C^6 D^3 E \asymp X \\ Y \ll L:= AB^2C^3D \ll \sqrt{X} \\ A \leq BE}} CD A^{\frac{2}{3}} (BE)^{\frac{1}{2}} \Big((ABE)^{\frac{1}{2}} + (AM)^{\frac{1}{2}} + A^{\frac{1}{2}} (BEM)^{\frac{1}{3}} \Big) \nonumber \\
        &  + X^\varepsilon \sum_{\substack{A,B,C,D,E \text{ dyadic} \\ A^2 B^4 C^6 D^3 E \asymp X \\ Y \ll L:= AB^2C^3D \ll \sqrt{X} \\ A \geq BE}} CD \Big( A^{\frac{1}{4}} (BE)^{\frac{3}{4}} + A^{\frac{1}{2}} (BE)^{\frac{2}{3}} \Big) \Big((ABE)^{\frac{1}{2}} + (BEM)^{\frac{1}{2}} + (BE)^{\frac{1}{2}} (AM)^{\frac{1}{3}} \Big). \nonumber
    \end{align}
    Observe that $ABE \asymp \frac{X}{LBC^3D^2}$ and $A = \frac{L}{B^2C^3D}$, so 
    \begin{equation*}
        A \leq BE \implies \frac{L}{B^2C^3D} = A \leq A^{2/3}(BE)^{1/3} \asymp \frac{X^{1/3}}{BC^2D} \implies L \ll X^{1/3}BC,
    \end{equation*}
    and conversely $A \geq BE \implies L \gg X^{1/3}BC$. Thus \eqref{A1_first_mom} is 
    \begin{align}
        & \ll X^\varepsilon \sum_{\substack{A,B,C,D,E \text{ dyadic} \\ A^2 B^4 C^6 D^3 E \asymp X \\ Y \ll L := AB^2C^3D \ll \sqrt{X} \\ A \leq BE}} \Big( \frac{X}{L^{\frac{5}{6}}} \frac{1}{B^{\frac{4}{3}}C^{\frac{5}{2}}D^{\frac{7}{6}}} + \frac{X^{\frac{1}{2}} M^{\frac{1}{2}} L^{\frac{1}{6}}}{B^{\frac{11}{6}}C^{\frac{5}{2}}D^{\frac{2}{3}}} + \frac{X^{\frac{5}{6}}M^{\frac{1}{3}}}{L^{\frac{1}{2}}}\frac{1}{B^{\frac{3}{2}}C^{\frac{5}{2}}D} \Big) \nonumber \\
        &  + X^\varepsilon \sum_{\substack{A,B,C,D,E \text{ dyadic} \\ A^2 B^4 C^6 D^3 E \asymp X \\ Y \ll L:= AB^2C^3D \ll \sqrt{X} \\ A \geq BE}} CD \Big( \Big(\frac{L}{B^2C^3D}\Big)^\frac{1}{4} \Big(\frac{X}{L^2}\Big)^{\frac{3}{4}} \Big(\frac{B}{D}\Big)^{\frac{3}{4}} + \Big(\frac{L}{B^2C^3D}\Big)^\frac{1}{2} \Big(\frac{X}{L^2}\Big)^{\frac{2}{3}} \Big(\frac{B}{D}\Big)^{\frac{2}{3}} \Big) \nonumber \\
        & \qquad \qquad \times \Big( \Big(\frac{X}{LBC^3D^2}\Big)^\frac{1}{2} + \Big(\frac{X}{L^2}\Big)^{\frac{1}{2}} M^\frac{1}{2} \Big(\frac{B}{D}\Big)^{\frac{1}{2}} +  \Big(\frac{X}{L^2}\Big)^{\frac{1}{2}} M^\frac{1}{3} \Big(\frac{B}{D}\Big)^{\frac{1}{2}} \Big(\frac{L}{B^2C^3D}\Big)^\frac{1}{3} \Big).   \nonumber 
    \end{align}
    Using $Y \ll L \ll X^{1/3}BC$, since all of the ranges are $\gg 1$ the first sum is 
    \begin{equation*}
        \ll X^\varepsilon \Big( \frac{X}{Y^{5/6}} + X^{5/9} M^{1/2} + \frac{X^{5/6}M^{1/3}}{Y^{1/2}} \Big).
    \end{equation*}
    For the second sum, we use $B \ll \frac{L}{C X^{1/3}}$ (and therefore $L \gg X^{1/3}$) to see that it is
    \begin{align*}
        \ll X^\varepsilon & \sup_{X^{1/3} \ll L \ll X^{1/2}} \Big( L^\frac{1}{4} \Big(\frac{X}{L^2}\Big)^{\frac{3}{4}} \Big(\frac{L}{X^\frac{1}{3}}\Big)^{\frac{1}{4}} + L^\frac{1}{2} \Big(\frac{X}{L^2}\Big)^{\frac{2}{3}} \Big) \nonumber \\
        & \qquad \times \Big( \Big(\frac{X}{L}\Big)^\frac{1}{2} + \Big(\frac{X}{L^2}\Big)^{\frac{1}{2}} M^\frac{1}{2} \Big(\frac{L}{X^\frac{1}{3}}\Big)^{\frac{1}{2}} +  \Big(\frac{X}{L^2}\Big)^{\frac{1}{2}} M^\frac{1}{3} L^\frac{1}{3} \Big) \nonumber \\
        \ll X^\varepsilon & ( X^{13/18} + X^{5/9} M^{1/2} + X^{2/3} M^{1/3} ). \nonumber 
    \end{align*}

    Combining these two bounds with \eqref{SRexpansion_new}, we obtain the desired estimate
    \begin{align*}
        \mathcal{S}_{R}(|\mathcal{M}(q) A_1(q)|; F) \ll X^\varepsilon \Big( \frac{X}{Y^{5/6}} + X^{13/18} + \frac{X^{5/6}M^{1/3}}{Y^{1/2}} + X^{2/3} M^{1/3} + X^{7/12} M^{1/2} \Big).
    \end{align*}
  
\end{proof}

\begin{proof}[Proof of \cref{SRestimate2}]
    Applying \cref{secondmoment} (with $R = M$ and $h=1$) we see that \eqref{boundintermed_new} for $j=2$ is 
    \begin{align} \label{boundintermed2}
        & \ll  X^{\varepsilon} \sum_{\substack{ A,B,C,D,E \text{ dyadic} \\ A^2 B^4 C^6 D^3 E \asymp X \\ Y\ll AB^2C^3D \ll \sqrt{X}}} CD \Big(\min(EB, A) (ABE)^{1/2}M + \min(EB, A)^{1/3} ABE M^{2/3}\Big) \nonumber\\
        &\ll X^{\varepsilon} \sum_{\substack{ A,B,C,D,E \text{ dyadic} \\ A^2 B^4 C^6 D^3 E \asymp X \\ Y\ll AB^2C^3D \ll \sqrt{X}}} CD \Big(A^{3/2}(BE)^{1/2} M + A^{1/3} \frac{X}{AB^3 C^6 D^{3}} M^{2/3} \Big) \nonumber \\
        & \ll X^{\varepsilon} \sum_{\substack{ A,B,C,D,E \text{ dyadic} \\ A^2 B^4 C^6 D^3 E \asymp X \\ Y\ll AB^2C^3D \ll \sqrt{X}}} \Big( (A^2B^4C^6D^3E)^{3/4} M + \frac{XM^{2/3}}{(AB^2C^3D)^{2/3}} \Big) \nonumber\\
        &\ll X^{\varepsilon} \Big( X^{3/4} M + \frac{XM^{2/3}}{Y^{2/3}} \Big).\nonumber
    \end{align}
    
    Therefore \eqref{SRexpansion_new} gives
    \begin{align*}
        \mathcal{S}_{R}(|\mathcal{M}(q)^2 A_2(q)|; F) &\ll X^\varepsilon \Big(X^{5/6}M + X^{3/4}M + X\Big(\frac{M}{Y}\Big)^{2/3}\Big) \\
        &\ll X^\varepsilon \Big(X^{5/6}M + X\Big(\frac{M}{Y}\Big)^{2/3}\Big),
    \end{align*}
    as desired.
    
\end{proof}


\section{Twisted sums of cubic Gauss sums} \label{Gauss_sums_section}

In this section our goal is to compute certain sums of twisted cubic Gauss sums $g_3(r, c)$. In our upcoming application it is crucial to deal with $r$ which can be a multiple of $\lambda$ and not necessarily squarefree. While there is a vast literature on the subject, we could not find a reference in this generality. This section works out such a general result using (by now) standard tools, ultimately relying on the foundational work of Patterson \cite{Pat1}. We will frequently use \cref{cubic-GS-lemma1} and \cref{localcomp} without further reference. 

For any $r, \alpha \in \Z[\omega]$ with $\alpha\equiv 1 \pmod{3}$ and squarefree, denote
\begin{equation} \label{psizetadef}
    \psi_\alpha(r, s) := \sum_{\substack{c \in \Z[\omega] \\ c\equiv 1 \pmod{3} \\ (c, \alpha)=1}} \frac{g_3(r, c)}{N(c)^s} \qquad  \text{and} \qquad \zeta_{\lambda}(s) := \sum_{\substack{c\in\Z[\omega] \\ c\equiv 1 \pmod{3}}} \frac{1}{N(c)^s},
\end{equation}
which converge absolutely if $\Re(s) > \frac{3}{2}$ and $\Re(s) > 1$, respectively. Also let 
\begin{equation} \label{psirsdef}
    \psi(r, s) := \psi_1(r, s) \qquad \text{and} \qquad \Delta_\alpha(s) := \prod_{\substack{\pi \text{ prime} \\ \pi\equiv 1 \pmod{3} \\ \pi \mid \alpha}} (1 - N(\pi)^{2 - 3s} ).
\end{equation}
It is convenient to express $\psi_\alpha$ in terms of the simpler $\psi$. This can be achieved via the following relations, which generalize \cite[Lemma 3]{HBP}.

\begin{lemma}[Removal of coprimality conditions] \label{psi_coprimality_lemma}
    For any $s\in \C$ with $\Re(s) > \frac{3}{2}$ and $\alpha, \beta, r \in \Z[\omega]$ satisfying $\alpha, \beta \equiv 1 \pmod{3}$, $\mu^2(\alpha)=1$, and $(\alpha, \beta r)$ = 1, we have:
    \begin{itemize}
        \item[(i)]  
        \begin{align*}
            \psi_{\alpha \beta}( \alpha^2 r, s) \Delta_\alpha(s) = \psi_{\beta} (\alpha^2 r, s);
        \end{align*}

        \item[(ii)]  
        \begin{align*}
            \psi_{\alpha \beta}( \alpha r, s) \Delta_\alpha(s) = \sum_{\substack{d \in \Z[\omega] \\ d \equiv 1 \pmod{3} \\ d \mid \alpha}} \mu(d) N(d)^{1 - 2s} \overline{g_3(\alpha r / d, d)} \psi_{\beta}(\alpha r / d, s);
        \end{align*}

        \item[(iii)]
        \begin{align*}
            \psi_{\alpha \beta}(r, s) \Delta_\alpha(s) = \sum_{\substack{d \in \Z[\omega] \\ d \equiv 1 \pmod{3} \\ d \mid \alpha}} \mu(d) N(d)^{-s} g_3(r, d) \psi_\beta(rd, s);
        \end{align*}
    \end{itemize}
  where $\psi_\alpha(r, s)$ is given in \eqref{psizetadef} and \eqref{psirsdef}. 
\end{lemma}

\begin{proof}
    Fix $r$ and $\beta$ as above. The proof is by induction on the number of prime factors of $\alpha$, with the base case $\alpha=1$ being trivial for each item. 
    
    Observe that if $\gamma, \pi \equiv 1 \pmod{3}$ and $\pi$ is a prime with $\pi \nmid \gamma$, then for any $\rho \in \Z[\omega]$ we have
    \begin{align*}
        \psi_{\gamma} (\rho, s) = \sum_{\ell=0}^\infty \sum_{\substack{c \in \Z[\omega] \\ c\equiv 1 \pmod{3} \\ (c, \pi\gamma) = 1}} \frac{g_3(\rho, c \pi^\ell)}{N(c\pi^\ell)^s}.
    \end{align*}
    Write $\rho = \delta \pi^k$ for $\delta \in \Z[\omega]$ with $(\delta, \pi)=1$. Using \cref{localcomp}, observe that if $k=0$ then only $\ell \in\{ 0, 1\}$ contribute, while if $k=1$ only $\ell \in\{ 0, 2\}$ contribute, and if $k = 2$ only $\ell \in\{ 0, 3\}$ contribute. Furthermore we have
    \begin{align*}
        g_3(\delta, c \pi) = g_3(\delta, \pi) g_3(\delta \pi, c), 
    \end{align*}
    \begin{align*}
        g_3(\delta \pi, c \pi^2) = g_3(\delta \pi, \pi^2) g_3(\delta \pi^3, c) = \chi_\pi(\delta) N(\pi) \overline{g_3(\pi)} g_3(\delta, c) = N(\pi) \overline{g_3(\delta, \pi)} g_3(\delta, c),
    \end{align*}
    \begin{align*}
        g_3(\delta \pi^2, c \pi^3) = g_3(\delta \pi^2, \pi^3) g_3(\delta \pi^5, c) = g_3(\pi^2, \pi^3) g_3(\delta \pi^2, c) = -N(\pi)^2 g_3(\delta \pi^2, c).
    \end{align*}
    These observations lead respectively to
    \begin{align}\label{pi_0_transform}
        \psi_{\gamma}(\delta, s) = \psi_{\gamma \pi} (\delta, s) + \frac{g_3(\delta, \pi)}{N(\pi)^s} \psi_{\gamma \pi} (\delta \pi, s),
    \end{align}
    \begin{align}\label{pi_1_transform}
        \psi_{\gamma}(\delta \pi, s) = \psi_{\gamma \pi} (\delta \pi, s) + \frac{N(\pi) \overline{g_3(\delta, \pi)}}{N(\pi)^{2s}} \psi_{\gamma \pi} (\delta, s),
    \end{align}
    \begin{align}\label{pi_2_transform}
        \psi_{\gamma}(\delta \pi^2, s) = \psi_{\gamma \pi} (\delta \pi^2, s) - \frac{N(\pi)^2}{N(\pi)^{3s}} \psi_{\gamma \pi} (\delta \pi^2, s) = \Delta_\pi(s)\psi_{\gamma \pi} (\delta \pi^2, s).
    \end{align}

    The proof of item (i) follows inductively from \eqref{pi_2_transform} for $(\gamma, \delta) = (\alpha \beta, \alpha^2 r)$. For the other items we first observe that since $|g_3(\delta, \pi)|^2 = N(\pi)$, combining \eqref{pi_0_transform} and \eqref{pi_1_transform} gives
    \begin{align}\label{remove_pi_both}
        \psi_{\gamma\pi} (\delta \pi, s) \Delta_\pi(s) = \psi_\gamma(\delta\pi, s) - \frac{N(\pi) \overline{g_3(\delta, \pi)}}{N(\pi)^{2s}} \psi_\gamma(\delta, s)
    \end{align}
    and
    \begin{align}\label{remove_pi_single}
        \psi_{\gamma\pi} (\delta, s) \Delta_\pi(s) = \psi_\gamma(\delta, s) - \frac{g_3(\delta, \pi)}{N(\pi)^{s}} \psi_\gamma(\delta \pi, s).
    \end{align}

    Then the proof of items (ii) and (iii) follows from \eqref{remove_pi_both} and \eqref{remove_pi_single} for $(\gamma, \delta) = (\alpha \beta, \alpha r)$ and $(\alpha \beta, r)$, respectively, using the induction hypothesis and assembling the divisors of $\pi \alpha$ from those of the form $d$ and those of the form $d\pi$, for $d \mid \alpha$. 
    
\end{proof}

\begin{corollary}[Removal of coprimality for cube-free twists]\label{general_coprimality_cor}
     For any $s\in \C$ with $\Re(s) > \frac{3}{2}$ and $a, b, c, r \in \Z[\omega]\setminus\{0\}$ satisfying $a, b, c \equiv 1 \pmod{3}$, $\mu^2(abc)=1$, and $(abc, r) = 1$, we have
     \begin{equation*}
         \psi_{abc}(ab^2 r, s) = \Delta_{abc}(s)^{-1} \mathop{\sum \sum}_{\substack{d, e \in \Z[\omega] \\ d, e\equiv 1\pmod{3} \\ d\mid a,\ e\mid c}} \mu(de) \frac{N(d)}{N(d^2e)^{s}} \overline{g_3\Big(\frac{ab^2r}{d}, d\Big)} g_3\Big(\frac{ab^2r}{d}, e\Big) \psi\Big(\frac{ab^2er}{d}, s\Big). 
     \end{equation*}
\end{corollary}

\begin{proof}
    By \cref{psi_coprimality_lemma} (i) we have
    \begin{equation*}
        \psi_{abc}(ab^2r, s) = \Delta_b(s)^{-1} \psi_{ac}(ab^2r, s). 
    \end{equation*}
    Then \cref{psi_coprimality_lemma} (ii) and the identity $\Delta_{ab}(s) = \Delta_{a}(s)\Delta_{b}(s)$ (since $(a, b)=1$) give
    \begin{equation*}
        \psi_{abc}(ab^2r, s) = \Delta_{ab}(s)^{-1} \sum_{\substack{d\in\Z[\omega] \\ d\equiv 1 \pmod 3 \\ d\mid a}} \mu(d) N(d)^{1-2s} \overline{g_3(ab^2r/d, d)} \psi_c(ab^2r/d, s).
    \end{equation*}
    A final application of \cref{psi_coprimality_lemma} (iii) then implies the desired result.
    
\end{proof}

For $r \in \Z[\omega]$, define
\begin{equation}\label{tau_fourier_coeff_def}
    \tau_3(r) := 
    \begin{cases}
        \overline{g_3(\lambda^2, c)} \: |\frac{d}{c} | \: 3^{\frac{n}{2}+2}  & \qquad \text{if } r = \pm \lambda^{3n-4} cd^3 \text{ for } n\geq 2, \\
        e ({-\frac{1}{9}}) \: \overline{g_3(\omega \lambda^2, c)} \: |\frac{d}{c} | \: 3^{\frac{n}{2}+2}  & \qquad \text{if } r = \pm \omega \lambda^{3n-4} cd^3 \text{ for } n\geq 2, \\
        e(\frac{1}{9}) \: \overline{g_3(\omega^2 \lambda^2, c)} \: |\frac{d}{c} | \: 3^{\frac{n}{2}+2}  & \qquad \text{if } r = \pm \omega^2 \lambda^{3n-4} cd^3 \text{ for } n\geq 2, \\
        \overline{g_3(1, c)} \: |\frac{d}{c} | \: 3^{\frac{n+5}{2}}  & \qquad \text{if } r = \pm \lambda^{3n-3} cd^3 \text{ for } n\geq 1, \\
        0  & \qquad \text{otherwise}, \\
    \end{cases}
\end{equation}
where $c, d \in \Z[\omega]$ satisfy $c \equiv d \equiv 1 \pmod{3}$ and $\mu^2(c)=1$. We remark that extending $\tau_3(r)$ to all $r \in \lambda^{-3} \mathbb{Z}[\omega]$ as in \cite[Theorem 8.1]{Pat1} would give the Fourier coefficients of the cubic theta function.

We can now collect some basic facts on the poles and convexity bound for the function $\psi$, in a straightforward generalization of \cite[Lemma 4]{HBP}, whose proof we mostly follow. 

\begin{lemma}[Poles and convexity bound] \label{psi_convexity_lemma}
    For any $0 \neq r \in \Z[\omega]$, the function $\psi(r, s)$ has meromorphic continuation in $s$ to all of $\C$. It is holomorphic for $\Re(s) > 1$, except for at most a simple pole at $s = \frac{4}{3}$, with
    \begin{equation*}
        \mathop{\mathrm{Res}}_{s=\frac{4}{3}} \psi(r, s) = \frac{c_0 \cdot \tau_3(r)}{N(r)^{1/6}}
    \end{equation*}
    for an absolute constant $c_0$ (given in \eqref{c_0_def}). Furthermore, we have the bound
    \begin{equation*}
        \psi(r, s) \ll_{\varepsilon} N(r)^{\frac{3}{4} - \frac{\Re(s)}{2} + \varepsilon} (1+|\Im(s)|)^{3 - 2\Re(s) + \varepsilon}
    \end{equation*}
    uniformly if $\frac{3}{2} \geq \Re(s) \geq 1+\varepsilon$ and $ |s-\frac{4}{3} | \geq \varepsilon$.
\end{lemma}

\begin{remark}
Note that the convexity bound for $\psi(r, s)$ established in \cite[Lemma 4]{HBP} is only valid for $r \in \Z[\omega]$ such that 
$r \equiv 1 \pmod{3}$. Additional technicalities arise when this congruence condition is relaxed. 
\end{remark}

\begin{proof}
    Let 
    \begin{equation*}
        G(s) := (2\pi)^{-2s} \Gamma\Big(s-\frac{1}{3}\Big) \Gamma\Big(s-\frac{2}{3}\Big)
    \end{equation*}
    and 
    \begin{equation*}
        F(r, s) := G(s) \psi(r, s) \zeta_{\lambda}(3s-2).
    \end{equation*}
    
    The meromorphic continuation and pole structure of $\psi(r, s)$ follows from the corresponding result for $F(r, s)$ given in \cite[Theorem 6.1]{Pat1}, and the fact that
    $G(s) \zeta_{\lambda}(3s-2)$ is both holomorphic and zero-free in the half-plane $\Re(s)>1$. Note that
    \begin{equation}\label{zeta_lambda_identity}
    \zeta_{\lambda}(s) = (1-3^{-s}) \cdot \zeta_{\Q(\omega)}(s),
    \end{equation}
    where $\zeta_{\Q(\omega)}(s)$ denotes the Dedekind zeta function of $\Q(\omega)$.
    
     The evaluation of $\mathop{\mathrm{Res}}_{s=\frac{4}{3}} \psi(r, s)$ is derived from \cite[Theorem 9.1]{Pat1}, which combined with \cite{Pat5} (see also \cite{Pat6}) in fact gives
    \begin{equation}\label{c_0_def}
        c_0 = \Big(2 \cdot 3^{13/2} \cdot \pi \cdot G(4/3)\cdot \zeta_{\lambda}(2)\Big)^{-1} = \frac{(2\pi)^{5/3}}{8\cdot 3^{9/2} \cdot \Gamma(2/3) \cdot \zeta_{\Q(\omega)}(2)}.
    \end{equation}

    We are left with proving the desired (convexity) bound for $\psi(r, s)$. Fix $2>\sigma_0 > \frac{3}{2}$ and observe that if $\Re(s)\geq \sigma_0$ then by absolute convergence and multiplicativity of $|g_3(r, c)|$ in $c$ we have the uniform bounds $\zeta_{\lambda}(3s-2) \ll 1$ and
    \begin{align*}
        |\psi(r, s)| \leq \prod_{\substack{\pi \text{ prime} \\ \pi\equiv 1 \pmod{3}}} \sum_{k=0}^\infty \frac{|g_3(\pi^{\nu_\pi(r)}, \pi^k)|}{N(\pi)^{k\sigma_0}} \leq \prod_{\substack{\pi \text{ prime} \\ \pi\equiv 1 \pmod{3}}} \Big(1 + \frac{N(\pi)^{1/2}}{N(\pi)^{\sigma_0}} + \sum_{k=2}^\infty \frac{N(\pi)^k}{N(\pi)^{k\sigma_0}} \Big) \ll_{\sigma_0} 1.
    \end{align*}
    
    By \cite[Theorem 6.1]{Pat1} there is a functional equation 
    \begin{align}\label{F_functional_eq}
        F(r, s) =  N(r)^{1-s} \Big[A(s)\cdot F_\infty(r, 2-s) + B(s) \cdot F(r, 2-s) \Big],
    \end{align}
    where 
    \begin{align*}
        A(s) := \frac{3^{8-9s} (1-3^{3-3s})}{1-3^{3s-4}}, \qquad  \qquad B(s) := \frac{2\cdot 3^{8-9s}}{1-3^{3s-4}},
    \end{align*}
    and
    \begin{equation}\label{F_infty_def}
        F_\infty(r, s) := \sum_{\nu^6 = 1} \sum_{b=2}^{\nu_{\lambda}(r) + 5} \Gamma(r, \nu \lambda^b) 3^{-bs} F(\nu \lambda^b r, s)
    \end{equation}
    for certain coefficients given in \cite[Proposition 5.1]{Pat1} which satisfy $|\Gamma(r, \nu \lambda^b)| \leq 3^b$. Together with Stirling's formula, this shows that if $\Re(s) = 2 - \sigma_0 < \frac{1}{2}$ then
    \begin{equation*}
        \psi(r, s) \zeta_{\lambda}(3s-2) \ll_{\sigma_0} N(r)^{\sigma_0 - 1} \Big|\frac{G(2-s)}{G(s)}\Big| \ll N(r)^{\sigma_0 - 1} (1+|\Im(s)|)^{4(\sigma_0 - 1)}.
    \end{equation*}

    We can now apply the Phragm{\'e}n-Lindel{\"o}f principle to the function
    \begin{equation*}
        s^{-1} (s-4/3 ) \psi(r, s) \zeta_{\lambda}(3s-2),
    \end{equation*}
    which is holomorphic for $2-\sigma_0 \leq \Re(s) \leq \sigma_0$, due once again to \cite[Proposition 6.1]{Pat1}, where it is clear from the proof that the pole of $F(r, s)$ at $s=\frac{2}{3}$ is at most simple, so can only come from $G(s)$. In conclusion,
    \begin{equation}\label{Z_convexity_implicit}
        \psi(r, s) \zeta_{\lambda}(3s-2) \ll_{\sigma_0} \frac{|s|}{|s-4/3 |} N(r)^{\frac{\sigma_0 - \Re(s)}{2}} (1+|\Im(s)|)^{2(\sigma_0 - \Re(s))}
    \end{equation}
    for $2-\sigma_0 \leq \Re(s) \leq \sigma_0$. Taking $\sigma_0 = \frac{3}{2}+\frac{\varepsilon}{2}$ and observing that $\zeta_{\lambda}(3s-2) \gg_\varepsilon 1$ for $\Re(s) \geq 1+\varepsilon$ gives the desired result.
    
\end{proof}

The convexity bound of \cref{psi_convexity_lemma} can be significantly refined when averaging over $r$. One can obtain a Lindel{\"o}f-type bound on average for the second moment in the twist aspect. Before proving that, we give a technical manipulation which is useful for bounding sums of twisted cubic Gauss sums.

\begin{lemma}\label{technical_lemma_split_gauss_sums}
    Let $H :\R \to \R$ be a Schwartz function and $k \in \Z[\omega]\setminus\{0\}$. For any $t\in \R$, $\sigma \geq 1$, and $C>0$ we have
    \begin{align}
        &\Big|\sum_{\substack{c \in \Z[\omega] \\ c \equiv 1 \pmod{3}}} \frac{g_3(k, c)}{N(c)^{\sigma+it}} H\Big(\frac{N(c)}{C}\Big) \Big|^2 \label{cubic_gauss_sum_technical_lemma_eq} \\
        &\ll_\varepsilon N(k)^\varepsilon \sum_{\substack{w \in \Z[\omega] \\ w \equiv 1 \pmod{3} \\ w \mid k^2}} \frac{1}{N(w)^{2\sigma-2}} 
        \Big|\sum_{\substack{n \in \Z[\omega] \\ n \equiv 1  \pmod{3}}} \mu^2(n) \frac{\overline{\chi_n(kw)} g_3(n)}{N(n)^{\sigma+it}} H\Big(\frac{N(wn)}{C}\Big) \Big|^2. &\nonumber
    \end{align}
\end{lemma}

\begin{proof}
Writing (uniquely) $c = wn$ for $w, n\equiv 1\pmod{3}$ with $w \mid k^\infty$ and $(n, k)=1$, observe that 
\begin{equation*}
    g_3(k, wn) = g_3(k, w) \overline{\chi_n(w)} g_3(k, n) =  g_3(k, w) \overline{\chi_n(kw)} g_3(n).
\end{equation*}
Notice that $w \mid k^\infty$ can be restricted to $w \mid k^2$, otherwise $g_3(k, w)=0$ by \cref{localcomp}.

 Thus the quantity in \eqref{cubic_gauss_sum_technical_lemma_eq} is equal to
\begin{align*}
    &\Big|\sum_{\substack{w \in \Z[\omega] \\ w \equiv 1 \pmod{3} \\ w\mid k^\infty}} \sum_{\substack{n \in \Z[\omega] \\ n \equiv 1  \pmod{3} \\ (n, k) = 1}} \frac{g_3(k, wn)}{N(wn)^{\sigma+it}} H\Big(\frac{N(wn)}{C}\Big) \Big|^2 \leq \Big(\sum_{\substack{w \in \Z[\omega] \\ w \equiv 1 \pmod{3} \\ w\mid k^2}} \frac{|g_3(k, w)|^2}{N(w)^2} \Big) \\
    &\qquad \times \Big(\sum_{\substack{w \in \Z[\omega] \\ w \equiv 1 \pmod{3} \\ w\mid k^2}} \frac{1}{N(w)^{2\sigma-2}} \Big|\sum_{\substack{n \in \Z[\omega] \\ n \equiv 1  \pmod{3}}} \frac{\overline{\chi_n(kw)} g_3(n)}{N(n)^{\sigma+it}} H\Big(\frac{N(wn)}{C}\Big) \Big|^2\Big), 
\end{align*}
where we removed the condition $(n, k)=1$ since it is enforced by $\overline{\chi_n(k)}$. Observe that $|g_3(k, w)| \leq N(w)$, so the first sum over $w$ is $\leq d(k^2) \ll N(k)^\varepsilon$. The result follows after observing that the term $g_3(n)$ forces $n$ to be squarefree.

\end{proof}

We are finally ready to prove the main estimate for this section.

\begin{lemma}[Lindel{\"o}f on average for second moment]\label{second_moment_lindelof_lemma}
    For any $0 \neq h \in \Z[\omega]$, $t\in \R$, $1 < \sigma \leq \frac{5}{4}$, and $M \geq 1$, we have the bound
    \begin{equation*}
        \sum_{\substack{0\neq m \in \Z[\omega] \\ N(m) \leq M}} |\psi(hm, \sigma + it)|^2 \ll_{\sigma, \varepsilon} N(h)^{\frac{1}{2}} M^{1+\varepsilon} (1+|t|)^{2}.
    \end{equation*}
\end{lemma}

\begin{proof}
           We initially follow the proof of \cite[Lemma 3]{HB}. For $0 \neq r \in \Z[\omega]$ and $\Re(s) > \frac{3}{2}$, denote
    \begin{equation} \label{convolution}
        Z(r, s) := \zeta_{\lambda}(3s-2) \psi(r, s) = \sum_{\substack{c, d \in \Z[\omega] \\ c,d \equiv 1 \pmod 3}}  \frac{g_3(r, c) N(d)^2}{N(cd^3)^s} =: \sum_{n=1}^\infty a_n(r) n^{-s}, 
    \end{equation}
    which by the discussion in \cref{psi_convexity_lemma} is holomorphic apart from at most a simple pole at $s=\frac{4}{3}$. Since $Z(\lambda^3r, s) = Z(r, s)$, observe that the functional equation \eqref{F_functional_eq} allows us to write
    \begin{equation*}
        F(r, s) = N(r)^{1-s} \sum_{\eta \mid \lambda^2}  c_\eta(r, 2-s) F(\eta r, 2-s),
    \end{equation*} 
    where by \eqref{F_infty_def} the coefficients satisfy
    \begin{equation*}
        c_\eta(r, 2-s) \ll \sum_{b=0}^{\infty} 3^{b(\Re(s)-1)} \ll_{\Re(s)} 1
    \end{equation*}
    for $\frac{3}{4} \leq \Re(s) < 1$. Setting
    \begin{equation*}
        \widetilde{Z}(r, s) := \sum_{\eta \mid \lambda^2} |Z(\eta r, s)|^2, 
    \end{equation*}
    we conclude (using Stirling's formula) that if $\frac{3}{4} \leq \Re(s) < 1$ then
    \begin{align}
        \widetilde{Z}(r, s) &\ll_{\Re(s)} N(r)^{2-2\Re(s)} \Big|\frac{G(2-s)}{G(s)}\Big|^2 \widetilde{Z}(r, 2-s) \nonumber \\
        &\ll N(r)^{2-2\Re(s)} (1+|\Im(s)|)^{8-8\Re(s)} \widetilde{Z}(r, 2-s) .\label{func_eq_bound_Z_tilde}
    \end{align}
    
    For $X \geq 1$ and $s = \sigma+it$ (where we recall that $1 < \sigma \leq \frac{5}{4}$), observe that
    \begin{equation*}
        \sum_{n=1}^\infty a_n(r) n^{-s} e^{-n/X} = \frac{1}{2\pi i} \int_{2-i\infty}^{2+i\infty} Z(r, s+w) X^w \Gamma(w) dw.
    \end{equation*}
    Moving the line of integration to $\Re(w) = 2-2\sigma$, we pick up a simple pole at $w=0$ with residue $Z(r, s)$, and possibly a simple pole at $w = \frac{4}{3}-s$ with residue
    \begin{equation*}
        c_0 \cdot \zeta_{\lambda}(2)\frac{\tau_3(r)}{N(r)^{1/6}} X^{4/3-s}\Gamma(4/3-s) \ll X^{4/3-\sigma} e^{-|t|} \frac{|\tau_3(r)|}{N(r)^{1/6}}.
    \end{equation*}
    
    Taking $r = \eta hm$ for each $\eta\mid \lambda^2$ gives
    \begin{equation*}
        \widetilde{Z}(hm, s) \ll T_1 + T_2 + T_3
    \end{equation*}
    for
    \begin{align*}
        T_1 := X^{8/3-2\sigma} e^{-2|t|} \sum_{\eta\mid \lambda^2} \frac{|\tau_3(\eta hm)|^2}{N(\eta hm)^{1/3}},  \qquad T_2 := \sum_{\eta\mid \lambda^2} \Big|\sum_{n=1}^\infty a_n(\eta h m) n^{-s} e^{-n/X}\Big|^2, 
    \end{align*}
    and
    \begin{align*}
        T_3 := \sum_{\eta\mid \lambda^2}\Big|\int_{2-2\sigma -i\infty}^{2-2\sigma +i\infty} Z(\eta hm, s+w) X^w \Gamma(w) dw \Big|^2.
    \end{align*}
    Using \eqref{func_eq_bound_Z_tilde} and the decay of the Gamma function, we conclude that
    \begin{align}
        T_3 &\ll X^{4-4\sigma}\int_{2-2\sigma -i\infty}^{2-2\sigma +i\infty} \widetilde{Z}(hm, s+w) |\Gamma(w)| |dw| \nonumber \\
        &\ll_\sigma X^{4-4\sigma} N(hm)^{2\sigma-2} \int_{-\infty}^{\infty} (1+|t+y|)^{8\sigma-8} \widetilde{Z}(hm, \sigma - it-iy) e^{-|y|} dy \nonumber \\
        & \ll  \Big(\frac{N(h)M(1+|t|)^4}{X^2}\Big)^{2(\sigma-1)} \int_{-\infty}^{\infty} \widetilde{Z}(hm, \sigma - it-iy) e^{-|y|/2} dy.\label{T3_first_step}
    \end{align}

    Now we introduce the sum over $m$, so
    \begin{equation}\label{basic_Z_S_triple_bound}
        \sum_{\substack{0\neq m \in \Z[\omega] \\ N(m) \leq M}} \widetilde{Z}(hm, s) \ll S_1 + S_2 + S_2 \qquad\text{for} \qquad S_i := \sum_{\substack{0\neq m \in \Z[\omega] \\ N(m) \leq M}} T_i.
    \end{equation}
    
    Let us first treat $S_1$. For $r \in \Z[\omega]$, observe from \eqref{tau_fourier_coeff_def} that $\tau_3(r) = 0$ unless $r = \nu \lambda^\ell cd^3$ for some $\nu^6=1$, $\ell\in\Z_{\geq 0}$, and $c\equiv d \equiv 1 \pmod{3}$ with $c$ squarefree, in which case $|\tau_3(r)|^2 N(r)^{-1/3} \ll N(c)^{-1/3}$. Thus writing $h = \xi \lambda^b h_1$ for $\xi^6=1$, $b\in\Z_{\geq 0}$, and $h_1 \equiv 1\pmod{3}$, we have
    \begin{align*}
        S_1 &\ll X^{8/3-2\sigma} \sum_{\substack{0\neq m \in \Z[\omega] \\ N(m) \leq 9 M}} \frac{|\tau_3(hm)|^2}{N(hm)^{1/3}} \ll X^{8/3-2\sigma} \sum_{\ell=b}^\infty \sum_{\substack{d\in \Z[\omega]\\ d\equiv 1\pmod{3}}} \sum_{\substack{c\in \Z[\omega] \\ c\equiv 1 \pmod{3} \\ N(c) \leq \frac{9 N(h) M }{3^\ell N(d)^3} \\ h_1 \mid c d^3}} \frac{1}{N(c)^{1/3}} \\
        & \ll X^{8/3-2\sigma} \sum_{\ell=b}^\infty \sum_{\substack{d\in \Z[\omega]\\ d\equiv 1\pmod{3}}} \frac{N((h_1, d^3))}{N(h_1)}\Big(\frac{N(h_1)M}{3^{\ell-b} N(d)^3}\Big)^{2/3} \\
        &\ll \frac{X^{8/3-2\sigma} M^{2/3}}{N(h_1)^{1/3}} \sum_{\substack{d\in \Z[\omega]\\ d\equiv 1\pmod{3}}} \frac{N((h_1, d^3))}{N(d)^2}.
    \end{align*}
    Note that
    \begin{align*}
        \sum_{\substack{d\in \Z[\omega]\\ d\equiv 1\pmod{3}}} \frac{N((h_1, d^3))}{N(d)^2} &\leq \prod_{\substack{\pi \text{ prime} \\ \pi \equiv 1 \pmod 3 \\ \pi \mid h_1}}  \Big( \sum_{0 \leq 3 \ell < \nu_\pi(h_1)} {N(\pi)^\ell} + \sum_{3 \ell \geq \nu_\pi(h_1)} N(\pi)^{\nu_\pi(h_1)-2 \ell} \Big) \\
        &\leq  \prod_{\substack{\pi \text{ prime} \\ \pi \equiv 1 \pmod 3 \\ \pi \mid h_1}}  \Big(4 N(\pi)^{\frac{\nu_\pi(h_1)}{3}}\Big)  \ll_\varepsilon N(h_1)^{\frac{1}{3}+\varepsilon},
    \end{align*}
    so we obtain
    \begin{align}
        S_1 \ll_\varepsilon N(h)^\varepsilon M^{2/3} X^{8/3-2\sigma}. \label{S1_final_bound}
    \end{align}
    
    Next we consider $S_2$. From \eqref{convolution} and Cauchy--Schwarz (as $\sigma > 1$) we have
    \begin{align*}
        S_2 &\ll \sum_{\substack{0\neq m \in \Z[\omega] \\ N(m) \leq 9 M}} \Big| \sum_{\substack{c, d \in \Z[\omega] \\ c, d\equiv 1\pmod{3}}} \frac{g_3(hm, c) N(d)^2}{N(cd^3)^{s}} e^{-N(cd^3)/X}\Big|^2 \\
        &\ll_\sigma \sum_{\substack{d \in \Z[\omega] \\ d\equiv 1\pmod{3}}}\frac{1}{N(d)^\sigma} \sum_{\substack{0\neq m \in \Z[\omega] \\ N(m) \leq 9 M}} \Big| \sum_{\substack{c \in \Z[\omega] \\ c\equiv 1\pmod{3}}} \frac{g_3(hm, c) }{N(c)^{s}} e^{-N(cd^3)/X}\Big|^2
    \end{align*}
    Applying \cref{technical_lemma_split_gauss_sums} with $k=hm$ and $C=X/N(d^3)$ gives
    \begin{align*}
        S_2 \ll_{\sigma, \varepsilon} & \ (N(h)M)^\varepsilon \sum_{\substack{d, w \in \Z[\omega] \\ d, w\equiv 1\pmod{3}}}\frac{1}{N(d)^\sigma N(w)^{2\sigma-2}} \\
        &\times\sum_{\substack{0\neq m \in \Z[\omega] \\ N(m) \leq 9 M \\ w \mid (hm)^2}} \Big| \sum_{\substack{n \in \Z[\omega] \\ n\equiv 1\pmod{3}}} \mu^2(n) \frac{\overline{\chi_n(hmw)} g_3(n) }{N(n)^{\sigma + it}} e^{-N(n w d^3)/X}\Big|^2
    \end{align*}
    Writing $w = ab^2$ for $a, b\equiv 1\pmod{3}$ with $\mu^2(a) = 1$, the condition $w \mid (hm)^2$ implies $ab \mid hm$ and consequently $q \mid m$, where we denote $q := \frac{ab}{(h, ab)}$. Finally we write $m = q v$ for $v\equiv 1\pmod{3}$ and split into dyadic ranges using Cauchy--Schwarz to conclude that
    \begin{align*}
        S_2 \ll_{\sigma, \varepsilon}& \ (N(h)M)^\varepsilon \sum_{\substack{a, b, d \in \Z[\omega] \\ a, b, d\equiv 1\pmod{3}}}\frac{1}{N(d)^\sigma N(ab^2)^{2\sigma-2}} \\
        &\times \sum_{\substack{\ell \in \Z_{\geq 0} \\ L := 2^\ell}} (\log{L})^2 \sum_{\substack{0\neq v \in \Z[\omega] \\ N(v) \leq \frac{9 M}{N(q)}}} \Big| \sum_{\substack{n \in \Z[\omega] \\ n\equiv 1\pmod{3} \\ N(n) \sim L}} \mu^2(n) \frac{\overline{\chi_n(hqv a b^2)} g_3(n) }{N(n)^{\sigma + it}} e^{-N(nab^2d^3)/X}\Big|^2.
    \end{align*}
    The cubic large sieve of \cref{HBcubicdual}, or more precisely the dual version in \eqref{B2bound}, implies
    \begin{align}
        S_2 \ll_{\sigma, \varepsilon}& \ (N(h)M)^\varepsilon \sum_{\substack{a, b, d \in \Z[\omega] \\ a, b, d\equiv 1\pmod{3}}}\frac{1}{N(d)^\sigma N(ab^2)^{2\sigma-2}} \nonumber \\
        &\times \sum_{\substack{\ell \in \Z_{\geq 0} \\ L := 2^\ell}} L^\varepsilon \Big(\frac{M}{N(q)} + L\Big(\frac{M}{N(q)}\Big)^{1/3} + \Big(\frac{LM}{N(q)}\Big)^{2/3}\Big) L^{2-2\sigma} e^{-\frac{L N(ab^2d^3)}{X}}. \nonumber
    \end{align}
    Recalling that $1<\sigma \leq \frac{5}{4}$ gives
    \begin{align}
        S_2 \ll_{\sigma, \varepsilon}& \ (N(h)MX)^\varepsilon \sum_{\substack{a, b, d \in \Z[\omega] \\ a, b, d\equiv 1\pmod{3}}} \frac{1}{N(d)^\sigma N(ab)^{2\sigma-2}} \label{final_S2_large_sieve_bound} \\
        &\times \Big(\frac{M}{N(q)} + \Big(\frac{X}{N(ab)}\Big)^{3-2\sigma} \Big(\frac{M}{N(q)}\Big)^{1/3} + \Big(\frac{X}{N(ab)}\Big)^{8/3-2\sigma} \Big(\frac{M}{N(q)}\Big)^{2/3}\Big). \nonumber
    \end{align}
    The sum over $d$ converges and can now be removed. We have
    \begin{equation*}
        N(q) = \frac{N(ab)}{N((h, ab))} \geq \frac{N(a)}{N((h, a))}\frac{N(b)}{N((h, b))}.
    \end{equation*}
    Observe also that for any $\Delta, \delta > 0$ with $\Delta+\delta > 1$, a divisor-type bound gives
    \begin{align*}
        \sum_{\substack{c \in \Z[\omega] \\ c \equiv 1 \pmod{3}}} \frac{N((h, c))^\Delta}{N(c)^{\Delta+\delta}} \leq \prod_{\substack{\mathfrak{p} \text{ prime} \\ \mathfrak{p} \mid (h)}} \Big(1 - N(\mathfrak{p})^{-\delta}\Big)^{-1} \prod_{\substack{\mathfrak{p} \text{ prime} \\ \mathfrak{p} \nmid (h)}} \Big(1 - N(\mathfrak{p})^{-\Delta-\delta}\Big)^{-1} \ll_{\Delta, \delta, \varepsilon} N(h)^\varepsilon.
    \end{align*}
    Applying these two inequalities, with $(\Delta, \delta) = (1, 2 \sigma-2), (\frac13, 1)$, and $(\frac23, \frac23)$, to the sums over $a$ and $b$ in \eqref{final_S2_large_sieve_bound} yields the final estimate
    \begin{align}
        S_2 &\ll_{\sigma, \varepsilon} (N(h)MX)^\varepsilon \Big(M+M^{1/3}X^{3-2\sigma} + M^{2/3} X^{8/3-2\sigma}\Big).  \label{S2_final_bound}
    \end{align}
    
    Let us finally dispose of $S_3$. Denoting 
    \begin{equation*}
        \mathcal{Z}_M(h, w) := \sum_{\substack{0\neq m \in \Z[\omega] \\ N(m) \leq M}} \widetilde{Z}(hm, w),
    \end{equation*}
    from \eqref{T3_first_step} we obtain
    \begin{align}\label{S3_final_bound}
        S_3 \ll_\sigma \Big(\frac{N(h)M(1+|t|)^4}{X^2}\Big)^{2(\sigma-1)} \mathcal{I}_M(h, \sigma+it)
    \end{align}
    for 
    \begin{align}\label{I_M_integral_def}
        \mathcal{I}_M(h, \sigma+it) := \int_{-\infty}^{\infty} \mathcal{Z}_M(h, \sigma-it - iy) e^{-|y|/2}dy. 
    \end{align}

    Putting \eqref{S1_final_bound}, \eqref{S2_final_bound}, and \eqref{S3_final_bound} into \eqref{basic_Z_S_triple_bound}, for any $1<\sigma\leq \frac{5}{4}$ and $t\in \R$ we get
    \begin{align}
        \mathcal{Z}_M(h, \sigma+it) \leq & \ C_1(\sigma) \cdot \Big(\frac{N(h)M(1+|t|)^4}{X^2}\Big)^{2(\sigma-1)} \mathcal{I}_M(h, \sigma+it) \nonumber \\
        & + C_2(\sigma, \varepsilon) \cdot (N(h)MX)^{\varepsilon} \Big(M + M^{1/3}X^{3-2\sigma} + M^{2/3}X^{8/3 - 2\sigma}\Big), \label{final_S_M_bound_csts}
    \end{align}
    where $C_1(\sigma) \geq 1$ depends only on $\sigma$ and $C_2(\sigma, \varepsilon)$ only on $\sigma$ and $\varepsilon$.

    Observe that
    \begin{align*}
        \mathfrak{S}_M(h, \sigma) := \sup_{y \in \R} \frac{\mathcal{Z}_M(h, \sigma + iy)}{(1+|y|)^2} 
    \end{align*}
    is finite and attained at some value $\widetilde{t} = \widetilde{t}_M(h, \sigma) \in \R$, since
    since the ratio is continuous as a function of $y$ and also tends to zero as $|y|\to \infty$ by 
    \eqref{Z_convexity_implicit}.
    We have 
    \begin{equation}\label{S_M_supremum_relation}
        \frac{\mathcal{Z}_M(h, \sigma+i\widetilde{t})}{(1+|\widetilde{t}|)^2} = \mathfrak{S}_M(h, \sigma) \geq  \frac{\mathcal{Z}_M(h, \sigma+iy)}{(1+|y|)^2}
    \end{equation}
    for every $y\in\R$, which together with \eqref{I_M_integral_def} shows that there exists an absolute constant $C_3 \geq 1$ such that
    \begin{equation*}
        \mathcal{I}_M(h, \sigma+i \widetilde{t}) \leq \mathcal{Z}_M(h, \sigma+i\widetilde{t}) \int_{-\infty}^\infty \Big(\frac{1+|\widetilde{t}+y|}{1+|\widetilde{t}|}\Big)^2 e^{-|y|/2} dy \leq C_3 \cdot \mathcal{Z}_M(h, \sigma + i\widetilde{t}). 
    \end{equation*}
    
    We now take $t = \widetilde{t}_M(h, \sigma)$ and 
    \begin{equation*}
        X = [2C_1(\sigma)C_3 ]^{\frac{1}{4(\sigma-1)}} \cdot N(h)^{1/2} M^{1/2} (1+|\widetilde{t}|)^2 \geq 1
    \end{equation*}
    in \eqref{final_S_M_bound_csts} to obtain (recalling that $\sigma > 1$) the bound
    \begin{align*}
        \mathcal{Z}_M(h, \sigma+i\widetilde{t}) &\ll_{\sigma, \varepsilon} (N(h)MX)^{\varepsilon} \Big(M + M^{1/3}X^{3-2\sigma} + M^{2/3}X^{8/3 - 2\sigma}\Big)\\
        &\ll_{\sigma, \varepsilon} N(h)^{1/2} M^{1+\varepsilon} (1+|\widetilde{t}|)^2,
    \end{align*}
    where we restricted to $0<\varepsilon<\frac{\sigma-1}{100}$ in the second step. By \eqref{S_M_supremum_relation} this implies
    \begin{align*}
        \mathcal{Z}_M(h, \sigma+it) \ll_{\sigma, \varepsilon} N(h)^{1/2} M^{1+\varepsilon} (1+|t|)^2,
    \end{align*}
    for every $t\in \R$. Combining this with the bound $\zeta_{\lambda}(3s-2) \gg_\sigma 1$ for $s=\sigma+it$ and $\sigma > 1$ finishes the proof.
      
\end{proof}

\begin{lemma}[Evaluation of truncated twisted sums of Gauss sums] \label{truncated_Gauss_sums_lemma}
    Let $H :\R_{\geq 0} \to \R_{\geq 0}$ be a smooth and compactly supported function. Let $M >0$,  $0\leq \sigma \leq 2$, and $t \in \R$. Then for any $\eta, a, b, c \in \Z[\omega]\setminus\{0\}$ satisfying $a, b, c \equiv 1 \pmod{3}$, $\mu^2(abc)=1$, and $\eta \mid 3$, we have
    \begin{align} 
        &\sum_{\substack{m \in \Z[\omega] \\ m \equiv 1 \pmod{3}}} \frac{\overline{\chi_m(\eta ab^2c^3)} g_3(m)}{N(m)^{\sigma+it}} H\Big(\frac{N(m)}{M}\Big) = \label{twisted_Gauss_sum_display} \\
        & \qquad \qquad \qquad \bbone_{b=1}\cdot C_\eta \cdot \widetilde{H}(4/3-\sigma - it) M^{4/3-\sigma-it} \frac{\overline{\widetilde{g}_3(\eta, a)}}{N(a)^{1/6}} \frac{\Delta_{ac}(1)}{\Delta_{ac}(4/3)} + \mathcal{R}_M(\eta ab^2c^3, \sigma+it), \nonumber
    \end{align}
    where the constant $C_\eta$ depends only on $\eta$ and for any $A \in \Z_{\geq 0}$ we have
    \begin{align*}
        \mathcal{R}_M(\eta ab^2c^3, \sigma+it) \ll_{A, H, \varepsilon} M^{1-\sigma+\varepsilon} \mathop{\sum \sum}_{\substack{d, e \in \Z[\omega] \\ d, e\equiv 1\pmod{3} \\ d\mid a,\ e\mid c}} \frac{1}{N(de)^{1/2+\varepsilon}} \int_{-\infty}^\infty \frac{|\psi(\frac{\eta ab^2e}{d}, 1+\varepsilon + it+iy)|}{(1+|y|)^A} dy.
    \end{align*}
    Here, $\widetilde{H}(w)$ denotes the Mellin transform of $H$, namely $\widetilde{H}(w) = \int_{0}^\infty H(x) x^{w} \frac{dx}{x}$, and $\psi(r,s)$ is given in \eqref{psizetadef} and \eqref{psirsdef}. 
\end{lemma}

\begin{proof}
    For $\Re(w) > \frac{3}{2}$ we have
    \begin{equation*}
        \sum_{\substack{m \in \Z[\omega] \\ m \equiv 1 \pmod{3}}} \frac{\overline{\chi_m(\eta ab^2c^3)} g_3(m)}{N(m)^{w}} = \sum_{\substack{m \in \Z[\omega] \\ m \equiv 1 \pmod{3} \\ (m, abc)=1}} \frac{g_3(\eta ab^2, m)}{N(m)^{w}} = \psi_{abc}(\eta ab^2, w).
    \end{equation*}
    Then \cref{general_coprimality_cor} and \cref{psi_convexity_lemma} show that $\psi_{abc}(\eta ab^2, w)$ is holomorphic for $\Re(w) > 1$, except for at most a simple pole at $w=\frac{4}{3}$.
    
    Now let $s=\sigma+it$. By Mellin inversion, the quantity in \eqref{twisted_Gauss_sum_display} is equal to
    \begin{align*}
        \frac{1}{2\pi i} \int_{2 -i\infty}^{2+i\infty} \psi_{abc}(\eta ab^2, w)  \widetilde{H}(w-s) M^{w-s} dw,
    \end{align*}
    where $\widetilde{H}$ denotes the Mellin transform of $H$, which is rapidly decaying uniformly on vertical strips. Shifting to $\Re(w) = 1+\varepsilon$, we pick up a possible simple pole at $w = \frac{4}{3}$ with residue
    \begin{equation*}
        c_0 \cdot \widetilde{H}(4/3-s) M^{4/3-s} \Delta_{abc}(4/3)^{-1} \mathop{\sum \sum}_{\substack{d, e \in \Z[\omega] \\ d, e\equiv 1\pmod{3} \\ d\mid a,\ e\mid c}} \mu(de) \frac{\overline{g_3(\frac{\eta ab^2}{d}, d)} g_3(\frac{\eta ab^2}{d}, e)}{N(d)^{5/3}N(e)^{4/3}} \frac{\tau_3(\frac{\eta ab^2e}{d})}{N(\frac{\eta ab^2e}{d})^{1/6}},  
    \end{equation*}
    where $c_0$ is the (absolute) constant given in \eqref{c_0_def}. Since $\mu(abe)^2 = 1$ and $d \mid a$, we see from \eqref{tau_fourier_coeff_def} that $\tau_3(\frac{\eta ab^2e}{d}) = 0$ unless $b=1$, in which case
    \begin{equation*}
        \tau_3 \Big(\frac{\eta ae}{d}\Big) = c_\eta \frac{\overline{g_3(\eta, ae/d)}}{N(ae/d)^{1/2}}
    \end{equation*}
    for a constant $c_\eta$ depending only on $\eta \mid 3$ (either $c_\eta = 0$ or $|c_\eta| = 3^3$). Writing $f = \frac{a}{d}$, so $\mu^2(def)=1$, we then compute
    \begin{align*}
        \overline{g_3(\eta f, d)} g_3(\eta f, e) \overline{g_3(\eta, ef)} &= \chi_{d}(\eta f) \overline{\chi_{e}(\eta f)} \chi_{ef}(\eta) \chi_{e}(f) \overline{g_3(d)} g_3(e) \overline{g_3(e)} \overline{g_3(f)} \\
        & = \chi_{df}(\eta) N(e) \overline{g_3(df)} = N(e) \overline{g_3(\eta, a)}.
    \end{align*}

    Therefore, the residue at $w = \frac{4}{3}$ is equal to
    \begin{align*}
        &\ \bbone_{b=1}\cdot \frac{c_0 c_\eta}{N(\eta)^{1/6}} \cdot \frac{\widetilde{H}(4/3-s) M^{4/3-s} \overline{g_3(\eta, a)}}{N(a)^{2/3} \Delta_{ac}(4/3)} \mathop{\sum \sum}_{\substack{d, e \in \Z[\omega] \\ d, e\equiv 1\pmod{3} \\ d\mid a,\ e\mid c}} \frac{\mu(d)}{N(d)} \frac{\mu(e)}{N(e)} \\
        = &\ \bbone_{b=1}\cdot \frac{c_0 c_\eta}{N(\eta)^{1/6}} \cdot \widetilde{H}(4/3-s) M^{4/3-s} \frac{\overline{\widetilde{g}_3(\eta, a)}}{N(a)^{1/6}} \frac{\Delta_{ac}(1)}{\Delta_{ac}(4/3)}.
    \end{align*}

    For the remaining integral over $\Re(w) = 1+\varepsilon$, we simply use \cref{general_coprimality_cor} and bound the Gauss sums trivially using \cref{cubic-GS-lemma1}(\ref{changevar})    and \eqref{sqrootcancel} 
     (the moduli $d$ and $e$ are squarefree). Since
    \begin{equation*}
        \Delta_{abc}(1+\varepsilon+iy) \gg_\varepsilon 1 \qquad  \text{and} \qquad \widetilde{H}(1-\sigma + \varepsilon + iy) \ll_{A, H, \varepsilon} (1 + |y|)^{-A}
    \end{equation*} 
    for any $A \in \Z_{\geq 0}$ (uniformly in $y \in \R$ and $0 \leq \sigma \leq 2$), we obtain the desired result.
 
\end{proof}


\section{Second moment asymptotics: proof of \texorpdfstring{\cref{SMestimate2}}{}} \label{second_main_sum_section}

In this section we use the results developed in \cref{Gauss_sums_section} to complete the crucial task of evaluating the main sum for the twisted second moment.

\begin{proof}[Proof of \cref{SMestimate2}]
Recalling \eqref{SMdef}, we have that
\begin{equation*}
\mathcal{S}_M(\chi_q(\mathfrak{b}_1) \overline{\chi_q(\mathfrak{b}_2)} A_2(q); F) = 
\sum_{\substack{ q \in \mathbb{Z}[\omega] \\ q \equiv 1 \pmod{9}}}  M_Y(q)  \chi_q(\mathfrak{b}_1) \overline{\chi_q(\mathfrak{b}_2)} A_2(q) F \Big( \frac{N(q)}{X} \Big).
\end{equation*}
Furthermore, we open $A_2(q)$ using \eqref{A2def},
and then interchange the order of summation to obtain
\begin{equation} \label{SMopen}
 \mathcal{S}_M ( \chi_q(\mathfrak{b}_1) \overline{\chi_q(\mathfrak{b}_2)} A_2(q); F) 
 = \sumtwo_{0 \neq \substack{\mathfrak{n}_1, \mathfrak{n}_2 \unlhd \Z[\omega]}} 
N(\mathfrak{n}_1 \mathfrak{n}_2)^{-1/2}  \mathcal{S}_M ( \chi_q(\mathfrak{b}_1 \mathfrak{n}_1 ) \overline{\chi_q(\mathfrak{b}_2 \mathfrak{n}_2 )}; F_{\mathfrak{n}_1,\mathfrak{n}_2 }),
\end{equation}
where
\begin{equation}\label{F_def}
F_{\mathfrak{n}_1,\mathfrak{n}_2 }(t):= F(t) \Phi_2 \Big( \frac{N(\mathfrak{n}_1 \mathfrak{n}_2)}{3 X t} \Big).
\end{equation}

Opening the right side of \eqref{SMopen} and using \eqref{MYRYdef}, we have that
\begin{align*}
& \mathcal{S}_M ( \chi_q(\mathfrak{b}_1 \mathfrak{n}_1 ) \overline{\chi_q(\mathfrak{b}_2 \mathfrak{n}_2)}; F_{\mathfrak{n}_1,\mathfrak{n}_2}) \\
&=\sum_{\substack{\ell \in \Z[\omega] \\ \ell \equiv 1 \pmod 3 \\ N(\ell) \leq Y}} \mu(\ell) 
 \sum_{\substack{ m \in \Z[\omega]  \\ \ell^2 m \equiv 1 \pmod{9}}}  \chi_{\ell^2 m}(\mathfrak{b}_1 \mathfrak{n}_1) \overline{\chi_{\ell^2 m}(\mathfrak{b}_2 \mathfrak{n}_2)} F_{\mathfrak{n}_1,\mathfrak{n}_2} \Big( \frac{N(\ell^2 m)}{X} \Big).
 \end{align*}
For $i \in \{1,2\}$ we write $\mathfrak{b}_i = b_i \mathbb{Z}[\omega]$ and $\mathfrak{n}_i=\lambda^{g_i} n_i \mathbb{Z}[\omega]$ for some
$g_i \in \mathbb{Z}_{\geq 0}$ and 
$b_i,n_i \equiv 1 \pmod{3}$, where $\mu^2(b_i)=1$ since we are assuming $\mathfrak{b}_i$ are squarefree. By \eqref{cubesupp} and cubic reciprocity,
\begin{equation*}
\chi_{\ell^2 m}(\mathfrak{b}_i \mathfrak{n}_i)=\chi_{\ell^2 m}(\lambda^{g_i} b_i n_i)=\chi_{\ell^2 m}(b_i n_i)=\chi_{b_i n_i}(\ell^2 m)=\overline{\chi_{b_i n_i}(\ell)} \chi_{b_i n_i}(m),
\end{equation*}
for $i \in \{1, 2\}$. Thus 
\begin{align*}
& \mathcal{S}_M ( \chi_q(b_1 \lambda^{g_1} n_1 ) \overline{\chi_q(b_2 \lambda^{g_2} n_2 )}; F_{\lambda^{g_1}n_1,\lambda^{g_2} n_2 }) \nonumber \\
 &=\sum_{\substack{\ell \in \Z[\omega] \\ \ell \equiv 1 \pmod 3 \\ N(\ell) \leq Y}} \mu(\ell) 
\overline{\psi_{b_1 n_1,b_2 n_2}(\ell)} 
\sum_{\substack{ m \in \Z[\omega] \\  \ell^2 m \equiv 1 \pmod{9}}} \psi_{b_1 n_1,b_2 n_2}(m)
F_{\lambda^{g_1}n_1,\lambda^{g_2} n_2 } \Big( \frac{N(\ell^2 m)}{X}  \Big), 
 \end{align*}
where for $r_1,r_2 \equiv 1 \pmod{3}$ we write $\psi_{r_1,r_2}(\cdot):=\chi_{r_1}(\cdot) \overline{\chi_{r_2}(\cdot)}$.

We aim to apply Poisson summation to the sum over $m$, but first need to control some technicalities arising from lack of coprimality. In what follows whenever a variable is defined by a gcd the corresponding (principal) ideal will always be coprime to $\lambda$, so we take the unique generator of that ideal congruent to $1\pmod{3}$. 

For $i\in \{1, 2\}$ let $d := (b_1 n_1, b_2 n_2)$ and $r_i := \frac{b_i n_i}{d}$, so $(r_1, r_2) = 1$. For fixed $b_i \equiv 1 \pmod{3}$ this gives a one-to-one correspondence between pairs $(n_1, n_2)$ and triples $(d, r_1, r_2)$ satisfying $(r_1, r_2)=1$ and $b_i \mid d r_i$, where all of the variables run over elements of $\Z[\omega]$ congruent to $1\pmod{3}$. Indeed, the inverse map is $(d, r_1, r_2) \mapsto (\frac{r_1d}{b_1}, \frac{r_2 d}{b_2} )$. Performing this change of variables and replacing in \eqref{SMopen} gives
\begin{align}
    &\mathcal{S}_M ( \chi_q(b_1) \overline{\chi_q(b_2)} A_2(q); F)  \nonumber \\
    &= N(b_1b_2)^{1/2}\sumtwo_{g_1,g_2 \in \mathbb{Z}_{\geq 0}} \frac{1}{N(\lambda^{g_1+g_2})^{1/2}} \sum_{\substack{d \in \Z[\omega] \\ d \equiv 1\pmod{3}}} \frac{1}{N(d)}
    \sumtwo_{\substack{ r_1, r_2 \in \mathbb{Z}[\omega] \\ r_1, r_2 \equiv 1 \pmod{3} \\ (r_1, r_2)=1 \\ b_i \mid d r_i }} 
    \frac{1}{N(r_1r_2)^{1/2}} \nonumber \\
    & \times  \sum_{\substack{\ell \in \Z[\omega] \\ \ell \equiv 1 \pmod 3 \\ N(\ell) \leq Y}} \mu(\ell) \overline{\psi_{d r_1, d r_2}(\ell)} \sum_{\substack{ m \in \Z[\omega] \\  \ell^2 m \equiv 1 \pmod{9}}} \psi_{d r_1, d r_2}(m)
    F_{\frac{\lambda^{g_1} d r_1}{b_1},\frac{\lambda^{g_2} d r_2}{b_2} } \Big( \frac{N(\ell^2 m)}{X}  \Big). \label{new_m_sum}
\end{align}

For $i \in \{1, 2\}$ write (uniquely) $r_i = d_i m_i$ with $d_i, m_i \equiv 1\pmod{3}$ satisfying $d_i \mid d^\infty$, $(m_i, d) = (m_1, m_2) = (d_1, d_2) = 1$, and $b_i \mid dd_im_i$. Note that
\begin{equation*}
    \psi_{d r_1, d r_2} = \chi_{d r_1} \overline{\chi_{d r_2}} = \chi_{r_1} \overline{\chi_{r_2}}\textbf{1}_d = \chi_{m_1} \overline{\chi_{m_2}} \chi_{d_1}\overline{\chi_{d_2}}\textbf{1}_{e},
\end{equation*}
where $e \equiv 1\pmod 3$ is the squarefree element of $\Z[\omega]$ given by
\begin{equation}\label{e_def}
    e := \prod_{\substack{\pi \text{ prime}\\ \pi \equiv 1 \pmod 3 \\ \pi \mid d,\ \pi \nmid d_1d_2}} \pi = \frac{\mathrm{rad}(d)}{\mathrm{rad}(d_1d_2)}.
\end{equation}
Therefore, $\psi_{d r_1, d r_2}$ is $(m_1 m_2 d_1 d_2 e)$-periodic. Applying Poisson summation (\cref{radialpois}) with period $m_1 m_2 d_1 d_2 e$ in \eqref{new_m_sum}, we see that the sum over $m$ is equal to 
\begin{align} \label{poissonapp}
 \frac{4 \pi X}{3^{9/2} N(\ell^2 m_1 m_2 d_1 d_2 e)} \sum_{k \in \mathbb{Z}[\omega]} \ddot{\psi}_{d d_1 m_1, d d_2 m_2}(k) \check{e}\Big({-\frac{k\ell (m_1 m_2 d_1 d_2 e)^2}{9 \lambda}}\Big) & \nonumber \\
 \times \ddot{F}_{\frac{\lambda^{g_1} d d_1 m_1}{b_1},\frac{\lambda^{g_2} d d_2 m_2}{b_2} }
\Big( \frac{k \sqrt{X}}{\ell^2 m_1 m_2 d_1 d_2 e} \Big ),&
\end{align}
where we used $\overline{\ell^2} \equiv \ell \pmod{9}$ as $\ell \equiv 1\pmod{3}$. Crucially, since $m_1, m_2, d_1, d_2, e$ are pairwise coprime and congruent to $1\pmod{3}$, and $9\lambda = \lambda^5$, by the Chinese remainder theorem for the pairwise coprime moduli $m_1, m_2, d_1d_2e$ and cubic reciprocity we have
\begin{align*}
    \ddot{\psi}_{d d_1 m_1, d d_2 m_2}(k):=& \sum_{a \pmod{m_1 m_2 d_1 d_2 e}} \chi_{m_1} \overline{\chi_{m_2}} \chi_{d_1}\overline{\chi_{d_2}}\textbf{1}_{e}(9 \lambda a) \check{e}\Big({-\frac{ka}{m_1m_2 d_1 d_2 e}}\Big) \\
    = &\ \chi_{m_1}(9\lambda m_2 d_1 d_2 e) g_3(-k, m_1) \overline{\chi_{m_2}(9\lambda m_1 d_1 d_2 e)} \overline{g_3(k, m_2)} \\
    &\qquad\qquad \times \chi_{d_1}\overline{\chi_{d_2}}(9\lambda m_1 m_2)\sum_{a \pmod{d_1 d_2 e}} \chi_{d_1} \overline{\chi_{d_2}} \textbf{1}_{e}(a) \check{e}\Big({-\frac{ka}{d_1 d_2 e}}\Big) \nonumber \\
    = &\ N(d_1 d_2)^{3/2} N(e) \cdot \overline{\chi_{m_1}(\lambda d_1 e^2)} g_3(k, m_1) \cdot \chi_{m_2}(\lambda d_2 e^2)\overline{g_3(k, m_2)} \cdot \mathcal{E}(d_1, d_2; k, e),\\
\end{align*}
where another application of the Chinese remainder theorem for the pairwise coprime moduli $d_1,d_2,e$ gives
\begin{align}
    \mathcal{E}(d_1, d_2; k, e) :=&\ \frac{\overline{\psi_{d_1, d_2}(\lambda)}}{N(d_1 d_2)^{3/2} N(e)} \sum_{a \pmod{d_1 d_2 e}} \chi_{d_1} \overline{\chi_{d_2}} \textbf{1}_{e}(a) \check{e}\Big(\frac{ka}{d_1 d_2 e}\Big) \nonumber \\
    =&\ \psi_{d_1, d_2}(\lambda^2 e) \frac{\widetilde{g}_3(k, d_1)}{N(d_1)} \frac{\overline{\widetilde{g}_3(k, d_2)}}{N(d_2)} \frac{\mu\big(\frac{e}{(e, k)}\big) \varphi(e)}{N(e) \varphi\big(\frac{e}{(e, k)}\big)} \label{exp_sum_definition} \\
    \ll&\ \frac{|\widetilde{g}_3(k, d_1) \widetilde{g}_3(k, d_2)|}{N(d_1 d_2)}\frac{N((e, k))}{N(e)}.\label{exp_sum_uniform_bound}
\end{align}
In the middle step we used the evaluation of the Ramanujan sum modulo $e$ (which is squarefree) given in \cite[Lemma 5.5]{DR}.

We now substitute \eqref{poissonapp} into \eqref{new_m_sum}, changing variables as described above from $(d, r_1, r_2)$ satisfying $(r_1, r_2) = 1$ and $b_i \mid d r_i$ to $(d, d_1, d_2, m_1, m_2)$ satisfying $(m_1, m_2) = (d_1, d_2) = (m_i, d) = 1$, $d_i \mid d^\infty$, $b_i \mid dd_i m_i$, and all variables equivalent to $1 \pmod 3$, which one checks is one-to-one. We also split the sums over $m_i$ into congruence classes modulo $9$, since $\chi_{m_i}(\lambda)$ is $9$-periodic by \eqref{cubesupp}. This yields
\begin{align}
    &\mathcal{S}_M ( \chi_q(b_1) \overline{\chi_q(b_2)} A_2(q); F) = \frac{4 \pi X}{3^{9/2}} N(b_1b_2)^{1/2} \sumtwo_{g_1,g_2 \in \mathbb{Z}_{\geq 0}} \frac{1}{3^{(g_1+g_2)/2}} \nonumber \\
    & \times \sum_{\substack{d \in \Z[\omega] \\ d \equiv 1\pmod{3}}} \frac{1}{N(d)} \sumtwo_{\substack{d_1, d_2 \in \Z[\omega] \\ d_1, d_2 \equiv 1 \pmod{3} \\ (d_1, d_2)=1 \\ d_1d_2 \mid d^\infty}} \sumtwo_{\substack{c_1, c_2 \pmod{9} \\ c_1, c_2 \equiv 1\pmod{3}}} \overline{\chi_{c_1}}\chi_{c_2}(\lambda) \sum_{k \in \Z[\omega]} \mathcal{E}(d_1, d_2; k, e)   \nonumber \\
    & \times \sum_{\substack{\ell \in \Z[\omega] \\ \ell \equiv 1 \pmod 3 \\ N(\ell) \leq Y}} \check{e} \Big({-\frac{k\ell(c_1c_2 d_1 d_2 e)^2}{9\lambda}}\Big) \frac{\mu(\ell) \overline{\psi_{d d_1, d d_2}(\ell)}}{N(\ell)^2}   \sum_{\substack{m_1 \in \Z[\omega] \\ m_1 \equiv c_1 \pmod{9} \\ (m_1, d) = 1 \\ b_1 \mid dd_1m_1}} \frac{\overline{\chi_{m_1}(\ell d_1 e^2)} \widetilde{g}_3(k, m_1)}{N(m_1)} \nonumber \\
    & \times \sum_{\substack{m_2 \in \Z[\omega] \\ m_2 \equiv c_2 \pmod{9} \\ (m_2, d) = 1 \\ b_2 \mid dd_2m_2 \\ (m_1, m_2)=1}} \frac{\chi_{m_2}(\ell d_2 e^2) \overline{\widetilde{g}_3(k, m_2)}}{N(m_2)}  \ddot{F}_{\frac{\lambda^{g_1} d d_1 m_1}{b_1},\frac{\lambda^{g_2} d d_2 m_2}{b_2} }
    \Big( \frac{k \sqrt{X}}{\ell^2 m_1 m_2 d_1 d_2 e} \Big). \label{expression_to_split}
\end{align}

Therefore we may write
\begin{equation*}
    \mathcal{S}_M ( \chi_q(b_1) \overline{\chi_q(b_2)} A_2(q); F) = \mathcal{M}(b_1, b_2) + \mathcal{R}(b_1, b_2),
\end{equation*}
where $\mathcal{M}(b_1, b_2)$ corresponds to the term $k=0$ in \eqref{expression_to_split}, and $\mathcal{R}(b_1, b_2)$ corresponds to the terms $0\neq k \in \Z[\omega]$ in that expression.


\subsection{The main term \texorpdfstring{$\mathcal{M}(b_1, b_2)$}{}}

From \eqref{exp_sum_definition} we have
\begin{align*}
    \mathcal{E}(d_1, d_2; 0, e)  = \frac{\psi_{d_1, d_2}(\lambda^2 e)}{N(d_1 d_2 e)} \widetilde{g}_3(0, d_1) \overline{\widetilde{g}_3(0, d_2)} \varphi(e).
\end{align*}

We have $\widetilde{g}_3(0, m) = 0$ unless $m$ is a cube, in which case $\widetilde{g}_3(0, m) = \frac{\varphi(m)}{N(m)^{1/2}}$. Therefore $\mathcal{E}(d_1, d_2; 0, e) = 0$ unless both $d_1$ and $d_2$ are cubes, in which case $\mathcal{E}(d_1, d_2; 0, e) =  \frac{\varphi(d_1 d_2 e)}{N(d_1 d_2)^{3/2} N(e)}$. Making the change of variables $d_i \mapsto d_i^3$ and $m_i \mapsto m_i^3$ for $i\in \{1, 2\}$ (observe that the definition of $e$ in \eqref{e_def} remains unchanged), we conclude that
\begin{align}
    &\mathcal{M}(b_1, b_2) = \frac{4 \pi X}{3^{9/2}} N(b_1b_2)^{1/2} \sumtwo_{g_1,g_2 \in \mathbb{Z}_{\geq 0}} \frac{1}{3^{(g_1+g_2)/2}} \sum_{\substack{d \in \Z[\omega] \\ d \equiv 1\pmod{3}}} \frac{1}{N(d)} \sumtwo_{\substack{d_1, d_2 \in \Z[\omega] \\ d_1, d_2 \equiv 1 \pmod{3} \\ (d_1, d_2)=1 \\ d_1d_2 \mid d^\infty}} \frac{\varphi(e)}{N(e)}\frac{\varphi(d_1^3 d_2^3)}{N(d_1d_2)^{9/2}} \nonumber \\
    & \times \sum_{\substack{m_1 \in \Z[\omega] \\ m_1 \equiv 1 \pmod{3} \\ (m_1, d) = 1 \\ b_1 \mid d d_1^3 m_1^3}} \sum_{\substack{m_2 \in \Z[\omega] \\ m_2 \equiv 1 \pmod{3} \\ (m_2, d) = 1 \\ b_2 \mid d d_2^3 m_2^3 \\ (m_1, m_2)=1}} \frac{\varphi(m_1^3 m_2^3)}{N(m_1 m_2)^{9/2}}  \ddot{F}_{\frac{\lambda^{g_1} d d_1^3 m_1^3}{b_1},\frac{\lambda^{g_2} d d_2^3 m_2^3}{b_2} }
    ( 0 ) \sum_{\substack{ \ell \in \Z[\omega] \\ \ell \equiv 1 \pmod 3 \\ N(\ell) \leq Y}} \frac{\mu(\ell) \mathbf{1}_{d m_1 m_2}(\ell)}{N(\ell)^2}. \label{main_term_expression}
\end{align}
We can now directly evaluate the sum over $\ell$, obtaining
\begin{equation}\label{ell_sum_eval}
    \sum_{\substack{ \ell \in \Z[\omega] \\ \ell \equiv 1 \pmod 3 \\ N(\ell) \leq Y}} \frac{\mu(\ell) \mathbf{1}_{d m_1 m_2}(\ell)}{N(\ell)^2} = \zeta_\lambda(2)^{-1} \prod_{\substack{\pi \text{ prime} \\ \pi \equiv 1 \pmod 3 \\ \pi \mid d m_1 m_2}} \Big(1 - \frac{1}{N(\pi)^2}\Big)^{-1} + O\Big(\frac{1}{Y}\Big).
\end{equation}

Recall from \eqref{F_def} that $F_{\mathfrak{n}_1, \mathfrak{n}_2}(r) = F(r) \Phi_2 \big(\frac{N(\mathfrak{n}_1 \mathfrak{n}_2)}{3Xr} \big)$, where $\Phi_2$ is defined by \eqref{def-Phi} and  $F$ has support in $(1, 2)$. Since $J_0(0)=1$, replacing in \eqref{Vddot}, we have 
\begin{align}
    \ddot{F}_{\mathfrak{n}_1, \mathfrak{n}_2}(0) & = \int_1^{\sqrt{2}}  r F(r^2) \Phi_2 \Big(\frac{N(\mathfrak{n}_1 \mathfrak{n}_2)}{3Xr^2} \Big) dr \nonumber \\
    & =  \frac{1}{2\pi i} \int_{2 -i\infty}^{2 + i\infty} \Big(\frac{4 \pi^2 N(\mathfrak{n}_1 \mathfrak{n}_2)}{3X} \Big)^{-w} \Big(\frac{\Gamma (\frac{1}{2}+w )}{\Gamma (\frac{1}{2})}\Big)^2 \check{F}(w) \frac{dw}{2w}, \label{F_dot_dot_zero}
\end{align}
for $\check{F}(w) := \int_0^\infty t^w F(t) dt = 2 \int_1^{\sqrt{2}} r^{2w+1} F(r^2) dr$. In particular, by the rapid decay of $\Phi_2$ as in \eqref{Phi_bound}, we have (using $0 \leq F(t) \leq 1$ for the uniformity of the implied constant) the coarse bound
\begin{align}
    \ddot{F}_{\mathfrak{n}_1, \mathfrak{n}_2}(0) \ll_{A} \Big(1 +  \frac{N(\mathfrak{n}_1 \mathfrak{n}_2)}{X}\Big)^{-A}. \label{F_dot_dot_zero_bound}
\end{align}
The following simple observation will be very useful.

\begin{lemma}\label{gcd_sum_lemma}
    For any $C \geq 1/2$, $1\geq \Delta \geq 0$, $\delta \in \R$, and $g \in \Z[\omega]$ with $g \equiv 1 \pmod 3$,
    \begin{equation*}
        \sum_{\substack{c\in \Z[\omega] \\ c\equiv 1 \pmod 3 \\ N(c) \sim C}} \frac{N((g, c))^\Delta}{N(c)^\delta} \ll_{\delta, \varepsilon} N(g)^\varepsilon C^{1 - \delta}.
    \end{equation*}
\end{lemma}

\begin{proof}
    Stratifying by the values of $d = (g, c)$, by positivity we obtain
    \begin{align*}
        \sum_{\substack{c\in \Z[\omega] \\ c\equiv 1 \pmod 3 \\ N(c) \sim C}} \frac{N((g, c))^\Delta}{N(c)^\delta} &\leq \sum_{\substack{d\in \Z[\omega] \\ d\equiv 1 \pmod 3 \\ d\mid g,\ N(d) \leq 2C}} \frac{N(d)^\Delta}{N(d)^\delta} \sum_{\substack{c'\in \Z[\omega] \\ c'\equiv 1 \pmod 3 \\ N(c') \sim \frac{C}{N(d)}}} \frac{1}{N(c')^\delta} \\
        &\ll_\delta C^{1-\delta} \sum_{\substack{d\in \Z[\omega] \\ d\equiv 1 \pmod 3 \\ d\mid g,\ N(d) \leq 2C}} N(d)^{\Delta-1} \leq C^{1-\delta} d(g) \max\{1, (2C)^{\Delta-1}\}, 
    \end{align*}
    and the result follows by the divisor bound since $\Delta \leq 1$.
    
\end{proof}

Observe that the condition $b_i \mid d d_i^3 m_i^3$ is equivalent to $\frac{b_i}{(b_i, d)} \mid m_i$, since $\mu^2(b_i) = 1$ and $d_i \mid d^\infty$. Inserting \eqref{ell_sum_eval} into \eqref{main_term_expression}, we can use \eqref{F_dot_dot_zero_bound} to conclude that the error term arising from $O(\frac{1}{Y})$ is
\begin{align}
    &\ll \frac{X N(b_1b_2)^{1/2}}{Y} \sumthree_{\substack{d, d_1, d_2 \in \Z[\omega] \\ d, d_1, d_2 \equiv 1 \pmod{3} \\ (d_1, d_2)=1 \\ d_1 d_2 \mid d^\infty}} \sumtwo_{\substack{m_1, m_2 \in \Z[\omega] \\ m_1, m_2 \equiv 1 \pmod{3} \\ (m_1, m_2) = (m_i, d) = 1 \\ \frac{b_i}{(b_i, d)} \mid m_i}} \frac{\big(1 + \frac{N(d)^2 N(d_1 m_1 \cdot d_2 m_2)^3}{N(b_1b_2) X} \big)^{-100}}{N(d) N(d_1 m_1 \cdot d_2 m_2)^{3/2}} \nonumber \\
    &\ll \frac{X N(b_1b_2)^{1/2}}{Y} \sum_{\substack{d \in \Z[\omega] \\ d \equiv 1 \pmod{3}}} \frac{\big(1 + \frac{N(d^2)}{N(b_1b_2) X} \big)^{-100}}{N(d)} \frac{N((b_1, d)(b_2, d))^{3/2}}{N(b_1 b_2)^{3/2}} \nonumber \\
    &\ll \frac{X}{Y} \sum_{\substack{d \in \Z[\omega] \\ d \equiv 1 \pmod{3}}} \frac{N((b_1, d))+ N((b_2, d))}{N(d)} \Big(1 + \frac{N(d^2)}{N(b_1b_2) X} \Big)^{-100} \ll \frac{X^{1+\varepsilon}}{Y} \nonumber,
\end{align}
where the last step follows from \cref{gcd_sum_lemma}.

Applying \eqref{F_dot_dot_zero} for $\mathfrak{n}_i= \frac{\lambda^{g_i} d d_i^3 m_i^3}{b_i} \Z[\omega]$ then gives
\begin{align} 
    \mathcal{M}(b_1, b_2) = \frac{2 \pi X}{3^{9/2} \zeta_\lambda(2)} \frac{1}{2\pi i} \int_{2 - i \infty}^{2 + i \infty} & \Big(\frac{4 \pi^2}{3X} \Big)^{-w} \Big(\frac{\Gamma (\frac{1}{2}+w )}{\Gamma (\frac{1}{2} )} \Big)^2 \check{F}(w) \mathcal{G}_{b_1, b_2}(w) \frac{dw}{w} + O_\varepsilon\Big(\frac{X^{1+\varepsilon}}{Y}\Big), \label{after-mellin}
\end{align}
where, recalling the definition of $e$ in \eqref{e_def}, we set
\begin{align}
    & \mathcal{G}_{b_1, b_2}(w) := N(b_1 b_2)^{1/2 + w} \Bigg( \sumtwo_{g_1,g_2 \in \mathbb{Z}_{\geq 0}} \frac{1}{3^{(g_1+g_2)(1/2+w)}} \Bigg) \sum_{\substack{d \in \Z[\omega] \\ d \equiv 1\pmod{3}}} 
    \frac{1}{N(d)^{1+2w}} \nonumber \\
    & \times \prod_{\substack{\pi \text{ prime} \\ \pi \equiv 1 \pmod 3 \\ \pi \mid d}} \Big(1 - \frac{1}{N(\pi)^2}\Big)^{-1} \Bigg(\sumtwo_{\substack{d_1, d_2 \in \Z[\omega] \\ d_1, d_2 \equiv 1 \pmod{3} \\ (d_1, d_2)=1 \\ d_1d_2 \mid d^\infty}} \frac{\varphi(e)}{N(e)} \frac{\varphi(d_1^3 d_2^3)}{N(d_1d_2)^{9/2+3w}} \Bigg) \nonumber \\
    & \times \sum_{\substack{m_1 \in \Z[\omega] \\ m_1 \equiv 1 \pmod{3} \\ (m_1, d) = 1 \\ \frac{b_1}{(b_1, d)} \mid m_1}} \sum_{\substack{m_2 \in \Z[\omega] \\ m_2 \equiv 1 \pmod{3} \\ (m_2, d) = 1 \\ \frac{b_2}{(b_2, d)} \mid m_2 \\ (m_1, m_2)=1}} \frac{\varphi(m_1^3 m_2^3)}{N(m_1 m_2)^{9/2 + 3w}} \prod_{\substack{\pi \text{ prime} \\ \pi \equiv 1 \pmod 3 \\ \pi \mid m_1 m_2}} \Big(1 - \frac{1}{N(\pi)^2}\Big)^{-1}. \label{generating-series}
\end{align}
We will eventually move the line of integration to the left and collect a pole at $w=0$. The generating series $\mathcal{G}_{b_1, b_2}$ will have a simple pole at $w=0$, so we ultimately expect a double pole for the integrand in \eqref{after-mellin}.
 
Recall that $b_1, b_2$ are squarefree  (but not necessarily  coprime), which will be used often in the computations below.
For each $d$, the sums over $m_1, m_2$ in \eqref{generating-series} are empty unless $(\frac{b_1}{(b_1, d)}, \frac{b_2}{(b_2, d)})=1$. Assuming that this is the case, denote $b'_i := \frac{b_i}{(b_i, d)}$, so $(b'_1, b'_2) = (b'_i, d)=1$. The summation over $m_1, m_2$ in \eqref{generating-series} is then
\begin{align*}
    &  \sum_{\substack{m_1 \in \Z[\omega] \\ m_1 \equiv 1 \pmod{3} \\ (m_1, d) = 1 \\  b'_1 \mid  m_1}} \sum_{\substack{m_2 \in \Z[\omega] \\ m_2 \equiv 1 \pmod{3} \\ (m_2, d) = 1 \\  b'_2 \mid m_2 \\ (m_1, m_2)=1}} \frac{\varphi(m_1^3 m_2^3)}{N(m_1 m_2)^{9/2 + 3w}} \prod_{\substack{\pi \text{ prime} \\ \pi \equiv 1 \pmod 3 \\ \pi \mid m_1 m_2}} \Big(1 - \frac{1}{N(\pi)^2}\Big)^{-1} \\
    &= \prod_{\substack{\pi \text{ prime} \\ \pi \equiv 1 \pmod 3\\ \pi \nmid d b'_1 b'_2}}  \Bigg( 1 + 2 \sum_{k=1}^\infty \frac{\varphi(\pi^{3k})}{N(\pi^{3k})} \frac{\big(1 - \frac{1}{N(\pi)^2}\big)^{-1}}{N(\pi)^{k(3/2 + 3w)}} \Bigg) \prod_{\substack{\pi \text{ prime} \\ \pi \equiv 1 \pmod 3\\ \pi \mid b'_1 b'_2}}  \Bigg( \sum_{k=1}^\infty \frac{\varphi(\pi^{3k})}{N(\pi^{3k})} \frac{\big(1 - \frac{1}{N(\pi)^2}\big)^{-1}}{N(\pi)^{k(3/2 + 3w)}} \Bigg) \\
    &= \prod_{\substack{\pi \text{ prime} \\ \pi \equiv 1 \pmod 3\\ \pi \nmid d}}  E(\pi, w)  \prod_{\substack{\pi \text{ prime} \\ \pi \equiv 1 \pmod 3\\ \pi \mid b'_1 b'_2}} E(\pi, w)^{-1} F(\pi, w)\\
    &= \mathcal{E}(w) \prod_{\substack{\pi \text{ prime} \\ \pi \equiv 1 \pmod 3 \\ \pi \mid d}} E(\pi, w)^{-1}  \prod_{\substack{\pi \text{ prime} \\ \pi \equiv 1 \pmod 3\\ \pi \mid b'_1 b'_2}} E(\pi, w)^{-1}  F(\pi, w),
\end{align*}
where 
\begin{align*}
    E(\pi, w) & := 1 + 2 \sum_{k=1}^\infty \frac{\varphi(\pi^{3k})}{N(\pi^{3k})} \frac{\big(1 - \frac{1}{N(\pi)^2}\big)^{-1}}{N(\pi)^{k(3/2 + 3w)}} = 1 + 2 \frac{\big(1 + \frac{1}{N(\pi)}\big)^{-1}}{N(\pi)^{3/2 + 3w} - 1}, \\
    F(\pi, w) &:= \sum_{k=1}^\infty \frac{\varphi(\pi^{3k})}{N(\pi^{3k})} \frac{\big(1 - \frac{1}{N(\pi)^2}\big)^{-1}}{N(\pi)^{k(3/2 + 3w)}} = \frac{\big(1 + \frac{1}{N(\pi)}\big)^{-1}}{N(\pi)^{3/2 + 3w} - 1}, \\
    \mathcal{E}(w) &:= \prod_{\substack{\pi \text{ prime} \\ \pi \equiv 1 \pmod 3}} E(\pi, w) = \prod_{\substack{\pi \text{ prime} \\ \pi \equiv 1 \pmod 3}} \Bigg( 1 + 2 \frac{\big(1 + \frac{1}{N(\pi)}\big)^{-1}}{N(\pi)^{3/2 + 3w} - 1} \Bigg).
\end{align*}

Observe that $\mathcal{E}(w)$ is holomorphic for $\Re(w) >  - \frac16$. For brevity, we will often write $(E^{-1}F) (\pi, w)$ for $E(\pi, w)^{-1}  F(\pi, w)$, and similarly for other Euler factors. By \eqref{e_def} we have
\begin{align*}
    \sumtwo_{\substack{d_1, d_2 \in \Z[\omega] \\ d_1, d_2 \equiv 1 \pmod{3} \\ (d_1, d_2)=1 \\ d_1d_2 \mid d^\infty}} \frac{\varphi(e)}{N(e)} \frac{\varphi(d_1^3 d_2^3)}{N(d_1 d_2)^{9/2+3w}} &= \prod_{\substack{\pi \text{ prime} \\ \pi \equiv 1 \pmod 3 \\ \pi \mid d}} \frac{\varphi(\pi)}{N(\pi)}   \underbrace{\Big( 1 + 2  \sum_{k=1}^{\infty}  \frac{1}{N(\pi)^{k(3/2 + 3w)}} \Big)}_{=: \; G(\pi, w)} \\
    & = \frac{\varphi(d)}{N(d)} \prod_{\substack{\pi \text{ prime} \\ \pi \equiv 1 \pmod 3 \\ \pi \mid d}} G(\pi, w),
\end{align*}
where we used the identity $\frac{\varphi(\mathrm{rad}(d))}{N(\mathrm{rad}(d))} = \frac{\varphi(d)}{N(d)}$. Therefore, the sum over $d$ in \eqref{generating-series} is equal to
\begin{align}
    \mathcal{E}(w) \sum_{\substack{d \in \Z[\omega] \\ d \equiv 1\pmod{3}\\ \big(\frac{b_1}{(b_1, d)}, \frac{b_2}{(b_2, d)} \big)=1}} \frac{1}{N(d)^{1+ 2w}} \prod_{\substack{\pi \text{ prime} \\ \pi \equiv 1 \pmod 3 \\ \pi \mid d}} \frac{N(\pi) \cdot (E^{-1} G)(\pi, w)}{N(\pi) + 1} \prod_{\substack{\pi \text{ prime} \\ \pi \equiv 1 \pmod 3\\ \pi \mid \frac{b_1 b_2}{(b_1,d) (b_2, d)}}} (E^{-1}F)(\pi, w). \label{d-sum}  
\end{align}
 
Let $b := (b_1, b_2)$. Since $\big(\frac{b_1}{(b_1, d)}, \frac{b_2}{(b_2, d)} \big) = 1$ if and only if $b \mid d$, we can uniquely factorise $d = b b' d'$, for $b', d' \equiv 1 \pmod 3$ with $b' \mid b^\infty$ and $(d', b)=1$. Then \eqref{d-sum} is equal to
\begin{align}
    \mathcal{E}(w) \sumtwo_{\substack{b', d' \in \Z[\omega] \\ b', d' \equiv 1\pmod{3}\\ b' \mid b^\infty,\ (d', b) = 1}} \frac{1}{N(b b' d')^{1+2w}} \prod_{\substack{\pi \text{ prime} \\ \pi \equiv 1 \pmod 3 \\ \pi \mid b d'}} \frac{N(\pi) \cdot (E^{-1} G)(\pi, w)}{N(\pi) + 1} \prod_{\substack{\pi \text{ prime} \\ \pi \equiv 1 \pmod 3\\ \pi \mid b_1 b_2,\ \pi \nmid b d'}} (E^{-1}F)(\pi, w),\nonumber
\end{align}
where we used the fact that the $b_i$ are squarefree.
At this point the sum over $b'$ can be freely evaluated, and we obtain
\begin{equation*}
    \sum_{\substack{b' \in \Z[\omega] \\ b'\equiv 1\pmod{3}\\ b' \mid b^\infty}} \frac{1}{N(b b')^{1+2w}} = \prod_{\substack{\pi \text{ prime} \\ \pi \equiv 1 \pmod 3 \\ \pi \mid b}} \sum_{k=1}^\infty \frac{1}{N(\pi)^{k(1+2w)}} = \prod_{\substack{\pi \text{ prime} \\ \pi \equiv 1 \pmod 3 \\ \pi \mid b}} \frac{1}{N(\pi)^{1+2w} - 1}.
\end{equation*}
Thus \eqref{d-sum} is equal to
\begin{align*}
    & \mathcal{E}(w) \prod_{\substack{\pi \text{ prime} \\ \pi \equiv 1 \pmod 3 \\ \pi \mid b}}  \Big( 1 + \frac{1}{N(\pi)} \Big)^{-1} \frac{(E^{-1}G)(\pi, w)}{N(\pi)^{1+2w} - 1} \\
    &  \times  \prod_{\substack{\pi \text{ prime} \\ \pi \equiv 1 \pmod 3\\ \pi \mid b_1 b_2,\ \pi \nmid b }} (E^{-1}F)(\pi, w) \sum_{\substack{d' \in \Z[\omega] \\ d' \equiv 1\pmod{3}\\ (d', b) = 1}} \frac{1}{N(d')^{1+2 w}} \nonumber \\
    & \times \prod_{\substack{\pi \text{ prime} \\ \pi \equiv 1 \pmod 3 \\ \pi \mid d'}} \Big( 1 + \frac{1}{N(\pi)} \Big)^{-1} (E^{-1}G)(\pi, w)  \prod_{\substack{\pi \text{ prime} \\ \pi \equiv 1 \pmod 3\\ \pi \mid (b_1 b_2, d')}} (EF^{-1})(\pi, w).
\end{align*}
The sum over $d'$ in the display above equals
\begin{align*} 
    & \prod_{\substack{\pi \text{ prime} \\ \pi \equiv 1 \pmod 3 \\ \pi \nmid b_1 b_2}} \Big( 1 + \Big( 1 + \frac{1}{N(\pi)} \Big)^{-1} \frac{(E^{-1}G)(\pi, w)}{N(\pi)^{1+2w}-1} \Big) \\
    \times &\ \prod_{\substack{\pi \text{ prime} \\ \pi \equiv 1 \pmod 3 \\ \pi \mid b_1 b_2,\ \pi \nmid b}} \Big( 1 + \Big( 1 + \frac{1}{N(\pi)} \Big)^{-1} \frac{(F^{-1}G)(\pi, w)}{N(\pi)^{1+2w}-1} \Big),
\end{align*}
which we rewrite as
\begin{align*}
    &\prod_{\substack{\pi \text{ prime} \\ \pi \equiv 1 \pmod 3}} \Big( 1 + \Big( 1 + \frac{1}{N(\pi)} \Big)^{-1} \frac{(E^{-1}G)(\pi, w)}{N(\pi)^{1+2w}-1} \Big) \\
    \times &\ \prod_{\substack{\pi \text{ prime} \\ \pi \equiv 1 \pmod 3 \\ \pi \mid b_1 b_2}} \Big( 1 + \Big( 1 + \frac{1}{N(\pi)} \Big)^{-1} \frac{(E^{-1}G)(\pi, w)}{N(\pi)^{1+2w}-1} \Big)^{-1}\\
    \times &\ \prod_{\substack{\pi \text{ prime} \\ \pi \equiv 1 \pmod 3 \\ \pi \mid \frac{b_1 b_2}{b^2}}} \Big( 1 + \Big( 1 + \frac{1}{N(\pi)} \Big)^{-1} \frac{(F^{-1}G)(\pi, w)}{N(\pi)^{1+2w}-1} \Big).
\end{align*}

Inserting this into \eqref{d-sum} and \eqref{generating-series} we obtain 
\begin{align}
    \mathcal{G}_{b_1, b_2}(w) & = \Big( \frac{3^{1/2+w}}{3^{1/2+w}-1} \Big)^2 \cdot \mathcal{F}(w) \cdot \mathcal{H}_{b_1, b_2}(w) \label{G_product_expression},
\end{align}  
for
\begin{align*}
    \mathcal{F}(w) &:= \prod_{\substack{\pi \text{ prime} \\ \pi \equiv 1 \pmod 3}} E_0(\pi, w)
\end{align*}
and
\begin{align}
    \mathcal{H}_{b_1, b_2}(w) := N(b_1 b_2)^{1/2 + w} \prod_{\substack{\pi \text{ prime} \\ \pi \equiv 1 \pmod 3\\ \pi \mid b}} E_1(\pi , w) \prod_{\substack{\pi \text{ prime} \\ \pi \equiv 1 \pmod 3\\ \pi \mid \frac{b_1 b_2}{b^2}}} E_2(\pi, w) \prod_{\substack{\pi \text{ prime} \\ \pi \equiv 1 \pmod 3\\ \pi \mid {b_1 b_2}}} E_3(\pi, w), \label{H_Euler_prod_def}
\end{align}
where, assuming from now on that $\Re(w) \geq -\frac{1}{2}+\varepsilon$, 
\begin{align*} 
    E_0(\pi, w) &:= \Big( 1 + \Big( 1 + \frac{1}{N(\pi)} \Big)^{-1} \frac{(E^{-1}G)(\pi, w)}{N(\pi)^{1+2w}-1} \Big) \cdot E(\pi, w), \\
    E_1(\pi, w) &:= \Big( 1 + \frac{1}{N(\pi)} \Big)^{-1} \frac{(E^{-1}G)(\pi, w)}{N(\pi)^{1+2w}-1} = O_\varepsilon\Big(\frac{1}{N(\pi)^{1+2\Re(w)}}\Big),\\
    E_2(\pi, w) &:= \Big( 1 + \Big( 1 + \frac{1}{N(\pi)} \Big)^{-1} \frac{(F^{-1}G)(\pi, w)}{N(\pi)^{1+2w}-1} \Big) \cdot (E^{-1}F)(\pi, w) = O_\varepsilon\Big(\frac{1}{N(\pi)^{1+2\Re(w)}}\Big), \\
    E_3(\pi, w) &:= \Big( 1 + \Big( 1 + \frac{1}{N(\pi)} \Big)^{-1} \frac{(E^{-1}G)(\pi, w)}{N(\pi)^{1+2w}-1} \Big)^{-1} = 1 + O_\varepsilon\Big(\frac{1}{N(\pi)^{1+2 \Re(w)}}\Big).
\end{align*} 
We easily obtain the uniform bound
\begin{equation}
    \mathcal{H}_{b_1, b_2}(w) \ll N(b_1b_2)^\varepsilon N\Big(\frac{b_1b_2}{b^2}\Big)^{-1/2-\Re(w)} \ll X^\varepsilon N\Big(\frac{b_1b_2}{b^2}\Big)^{-1/2-\Re(w)}, \label{H_uniform_bound}
\end{equation}
and a tedious computation gives
\begin{align}
    E_0(\pi, w)  \Big(1 - \frac{1}{N(\pi)^{1+2w}}\Big) = &\ 1 - \frac{1}{N(\pi)^{1+2w}} \nonumber \\
    & \quad + \Big(1 + \frac{1}{N(\pi)}\Big)^{-1} \Big(\frac{2}{N(\pi)^{3/2+3w}-1} + \frac{1}{N(\pi)^{1+2w}} \Big) \label{P_Euler_factor}\\
    = &\ 1 + \frac{2}{N(\pi)^{3/2 + 3w} - 1} + O_\varepsilon\Big( \frac{1}{N(\pi)^{1+\varepsilon}}\Big) \nonumber \\
    = &\ \Big(1 - \frac{1}{N(\pi)^{3/2 + 3w}}\Big)^{-2}\Big(1 - \frac{1}{N(\pi)^{3 + 6w}}\Big)\Big(1 + O_\varepsilon\Big( \frac{1}{N(\pi)^{1+\varepsilon}}\Big) \Big). \nonumber
\end{align}
Thus
\begin{equation*}
    \mathcal{F}(w) := \prod_{\substack{\pi \text{ prime} \\ \pi \equiv 1 \pmod 3}} E_0(\pi, w) = \frac{\zeta_\lambda(1+2w) \zeta_\lambda(3/2+3w)^2}{\zeta_\lambda(3+6w)} \mathcal{J}(w),
\end{equation*}
where $\mathcal{J}(w)$ is holomorphic and uniformly bounded for $\Re(w) \geq -\frac{1}{2}+\varepsilon$. Inserting this into \eqref{G_product_expression}, we conclude that $\mathcal{G}_{b_1, b_2}(w)$ is meromorphic for $\Re(w) \geq -\frac{1}{3}+\varepsilon$ with a simple pole at $w=0$, a double pole at $w=-\frac{1}{6}$ (since one can check that $\mathcal{J}(-1/6)\neq 0$), and no other poles. 

We shift the line of integration in \eqref{after-mellin} to $\Re(w) = -\frac{1}{6} + \varepsilon$, and conclude -- using the convexity bound for $\zeta_{\lambda}(1+2w)$, the absolute convergence of the remaining Euler products, and \eqref{H_uniform_bound} -- that the remaining integral is
\begin{equation*}
    \ll X^{5/6+2\varepsilon} N\Big(\frac{b_1b_2}{b^2}\Big)^{-1/3} \int_{-1/6+\varepsilon-i\infty}^{-1/6+\varepsilon+i\infty} |\Gamma(1/2 + w)|^2 |\check{F}(w)| |w|^{100}  |dw| \ll X^{5/6+\varepsilon} N\Big(\frac{b_1b_2}{b^2}\Big)^{-1/3}.
\end{equation*}
We get from  \eqref{after-mellin} that
\begin{align} 
\mathcal{M}(b_1, b_2) = \frac{2 \pi X}{3^{9/2} \zeta_\lambda(2)} \frac{d}{dw} \Big[  \Big(\frac{4 \pi^2}{3X} \Big)^{-w} \Big(\frac{\Gamma (\frac{1}{2}+w )}{\Gamma (\frac{1}{2} )} \Big)^2 \check{F}(w) w \mathcal{G}_{b_1, b_2}(w)   \Big]_{w=0} & \nonumber \\ 
+\ O_\varepsilon \Big( \frac{X^{1+\varepsilon}}{Y} + X^{5/6+\varepsilon} N\Big(\frac{b_1b_2}{b^2}\Big)^{-1/3} \Big) & \nonumber
\end{align}
since there is a double pole of the integrand at $w=0$.

\begin{remark}
    We could shift all the way to $\Re(w) = -\frac{1}{3}+\varepsilon$, collecting the double pole of the integrand at $w=-\frac{1}{6}$ to show that the remaining integral is equal to
    \begin{equation*}
        X^{5/6} P_{b_1, b_2}(\log{X}) + O_{b_1, b_2, \varepsilon}(X^{2/3+\varepsilon})
    \end{equation*}
    for some explicit polynomial $P_{b_1, b_2}$ of degree $1$. This would not directly improve our final results, so we refrain from doing so for simplicity. However, observe that this matches the size of the second order main term in a previously mentioned conjecture of Diaconu \cite[Conjecture~4.5]{Dia}.
\end{remark}

It remains to compute the residue at $w = 0$, which is a double pole of the integrand in \eqref{after-mellin}. By the identity \eqref{zeta_lambda_identity} and the class number formula,
\begin{align}
    \mathop{\mathrm{Res}}_{w=0} \zeta_\lambda(1+2w) & = \frac{1}{2}\cdot \mathop{\mathrm{Res}}_{s=1} \zeta_\lambda(s) = \frac{(1-3^{-1})}{2} \cdot \mathop{\mathrm{Res}}_{s=1} \zeta_{\Q(\omega)}(s) \nonumber \\
    & = \frac{1}{3}\cdot \frac{2\pi \cdot h_{\Q(\omega)}}{w_{\Q(\omega)} \cdot \sqrt{|d_{\Q(\omega)}|}} = \frac{\pi}{9 \sqrt{3}}. \label{zeta_residue}
\end{align}
Around $w=0$ we have the series expansions
\begin{align*}
    &\Big(\frac{3X}{4 \pi^2} \Big)^{w} = 1 + w \log\Big(\frac{3X}{4 \pi^2} \Big) + O_X(w^2), \\
    &\Big(\frac{\Gamma (\frac{1}{2}+w )}{\Gamma (\frac{1}{2} )} \Big)^2 \Big( \frac{3^{1/2+w}}{3^{1/2+w}-1} \Big)^2 w \zeta_\lambda(1+2w) = \frac{\pi}{6\sqrt{3}(2-\sqrt{3})} + cw + O(w^2), \\
    & \check{F}(w) = \check{F}(0)\Big(1 + w \frac{\check{F}'}{\check{F}}(0) + O_F(w^2)\Big), \\
    & \mathcal{P}(w) := \frac{\zeta_\lambda(3/2+3w)^2}{\zeta_\lambda(3+6w)} \mathcal{J}(w) = \mathcal{P}(0) \Big( 1 + w \frac{\mathcal{P}'}{\mathcal{P}}(0) + O(w^2) \Big), \\
    & \mathcal{H}_{b_1, b_2}(w) = \mathcal{H}_{b_1, b_2}(0)\Big( 1 + w \frac{\mathcal{H}'_{b_1, b_2}}{\mathcal{H}_{b_1, b_2}}(0) + O_{b_1, b_2}(w^2) \Big),
\end{align*}
for some absolute constant $c$. Therefore the residue at $w = 0$ gives rise in \eqref{after-mellin} to a term
\begin{equation}\label{main_term_pole}
    \frac{\pi^2 X}{2^3\cdot 3^4 \cdot(2-\sqrt{3}) \cdot \zeta_{\Q(\omega)}(2)} \check{F}(0) \mathcal{P}(0) \mathcal{H}_{b_1, b_2}(0) \Big[ \log{X} + C_0  + \frac{\mathcal{H}'_{b_1, b_2}}{\mathcal{H}_{b_1, b_2}}(0) \Big],
\end{equation}
where $C_0$ is a linear combination of $1$ and $\frac{\check{F}'}{\check{F}}(0)$ with (absolute) constant coefficients (note that $\frac{\mathcal{P}'}{\mathcal{P}}(0)$ appears inside $C_0$). From \eqref{P_Euler_factor} we obtain
\begin{equation*}
    \mathcal{P}(0) = \prod_{\substack{\pi \text{ prime} \\ \pi \equiv 1 \pmod 3 \\ q\; :=\; N(\pi)}} \Big( 1 - \frac{1}{q(q+1)} + \frac{2q}{(q+1) (q^{3/2}-1)} \Big) > 0.
\end{equation*}

Now \eqref{H_Euler_prod_def} also gives
\begin{align}
    \frac{\mathcal{H}'_{b_1, b_2}}{\mathcal{H}_{b_1, b_2}}(w) = \log{N(b_1b_2)} + \sum_{\substack{\pi \text{ prime} \\ \pi \equiv 1 \pmod 3\\ \pi \mid b}} \frac{(E_1 E_3)'}{E_1 E_3}(\pi , w) + \sum_{\substack{\pi \text{ prime} \\ \pi \equiv 1 \pmod 3\\ \pi \mid \frac{b_1 b_2}{b^2}}} \frac{(E_2 E_3)'}{E_2 E_3}(\pi , w). \label{H_log_deriv}
\end{align}
Denoting $q := N(\pi)$, we can compute
\begin{align}
    (E_1 E_3)(\pi, w) = \frac{q (q^{3/2+3w} + 1)}{q^{7/2+5w} + q^{5/2+5w} - q^{3/2+3w} + q^{2+2w} - q^{1+2w} + 1} \label{D_1_eq}
\end{align}
and
\begin{align}
    (E_2 E_3)(\pi, w) = \frac{q^{2+2w} (q^{1/2+w} + 1)}{q^{7/2+5w} + q^{5/2+5w} - q^{3/2+3w} + q^{2+2w} - q^{1+2w} + 1} . \label{D_2_eq}
\end{align}
It follows that
\begin{align*}
    \frac{(E_1 E_3)'}{(E_1 E_3)}(\pi, 0) = -2 \log{q} + D_{1}(\pi) \frac{\log{q}}{q}
\end{align*}
and
\begin{align*}
    \frac{(E_2 E_3)'}{(E_2 E_3)}(\pi, 0) = -2 \log{q} + D_{2}(\pi) \frac{\log{q}}{\sqrt{q}},
\end{align*}
where $D_i(\pi) \ll 1$ for $i \in \{1, 2\}$. Thus \eqref{H_log_deriv}, \eqref{D_1_eq}, and \eqref{D_2_eq} imply
\begin{align*}
    \frac{\mathcal{H}'_{b_1, b_2}}{\mathcal{H}_{b_1, b_2}}(0) = - \log{N\Big(\frac{b_1 b_2}{b^2}\Big)} + \sum_{\substack{\pi \text{ prime} \\ \pi \equiv 1 \pmod 3\\ \pi \mid b}} D_{1}(\pi) \frac{\log{N(\pi)}}{N(\pi)} + \sum_{\substack{\pi \text{ prime} \\ \pi \equiv 1 \pmod 3\\ \pi \mid \frac{b_1 b_2}{b^2}}} D_{2}(\pi) \frac{\log{N(\pi)}}{\sqrt{N(\pi)}}.
\end{align*}

Furthermore, 
\begin{align*}
    \mathcal{H}_{b_1, b_2}(0) = N\Big(\frac{b_1 b_2}{b^2}\Big)^{-1/2} g(b) h\Big(\frac{b_1 b_2}{b^2}\Big),
\end{align*}
where $g$ and $h$ are the multiplicative functions given, for $k\geq 1$, by
\begin{align*}
    g(\pi^k) := N(\pi)\cdot (E_1 E_3)(\pi, 0) = 1 - \frac{(q^{3/2} - 1)(q - 1)}{q^{7/2} + q^{5/2} + q^{2} - q^{3/2} - q + 1} = 1 + O\Big(\frac{1}{N(\pi)}\Big)
\end{align*}
and
\begin{align*}
    h(\pi^k) := N(\pi) \cdot (E_2 E_3)(\pi, 0) = 1 + \frac{(q^2 - q^{3/2} + 1)(q - 1)}{q^{7/2} + q^{5/2} + q^{2} - q^{3/2} - q + 1} = 1 + O\Big(\frac{1}{\sqrt{N(\pi)}}\Big).
\end{align*}

Inserting this into \eqref{main_term_pole} and consolidating our work, we conclude that
\begin{align*}
    \mathcal{M}(b_1, b_2) = D \check{F}(0) X N\Big(\frac{b_1 b_2}{b^2}\Big)^{-1/2} g(b) h\Big(\frac{b_1 b_2}{b^2}\Big) \Big[ \log{\Big(\frac{X N(b^2)}{N(b_1 b_2)}\Big)} + \mathcal{O}(b_1, b_2) \Big] & \\
    +\ O_\varepsilon\Big(\frac{X^{1+\varepsilon}}{Y} + X^{5/6+\varepsilon} N\Big(\frac{b_1b_2}{b^2}\Big)^{-1/3}\Big)&
\end{align*}
for
\begin{equation*}
    D :=  \frac{\pi^2 \cdot \mathcal{P}(0)}{2^3\cdot 3^4 \cdot(2-\sqrt{3}) \cdot \zeta_{\Q(\omega)}(2)} = \frac{\pi^2 \cdot \mathcal{P}(0)}{648(2-\sqrt{3}) \cdot \zeta_{\Q(\omega)}(2)}
\end{equation*}
and
\begin{equation*}
    \mathcal{O}(b_1, b_2) := C_0 + \sum_{\substack{\pi \text{ prime} \\ \pi \equiv 1 \pmod 3\\ \pi \mid b}} D_{1}(\pi) \frac{\log{N(\pi)}}{N(\pi)} + \sum_{\substack{\pi \text{ prime} \\ \pi \equiv 1 \pmod 3\\ \pi \mid \frac{b_1 b_2}{b^2}}} D_{2}(\pi) \frac{\log{N(\pi)}}{\sqrt{N(\pi)}},
\end{equation*}
which gives the desired main term.

 
\subsection{The error term \texorpdfstring{$\mathcal{R}(b_1, b_2)$}{}: initial manipulations}

It remains to show the bound \eqref{R_average_bound} for the error terms. Note that it suffices to show that
\begin{align}
    \mathfrak{S}(B_1, B_2) &:= \mathop{\sum\sum}_{\substack{b_1, b_2 \in \Z[\omega] \\ b_1, b_2 \equiv 1 \pmod 3  \\ N(b_1) \sim B_1,\ N(b_2) \sim B_2}} \mu^2(b_1) \mu^2(b_2) c(b_1, b_2) \mathcal{R}(b_1, b_2) \nonumber \\
    &\ll_{F, \varepsilon} X^{1/2+\varepsilon} (B_1B_2)^{1/2} \Big(X^{1/6} Y^{1/2}B^{3/2} + Y B^2 + X^{1/3}B\Big) \label{B_1B_2_final_bound}
\end{align}
for $\frac{1}{2} \leq B_1, B_2 \leq X^{100}$, $B := \max\{B_1, B_2\}$, and arbitrary coefficients with $|c(b_1, b_2)| \leq 1$. Our aim for the rest of this section is to prove \eqref{B_1B_2_final_bound}.

We have almost separated the variables $m_1$ and $m_2$ in \eqref{expression_to_split}, except for the condition $(m_1, m_2) = 1$ and the Archimedean transform $\ddot{F}$. This can be remedied by standard maneuvers via M{\"o}bius inversion and Mellin inversion, respectively. First notice the following simple but important bound.

\begin{lemma}\label{F_dot_dot_bound_lemma}
    For any $\mathfrak{n}_1, \mathfrak{n}_2 \unlhd \Z[\omega]$, $u \in \C$, and $A \in \Z_{\geq 0}$, we have the uniform bound
    \begin{align*}
    \ddot{F}_{\mathfrak{n}_1, \mathfrak{n}_2}(u) \ll_{F, A} \Big(1 + |u| + \frac{N(\mathfrak{n}_1 \mathfrak{n}_2)}{X}\Big)^{-A}.
\end{align*}
\end{lemma}

\begin{proof}
From \eqref{Vddot} and integration by parts, antidiferentiating the Bessel function as in \cite[(4.11)]{DR}, for any $j \in \Z_{\geq 0}$ (and assuming $u\neq 0$ if $j \neq 0$) we have
\begin{equation*}
    \ddot{F}_{\mathfrak{n}_1, \mathfrak{n}_2}(u) = (-1)^j \Big(\frac{9\sqrt{3}}{2\pi}\Big)^j \frac{1}{|u|^j} \int_0^\infty F_{\mathfrak{n}_1, \mathfrak{n}_2}^{(j)} (r^2) \cdot r^{j+1} J_j\Big(\frac{4\pi r |u|}{9 \sqrt{3}}\Big) dr.
\end{equation*}
Recall from \eqref{F_def} that $F_{\mathfrak{n}_1, \mathfrak{n}_2}(t) = F(t) \Phi_2\Big(\frac{N(\mathfrak{n}_1 \mathfrak{n}_2)}{3Xt}\Big)$. Then \eqref{Phi_bound} and the fact that $F$ has support in $(1, 2)$ give 
\begin{equation*}
    F_{\mathfrak{n}_1, \mathfrak{n}_2}^{(j)}(t) \ll_{F, j, A} \Big(1 + \frac{N(\mathfrak{n}_1 \mathfrak{n}_2)}{X}\Big)^{-A}.
\end{equation*}
Putting those together for $j=0$ if $|u| \leq 1$ or $j=A$ if $|u|\geq 1$ gives the desired result.

\end{proof}

\begin{remark}\label{F_bound_remark}
    From \cref{F_dot_dot_bound_lemma} and $\log N(b_1b_2\ell) \ll \log{X}$, we have that if $N(m_1m_2) \gg \frac{X^{1+\varepsilon} N(b_1b_2)}{N(d_1 d_2 d^2)}$ or $N(k) \gg \frac{N(\ell^2 m_1 m_2 d_1 d_2 e)}{X^{1-\varepsilon}}$, then
    \begin{equation*}
        \ddot{F}_{\frac{\lambda^{g_1} d d_1 m_1}{b_1},\frac{\lambda^{g_2} d d_2 m_2}{b_2} }  \Big( \frac{k \sqrt{X}}{\ell^2 m_1 m_2 d_1 d_2 e} \Big) \ll_{F, A, \varepsilon} X^{-A} N(m_1 m_2 d_1 d_2 d k)^{-A}.
    \end{equation*}
\end{remark}

We now assume $u \neq 0$. Shifting the line of integration in \eqref{def-Phi}, by absolute convergence and Stirling's formula we also obtain
\begin{align*}
    \ddot{F}_{\mathfrak{n}_1, \mathfrak{n}_2}(u) &= \frac{1}{2\pi i} \int_{\varepsilon -i\infty}^{\varepsilon + i\infty} \Big(\frac{4 \pi^2 N(\mathfrak{n}_1 \mathfrak{n}_2)}{3X}\Big)^{-w} \Big(\frac{\Gamma (\frac{1}{2}+w )}{\Gamma (\frac{1}{2} )}\Big)^2 \int_0^\infty r^{2w} F(r^2)  r J_0\Big(\frac{4\pi r|u|}{9\sqrt{3}}\Big) dr \frac{dw}{w} \\
    & = \frac{1}{2\pi i} \int_{\varepsilon -iX^\varepsilon}^{\varepsilon + iX^\varepsilon} \Big(\frac{4 \pi^2 N(\mathfrak{n}_1 \mathfrak{n}_2)}{3X}\Big)^{-w} \Big(\frac{\Gamma (\frac{1}{2}+w)}{\Gamma (\frac{1}{2})}\Big)^2 \mathcal{I}(w, u) \frac{dw}{w} + O_{A, \varepsilon}(X^{-A})
\end{align*}
for all sufficiently large $X>0$, where $\mathcal{I}(w, u)$ is given in \eqref{Iwudef}. Note that
the truncation is justified since $\mathcal{I}(w, u) \ll 1$ uniformly for $\Re(w) = \varepsilon$. Furthermore, using integration by parts and the Mellin--Barnes integral as in \cite[(7.18)]{DR} gives
\begin{align} \label{Iwudef}
    \mathcal{I}(w, u) := &\ \int_0^\infty r^{2w} F(r^2)  r J_0\Big(\frac{4\pi r|u|}{9\sqrt{3}}\Big) dr \\
    =&\ \frac{(-1)^j}{2\pi i} \int_0^\infty \int_{-\varepsilon -i\infty}^{-\varepsilon + i\infty} G_w^{(j)} (r^2) r^{2j+1}\frac{\Gamma(-s)}{\Gamma(j+s+1)} \Big(\frac{2\pi r|u|}{9\sqrt{3}}\Big)^{2s} ds\: dr
\end{align}
for $G_w(y) := y^w F(y)$, $u \neq 0$, and $j\in \Z_{\geq 1}$. Observe that $G_w^{(j)}$ is supported in $(1, 2)$ and satisfies the uniform bound
\begin{equation*}
    G_w^{(j)}(y) \ll_{F, j} (1+|w|)^j.
\end{equation*}
Choosing $j$ sufficiently large in terms of $\varepsilon$ (but fixed), Stirling's formula implies that
\begin{equation}\label{F_transform_bulk}
    \ddot{F}_{\mathfrak{n}_1, \mathfrak{n}_2}(u) = \int_{\varepsilon -iX^\varepsilon}^{\varepsilon + iX^\varepsilon}\int_1^{\sqrt{2}} \int_{-\varepsilon -iX^{\varepsilon}}^{-\varepsilon + iX^{\varepsilon}} \mathcal{G}_2(w, r, s)  \frac{|u|^{2s}}{N(\mathfrak{n}_1 \mathfrak{n}_2)^{w}} ds\: dr\: dw + O_{F, \varepsilon}((1 + |u|^{-2\varepsilon})X^{-2000}).
\end{equation}
Note that to obtain the above display we also used that 
\begin{align*}
    \mathcal{G}_2(w, r, s) :=&\ \frac{(-1)^j}{(2\pi i)^2} \Big(\frac{4\pi^2}{3X}\Big)^{-w} \Big(\frac{\Gamma (\frac{1}{2}+w )}{\Gamma (\frac{1}{2} )}\Big)^2 \frac{G_w^{(j)}(r^2) r^{2j+1}}{w} \frac{\Gamma(-s)}{\Gamma(j+s+1)} \Big(\frac{2\pi r}{9\sqrt{3}}\Big)^{2s} \nonumber \\
    \ll_{F, j} &\ \frac{X^\varepsilon \Gamma(\frac{1}{2}+w)^2 (1+|w|)^j \Gamma(-s)}{(j+s)(j-1+s) \cdots (1+s) \Gamma(1+s)} \ll_j \frac{X^\varepsilon}{(1+ |\Im(s)|)^j}
\end{align*}
for $\Re(w) = \varepsilon$ and $\Re(s) = -\varepsilon$.
In particular, we obtain the uniform bound
\begin{align}\label{G_func_uniform_bound}
    \mathcal{G}_2(w, r, s) \ll_F X^\varepsilon.
\end{align}

It is convenient to introduce a smooth non-negative function $H$ with compact support in $(1, 3)$ which furnishes a partition of unity
\begin{equation}\label{partition_of_unity}
    1 = \sum_{h \in \Z} H\Big(\frac{x}{2^h}\Big) \qquad \text{for every } x>0. 
\end{equation}
Applying M{\"o}bius inversion to the condition $(m_1, m_2) = 1$ and adding partitions of unity to the sums over $m_1$ and $m_2$, we conclude that if $k\neq 0$ and $N(\ell) \sim L \gg 1$ then the sums over $m_1$ and $m_2$ in \eqref{expression_to_split} are equal to
\begin{align}
    &\sumtwo_{\substack{\alpha_1, \alpha_2 \in \Z \\ M_i = 2^{\alpha_i} \gg 1 \\ M_1 M_2 \ll \frac{X^{1+\varepsilon} N(b_1b_2)}{N(d_1 d_2 d^2)}}} \bbone_{N(k) \ll \frac{M_1 M_2 L^2 N(d_1 d_2 e)}{X^{1-\varepsilon}}} \sum_{\substack{f \in \Z[\omega] \\ f \equiv 1 \pmod{3} \\ (f, d)=1 \\ N(f) \ll M_1+M_2}} \mu(f) \int_{\varepsilon -iX^\varepsilon}^{\varepsilon + iX^\varepsilon}\int_1^{\sqrt{2}} \int_{-\varepsilon -iX^{\varepsilon}}^{-\varepsilon + iX^{\varepsilon}} \frac{\mathcal{G}_2(w, r, s) N(b_1b_2)^w}{N(\lambda^{g_1+g_2} d_1 d_2 d^2)^w} \nonumber \\
    &\times \Big(\frac{N(k)X}{N(\ell^2 d_1 d_2 e)}\Big)^s \Bigg(\sum_{\substack{m_1 \in \Z[\omega] \\ m_1 \equiv c_1 \pmod{9} \\ (m_1, d) = 1 \\ b_1 \mid dd_1m_1 \\ f \mid m_1}} \frac{\overline{\chi_{m_1}(\ell d_1 e^2)} \widetilde{g}_3(k, m_1)}{N(m_1)^{1+w+s}} H\Big(\frac{N(m_1)}{M_1}\Big) \Bigg) \nonumber \\
    &\times \Bigg( \sum_{\substack{m_2 \in \Z[\omega] \\ m_2 \equiv c_2 \pmod{9} \\ (m_2, d) = 1 \\ b_2 \mid dd_2m_2 \\ f \mid m_2}} \frac{\chi_{m_2}(\ell d_2 e^2) \overline{\widetilde{g}_3(k, m_2)}}{N(m_2)^{1+w+s}} H\Big(\frac{N(m_2)}{M_2}\Big) \Bigg) ds\: dr\: dw  \label{m_separation_Mellin}
\end{align}
plus the contributions from the error term in \eqref{F_transform_bulk} and the remaining ranges of $M_1, M_2$, and $N(k)$, which by \cref{F_bound_remark} contribute to \eqref{expression_to_split} a negligible error term of $O_{F, \varepsilon}(X^{-1000})$. The bound on $f$ and the condition $(f, d) = 1$ were added since otherwise the summands are zero. 

Thus
\begin{equation*}
    \mathcal{R}(b_1, b_2) = \sumtwo_{\substack{\alpha_1, \alpha_2 \in \Z \\ M_i = 2^{\alpha_i} \gg 1 \\ M_1 M_2 \ll X^{1+\varepsilon} N(b_1b_2)}} \sum_{\substack{\alpha \in \Z \\ L = 2^{\alpha} \gg 1 \\ L \leq Y}} \mathcal{R}^L_{M_1, M_2}(b_1, b_2) + O_{F, \varepsilon}(X^{-1000}),
\end{equation*}
where $\mathcal{R}^L_{M_1, M_2}(b_1, b_2)$ corresponds to the terms of $\mathcal{R}(b_1, b_2)$ with $N(m_i)$ localized at $M_i$ (using the partition of unity $H$) and $N(\ell)$ localized at $L$ (using a dyadic decomposition $N(\ell) \sim L$). Therefore
\begin{align}
    \mathfrak{S}(B_1, B_2) = \sumtwo_{\substack{\alpha_1, \alpha_2 \in \Z \\ M_i = 2^{\alpha_i} \gg 1 \\ M_1 M_2 \ll X^{1+\varepsilon} B_1 B_2}} \sum_{\substack{\alpha \in \Z \\ L = 2^{\alpha} \gg 1 \\ L \leq Y}} \mathfrak{S}_{M_1, M_2}^L(B_1, B_2) + O_{F, \varepsilon}(X^{-500}), \label{B_1_B_2_sum_decomposition}
\end{align}
where
\begin{align}\label{localized_sum_def}
    \mathfrak{S}_{M_1, M_2}^L(B_1, B_2) := \mathop{\sum\sum}_{\substack{b_1, b_2 \in \Z[\omega] \\ b_1, b_2 \equiv 1 \pmod 3  \\ N(b_1) \sim B_1,\ N(b_2) \sim B_2}} \mu^2(b_1) \mu^2(b_2) c(b_1, b_2) \mathcal{R}_{M_1, M_2}^L(b_1, b_2).
\end{align}

By the supplement to cubic reciprocity in \eqref{cubesupp}, if $m, c\equiv 1\pmod{3}$ then
\begin{equation}\label{congruence_with_chars}
    \bbone_{m\equiv c \pmod 9} = \frac{1}{9} \sum_{a, b=0}^2 \chi_c(\omega^a \lambda^b) \overline{\chi_m(\omega^a \lambda^b)} = \frac{1}{18} \sum_{\eta \mid 3} \chi_c(\eta) \overline{\chi_m(\eta)},
\end{equation}
so we can detect the conditions $m_i \equiv c_i\pmod{9}$ with linear combinations of cubic characters using \eqref{congruence_with_chars}. With $m_1$ and $m_2$ separated in \eqref{m_separation_Mellin}, we swap the integrals and the sum over $f$ with the sums over $\ell$ and $k$ in \eqref{expression_to_split}, and then put absolute values around the sum over $k$. We also introduce the sums over $b_1$ and $b_2$ as in \eqref{localized_sum_def}, obtaining an upper bound for $\mathfrak{S}_{M_1, M_2}^L(B_1, B_2)$. Recalling \eqref{G_func_uniform_bound}, and writing
\begin{equation}\label{f_i_def}
    f_i := \Big[\frac{b_i}{(b_i, dd_i)}, f \Big] = \Big[\frac{b_i}{(b_i, d)}, f \Big],
\end{equation}
which is squarefree (since so are $b_i$ and $f$), the resulting bound is
\begin{align}
    &\mathfrak{S}_{M_1, M_2}^L(B_1, B_2) \ll_F X^{1+\varepsilon} (B_1 B_2)^{1/2}  \sup_{\substack{t, t_1, t_2 \in \R \\ |t|, |t_1|, |t_2| \leq X^{\varepsilon}}}  \mathop{\sum\sum}_{\eta_1, \eta_2 \mid 3} \sum_{\substack{c \pmod 9 \\ c \equiv 1 \pmod{3}}} \mathop{\sum\sum\sum}_{\substack{d, d_1, d_2 \in \Z[\omega] \\ d, d_1, d_2 \equiv 1 \pmod{3} \\ (d_1, d_2)=1,\ d_1d_2 \mid d^\infty \\ N(d_1 d_2 d^2) \ll \frac{X^{1+\varepsilon} B_1 B_2}{M_1 M_2} \\ N(d) \gg \frac{B_i}{M_i}}} \frac{1}{N(d)} \nonumber \\
    & \times  \sum_{\substack{f \in \Z[\omega] \\ f \equiv 1 \pmod{3} \\ N(f) \ll M_1+M_2 \\ (f, d) = 1}} \mu^2(f) \sum_{\substack{0 \neq k \in \Z[\omega] \\ N(k) \ll \frac{M_1M_2 L^2 N(d_1 d_2 e)}{X^{1-\varepsilon}}}} |\mathcal{E}(d_1, d_2; k, e)| \Bigg|  \mathop{\sum\sum}_{\substack{b_1, b_2 \in \Z[\omega] \\ b_1, b_2 \equiv 1 \pmod 3  \\ N(b_1) \sim B_1\\ N(b_2) \sim B_2}} \mu^2(b_1) \mu^2(b_2) c(b_1, b_2) \nonumber \\
    & \qquad \times \sum_{\substack{\ell \in \Z[\omega] \\ \ell \equiv c \pmod 9 \\ N(\ell) \sim L \\ N(\ell) \leq Y \\ (\ell, d) = 1 }} \frac{\mu(\ell) \overline{\psi_{d_1, d_2}(\ell)}}{N(\ell)^{2-\varepsilon + it}} \Bigg( \sum_{\substack{m_1 \in \Z[\omega] \\ m_1 \equiv 1 \pmod{3} \\ (m_1, d)  = 1 \\ f_1 \mid m_1}} \frac{\overline{\chi_{m_1}(\eta_1 d_1 e^2 \ell)} \widetilde{g}_3(k, m_1)}{N(m_1)^{1+it_1}} H\Big(\frac{N(m_1)}{M_1}\Big) \Bigg) \nonumber \\ 
    &  \qquad \qquad \qquad  \times \Bigg(\sum_{\substack{m_2 \in \Z[\omega] \\ m_2 \equiv 1 \pmod{3} \\ (m_2, d)  = 1 \\ f_2 \mid m_2}} \frac{\chi_{m_2}(\eta_2 d_2 e^2 \ell) \overline{\widetilde{g}_3(k, m_2)}}{N(m_2)^{1-it_2}} H\Big(\frac{N(m_2)}{M_2}\Big) \Bigg) \Bigg| + X^{-500}. \label{R_after_abs_values}
\end{align}
We added a congruence condition on $\ell \pmod 9$ to remove the exponential phase. The upper bounds on $M_1M_2$, $d_1d_2d^2$, and $k$ follow from the ranges in \eqref{m_separation_Mellin}. The lower bound on $d$ comes from the divisibility condition $\frac{b_i}{(b_i, d)} \mid f_i \mid m_i$, so that 
\begin{equation*}
    M_i \gg N(m_i) \geq N\Big(\frac{b_i}{(b_i, d)}\Big) \geq \frac{N(b_i)}{N(d)} \gg \frac{B_i}{N(d)}. 
\end{equation*}


\subsection{The error term \texorpdfstring{$\mathcal{R}(b_1, b_2)$}{}: preparing to evaluate}

 For $i\in\{1, 2\}$, write (uniquely) $m_i = w_iv_i$ for $w_i, v_i\equiv 1\pmod{3}$ with $w_i \mid k^\infty$ and $(v_i, k)=1$, and observe that 
\begin{equation*}
    \widetilde{g}_3(k, w_iv_i) = \widetilde{g}_3(k, w_i) \overline{\chi_{v_i}(w_i)} \widetilde{g}_3(k, v_i) =  \widetilde{g}_3(k, w_i) \overline{\chi_{v_i}(kw_i)} \widetilde{g}_3(v_i).
\end{equation*}
Thus the sum over $m_i$ in \eqref{R_after_abs_values} is (up to complex conjugation) equal to
\begin{align*}
    &\sum_{\substack{w_i \in \Z[\omega] \\ w_i \equiv 1 \pmod{3} \\ (w_i, d)=1 \\ w_i\mid k^\infty}} \frac{\overline{\chi_{w_i}(\eta_i d_i e^2 \ell)}\widetilde{g}_3(k, w_i)}{N(w_i)^{1+it_i}} \sum_{\substack{v_i \in \Z[\omega] \\ v_i \equiv 1  \pmod{3} \\ (v_i, d)=1 \\ f_i \mid w_i v_i}} \frac{\overline{\chi_{v_i}(\eta_i d_i e^2 w_i k \ell)} \widetilde{g}_3(v_i)}{N(v_i)^{1+it_i}} H\Big(\frac{N(w_i v_i)}{M_i}\Big), 
\end{align*}
where we removed the condition $(v_i, k)=1$ since it is enforced by $\overline{\chi_{v_i}(k)}$.

We write
\begin{equation}\label{q_i_def}
    q_i := \frac{f_i}{(f_i, w_i)},
\end{equation} 
which is squarefree (since so is $f_i$), so the divisibility condition for the sum over $v_i$ becomes $q_i \mid v_i$. Note that we automatically have $(q_i, d) = 1$, since $b_i$ is squarefree and $(f, d)=1$, so by \eqref{f_i_def} we have $(f_i, d)=1$. 

Setting $v_i = q_in_i$ for $n_i\equiv 1\pmod{3}$ with $(n_i, d)=1$, 
observe that $v_i$ is squarefree (otherwise $g_3(v_i) = 0$), so $(q_i, n_i)=1$ and we once again have
\begin{equation*}
    \widetilde{g}_3(q_in_i) =  \widetilde{g}_3(q_i) \overline{\chi_{n_i}(q_i)} \widetilde{g}_3(n_i).
\end{equation*}
Using this identity (where we can drop the condition $(q_i, n_i) = 1$ due to the presence of $\overline{\chi_{n_i}(q_i)}$), commuting the sums over $w_1$ and $w_2$ past the sums over $b_1$, $b_2$, $\ell$, and removing them from the absolute value using the triangle inequality, we conclude that the expression in absolute values in \eqref{R_after_abs_values} is 
\begin{align}
    & \ll \mathop{\sum\sum}_{\substack{w_1, w_2 \in \Z[\omega] \\ w_1, w_2 \equiv 1 \pmod{3} \\ (w_1 w_2, d) = 1 \\ w_1 w_2 \mid k^\infty}} \frac{|\widetilde{g}_3(k, w_1)\widetilde{g}_3(k, w_2)|}{N(w_1w_2)}  \Bigg| \sum_{\substack{\ell \in \Z[\omega] \\ \ell \equiv c \pmod 9 \\ N(\ell) \sim L \\ N(\ell) \leq Y }} \frac{\mu(\ell) \overline{\psi_{d_1 e^2 w_1, d_2 e^2 w_2}(\ell)}}{N(\ell)^{2-\varepsilon + it}} \nonumber \\
    & \times \mathop{\sum\sum}_{\substack{b_1, b_2 \in \Z[\omega] \\ b_1, b_2 \equiv 1 \pmod 3  \\ N(b_1) \sim B_1\\ N(b_2) \sim B_2}} \mu^2(b_1) \mu^2(b_2) c(b_1, b_2)  \frac{\overline{\chi_{q_1}(\eta_1 d_1 e^2 w_1 k \ell)}\chi_{q_2}(\eta_2 d_2 e^2 w_2 k \ell)\widetilde{g}_3(q_1) \overline{\widetilde{g}_3(q_2)}}{N(q_1)^{1+it_1} N(q_2)^{1-it_2}} \nonumber \\
    &\qquad \qquad \times \Bigg( \sum_{\substack{n_1 \in \Z[\omega] \\ n_1 \equiv 1  \pmod{3} \\ (n_1, d)=1}} \frac{\overline{\chi_{n_1}(\eta_1 d_1 e^2 w_1 q_1 k \ell)} \widetilde{g}_3(n_1)}{N(n_1)^{1+it_1}} H\Big(\frac{N(w_1q_1n_1)}{M_1}\Big) \Bigg) \nonumber \\ 
    & \qquad \qquad \times \overline{\Bigg( \sum_{\substack{n_2 \in \Z[\omega] \\ n_2 \equiv 1  \pmod{3} \\ (n_2, d)=1}} \frac{\overline{\chi_{n_2}(\eta_2 d_2 e^2 w_2 q_2 k \ell)} \widetilde{g}_3(n_2)}{N(n_2)^{1+it_2}} H\Big(\frac{N(w_2 q_2 n_2)}{M_2}\Big) \Bigg)} \Bigg|. \label{almost_ready_for_poles_eq}
\end{align}

It is time to add the sum over $k$ to \eqref{almost_ready_for_poles_eq}, which we commute with the sums over $w_1, w_2$, picking up the condition $w_1 w_2 \mid k^\infty$. Now write $u := (k, \ell)$, so that $u \equiv 1 \pmod{3}$ and $u$ is squarefree (since so is $\ell$). Then we can change variables to $\ell = ul$, remove the $u$ sum from the absolute value via triangle inequality, commute it with the sum over $k$, and also change variables to $k = uk'$, where $0 \neq l, k' \in \Z[\omega]$, $lu \equiv c \pmod{9}$, and $(uk', l)  = 1$. For ease of notation we denote
\begin{equation}\label{Delta_i_def}
    \Delta_i := d_i e^2 w_i \equiv 1\pmod{3}.
\end{equation}

Then write $k' = \eta \lambda^{3j} \kappa \alpha \beta^2 \gamma^3 \delta^3$, where $\eta \mid 3$, $j\in \Z_{\geq 0}$, $\kappa, \alpha, \beta, \gamma, \delta \equiv 1\pmod{3}$, $\kappa \mid (u\Delta_1 \Delta_2)^\infty$, $(\alpha\beta\gamma, u\Delta_1\Delta_2) = 1$, $\mu^2(\alpha\beta\gamma) = 1$, and $\delta \mid (\alpha\beta\gamma)^\infty$. 
Observe also that $(u, \Delta_1 \Delta_2) = 1$ is enforced by $\overline{\psi_{\Delta_1, \Delta_2}(\ell)}$, so we can write (uniquely) $\kappa = \Upsilon \theta'$ for $\Upsilon \mid u^\infty$ and $\theta' \mid (\Delta_1 \Delta_2)^\infty$. The condition $w_1 w_2 \mid k^\infty$ is then equivalent to $w_1 w_2 \mid \theta'^\infty$ and $\theta' \mid (d w_1w_2)^\infty$. Since $(d, w_1w_2)=1$, we can write (uniquely) $\theta' = \theta \iota$ for $\theta^\infty = (w_1w_2)^\infty$ and $\iota \mid d^\infty$. Here and in what follows we use the shorthand $\theta^\infty = (w_1w_2)^\infty$ for $\mathrm{rad}(\theta) = \mathrm{rad}(w_1w_2)$. Then by \eqref{exp_sum_uniform_bound},
\begin{equation*}
    |\widetilde{g}_3(k, w_1)\widetilde{g}_3(k, w_2) \mathcal{E}(d_1, d_2; k, e)| \leq |\widetilde{g}_3(\theta, w_1)\widetilde{g}_3(\theta, w_2)| \frac{|\widetilde{g}_3(\iota, d_1) \widetilde{g}_3(\iota, d_2)|}{N(d_1d_2)} \frac{N((\iota, e))}{N(e)}.
\end{equation*}
Thus inserting these into \eqref{almost_ready_for_poles_eq} we see that the sums over $f$ and $k$ in \eqref{R_after_abs_values} are
\begin{align}
    & \ll X^\varepsilon \sum_{\substack{f \in \Z[\omega] \\ f \equiv 1 \pmod{3} \\ N(f) \ll M_1+M_2 \\ (f, d) = 1}} \mu^2(f) \mathop{\sum\sum}_{\substack{w_1, w_2 \in \Z[\omega] \\ w_1, w_2 \equiv 1 \pmod{3} \\ (w_1 w_2, d) = 1 \\ N(w_i) \ll M_i}} \mathop{\sum\sum\sum \sum\sum\sum\sum}_{\substack{u, \Upsilon, \theta, \iota, \alpha, \beta, \gamma \in \Z[\omega] \\ u, \Upsilon, \theta, \iota, \alpha, \beta, \gamma \equiv 1 \pmod 3 \\ \Upsilon \mid u^\infty,\ \theta^\infty = (w_1 w_2)^\infty,\ \iota \mid d^\infty \\ (u, \Delta_1\Delta_2) = (\alpha \beta \gamma, u\Delta_1\Delta_2) = 1 \\ \mu^2(u) = \mu^2(\alpha \beta \gamma) = 1 \\ N(u\Upsilon\theta\iota\alpha\beta^2\gamma^3) \ll \frac{M_1M_2 L^2 N(d_1d_2e)}{X^{1-\varepsilon}}}} \frac{|\widetilde{g}_3(\theta, w_1)\widetilde{g}_3(\theta, w_2)|}{N(w_1w_2)}
     \nonumber \\
    & \times   \frac{|\widetilde{g}_3(\iota, d_1) \widetilde{g}_3(\iota, d_2)|}{N(d_1d_2)} \frac{N((\iota, e))}{N(e)}\frac{1}{N(u)^2} \Bigg|  \sum_{\substack{l \in \Z[\omega] \\ lu \equiv c \pmod 9 \\ N(u l) \sim L \\ N(ul) \leq Y \\ (u\alpha\beta\gamma, l) = 1 }} \frac{\mu(l) \overline{\psi_{\Delta_1, \Delta_2}(l)}}{N(l)^{2-\varepsilon + it}} \mathop{\sum\sum}_{\substack{b_1, b_2 \in \Z[\omega] \\ b_1, b_2 \equiv 1 \pmod 3  \\ N(b_1) \sim B_1\\ N(b_2) \sim B_2}} \mu^2(b_1) \mu^2(b_2) \nonumber \\
    & \times c(b_1, b_2)  \frac{\overline{\chi_{q_1}(\eta_1 \Delta_1 u^2 \Upsilon\theta \iota \alpha \beta^2 \gamma^3 l)} \widetilde{g}_3(q_1) \chi_{q_2}(\eta_2 \Delta_2 u^2 \Upsilon\theta \iota \alpha \beta^2 \gamma^3 l) \overline{\widetilde{g}_3(q_2)}}{N(q_1)^{1+it_1} N(q_2)^{1-it_2}} \nonumber \\
    &   \times \Bigg( \sum_{\substack{n_1 \in \Z[\omega] \\ n_1 \equiv 1  \pmod{3}}} \frac{\overline{\chi_{n_1}(\eta_1 q_1 d^3 \Delta_1 u^2 \Upsilon\theta \iota \alpha \beta^2 \gamma^3 l)} \widetilde{g}_3(n_1)}{N(n_1)^{1+it_1}} H\Big(\frac{N(w_1q_1n_1)}{M_1}\Big) \Bigg) \nonumber \\ 
    &   \times \overline{\Bigg( \sum_{\substack{n_2 \in \Z[\omega] \\ n_2 \equiv 1  \pmod{3}}} \frac{\overline{\chi_{n_2}(\eta_2 q_2 d^3 \Delta_2 u^2 \Upsilon\theta \iota \alpha \beta^2 \gamma^3 l)} \widetilde{g}_3(n_2)}{N(n_2)^{1+it_2}} H\Big(\frac{N(w_2 q_2 n_2)}{M_2}\Big) \Bigg)} \Bigg|. \label{ready_for_poles_eq}
\end{align}
Observe that writing $\eta\eta_i = \nu_i \lambda^{3j_i}$ with $\nu_i \mid 3$ and $j_i \in \{0, 1\}$, we should have $\overline{\chi_{n_i}(\nu_i)}$ instead of $\overline{\chi_{n_i}(\eta_i)}$ above. Since we are taking a sum over $\eta_1, \eta_2 \mid 3$, we can freely replace them in this way (and remove the sum over $\eta\mid 3$) at the cost of a constant factor. The sums over $j\in \Z_{\geq 0}$ and $\delta \equiv 1 \pmod{3}$ with $\delta \mid (\alpha\beta\gamma)^\infty$ were removed since the quantity inside absolute values does not depend on $j$ or $\delta$ and there are $\ll X^\varepsilon$ options for them. There we relaxed the condition $N(u\Upsilon\theta\iota\alpha\beta^2\gamma^3) \ll \frac{M_1M_2 L^2 N(d_1d_2e)}{X^{1-\varepsilon}}$ by positivity. Finally, the conditions $(n_i, d)=1$ were removed by addition of $\overline{\chi_{n_i}(d^3)}$.


\subsection{The error term \texorpdfstring{$\mathcal{R}(b_1, b_2)$}{}: evaluating sums of Gauss sums}

The quantity in \eqref{ready_for_poles_eq}, when inserted into \eqref{R_after_abs_values}, is symmetric in $M_1$ and $M_2$ (up to a relabeling of variables, and in particular an exchange of $B_1$ and $B_2$). Our strategy will be to evaluate the longer sum, say over $n_1$ (an analogous argument covers the complementary case) using the results of \cref{Gauss_sums_section}. Depending on the length of the $n_2$ sum, we may prefer to do the same for it or to directly use the cubic large sieve to take advantage of the extra averaging over $l$ and the variables outside the absolute values (especially $\alpha$).

For $i\in\{1, 2\}$, denote
\begin{equation}\label{alpha_i_def}
    d^3 \Delta_i u^2 \Upsilon\theta \iota  = \alpha_i \beta_i^2\gamma_i^3\delta_i^3,
\end{equation}
where $\alpha_i, \beta_i, \gamma_i, \delta_i \equiv 1 \pmod{3}$ satisfy $\mu^2(\alpha_i\beta_i\gamma_i)=1$ and $\delta_i \mid (\alpha_i\beta_i\gamma_i)^\infty$. Then
\begin{equation*}
    \overline{\chi_{n_i}(\eta_i q_i d^3 \Delta_i u^2 \Upsilon\theta \iota \alpha \beta^2 \gamma^3 l)} = \overline{\chi_{n_i}(\eta_i q_i \alpha_i \alpha l(\beta_i\beta)^2(\gamma_i\gamma)^3)}
\end{equation*}
and $\mu^2(q_i \alpha_i \alpha l \beta_i\beta\gamma_i\gamma)=1$. Indeed, this follows from 

\begin{itemize}
    \item  $\mu^2(\alpha_i\beta_i\gamma_i)=1$ and  $\alpha_i \beta_i \gamma_i \mid (u d w_1 w_2)^{\infty}$;
    \item $\mu^2(\alpha \beta \gamma)=1$ and $(\alpha \beta \gamma, u d w_1 w_2) = (\alpha \beta \gamma,u \Delta_1 \Delta_2) = 1$;
    \item $\mu^2(l)=1$ and $(l,u)=1$, since $\ell=u l$ and $\mu^2(\ell)=1$;
    \item $(l,d w_1 w_2) = 1$, due to the presence of the character $\overline{\psi_{\Delta_1, \Delta_2}(l)}$;
    \item $(l,\alpha \beta \gamma) = 1$, since $(l,k^{\prime})=1$ and $\alpha\beta\gamma \mid k'$;
    \item $\mu^2(q_i) = 1$ and $(q_i, d) = 1$, due to the discussion following \eqref{q_i_def};
    \item $(q_i, u w_1 w_2 \alpha \beta \gamma l) = 1$, due to the presence of $\chi_{q_i}(\eta_i \Delta_i u^2 \Upsilon\theta \iota \alpha \beta^2 \gamma^3 l)$ in \eqref{ready_for_poles_eq} and $\theta^\infty = (w_1w_2)^\infty$.
\end{itemize}

We may therefore apply \cref{truncated_Gauss_sums_lemma} (keeping in mind that the cubic Gauss sum here is normalized, whereas there it is not) to obtain
\begin{align}\label{decomposition_Gauss_sums_eq}
    \sum_{\substack{n_i \in \Z[\omega] \\ n_i \equiv 1  \pmod{3}}} \frac{\overline{\chi_{n_i}(\eta_i q_i \alpha_i \alpha l(\beta_i\beta)^2(\gamma_i\gamma)^3)} \widetilde{g}_3(n_i)}{N(n_i)^{1+it_i}} H\Big(\frac{N(w_i q_i n_i)}{M_i}\Big) = \mathcal{P}_i + \mathcal{I}_i
\end{align}
for
\begin{align}\label{P_i_def}
    \mathcal{P}_i := \bbone_{\beta\beta_i=1}\cdot C_{\eta_i} \cdot \widetilde{H}(-1/6 - it_i) \Big(\frac{M_i}{N(w_iq_i)}\Big)^{-1/6-it_i} \frac{\overline{\widetilde{g}_3(\eta_i, q_i \alpha_i\alpha l)}}{N(q_i \alpha_i\alpha l)^{1/6}} \frac{\Delta_{q_i \alpha_i\alpha l\gamma_i \gamma}(1)}{\Delta_{q_i \alpha_i\alpha l \gamma_i \gamma}(4/3)}
\end{align}
and
\begin{align*}
    \mathcal{I}_i &\ll_{A, \varepsilon} \Big(\frac{M_i}{N(w_iq_i)}\Big)^{-1/2+\varepsilon} \mathop{\sum\sum\sum\sum}_{\substack{\xi_i, \xi'_i, \rho_i, \rho'_i \in \Z[\omega] \\ \xi_i, \xi'_i, \rho_i, \rho'_i \equiv 1\pmod{3} \\ \xi_i\mid \alpha_i \beta_i^2,\ \xi'_i \mid \gamma_i \\ \rho_i \mid \alpha l \beta^2\gamma,\ \rho'_i \mid q_i}} \int_{-X^\varepsilon}^{X^\varepsilon}\frac{|\psi(\eta_i \xi_i \xi'_i \rho_i \rho'_i, 1+\varepsilon + it_i+iy) |}{{N(\xi'_i)^{1/2}}} dy + X^{-A},
\end{align*}
where we used the convexity bound from \cref{psi_convexity_lemma} to bound the tails of the integral, and for each $\widetilde{d} \mid q_i \alpha_i \alpha l$ and $\widetilde{e} \mid \gamma_i \gamma $ coming from \cref{truncated_Gauss_sums_lemma} we uniquely factorised $\frac{q_i \alpha_i \alpha l (\beta_i \beta)^2 \widetilde{e}}{\widetilde{d}} = \xi_i \xi'_i \rho_i \rho'_i$ (using the coprimality induced by $\mu^2(q_i \alpha_i \alpha l \beta_i\beta\gamma_i\gamma)=1$) with the variables satisfying the divisibility conditions indexing the sum above. 

From \eqref{ready_for_poles_eq}, $w_i$ only contributes if $\widetilde{g}_3(\theta, w_i) \neq 0$ for some $\theta$ with $\theta^\infty = (w_1w_2)^\infty$. Since $w_i \mid \theta^\infty$, by \cref{localcomp} we see that necessarily $w_i$ is squarefull, so we can write (uniquely) $w_i = h_i^2 k_i^3 o_i$ with $\mu^2(h_ik_i)=1$ and $o_i\mid k_i^\infty$. Furthermore, since $\widetilde{g}_3(\theta, w_i) \neq 0$, \cref{localcomp} shows that $\nu_\pi(\theta) \geq \nu_\pi(w_i) - 1$ for each prime $\pi$. Therefore 
\begin{equation}\label{theta_divisibility}
    h_ik_i^2o_i \mid \theta.
\end{equation}

We will show that $\xi_i \mid \alpha_i \beta_i^2$ implies $\xi_i \mid \frac{d_i e^2 u^2 \theta \iota}{h_i (e, \iota)^3}$. Indeed, $\alpha_i \beta_i^2$ is the cube-free part (i.e.\ the quotient by the cube divisor congruent to $1 \pmod{3}$ of largest norm) of $\alpha_i \beta_i^2\gamma_i^3\delta_i^3 = d^3 \Delta_i u^2\Upsilon\theta \iota  = d^3 d_i e^2 w_i u^2 \Upsilon\theta \iota$. Note that $d^3 d_i e^2 \iota \mid d^\infty$, $u^2 \Upsilon \mid u^\infty$, and $w_i\theta \mid \theta^\infty$, so they are pairwise coprime. The cube-free part of $d^3 d_i e^2 \iota$ divides $\frac{d_i e^2 \iota}{(e, \iota)^3}$, and that of $u^2 \Upsilon$ divides $u^2$. We now need to consider the cube-free part of $w_i \theta$. By \eqref{theta_divisibility}, for each prime $\pi \mid h_i$ we have $\nu_{\pi}(w_i \theta) = \nu_{\pi}(h_i^2 \theta) \equiv \nu_\pi(\frac{\theta}{h_i}) \pmod{3}$, while for $\pi \mid k_i$ we have $\nu_\pi(\frac{\theta}{h_i}) = \nu_\pi(\theta) \geq 2$, and for $\pi \nmid h_i k_i$ we have $\nu_\pi(w_i \theta) = \nu_\pi(\frac{\theta}{h_i})$. Thus the cube-free part of $w_i\theta$ divides $\frac{\theta}{h_i}$. We conclude that $\xi_i\mid \frac{d_i e^2 u^2 \theta \iota}{h_i (e, \iota)^3}$, as desired. A simpler argument, using that $\xi'_i$ is squarefree and $\xi'_i \mid d^3 d_i e^2 w_i u^2 \Upsilon\theta \iota$, gives the coarse restriction $\xi'_i \mid d u \theta$. Our bound is then
\begin{align}\label{I_i_integral_bound}
    \mathcal{I}_i &\ll_{A, \varepsilon} \Big(\frac{M_i}{N(w_iq_i)}\Big)^{-1/2+\varepsilon} \mathop{\sum\sum\sum\sum}_{\substack{\xi_i, \xi'_i, \rho_i, \rho'_i \in \Z[\omega] \\ \xi_i, \xi'_i, \rho_i, \rho'_i \equiv 1\pmod{3} \\ \xi_i\mid \frac{d_i e^2 u^2 \theta \iota}{h_i (e, \iota)^3},\ \xi'_i \mid d u \theta \\ \rho_i \mid \alpha l \beta^2\gamma,\ \rho'_i \mid q_i}} \int_{-X^\varepsilon}^{X^\varepsilon}\frac{|\psi(\eta_i \xi_i \xi'_i \rho_i \rho'_i, 1+\varepsilon + it_i+iy) |}{N(\xi'_i)^{1/2}} dy + X^{-A}.
\end{align}

We use \eqref{decomposition_Gauss_sums_eq} for $i=1$, and deal with the contributions of $\mathcal{P}_1$ and $\mathcal{I}_1$ to \eqref{ready_for_poles_eq} separately. For the contribution of $\mathcal{I}_1$, we also decompose the sum over $n_2$ using \eqref{decomposition_Gauss_sums_eq}. Then apply the triangle inequality to the sums over $l$, $b_1$, $b_2$, obtaining sums (over all the variables) of $|\mathcal{I}_1||\mathcal{P}_2| + |\mathcal{I}_1||\mathcal{I}_2|$. Plugging this back into \eqref{ready_for_poles_eq}, the contribution of $\mathcal{I}_1$ to that display is
\begin{align}
    & \ll \frac{X^\varepsilon}{L^2} \mathop{\sum\sum}_{\substack{w_1, w_2 \in \Z[\omega] \\ w_1, w_2 \equiv 1 \pmod{3} \\ (w_1 w_2, d) = 1 \\ N(w_i) \ll M_i}} \mathop{\sum\sum\sum\sum}_{\substack{u, \Upsilon, \theta, \iota \in \Z[\omega] \\ u, \Upsilon, \theta, \iota \equiv 1 \pmod 3 \\ \theta^\infty = (w_1 w_2)^\infty,\ \iota \mid d^\infty,\ \Upsilon \mid u^\infty \\ \mu^2(u) = (u, \Delta_1 \Delta_2)=1}} \frac{|\widetilde{g}_3(\theta, w_1)\widetilde{g}_3(\theta, w_2)|}{N(w_1 w_2)}\frac{|\widetilde{g}_3(\iota, d_1) \widetilde{g}_3(\iota, d_2)|}{ N(d_1 d_2)}\frac{N((\iota, e))}{N(e)} \nonumber \\
    & \times \mathop{\sum\sum\sum\sum}_{\substack{\alpha, \beta, \gamma, l \in \Z[\omega] \\ \alpha, \beta, \gamma, l \equiv 1\pmod 3 \\ \mu^2(\alpha \beta \gamma l) = (\alpha \beta \gamma l, u \Delta_1 \Delta_2) = 1 \\ N(u \Upsilon\theta \iota\alpha\beta^2\gamma^3) \ll \frac{M_1M_2 L^2 N(d_1d_2e)}{X^{1-\varepsilon}} \\ N(ul) \sim L }} \sum_{\substack{f \in \Z[\omega] \\ f \equiv 1 \pmod {3} \\ N(f) \ll M_1+M_2 \\ (f, d) = 1 \\ \mu^2(f) = 1}} \mathop{\sum\sum}_{\substack{b_1, b_2 \in \Z[\omega] \\ b_1, b_2 \equiv 1 \pmod 3  \\ N(b_1) \sim B_1\\ N(b_2) \sim B_2 \\ \mu^2(b_1) = \mu^2(b_2) = 1}} \frac{\mathbf{1}_{d u \theta \alpha \beta \gamma l}(q_1q_2) \cdot |\mathcal{I}_1|\cdot (|\mathcal{P}_2| + |\mathcal{I}_2| )}{N(q_1q_2)},\label{l_outside_bound} 
\end{align}
where the indexing conditions on $l$ come from $(l, u\alpha\beta\gamma)=1$ and the presence of $\mu(l) \overline{\psi_{\Delta_1, \Delta_2}(l)}$ before our application of the triangle inequality in \eqref{ready_for_poles_eq}. Indexing conditions will often be dropped (by positivity) without comment in what follows.

We estimate \eqref{l_outside_bound} -- that is, the contribution from $\mathcal{I}_1$ -- in Sections \ref{cross_term_sec} and \ref{pure_int_sec}. We then explain and estimate (in two different ways) the contribution from $\mathcal{P}_1$ in Sections \ref{P1_cont_secA} and \ref{P1_cont_secB}. Those estimates are ultimately combined and optimized in \cref{opt_sec}.


\subsubsection{The cross term \texorpdfstring{$|\mathcal{I}_1||\mathcal{P}_2|$}{}}
\label{cross_term_sec}

Our starting point is the expression in \eqref{l_outside_bound}. Observe from \eqref{P_i_def}, discarding $N(\alpha_i)^{-1/6}$, that
\begin{align*}
    \mathcal{P}_i \ll \bbone_{\beta\beta_i=1}\cdot X^\varepsilon \cdot \Big(\frac{M_i}{N(w_i q_i)}\Big)^{-1/6} \frac{1}{N(q_i\alpha l)^{1/6}},
\end{align*}
Furthermore \eqref{alpha_i_def} and \eqref{Delta_i_def} combined with the condition $\beta_i = 1$ imply
\begin{equation}\label{conditional_alpha_i_relation}
    (d^3 d_i e^2 \iota) \cdot (u^2 \Upsilon) \cdot (w_i \theta) = \alpha_i (\gamma_i\delta_i)^3,
\end{equation}
where the three terms on the left are pairwise coprime. From $\mu^2(u) = \mu^2(e) = 1$, and $(e, d_1d_2)=1$, \eqref{conditional_alpha_i_relation} implies $u \mid \Upsilon$ and $e \mid \iota$. Thus for the terms containing some $\mathcal{P}_i$ we can write $\Upsilon = u\Upsilon'$ for $\Upsilon' \mid u^\infty$ and $\iota = e \iota'$ for $\iota' \mid d^\infty$. 

Using the remarks above for $i=2$, the contribution of the cross term $|\mathcal{I}_1||\mathcal{P}_2|$ to \eqref{l_outside_bound} is then
\begin{align}
    & \ll \frac{X^\varepsilon}{L^2 M_1^{1/2} M_2^{1/6}} \mathop{\sum\sum}_{\substack{w_1, w_2 \in \Z[\omega] \\ w_1, w_2 \equiv 1 \pmod{3} \\ (w_1 w_2, d) = 1 \\ N(w_i) \ll M_i}} \mathop{\sum\sum\sum\sum}_{\substack{u, \Upsilon', \theta, \iota' \in \Z[\omega] \\ u, \Upsilon', \theta, \iota' \equiv 1 \pmod 3 \\ \theta^\infty = (w_1 w_2)^\infty \\ \iota' \mid d^\infty,\ \Upsilon' \mid u^\infty }} \frac{|\widetilde{g}_3(\theta, w_1)\widetilde{g}_3(\theta, w_2)|}{N(w_1)^{1/2}N(w_2)^{5/6}}\frac{|\widetilde{g}_3(\iota', d_1) \widetilde{g}_3(\iota', d_2)|}{N(d_1d_2)} \nonumber \\
    & \times  \mathop{\sum\sum\sum\sum}_{\substack{\xi_1, \xi'_1, \rho_1, \rho'_1 \in \Z[\omega] \\ \xi_1, \xi'_1, \rho_1, \rho'_1 \equiv 1\pmod{3} \\ \xi_1\mid \frac{d_1 u^2 \theta \iota'}{h_1},\ \xi'_1 \mid d u \theta \\ (\rho'_1, d\theta) = 1}} \int_{-X^\varepsilon}^{X^\varepsilon}\frac{|\psi(\eta_1 \xi_1\xi'_1 \rho_1 \rho'_1, 1+\varepsilon + it_1+iy) |}{N(\xi'_1)^{1/2}} dy  \mathop{\sum\sum\sum}_{\substack{\alpha, \gamma, l \in \Z[\omega] \\ \alpha, \gamma, l \equiv 1\pmod 3 \\ N(u^2\Upsilon'\theta \iota' \alpha\gamma^3) \ll \frac{M_1M_2 L^2 N(d_1d_2)}{X^{1-\varepsilon}} \\ N(ul) \sim L,\ \rho_1\mid \alpha l \gamma}} \nonumber \\
    & \times \frac{1}{N(\alpha l)^{1/6}} \sum_{\substack{b_1 \in \Z[\omega] \\ b_1 \equiv 1 \pmod 3  \\ N(b_1) \sim B_1}} \mu^2(b_1) \sum_{\substack{f \in \Z[\omega] \\ f \equiv 1 \pmod {3} \\ N(f) \ll M_1+M_2 \\ (f, d) = 1,\ \rho'_1 \mid q_1}}  \frac{\mu^2(f)}{N(q_1)^{1/2}} \sum_{\substack{b_2 \in \Z[\omega] \\ b_2 \equiv 1 \pmod 3 \\ N(b_2) \sim B_2}} \frac{\mu^2(b_2)}{N(q_2)},\label{cross_term_initial_eq} 
\end{align}
where the condition $(\rho'_1, d\theta) = 1$ comes from $\rho'_1 \mid q_1$ and the presence of $\mathbf{1}_{d\theta}(q_1)$.

Now split the sums over dyadic ranges $N(u) \sim U, N(\Upsilon') \sim Y', N(\theta) \sim \Theta, N(\iota') \sim I', N(\alpha) \sim A, N(\gamma) \sim C$, where all of the ranges are $\gg 1$, and satisfy
\begin{equation}\label{cross_dyadic_ranges}
    U^2Y'\Theta I' A C^3 \ll \frac{M_1M_2L^2 N(d_1 d_2)}{X^{1-\varepsilon}}.
\end{equation}

Recalling that $f_i := [\frac{b_i}{(b_i, d)}, f ]$ and $(f, d)=1$, observe that
\begin{equation*}
    f_i = \frac{b_i f}{(b_i, d) (f, \frac{b_i}{(b_i, d)})} = \frac{b_i}{(b_i, d)} \cdot \frac{f}{(f, b_i)} = \frac{b_i f}{(b_i, df)}.
\end{equation*}
Since $b_i$ and $f$ are squarefree and $q_i := \frac{f_i}{(f_i, w_i)}$, 
\begin{equation}\label{q_i_split}
    q_i = \frac{b_i}{(b_i, d w_i)} \cdot \frac{f}{(f, b_i w_i)} = \frac{b_i}{(b_i, d f w_i)} \cdot \frac{f}{(f, w_i)}.
\end{equation}

We can evaluate the sums over $b_2$, $f$, and $b_1$ in \eqref{cross_term_initial_eq} using \cref{gcd_sum_lemma} and \eqref{q_i_split}, which give the bound $\ll X^\varepsilon N(\frac{(f, w_2)}{f} )$ for the sum over $b_2$. To evaluate the sum over $f$ we note that $\rho'_1 \mid q_1 \mid b_1f \implies \widetilde{\rho}_1 := \frac{\rho'_1}{(b_1, \rho'_1)} \mid f$, so write $f = \widetilde{\rho}_1 \widetilde{f}$. Observe that $(\widetilde{\rho}_1, b_1 w_1 w_2) = 1$, since we have the conditions $\theta^\infty = (w_1w_2)^\infty$ and $(\rho'_1, \theta) = 1$, and $q_1$ is squarefree (hence so is $\rho'_1$). Therefore, the sum over $f$ (including the contribution from $b_2$) is
\begin{align}
    &\ll X^\varepsilon \sum_{\substack{\widetilde{f} \in \Z[\omega] \\ \widetilde{f} \equiv 1 \pmod 3  \\ N(\widetilde{f}) \ll \frac{M_1+M_2}{N(\widetilde{\rho}_1)}}} N\Big(\frac{(b_1, dw_1)}{b_1}\frac{(\widetilde{\rho}_1 \widetilde{f}, b_1w_1)}{\widetilde{\rho}_1 \widetilde{f}}\Big)^{1/2} N\Big(\frac{(\widetilde{\rho}_1 \widetilde{f}, w_2)}{\widetilde{\rho}_1 \widetilde{f}}\Big) \nonumber \\
    &= \frac{X^\varepsilon}{N(\widetilde{\rho}_1)^{3/2}} N\Big(\frac{(b_1, dw_1)}{b_1}\Big)^{1/2} \sum_{\substack{\widetilde{f} \in \Z[\omega] \\ \widetilde{f} \equiv 1 \pmod 3  \\ N(\widetilde{f}) \ll \frac{M_1+M_2}{N(\widetilde{\rho}_1)}}}  N\Big(\frac{(\widetilde{f}, b_1w_1)}{\widetilde{f}}\Big)^{1/2} N\Big(\frac{(\widetilde{f}, w_2)}{\widetilde{f}}\Big) \nonumber \\
    &\ll \frac{X^\varepsilon}{N(\widetilde{\rho}_1)} N\Big(\frac{(b_1, dw_1)}{b_1}\Big)^{1/2} \cdot \bbone_{N(\widetilde{\rho}_1) \ll M_1+M_2}, \label{f_sum_gcd_ineq}
\end{align}
where we once again used \cref{gcd_sum_lemma} and \eqref{q_i_split}, and discarded $N(\widetilde{\rho}_1)^{-1/2}$ in the last step.

Finally, since the condition $(\rho'_1, d \theta)=1$ is present we have $(b_1, \rho'_1)(b_1, dw_1) = (b_1, \rho'_1 d w_1)$, so the sum over $b_1$ (including the contributions of $b_2$ and $f$) in \eqref{cross_term_initial_eq} is
\begin{align*}
    \ll \frac{X^\varepsilon \cdot \bbone_{N(\rho'_1) \ll B_1(M_1+M_2)}}{N(\rho'_1)} \sum_{\substack{b_1 \in \Z[\omega] \\ b_1 \equiv 1 \pmod 3  \\ N(b_1) \sim B_1}} \frac{N((b_1, \rho'_1 d w_1))}{N(b_1)^{1/2}} \ll \frac{X^\varepsilon B_1^{1/2}}{N(\rho'_1)} \cdot \bbone_{N(\rho'_1) \ll B_1(M_1+M_2)}.
\end{align*}

The sums over $\alpha, \gamma, l$ in \eqref{cross_term_initial_eq}, restricted to our dyadic ranges, contribute 
\begin{equation*}
    \mathop{\sum\sum\sum}_{\substack{\alpha, \gamma, l \in \Z[\omega] \\ \alpha, \gamma, l \equiv 1\pmod 3 \\ N(\alpha) \sim A,\ N(\gamma) \sim C,\ N(ul) \sim L \\ \rho_1\mid \alpha l \gamma}} \frac{1}{N(\alpha l)^{1/6}}\ll \frac{X^\varepsilon (AL)^{5/6}C}{N(\rho_1) U^{5/6}} \cdot \bbone_{N(\rho_1) \ll \frac{ALC}{U}}.
\end{equation*}

We can now evaluate the sums over $\xi_1, \xi'_1, \rho_1, \rho'_1$. Apply Cauchy--Schwarz to all four sums, and clump $\rho_1\rho'_1$ into a single variable $\rho$. Since a divisor bound shows there are $\ll X^\varepsilon$ options for $\xi_1$ and $\xi'_1$, and furthermore $N(\xi_1) \ll \frac{N(d_1) U^2\Theta I'}{N(h_1)}$, we conclude that the second and third lines of \eqref{cross_term_initial_eq} are
\begin{align*}
    &\ll X^\varepsilon \Big(\frac{AL}{U}\Big)^{5/6} C B_1^{1/2}  \sup_{\substack{w \in \R \\ |w| \ll X^\varepsilon}} \Bigg(\mathop{\sum\sum\sum}_{\substack{\xi_1, \xi'_1, \rho \in \Z[\omega] \\ \xi_1, \xi'_1, \rho \equiv 1\pmod{3} \\ \xi_1\mid \frac{d_1 u^2 \theta \iota'}{h_1},\ \xi'_1 \mid d u \theta,\ N(\rho) \ll X^{1000}}} \frac{|\psi(\eta_1 \xi_1\xi'_1 \rho, 1+\varepsilon + iw) |^2}{N(\xi'_1 \rho)}\Bigg)^{1/2} \\
    & \qquad \times \Bigg(\mathop{\sum\sum\sum}_{\substack{\xi_1, \xi'_1, \rho \in \Z[\omega] \\ \xi_1, \xi'_1, \rho \equiv 1\pmod{3} \\ \xi_1\mid \frac{d_1 u^2 \theta \iota'}{h_1},\ \xi'_1 \mid d u \theta,\ N(\rho) \ll X^{1000}}} \frac{1}{N(\rho)}\Bigg)^{1/2} \ll X^\varepsilon \Big(\frac{AL}{U}\Big)^{5/6} C B_1^{1/2} \Big(\frac{N(d_1) U^2\Theta I'}{N(h_1)}\Big)^{1/4},
\end{align*}
where we applied partial summation and \cref{second_moment_lindelof_lemma} (with $h = \eta_1 \xi_1 \xi^{\prime}_1$) to the sum over $\rho$ in the first bracketed sum above.

Inserting this back into \eqref{cross_term_initial_eq}, the sums over $u$ and $\Upsilon'$ can be evaluated since the summand is independent of them, and are easily seen to contribute a factor $\ll X^\varepsilon U$. Removing the normalization from the Gauss sums for convenience, we conclude that \eqref{cross_term_initial_eq} is
\begin{align}
    &\ll \frac{X^\varepsilon U^{2/3} \Theta^{1/4} I'^{1/4} A^{5/6}C B_1^{1/2}}{L^{7/6} M_1^{1/2} M_2^{1/6}} \Bigg( \sum_{\substack{\iota' \in \Z[\omega] \\ \iota' \equiv 1 \pmod 3 \\ \iota' \mid d^\infty,\ N(\iota') \sim I'}} \frac{|g_3(\iota', d_1) g_3(\iota', d_2)|}{N(d_1)^{5/4}N(d_2)^{3/2}} \Bigg) \nonumber \\
    &\times \Bigg(\mathop{\sum\sum\sum\sum\sum\sum}_{\substack{h_1, h_2, k_1, k_2, o_1, o_2 \in \Z[\omega] \\ h_1, h_2, k_1, k_2, o_1, o_2 \equiv 1 \pmod{3} \\ \mu^2(h_1k_1) = \mu^2(h_2k_2)=1 \\ o_1 \mid k_1^\infty,\ o_2 \mid k_2^\infty,\ (h_ik_i, d) = 1 \\ N(h_i^2 k_i^3 o_i) \ll M_i}} \sum_{\substack{\theta \in \Z[\omega] \\ \theta \equiv 1 \pmod 3 \\ \theta^\infty = (h_1 h_2 k_1 k_2)^\infty \\ N(\theta) \sim \Theta}} \frac{|g_3(\theta, h_1^2 k_1^3 o_1)g_3(\theta, h_2^2 k_2^3 o_2)|}{N(h_1)^{9/4}N(k_1^3 o_1) N(h_2^2 k_2^3 o_2)^{4/3}}\Bigg). \label{cross_term_intermediate_eq} 
\end{align}

Since $\iota' \mid d^\infty = (d_1d_2e)^\infty$, recalling that $d_1, d_2, e$ are pairwise coprime we can write $\iota' = \iota'_1 \iota'_2 \iota''$ where $\iota'_i \mid d_i^\infty$ and $\iota'' \mid e^\infty$. Then use the general bound 
\begin{equation}\label{iota_Gauss_sum_bound}
    |g_3(\iota', d_i)| = |g_3(\iota'_i, d_i)| \leq N(\iota'_i d_i)^{1/2},
\end{equation}
which follows from \cref{cubic-GS-lemma1}, \cref{localcomp}, and \eqref{sqrootcancel}, to conclude that the sum over $N(\iota') \sim I'$ in \eqref{cross_term_intermediate_eq} is
\begin{equation*}
    \ll X^\varepsilon \frac{N(\iota'_1 \iota'_2)^{1/2}}{N(d_1)^{3/4}N(d_2)} \ll X^\varepsilon \frac{I'^{1/2}}{N(d_1)^{3/4}N(d_2)}.
\end{equation*}

Since $h_i$ is squarefree and $(h_i, k_io_i) = 1$, observe that
\begin{equation}\label{h_k_sums_Gauss_sum bound}
    |g_3(\theta, h_i^2 k_i^3 o_i)| \leq N(h_i)^{3/2} N(k_i^3 o_i).
\end{equation}
We also have $h_i k_i^2 o_i \mid \theta$ by \eqref{theta_divisibility}, hence $N(h_i k_i^2 o_i) \ll \Theta$. From $\theta^\infty = (w_1 w_2)^\infty = (h_1 h_2 k_1 k_2)^\infty$ and $N(\theta) \sim \Theta$, we coarsely infer that there are $\ll X^\varepsilon$ options for $\theta$. Therefore \eqref{cross_term_intermediate_eq} is 
\begin{align}
    \ll \frac{X^\varepsilon U^{2/3} \Theta^{1/4}I'^{3/4} A^{5/6}C B_1^{1/2}}{L^{7/6} M_1^{1/2} M_2^{1/6} N(d_1)^{3/4} N(d_2)}  & \mathop{\sum\sum\sum}_{\substack{h_1, k_1, o_1 \in \Z[\omega] \\ h_1, k_1, o_1 \equiv 1 \pmod{3} \\ o_1 \mid k_1^\infty,\ N(o_1) \ll \Theta \\ N(h_1 k_1^2) \ll \Theta}} \frac{1}{N(h_1)^{3/4}} \mathop{\sum\sum\sum}_{\substack{h_2, k_2, o_2 \in \Z[\omega] \\ h_2, k_2, o_2 \equiv 1 \pmod{3} \\ o_2 \mid k_2^\infty,\ N(o_2) \ll \Theta \\ N(h_2 k_2^2) \ll \Theta}} \frac{1}{N(h_2)^{7/6} N(k_2^3 o_2)^{1/3}}. \nonumber
\end{align}
The sums over $o_1$ and $o_2$ contribute $\ll X^{\varepsilon}$ by Rankin's trick, and can be removed. The sums over $h_2$ and $k_2$ contribute $\ll X^\varepsilon$, and can also be removed. If we localize in dyadic ranges $N(h_1) \sim H_1$ and $N(k_1) \sim K_1$, so that both are $\gg 1$ and satisfy $H_1 K_1^2 \ll \Theta$, the corresponding contribution of the sums over $h_1$ and $k_1$ is
\begin{equation*}
    \ll X^\varepsilon H_1^{1/4} K_1 \ll X^\varepsilon (H_1K_1^2)^{1/2} \ll X^\varepsilon \Theta^{1/2}.
\end{equation*}

Thus \eqref{cross_term_initial_eq} is
\begin{align}
    &\ll \frac{X^\varepsilon U^{2/3} \Theta^{3/4}I'^{3/4} A^{5/6}C B_1^{1/2}}{L^{7/6} M_1^{1/2} M_2^{1/6} N(d_1)^{3/4} N(d_2)} \ll \frac{X^\varepsilon \Big(U^2 Y' \Theta I' A C^3 \Big)^{5/6} B_1^{1/2}}{L^{7/6} M_1^{1/2} M_2^{1/6} N(d_1)^{3/4} N(d_2)} \nonumber \\
    & \ll \frac{X^\varepsilon B_1^{1/2}}{L^{7/6} M_1^{1/2} M_2^{1/6} N(d_1)^{3/4}N(d_2)} \Big(\frac{M_1M_2L^2 N(d_1 d_2)}{X^{1-\varepsilon}}\Big)^{5/6} \nonumber \\
    &\ll \frac{X^\varepsilon L^{1/2} M_2^{1/3} B_1^{5/6} B_2^{1/3}}{X^{1/2}}\Big(\frac{M_1M_2}{X B_1 B_2}\Big)^{1/3} \frac{N(d_1)^{1/12}}{N(d_2)^{1/6}}, \label{cross_term_advanced_eq}
\end{align}
where we used the restriction \eqref{cross_dyadic_ranges} on the ranges. Inserting \eqref{cross_term_advanced_eq} into \eqref{R_after_abs_values}, where we recall that the ranges $L, M_1, M_2$ satisfy
\begin{equation}\label{M_1_M_2_bound}
    1\leq L \leq Y, \qquad \qquad M_1, M_2 \gg 1, \qquad \qquad M_1 M_2 \ll X^{1+\varepsilon} B_1 B_2,
\end{equation} 
we see that the contribution of the cross term $|\mathcal{I}_1||\mathcal{P}_2|$ to $\mathfrak{S}^L_{M_1, M_2}(B_1, B_2)$ is 
\begin{align}
    &\ll_F X^{1/2+\varepsilon} L^{1/2} M_2^{1/3} B_1^{4/3} B_2^{5/6} \Big(\frac{M_1M_2}{X B_1 B_2}\Big)^{1/3} \sum_{\substack{d\in \Z[\omega] \\ d\equiv 1 \pmod{3}}} \frac{1}{N(d)} \mathop{\sum\sum}_{\substack{d_1, d_2 \in \Z[\omega] \\ d_1, d_2 \equiv 1 \pmod{3} \\ (d_1, d_2)=1,\ d_1d_2 \mid d^\infty \\ N(d_1 d_2 d^2) \ll \frac{X^{1+\varepsilon} B_1 B_2}{M_1 M_2}}} \frac{N(d_1)^{1/12}}{N(d_2)^{1/6}} \nonumber \\
    &\ll X^{1/2+\varepsilon} L^{1/2} M_2^{1/3} B_1^{4/3} B_2^{5/6} \sum_{\substack{d\in \Z[\omega] \\ d\equiv 1 \pmod{3} \\ N(d^2) \ll \frac{X^{1+\varepsilon} B_1 B_2}{M_1M_2}}} \frac{1}{N(d)^{5/3}} \ll X^{1/2+\varepsilon} L^{1/2} M_2^{1/3} B_1^{4/3} B_2^{5/6}. \label{first_cross_term_result}
\end{align}


\subsubsection{The pure integral term \texorpdfstring{$|\mathcal{I}_1||\mathcal{I}_2|$}{}}
\label{pure_int_sec}

Now consider the contribution of the pure integral term $|\mathcal{I}_1||\mathcal{I}_2|$ to $\mathfrak{S}^L_{M_1, M_2}(B_1, B_2)$. We once again start from the corresponding term in \eqref{l_outside_bound}. Then split the sums present there over dyadic ranges $N(u) \sim U, N(\Upsilon) \sim Y, N(\theta) \sim \Theta, N(\iota) \sim I, N(\alpha) \sim A, N(\beta) \sim B, N(\gamma) \sim C$, where all of the ranges are $\gg 1$, and satisfy
\begin{equation}\label{integral_term_dyadic_ranges}
    U Y \Theta I A B^2 C^3 \ll \frac{M_1 M_2 L^2 N(d_1 d_2 e)}{X^{1-\varepsilon}}.
\end{equation}

Apply Cauchy--Schwarz to all of the sums on the second line of \eqref{l_outside_bound}, so the contribution of the term $|\mathcal{I}_1||\mathcal{I}_2|$ to that display (in the ranges described above) is
\begin{align}
    \ll \frac{X^\varepsilon}{L^2} \mathop{\sum\sum}_{\substack{w_1, w_2 \in \Z[\omega] \\ w_1, w_2 \equiv 1 \pmod{3} \\ (w_1 w_2, d) = 1 \\ N(w_i) \ll M_i}} & \mathop{\sum\sum\sum\sum}_{\substack{u, \Upsilon, \theta, \iota \in \Z[\omega] \\ u, \Upsilon, \theta, \iota \equiv 1 \pmod 3 \\ \theta^\infty = (w_1 w_2)^\infty \\ \iota \mid d^\infty,\ \Upsilon \mid u^\infty \\ N(u) \sim U,\ N(\Upsilon) \sim Y \\ N(\theta)\sim \Theta,\ N(\iota) \sim I}} \frac{|\widetilde{g}_3(\theta, w_1)\widetilde{g}_3(\theta, w_2)|}{N(w_1 w_2)}  \frac{|\widetilde{g}_3(\iota, d_1) \widetilde{g}_3(\iota, d_2)|}{ N(d_1 d_2)}\frac{N((\iota, e))}{N(e)} \mathcal{S}_1^{\frac{1}{2}} \mathcal{S}_2^{\frac{1}{2}}, \label{integral_term_initial_eq}
\end{align}
where we dropped several conditions by positivity and
\begin{align*}
    \mathcal{S}_i := \mathop{\sum\sum\sum\sum}_{\substack{\alpha, \beta, \gamma, l \in \Z[\omega] \\ \alpha, \beta, \gamma, l \equiv 1\pmod 3 \\ N(\alpha) \sim A,\ N(\beta) \sim B \\ N(\gamma)\sim C,\ N(ul) \sim L}} \sum_{\substack{f \in \Z[\omega] \\ f \equiv 1 \pmod {3} \\ N(f) \ll M_1+M_2 \\ (f, d) = 1,\ \mu^2(f) = 1}} \mathop{\sum\sum}_{\substack{b_1, b_2 \in \Z[\omega] \\ b_1, b_2 \equiv 1 \pmod 3  \\ N(b_1) \sim B_1,\ N(b_2) \sim B_2 \\ \mu^2(b_1) = \mu^2(b_2) = 1}} \frac{\mathbf{1}_{d u \theta \alpha \beta \gamma l}(q_1q_2) \cdot |\mathcal{I}_i|^2}{N(q_1q_2)}.
\end{align*}

For simplicity let us consider $\mathcal{S}_1$, with the treatment of $\mathcal{S}_2$ being analogous by symmetry. Applying \eqref{I_i_integral_bound}, Cauchy--Schwarz to the corresponding sum over divisors, and a divisor bound, 
\begin{align}
    \mathcal{S}_1 &\ll \frac{X^\varepsilon N(w_1)}{M_1} \mathop{\sum\sum\sum\sum}_{\substack{\xi_1, \xi'_1, \rho_1, \rho'_1 \in \Z[\omega] \\ \xi_1, \xi'_1, \rho_1, \rho'_1 \equiv 1\pmod{3} \\ \xi_1\mid \frac{d_1 e^2 u^2 \theta \iota}{h_1 (e, \iota)^3},\ \xi'_1 \mid d u \theta \\ (\rho'_1, d\theta) = 1}} \int_{-X^\varepsilon}^{X^\varepsilon}\frac{|\psi(\eta_1 \xi_1\xi'_1 \rho_1 \rho'_1, 1+\varepsilon + it_1+iy) |^2}{N(\xi'_1)} dy  \nonumber \\
    & \times \mathop{\sum\sum\sum\sum}_{\substack{\alpha, \beta, \gamma, l \in \Z[\omega] \\ \alpha, \beta, \gamma, l \equiv 1\pmod 3 \\ N(\alpha) \sim A,\ N(\beta) \sim B \\ N(\gamma)\sim C,\ N(ul) \sim L \\ \rho_1\mid \alpha l \beta^2 \gamma}} \sum_{\substack{b_1 \in \Z[\omega] \\ b_1 \equiv 1 \pmod 3  \\ N(b_1) \sim B_1}} \mu^2(b_1) \sum_{\substack{f \in \Z[\omega] \\ f \equiv 1 \pmod {3} \\ N(f) \ll M_1+M_2 \\ (f, d) = 1,\ \rho'_1 \mid q_1}} \mu^2(f) \sum_{\substack{b_2 \in \Z[\omega] \\ b_2 \equiv 1 \pmod 3 \\ N(b_2) \sim B_2}} \frac{\mu^2(b_2)}{N(q_2)}.\label{integral_term_second_line} 
\end{align}
As before, the condition $(\rho'_1, d\theta) = 1$ was added due to the presence of the factor $\mathbf{1}_{d\theta}(q_1)$, which was subsequently dropped by positivity.

Arguing as in \eqref{f_sum_gcd_ineq}, we see that the sums over $b_2$, $f$, and $b_1$ are
\begin{equation*}
    \ll \frac{X^\varepsilon B_1}{N(\rho'_1)} \cdot \bbone_{N(\rho'_1) \ll B_1(M_1+M_2)},
\end{equation*}
and by a divisor bound the sums over $\alpha, \beta, \gamma, l$ contribute an extra factor
\begin{equation*}
    \ll \frac{X^\varepsilon ALB^2C}{N(\rho_1) U} \cdot \bbone_{N(\rho_1) \ll \frac{ALB^2C}{U}}.
\end{equation*}

Observe that $N(\xi_1) \ll \frac{N(d_1 e^2) U^2\Theta I}{N(h_1 (e, \iota)^3)}$ and there are $\ll X^\varepsilon$ options for $\xi_1$ and $\xi'_1$. Using \cref{second_moment_lindelof_lemma} and partial summation for the sums over $\rho_1$ and $\rho'_1$ in \eqref{integral_term_second_line}, we obtain
\begin{equation*}
    \mathcal{S}_1 \ll \frac{X^\varepsilon N(w_1)}{M_1} \frac{ALB^2C B_1}{U} \Big(\frac{N(d_1 e^2) U^2\Theta I}{N(h_1 (e, \iota)^3)}\Big)^{1/2}.
\end{equation*}
Therefore
\begin{align*}
    \mathcal{S}_1^{1/2} \mathcal{S}_2^{1/2} \ll \frac{X^\varepsilon \Theta^{1/2} I^{1/2} AB^2C (B_1B_2)^{1/2} L}{(M_2 M_2)^{1/2}} \frac{N(e) N(w_1w_2)^{1/2}}{N((e, \iota))^{3/2}} \Big(\frac{N(d_1 d_2)}{N(h_1 h_2)}\Big)^{1/4}.
\end{align*}

Inserting this back into \eqref{integral_term_initial_eq}, the sums over $u$ and $\Upsilon$ can be evaluated and contribute $\ll X^\varepsilon U$. Removing the normalization of Gauss sums, that display is
\begin{align}
    &\ll \frac{X^\varepsilon U \Theta^{1/2} I^{1/2} A B^2 C (B_1B_2)^{1/2}}{L (M_1 M_2)^{1/2}} \sum_{\substack{\iota \in \Z[\omega] \\ \iota \equiv 1 \pmod 3 \\ \iota \mid d^\infty,\ N(\iota) \sim I}} \frac{|g_3(\iota, d_1) g_3(\iota, d_2)|}{N(d_1d_2)^{5/4}} \nonumber \\
    &\times \mathop{\sum\sum\sum\sum\sum\sum}_{\substack{h_1, h_2, k_1, k_2, o_1, o_2 \in \Z[\omega] \\ h_1, h_2, k_1, k_2, o_1, o_2 \equiv 1 \pmod{3} \\ \mu^2(h_1k_1) = \mu^2(h_2k_2)=1 \\ o_1 \mid k_1^\infty,\ o_2 \mid k_2^\infty,\ (h_ik_i, d) = 1 \\ N(h_i^2 k_i^3 o_i) \ll M_i}} \sum_{\substack{\theta \in \Z[\omega] \\ \theta \equiv 1 \pmod 3 \\ \theta^\infty = (h_1 h_2 k_1 k_2)^\infty \\ N(\theta) \sim \Theta}} \frac{|g_3(\theta, h_1^2 k_1^3 o_1)g_3(\theta, h_2^2 k_2^3 o_2)|}{N(h_1h_2)^{9/4}N(k_1 k_2)^3 N(o_1 o_2)}. \label{integral_term_intermediate}  
\end{align}

By \eqref{theta_divisibility} we have $h_ik_i^2 o_i \mid \theta$, so $N(h_i k_i^2 o_i) \ll \Theta$. Write $\widetilde{k} := (k_1, k_2)$ and $k_i = \widetilde{k} k'_i$, so $\mu^2(\widetilde{k} k'_1 k'_2) =1$ since $k_1$ and $k_2$ are squarefree. Thus $(\widetilde{k} k'_1 k'_2)^2 = [k_1^2, k_2^2]\mid \theta$. Note that for any prime $\pi \mid h_i$ we have $\nu_\pi(h_i^2k_i^3o_i) = 2$. Since $g_3(\theta, h_i^2k_i^3o_i)\neq 0$, \cref{localcomp} implies $\nu_\pi(\theta) = 1$, so $\pi \nmid \widetilde{k} k'_1 k'_2$ (otherwise we would have $\pi^2 \mid (\widetilde{k} k'_1 k'_2)^2 \mid \theta$). Therefore, $h_i (\widetilde{k} k'_1 k'_2)^2 \mid \theta$, and we conclude that $N(h_i(\widetilde{k} k'_1 k'_2)^2) \ll \Theta$. Using \eqref{h_k_sums_Gauss_sum bound}, the sum over the $h_i, k_i, o_i$ in \eqref{integral_term_intermediate} is 
\begin{align}\label{h_k_integral_term_sum}
    \ll \mathop{\sum\sum\sum}_{\substack{\widetilde{k}, k'_1, k'_2 \in \Z[\omega] \\ \widetilde{k}, k'_1, k'_2 \equiv 1 \pmod 3 \\ \mu^2(\widetilde{k} k'_1 k'_2) = 1 \\ N(\widetilde{k} k'_1 k'_2) \ll \Theta^{1/2}}} \mathop{\sum\sum\sum\sum}_{\substack{h_1, h_2, o_1, o_2 \in \Z[\omega] \\ h_1, h_2, o_1, o_2 \equiv 1 \pmod{3} \\ o_i \mid (\widetilde{k}k'_i)^\infty,\ N(o_i) \ll \Theta \\ N(h_i) \ll \frac{\Theta}{N(\widetilde{k} k'_1 k'_2)^2}}} \frac{1}{N(h_1h_2)^{3/4}}.
\end{align}

The sums over $o_1$ and $o_2$ contribute $\ll X^\varepsilon$, and splitting into dyadic intervals $N(h_i) \sim H_i$, $N(k'_i) \sim K'_i$, $N(\widetilde{k}) \sim \widetilde{K}$, with all ranges $\gg 1$ and satisfying
\begin{equation*}
    H_i \ll \frac{\Theta}{(\widetilde{K} K'_1 K'_2)^2} \qquad \text{and} \qquad \widetilde{K} K'_1 K'_2 \ll \Theta^{1/2},
\end{equation*}
we see that \eqref{h_k_integral_term_sum} for those ranges is
\begin{align*}
    \ll X^\varepsilon \widetilde{K} K'_1 K'_2 (H_1 H_2)^{1/4} \ll X^\varepsilon \widetilde{K} K'_1 K'_2 \frac{\Theta^{1/2}}{\widetilde{K} K'_1 K'_2} \ll X^\varepsilon \Theta^{1/2},
\end{align*}
so the same bound holds over all ranges.

Finally, the remaining sum over $\iota$ in \eqref{integral_term_intermediate} can be bounded by the argument used in \eqref{iota_Gauss_sum_bound} (but for $\iota$ in place of $\iota'$) and is readily seen to be
\begin{equation*}
    \ll \frac{X^\varepsilon I^{1/2}}{N(d_1d_2)^{3/4}},
\end{equation*}
so inserting our work so far into \eqref{integral_term_intermediate}, the resulting bound for \eqref{integral_term_initial_eq} in the given dyadic ranges is 
\begin{align*}
    &\ll \frac{X^\varepsilon U \Theta I A B^2 C (B_1B_2)^{1/2}}{L (M_1 M_2)^{1/2} N(d_1d_2)^{3/4}}.
\end{align*}

Using \eqref{integral_term_dyadic_ranges} and plugging this into \eqref{R_after_abs_values}, we conclude that the contribution of $|\mathcal{I}_1||\mathcal{I}_2|$ to $\mathfrak{S}^L_{M_1, M_2}(B_1, B_2)$ is 
\begin{align}
    & \ll_F X^{1/2+\varepsilon} L (B_1 B_2)^{3/2} \Big(\frac{M_1 M_2}{X B_1 B_2}\Big)^{1/2} \mathop{\sum\sum\sum}_{\substack{d, d_1, d_2 \in \Z[\omega] \\ d, d_1, d_2 \equiv 1 \pmod{3} \\ (d_1, d_2)=1,\ d_1d_2 \mid d^\infty \\ N(d_1 d_2 d^2) \ll \frac{X^{1+\varepsilon} B_1 B_2}{M_1 M_2}}} \frac{N(e)}{N(d)} N(d_1d_2)^{1/4} \nonumber \\ 
    & \ll X^{1/2+\varepsilon} L (B_1 B_2)^{3/2} \sum_{\substack{d \in \Z[\omega] \\ d \equiv 1 \pmod{3} \\ N(d^2) \ll \frac{X^{1+\varepsilon} B_1 B_2}{M_1 M_2}}} \frac{1}{N(d)} \ll X^{1/2+\varepsilon} L (B_1 B_2)^{3/2}. \label{integral_term_final_bound}
\end{align}
In the displays above we used Rankin's trick and the inequality
\begin{equation*}
N(e) N(d_1 d_2)^{1/4} \ll N(d) N(d_1 d_2)^{1/4} \ll N(d) N(d_1 d_2)^{1/2} \ll \Big( \frac{X^{1+\varepsilon} B_1 B_2}{M_1 M_2} \Big)^{1/2}.
\end{equation*}


\subsection{The error term \texorpdfstring{$\mathcal{R}(b_1, b_2)$}{}: first bound for \texorpdfstring{$\mathcal{P}_1$}{} contribution}
\label{P1_cont_secA}

We are now left with bounding the contribution to $\mathfrak{S}^L_{M_1, M_2}(B_1, B_2)$ of the polar term $\mathcal{P}_1$ which arises from evaluation of the sum over $n_1$ in \eqref{ready_for_poles_eq}. By the discussion around \eqref{conditional_alpha_i_relation} for $i=1$, in this case we have $\beta = \beta_1 = 1$ and therefore $u\mid \Upsilon$ and $e \mid \iota$, so we can write $\Upsilon = u \Upsilon'$ for $\Upsilon' \mid u^\infty$ and $\iota = e \iota'$ for $\iota' \mid d^\infty$. 

Using the definition \eqref{P_i_def} for $\mathcal{P}_1$, observing that
\begin{equation*}
    \overline{\widetilde{g}_3(\eta_1, q_1 \alpha_1\alpha l)} = \overline{\widetilde{g}_3(\eta_1, q_1\alpha_1\alpha)} \chi_{q_1 \alpha_1 \alpha}(l) \overline{\widetilde{g}_3(\eta_1, l)}
\end{equation*}
by coprimality, and removing the sum over $b_1$ from inside the absolute value in \eqref{ready_for_poles_eq} via the triangle inequality, we conclude that the contribution to \eqref{ready_for_poles_eq} of the term corresponding to $\mathcal{P}_1$ in the sum over $n_1$ is
\begin{align}
    & \ll \frac{X^\varepsilon}{M_1^{1/6}} \mathop{\sum\sum}_{\substack{w_1, w_2 \in \Z[\omega] \\ w_1, w_2 \equiv 1 \pmod{3} \\ (w_1 w_2, d) = 1 \\ N(w_i) \ll M_i}} \mathop{\sum\sum \sum\sum\sum\sum}_{\substack{u, \Upsilon', \theta, \iota', \alpha, \gamma \in \Z[\omega] \\ u, \Upsilon', \theta, \iota', \alpha, \gamma \equiv 1 \pmod 3 \\ \Upsilon' \mid u^\infty,\ \theta^\infty = (w_1 w_2)^\infty,\ \iota' \mid d^\infty \\ (u, \Delta_1\Delta_2) = (\alpha \gamma, u\Delta_1\Delta_2) = 1 \\ \mu^2(u) = \mu^2(\alpha \gamma) = 1 \\ N(u^2\Upsilon'\theta\iota'\alpha\gamma^3) \ll \frac{M_1 M_2 L^2 N(d_1 d_2)}{X^{1-\varepsilon}}}} \frac{|\widetilde{g}_3(\theta, w_1)\widetilde{g}_3(\theta, w_2)|}{N(w_1)^{5/6} N(w_2)} \frac{|\widetilde{g}_3(\iota', d_1) \widetilde{g}_3(\iota', d_2)|}{N(d_1d_2)}
     \nonumber \\
    & \times  \frac{\bbone_{\beta_1=1}}{N(u)^2 N(\alpha)^{1/6}} \sum_{\substack{b_1 \in \Z[\omega] \\ b_1 \equiv 1 \pmod 3 \\ N(b_1) \sim B_1}} \mu^2(b_1) \sum_{\substack{f \in \Z[\omega] \\ f \equiv 1 \pmod{3} \\ N(f) \ll M_1+M_2 \\ (f, d) = 1}} \frac{\mu^2(f)}{N(q_1)} \Bigg|  \sum_{\substack{l \in \Z[\omega] \\ lu \equiv c \pmod 9 \\ N(u l) \sim L \\ N(ul) \leq Y \\ (u\alpha\gamma, l) = 1 }} \frac{\mu(l) \overline{\psi_{\Delta_1, \Delta_2}(l)}}{N(l)^{2-\varepsilon + it}}  \nonumber \\
    & \times \frac{\mathbf{1}_{q_1}(l)\chi_{\alpha_1 \alpha}(l) \overline{\widetilde{g}_3(\eta_1, l)}}{N(l)^{1/6}} \frac{\Delta_l(1)}{\Delta_l(4/3)} \sum_{\substack{b_2 \in \Z[\omega] \\ b_2 \equiv 1 \pmod 3 \\ N(b_2) \sim B_2}} \mu^2(b_2) c(b_1, b_2) \frac{\chi_{q_2}(\eta_2\Delta_2 u^3 \Upsilon' \theta e \iota' \alpha \gamma^3 l) \overline{\widetilde{g}_3(q_2)}}{N(q_2)^{1-it_2}} \nonumber \\
    & \qquad \times \Bigg( \sum_{\substack{n_2 \in \Z[\omega] \\ n_2 \equiv 1  \pmod{3}}} \frac{\chi_{n_2}(\eta_2 q_2 d^3 \Delta_2 u^3 \Upsilon' \theta e \iota' \alpha \gamma^3 l) \overline{\widetilde{g}_3(n_2)}}{N(n_2)^{1-it_2}} H\Big(\frac{N(w_2 q_2 n_2)}{M_2}\Big) \Bigg) \Bigg|. \label{k_sum_polar_term}
\end{align}

We have two distinct strategies for bounding \eqref{k_sum_polar_term}, depending on the relative sizes of the parameters $M_1, M_2, B_1, B_2, L$. The first strategy is to simply evaluate the sum over $n_2$ exactly as we did for $n_1$, that is via \eqref{decomposition_Gauss_sums_eq}. In this case we once again apply a triangle inequality on the sums over $l$ and $b_2$, to obtain a second cross term $|\mathcal{P}_1||\mathcal{I}_2|$ and a pure polar term $|\mathcal{P}_1||\mathcal{P}_2|$.


\subsubsection{The cross term \texorpdfstring{$|\mathcal{P}_1||\mathcal{I}_2|$}{}}

The contribution to $\mathfrak{S}^L_{M_1, M_2}(B_1, B_2)$ of the cross term $|\mathcal{P}_1||\mathcal{I}_2|$ is
\begin{align}
    \ll_F X^{1/2+\varepsilon} L^{1/2} M_1^{1/3} B_1^{5/6} B_2^{4/3} \label{second_cross_term_result}
\end{align}
by the same argument which led to \eqref{first_cross_term_result}.


\subsubsection{The pure polar term \texorpdfstring{$|\mathcal{P}_1||\mathcal{P}_2|$}{}}

By \eqref{P_i_def} and dropping several conditions by positivity, the contribution of the pure polar term $|\mathcal{P}_1||\mathcal{P}_2|$ to \eqref{k_sum_polar_term}, and therefore to \eqref{ready_for_poles_eq}, is
\begin{align}
    & \ll \frac{X^\varepsilon}{L^2 (M_1 M_2)^{1/6}} \mathop{\sum\sum}_{\substack{w_1, w_2 \in \Z[\omega] \\ w_1, w_2 \equiv 1 \pmod{3} \\ (w_1 w_2, d) = 1 \\ N(w_i) \ll M_i}} \Bigg(\mathop{\sum\sum\sum}_{\substack{f, b_1, b_2 \in \Z[\omega] \\ f, b_1, b_2 \equiv 1 \pmod 3  \\ N(b_1) \sim B_1,\ N(b_2) \sim B_2 \\ N(f) \ll M_1+M_2,\ (f, d)=1}} \frac{\mu^2(f) \mu^2(b_1) \mu^2(b_2)}{N(q_1q_2)}\Bigg) \mathop{\sum\sum \sum}_{\substack{u, \Upsilon', \theta \in \Z[\omega] \\ u, \Upsilon', \theta \equiv 1 \pmod 3 \\ \Upsilon' \mid u^\infty,\ \theta^\infty = (w_1 w_2)^\infty}} 
     \nonumber \\
    & \times \frac{|\widetilde{g}_3(\theta, w_1)\widetilde{g}_3(\theta, w_2)|}{N(w_1 w_2)^{5/6}} \sum_{\substack{\iota' \in \Z[\omega] \\ \iota' \equiv 1 \pmod 3 \\ \iota' \mid d^\infty}} \frac{|\widetilde{g}_3(\iota', d_1) \widetilde{g}_3(\iota', d_2)|}{N(d_1d_2)} \mathop{\sum\sum\sum}_{\substack{\alpha, \gamma, l \in \Z[\omega] \\ \alpha, \gamma, l \equiv 1 \pmod 3 \\ N(u^2\Upsilon'\theta\iota'\alpha\gamma^3) \ll \frac{M_1M_2 L^2 N(d_1d_2)}{X^{1-\varepsilon}} \\ N(ul)\sim L}} \frac{1}{N(\alpha l)^{1/3}} . \label{pure_polar_initial}
\end{align}

Localize in dyadic ranges $N(u) \sim U, N(\Upsilon') \sim Y', N(\theta) \sim \Theta, N(\iota') \sim I', N(\alpha) \sim A, N(\gamma) \sim C$, where all of the ranges are $\gg 1$, and satisfy
\begin{equation}\label{polar_dyadic_ranges}
    U^2Y'\Theta I' A C^3 \ll \frac{M_1M_2L^2 N(d_1 d_2)}{X^{1-\varepsilon}}.
\end{equation}

We can evaluate the sums over $l$ to get $\ll \frac{L^{2/3}}{U^{2/3}}$, over $\alpha$ and $\gamma$ to get $A^{2/3}C$, over $\iota'$ -- using \eqref{iota_Gauss_sum_bound} and observing the difference in normalization -- to get $\ll \frac{X^\varepsilon I'^{1/2}}{N(d_1d_2)}$, 
and over $u$ and $\Upsilon'$ to get $\ll X^\varepsilon U$. The sums over $f, b_1, b_2$ are easily seen to be $\ll X^\varepsilon$ using \eqref{q_i_split} and three applications of \cref{gcd_sum_lemma}. Thus, noting the different normalization of Gauss sums, \eqref{pure_polar_initial} is 
\begin{align}
    & \ll \frac{X^\varepsilon U^{1/3}I'^{1/2}A^{2/3}C}{L^{4/3} (M_1 M_2)^{1/6} N(d_1d_2)} \mathop{\sum\sum\sum\sum\sum\sum}_{\substack{h_1, h_2, k_1, k_2, o_1, o_2 \in \Z[\omega] \\ h_1, h_2, k_1, k_2, o_1, o_2 \equiv 1 \pmod{3} \\ \mu^2(h_1k_1) = \mu^2(h_2k_2)=1 \\ o_1 \mid k_1^\infty,\ o_2 \mid k_2^\infty,\ N(h_i k_i^2 o_i) \ll \Theta}} \sum_{\substack{\theta \in \Z[\omega] \\ \theta \equiv 1 \pmod 3 \\ \theta^\infty = (h_1 h_2 k_1 k_2)^\infty \\ N(\theta) \sim \Theta}} \frac{|g_3(\theta, h_1^2 k_1^3 o_1)g_3(\theta, h_2^2 k_2^3 o_2)|}{N(h_1^2 h_2^2 k_1^3 k_2^3 o_1 o_2)^{4/3}}, \nonumber
\end{align}
where the bound $N(h_i k_i^2 o_i) \ll \Theta$ comes as before from the fact that $h_i k_i^2 o_i \mid \theta$ by \eqref{theta_divisibility}. We can use \eqref{h_k_sums_Gauss_sum bound} to bound all of the remaining sums above by $\ll X^\varepsilon$, and obtain a final estimate for \eqref{pure_polar_initial} in the chosen dyadic ranges of the form
\begin{align*}
    \ll \frac{X^\varepsilon U^{1/3}I'^{1/2}A^{2/3}C}{L^{4/3} (M_1 M_2)^{1/6} N(d_1d_2)} \ll \frac{X^\varepsilon}{L^{4/3} (M_1 M_2)^{1/6} N(d_1d_2)} \Big(\frac{M_1 M_2 L^2 N(d_1d_2)}{X^{1-\varepsilon}}\Big)^{2/3},
\end{align*}
where we used \eqref{polar_dyadic_ranges}. Thus this holds over all ranges. Inserting it into \eqref{R_after_abs_values}, we conclude that the contribution of the pure polar term $|\mathcal{P}_1||\mathcal{P}_2|$ to $\mathfrak{S}^L_{M_1, M_2}(B_1, B_2)$ is
\begin{align}
    & \ll_F X^{5/6+\varepsilon} B_1 B_2 \Big(\frac{M_1 M_2}{X B_1 B_2}\Big)^{1/2} \mathop{\sum\sum\sum}_{\substack{d, d_1, d_2 \in \Z[\omega] \\ d, d_1, d_2 \equiv 1 \pmod{3} \\ (d_1, d_2)=1,\ d_1d_2 \mid d^\infty \\ N(d_1 d_2 d^2) \ll \frac{X^{1+\varepsilon} B_1 B_2}{M_1 M_2}}} \frac{1}{N(d)}  \frac{1}{N(d_1 d_2)^{1/3}} \nonumber \\ 
    & \ll X^{5/6+\varepsilon} B_1 B_2 \sum_{\substack{d \in \Z[\omega] \\ d \equiv 1 \pmod{3} \\ N(d^2) \ll \frac{X^{1+\varepsilon} B_1 B_2}{M_1 M_2}}} \frac{1}{N(d)^2} \ll X^{5/6+\varepsilon} B_1 B_2. \label{pure_polar_final_bound}
\end{align}
In the displays above we used $\big( \tfrac{M_1 M_2}{X B_1 B_2} \big)^{1/2} \ll X^{\varepsilon/2} \big( \tfrac{1}{N(d_1 d_2 d^2)} \big)^{1/2}$
and Rankin's trick. Here and below, the additional powers of $X^\varepsilon$ are absorbed into the existing factors by adjusting $\varepsilon$.


\subsection{The error term \texorpdfstring{$\mathcal{R}(b_1, b_2)$}{}: second bound for \texorpdfstring{$\mathcal{P}_1$}{} contribution}
\label{P1_cont_secB}

We give an alternative treatment of \eqref{k_sum_polar_term}, which will lead to a better bound for some ranges of the parameters $M_1, M_2, B_1, B_2, L$. The second strategy is to observe that the only term which depends on $\alpha$ in the sums over $l, b_2, n_2$ in \eqref{k_sum_polar_term} is $\chi_{\alpha}(l q_2 n_2)$, so we can apply the cubic large sieve by grouping these three variables against $\alpha$.

Observe that the quantity inside absolute values in \eqref{k_sum_polar_term} is equal to
\begin{equation*}
    \sum_{\substack{l, q, n_2 \in \Z[\omega] \\ l, q, n_2 \equiv 1 \pmod 3 \\ N(l q n_2) \ll \frac{L M_2}{N(u w_2)}}} \mu^2(l q n_2) \cdot c_{l, q, n_2} \cdot \chi_\alpha(l q n_2),
\end{equation*}
where
\begin{align}
    c_{l, q, n_2} := &\  \bbone_{\substack{lu \equiv c \pmod 9 \\ N(u l) \sim L \\ N(ul) \leq Y \\ (q_1 u \gamma, l) = 1 }} \frac{\mu(l) \overline{\psi_{\Delta_1, \Delta_2}(l)}}{N(l)^{2-\varepsilon + it}} \frac{\chi_{\alpha_1}(l) \overline{\widetilde{g}_3(\eta_1, l)}}{N(l)^{1/6}} \frac{\Delta_l(1)}{\Delta_l(4/3)} \sum_{\substack{b_2 \in \Z[\omega] \\ b_2 \equiv 1 \pmod 3 \\ N(b_2) \sim B_2}} \mu^2(b_2) c(b_1, b_2) \cdot \bbone_{q_2 = q} \nonumber \\
    & \times \frac{\chi_{q n_2}(\eta_2 d^3 \Delta_2 u^3 \Upsilon' \theta e \iota' \gamma^3 l) \overline{\widetilde{g}_3(q n_2)}}{N(q n_2)^{1-it_2}} H\Big(\frac{N(w_2 q n_2)}{M_2}\Big). \nonumber
\end{align}
We dropped the condition $(\alpha, l)=1$ since it is automatically enforced by $\chi_\alpha(l)$, and used the fact that we automatically have $(q_2, d)=1$, as discussed below \eqref{q_i_def} (since $b_2$ is squarefree and $(f, d)=1$), to add a factor of $\chi_{q}(d^3)$. The coefficients $c_{l, q, n_2}$ depend on the variables $d, d_1, d_2, w_1, w_2, u, \Upsilon', \theta, \iota', \gamma, b_1, f, l, q, n_2$, but not on $\alpha$. 

By \eqref{q_i_split}, for any given $d$, $f$, $w_2$, and $q$, the number of squarefree $b_2 \equiv 1 \pmod 3$ for which $q_2 = q$ is $\leq d(dfw_2)\ll X^\varepsilon$, where $d(\cdot)$ denotes the divisor function. Evaluating the sums over $l$ and $n_2$ and applying this observation gives
\begin{align}
    \sum_{\substack{l, q, n_2 \in \Z[\omega]}} \mu^2(l q n_2) \cdot |c_{l, q, n_2}|^2 & \ll X^\varepsilon \Big(\frac{N(u)}{L}\Big)^{10/3} \frac{N(w_2)}{M_2} \sum_{0 \neq q \in \Z[\omega]} \frac{1}{N(q)}\Bigg(\sum_{\substack{b_2 \in \Z[\omega] \\ b_2 \equiv 1 \pmod 3 \\ N(b_2) \sim B_2}} \mu^2(b_2) \cdot \bbone_{q_2 = q} \Bigg)^2 \nonumber \\
    & \ll X^\varepsilon \Big(\frac{N(u)}{L}\Big)^{10/3} \frac{N(w_2)}{M_2} \sum_{\substack{0 \neq q \in \Z[\omega] \\ N(q) \ll B_2 N(f)}} \frac{1}{N(q)} \nonumber \\
    & \ll X^\varepsilon \Big(\frac{N(u)}{L}\Big)^{10/3} \frac{N(w_2)}{M_2}. \label{c_l^2_norm_bound}
\end{align}

As before, let us localize \eqref{k_sum_polar_term} in dyadic ranges $N(u) \sim U, N(\Upsilon') \sim Y', N(\theta) \sim \Theta, N(\iota') \sim I', N(\alpha) \sim A, N(\gamma) \sim C$, where all of the ranges are $\gg 1$, and satisfy \eqref{polar_dyadic_ranges}. Thus \eqref{k_sum_polar_term} in these dyadic ranges is, after dropping some conditions by positivity,
\begin{align}
    & \ll \frac{X^\varepsilon}{M_1^{1/6} U^2 A^{1/6}} \mathop{\sum\sum}_{\substack{w_1, w_2 \in \Z[\omega] \\ w_1, w_2 \equiv 1 \pmod{3} \\ (w_1 w_2, d) = 1 \\ N(w_i) \ll M_i}} \mathop{\sum\sum\sum\sum\sum}_{\substack{u, \Upsilon', \theta, \iota', \gamma \in \Z[\omega] \\ u, \Upsilon', \theta, \iota', \gamma \equiv 1 \pmod 3 \\ \Upsilon' \mid u^\infty,\ \theta^\infty = (w_1 w_2)^\infty,\ \iota' \mid d^\infty \\ N(u) \sim U,\ N(\Upsilon') \sim Y',\ N(\theta) \sim \Theta \\ N(\iota') \sim I',\ N(\gamma) \sim C}} \frac{|\widetilde{g}_3(\theta, w_1)\widetilde{g}_3(\theta, w_2)|}{N(w_1)^{5/6} N(w_2)} \frac{|\widetilde{g}_3(\iota', d_1) \widetilde{g}_3(\iota', d_2)|}{N(d_1d_2)}
     \nonumber \\
    & \times \sum_{\substack{f, b_1 \in \Z[\omega] \\ f, b_1 \equiv 1 \pmod{3} \\ N(f) \ll M_1+M_2 \\ (f, d) = 1,\ N(b_1) \sim B_1}} \frac{\mu^2(f)\mu^2(b_1)}{N(q_1)} \sum_{\substack{\alpha \in \Z[\omega] \\ \alpha \equiv 1 \pmod 3 \\ N(\alpha) \sim A}} \mu^2(\alpha) \Bigg| \sum_{\substack{l, q, n_2 \in \Z[\omega] \\ l, q, n_2 \equiv 1 \pmod 3 \\ N(l q n_2) \ll \frac{L M_2}{U N(w_2)}}} \mu^2(l q n_2) \cdot c_{l, q, n_2} \cdot \chi_\alpha(l q n_2) \Bigg|. \label{alpha_cubic_large_sieve}
\end{align}

Applying Cauchy--Schwarz and the cubic large sieve as in \cref{cubic-large-sieve}, from \eqref{c_l^2_norm_bound} and a divisor bound we see that
\begin{align}
    &\sum_{\substack{\alpha \in \Z[\omega] \\ \alpha \equiv 1 \pmod 3 \\ N(\alpha) \sim A}} \mu^2(\alpha) \Bigg| \sum_{\substack{l, q, n_2 \in \Z[\omega] \\ l, q, n_2 \equiv 1 \pmod 3 \\ N(l q n_2)\ll \frac{L M_2}{U N(w_2)}}} \mu^2(l q n_2) \cdot c_{l, q, n_2} \cdot \chi_\alpha(l q n_2) \Bigg| \nonumber \\
    & \ll X^\varepsilon A^{1/2} \Big(A^{1/2} + \Big(\frac{L M_2}{U N(w_2)}\Big)^{1/2} + \Big(\frac{A L M_2}{U N(w_2)}\Big)^{1/3}\Big) \Big(\frac{U}{L}\Big)^{5/3} \Big(\frac{N(w_2)}{M_2}\Big)^{1/2}  \nonumber \\
    &\ll X^\varepsilon \Big(\frac{U^{5/3} A}{L^{5/3} M_2^{1/2}} \cdot N(w_2)^{1/2} + \frac{U^{7/6} A^{1/2}}{L^{7/6}} + \frac{A^{5/6} U^{4/3}}{L^{4/3} M_2^{1/6}} \cdot N(w_2)^{1/6} \Big). \nonumber
\end{align}
Inserting this into \eqref{alpha_cubic_large_sieve}, we can use \eqref{q_i_split} and \cref{gcd_sum_lemma} as before to see that the sums over $b_1$ and $f$ are $\ll X^\varepsilon$. Also \eqref{iota_Gauss_sum_bound} or alternatively a trivial bound show that 
\begin{equation*}
    \sum_{\substack{\iota' \in \Z[\omega] \\ \iota' \equiv 1 \pmod 3 \\ \iota' \mid d^\infty,\ N(\iota') \sim I'}} \frac{|g_3(\iota', d_1) g_3(\iota', d_2)|}{N(d_1d_2)^{3/2}} \ll \min\Big\{\frac{X^\varepsilon I'^{1/2}}{N(d_1d_2)}, \frac{X^\varepsilon}{N(d_1d_2)^{1/2}} \Big\}.
\end{equation*}
The sum over $\gamma$ contributes $\ll C$, and the sums over $u$ and $\Upsilon'$ are $\ll X^\varepsilon U$. Finally, \eqref{theta_divisibility} and the bound \eqref{h_k_sums_Gauss_sum bound} give (note that here the Gauss sums are normalized)
\begin{align*}
    \mathop{\sum\sum}_{\substack{w_1, w_2 \in \Z[\omega] \\ w_1, w_2 \equiv 1 \pmod{3} \\ (w_1 w_2, d) = 1 \\ N(w_i) \ll M_i}} \sum_{\substack{\theta \in \Z[\omega] \\ \theta \equiv 1 \pmod 3 \\ \theta^\infty = (w_1 w_2)^\infty\\ N(\theta) \sim \Theta}} \frac{|\widetilde{g}_3(\theta, w_1)\widetilde{g}_3(\theta, w_2)|}{N(w_1)^{5/6} N(w_2)^{1/2}} \ll X^\varepsilon \mathop{\sum\sum\sum}_{\substack{h_2, k_2, o_2 \in \Z[\omega] \\ h_2, k_2, o_2 \equiv 1 \pmod{3} \\ o_2 \mid k_2^\infty,\ N(o_2) \ll \Theta \\ N(h_2 k_2^2) \ll \Theta}} \frac{1}{N(h_2)^{1/2}} \ll X^\varepsilon \Theta^{1/2}.
\end{align*}
Similarly, if we replace $N(w_2)^{1/2}$ with $N(w_2)^{\delta}$ for $\delta \geq \frac{5}{6}$, the bound above becomes $\ll X^\varepsilon$. Combining all of these bounds, \eqref{alpha_cubic_large_sieve} is
\begin{align}
    & \ll X^\varepsilon \Big(\frac{U^{2/3} I'^{1/2} \Theta^{1/2} A^{5/6} C}{L^{5/3} M_1^{1/6} M_2^{1/2} N(d_1d_2)} + \frac{U^{1/6}  A^{1/3} C}{L^{7/6}  M_1^{1/6} N(d_1d_2)^{1/2}} + \frac{U^{1/3} I'^{1/2} A^{2/3} C}{L^{4/3}  (M_1 M_2)^{1/6} N(d_1d_2)} \Big) \nonumber \\
    & \ll X^\varepsilon \Big(\frac{Z^{5/6}}{L^{5/3} M_1^{1/6} M_2^{1/2} N(d_1d_2)} + \frac{Z^{1/3}}{L^{7/6} M_1^{1/6} N(d_1d_2)^{1/2}} + \frac{Z^{2/3}}{L^{4/3} (M_1 M_2)^{1/6} N(d_1d_2)} \Big) \label{polar_cubic_sieve_intermediate}
\end{align}
for $Z := \frac{M_1M_2 L^2 N(d_1d_2)}{X^{1-\varepsilon}}$, where we used the restriction \eqref{polar_dyadic_ranges} on the ranges. Thus the same estimate holds over all ranges.

Plugging this bound into \eqref{R_after_abs_values}, the first term of \eqref{polar_cubic_sieve_intermediate} contributes
\begin{align*}
    &\ll_F \frac{X^{5/6+\varepsilon} (B_1 B_2)^{7/6}}{M_2^{1/3}} \Big(\frac{M_1 M_2}{X B_1 B_2}\Big)^{2/3} \mathop{\sum\sum\sum}_{\substack{d, d_1, d_2 \in \Z[\omega] \\ d, d_1, d_2 \equiv 1 \pmod{3} \\ (d_1, d_2)=1,\ d_1d_2 \mid d^\infty \\ N(d_1 d_2 d^2) \ll \frac{X^{1+\varepsilon} B_1 B_2}{M_1 M_2} \\ N(d) \gg \frac{B_i}{M_i}}} \frac{1}{N(d)}  \frac{1}{N(d_1 d_2)^{1/6}} \\
    &\ll \frac{X^{5/6+\varepsilon} (B_1 B_2)^{7/6}}{M_2^{1/3}}  \sum_{\substack{d \in \Z[\omega] \\ d \equiv 1 \pmod{3} \\ \frac{B_2}{M_2} \ll N(d) \ll \big(\frac{X^{1+\varepsilon} B_1 B_2}{M_1 M_2}\big)^{1/2}}} \frac{1}{N(d)^{7/3}} \ll \frac{X^{5/6+\varepsilon} (B_1 B_2)^{7/6}}{M_2^{1/3}} \min\Big\{1, \Big(\frac{M_2}{B_2}\Big)^{4/3}\Big\} \\
    &\ll X^{5/6+\varepsilon} B_1^{7/6} B_2^{5/6}.
\end{align*}
Here we used $\big( \tfrac{M_1 M_2}{X B_1 B_2} \big)^{2/3} \ll X^{2\varepsilon/3} \big( \tfrac{1}{N(d_1 d_2 d^2)} \big)^{2/3}$. Similarly, the second term of \eqref{polar_cubic_sieve_intermediate} contributes
\begin{align*}
    &\ll_F \frac{X^{5/6+\varepsilon} M_2^{1/6} (B_1 B_2)^{2/3}}{L^{1/2}} \Big(\frac{M_1 M_2}{X B_1 B_2}\Big)^{1/6} \mathop{\sum\sum\sum}_{\substack{d, d_1, d_2 \in \Z[\omega] \\ d, d_1, d_2 \equiv 1 \pmod{3} \\ (d_1, d_2)=1,\ d_1d_2 \mid d^\infty \\ N(d_1 d_2 d^2) \ll \frac{X^{1+\varepsilon} B_1 B_2}{M_1 M_2}}} \frac{1}{N(d)}  \frac{1}{N(d_1 d_2)^{1/6}} \\
    &\ll \frac{X^{5/6+\varepsilon} M_2^{1/6} (B_1 B_2)^{2/3}}{L^{1/2}}.
\end{align*}
This time we used $\big( \tfrac{M_1 M_2}{X B_1 B_2} \big)^{1/6} \ll X^{\varepsilon/6} \big( \tfrac{1}{N(d_1 d_2 d^2)} \big)^{1/6}$. The third term of \eqref{polar_cubic_sieve_intermediate} contributes
\begin{align*}
    \ll_F X^{5/6+\varepsilon} B_1 B_2 \Big(\frac{M_1 M_2}{X B_1 B_2}\Big)^{1/2} \mathop{\sum\sum\sum}_{\substack{d, d_1, d_2 \in \Z[\omega] \\ d, d_1, d_2 \equiv 1 \pmod{3} \\ (d_1, d_2)=1,\ d_1d_2 \mid d^\infty \\ N(d_1 d_2 d^2) \ll \frac{X^{1+\varepsilon} B_1 B_2}{M_1 M_2}}} \frac{1}{N(d)}  \frac{1}{N(d_1 d_2)^{1/3}} \ll X^{5/6+\varepsilon} B_1 B_2.
\end{align*}
Once again we used Rankin's trick and the inequality $\big( \tfrac{M_1 M_2}{X B_1 B_2} \big)^{1/2} \ll X^{\varepsilon/2} \big( \tfrac{1}{N(d_1 d_2 d^2)} \big)^{1/2}$.

Putting those three displays together, the contribution to $\mathfrak{S}^L_{M_1, M_2}(B_1, B_2)$ of the polar term $|\mathcal{P}_1(\mathcal{P}_2 + \mathcal{I}_2)|$ is
\begin{align}
    \ll_F X^{5/6+\varepsilon} \Big(B_1^{7/6} B_2^{5/6} + \frac{M_2^{1/6} (B_1 B_2)^{2/3}}{L^{1/2}} + B_1 B_2\Big). \label{polar_cubic_final_bound}
\end{align} 


\subsection{The error term \texorpdfstring{$\mathcal{R}(b_1, b_2)$}{}: final optimization}
\label{opt_sec}

Combining \eqref{polar_cubic_final_bound} with our previous bounds given in \eqref{second_cross_term_result} and \eqref{pure_polar_final_bound}, we conclude that for $B := \max\{B_1, B_2\}$, the contribution to $\mathfrak{S}^L_{M_1, M_2}(B_1, B_2)$ of the polar terms $|\mathcal{P}_1(\mathcal{P}_2 + \mathcal{I}_2)|$ is
\begin{align}
    &\ll_F X^{5/6+\varepsilon} \min\Big\{\frac{L^{1/2} M_1^{1/3} B_1^{5/6} B_2^{4/3}}{X^{1/3}} + B_1 B_2, B_1^{7/6} B_2^{5/6} + \frac{M_2^{1/6} (B_1B_2)^{2/3}}{L^{1/2}} + B_1B_2 \Big\} \nonumber \\
    &\ll X^{5/6+\varepsilon}\Big[(B_1B_2)^{5/6}B^{1/3} +  (B_1B_2)^{2/3} \min\Big\{\frac{L^{1/2} M_1^{1/3} B_1^{1/6}B_2^{2/3}}{X^{1/3}}, \frac{M_2^{1/6}}{L^{1/2}}\Big\} \Big] \nonumber \\
    &\ll X^{5/6+\varepsilon}(B_1B_2)^{1/2} \Big[(B_1B_2 B)^{1/3} +  (B_1B_2)^{1/6}\Big(\frac{L^{1/2} M_1^{1/3} B_1^{1/6}B_2^{2/3}}{X^{1/3}}\Big)^{1/3}\Big(\frac{M_2^{1/6}}{L^{1/2}}\Big)^{2/3} \Big] \nonumber \\
    &\ll X^{5/6+\varepsilon}(B_1B_2)^{1/2} \Big[B + \frac{(M_1M_2)^{1/9}(B_1B_2)^{2/9}B_2^{1/6}}{X^{1/9}L^{1/6}}\Big] \nonumber\\
    &\ll X^{5/6+\varepsilon}(B_1B_2)^{1/2}  \Big[B + \frac{(B_1B_2)^{1/3}B_2^{1/6}}{L^{1/6}}\Big] \ll X^{5/6+\varepsilon}(B_1B_2)^{1/2}B,
    \label{polar_term_final_bound}
\end{align}
where we used the restriction $M_1M_2 \ll X^{1+\varepsilon}B_1B_2$ from \eqref{M_1_M_2_bound}.

Combining the bounds \eqref{first_cross_term_result} for the cross term $|\mathcal{I}_1||\mathcal{P}_2|$, \eqref{integral_term_final_bound} for the integral term $|\mathcal{I}_1||\mathcal{I}_2|$, and \eqref{polar_term_final_bound} for the polar terms $|\mathcal{P}_1(\mathcal{P}_2 + \mathcal{I}_2)|$, we conclude that
\begin{align}
    \mathfrak{S}_{M_1, M_2}^{L}(B_1, B_2) &\ll_F X^{1/2+\varepsilon} L^{1/2} M_2^{1/3} B_1^{4/3} B_2^{5/6} + X^{1/2+\varepsilon} L (B_1 B_2)^{3/2} + X^{5/6+\varepsilon} (B_1B_2)^{1/2} B \nonumber \\
    &\ll X^{1/2+\varepsilon} (B_1B_2)^{1/2}\Big( L^{1/2} M_2^{1/3} B^{7/6} + L B^2 + X^{1/3} B\Big).\label{synthesis_temporary}
\end{align}
By the symmetry of our initial bound \eqref{R_after_abs_values}, we may obtain \eqref{synthesis_temporary} with $M_2$ replaced by $M_1$, so without loss of generality we may assume $M_2 \leq M_1$, in which case
\begin{equation*}
    M_2 \leq (M_1M_2)^{1/2} \ll X^{1/2+\varepsilon} (B_1B_2)^{1/2} \leq X^{1/2+\varepsilon}B.
\end{equation*}
Inserting this into \eqref{synthesis_temporary}, we finally obtain
\begin{align}
    \mathfrak{S}_{M_1, M_2}^{L}(B_1, B_2) &\ll_F X^{1/2+\varepsilon} (B_1B_2)^{1/2} \Big(X^{1/6} L^{1/2} B^{3/2} + L B^2 + X^{1/3}B\Big).\nonumber
\end{align}
Summing over the dyadic ranges in \eqref{B_1_B_2_sum_decomposition} then gives the desired result \eqref{B_1B_2_final_bound}.

\end{proof}


\section{First moment asymptotics: proof of \texorpdfstring{\cref{SMestimate1}}{}} \label{first_moment_asymptotics_section}

We can now take advantage of the machinery developed in \cref{second_main_sum_section} to quickly deal with the main sum for the twisted first moment.

\begin{proof}[Proof of \cref{SMestimate1}]
By \eqref{SMdef}, 
\begin{align*}
&\mathcal{S}_M \Big( \chi_{q}(\mathfrak{b}) \Big[A_1(q) + \widetilde{g}_3(q) \cdot \overline{A_1(q)}\Big]; F \Big) \\
& = \sum_{\substack{ q \in \mathbb{Z}[\omega] \\ q \equiv 1 \pmod{9}}}  M_Y(q)  \chi_{q}(\mathfrak{b}) \Big[A_1(q) + \widetilde{g}_3(q) \cdot \overline{A_1(q)}\Big] F \Big( \frac{N(q)}{X} \Big).
\end{align*}
Express $A_1(q)$ using \eqref{A1def} to conclude that the expression above is equal to
\begin{equation} \label{SMopen_first_mom}
 \sum_{\substack{0 \neq \mathfrak{n} \unlhd \Z[\omega]}} 
N(\mathfrak{n})^{-1/2}  \Big[ \mathcal{S}_M \big( \chi_q(\mathfrak{b} \mathfrak{n}); F_{\mathfrak{n}} \big) + \mathcal{S}_M \big( \widetilde{g}_3(q) \overline{\chi_q(\mathfrak{b}^2 \mathfrak{n})}; F_{\mathfrak{n}} \big) \Big],
\end{equation}
where we used the fact that the function $\Phi_1$ given in \eqref{def-Phi} is real-valued and set
\begin{equation*}
F_{\mathfrak{n} }(t):= F(t) \Phi_1 \Big( \frac{N(\mathfrak{n})}{\sqrt{3 X t}} \Big).
\end{equation*}
Opening up the two terms of \eqref{SMopen_first_mom} and using \eqref{MYRYdef}, we have that
\begin{align*}
    \mathcal{S}_M \big( \chi_q(\mathfrak{b} \mathfrak{n}); F_{\mathfrak{n}} \big) =\sum_{\substack{\ell \in \Z[\omega] \\ \ell \equiv 1 \pmod 3 \\ N(\ell) \leq Y}} \mu(\ell) \sum_{\substack{ m \in \Z[\omega]  \\ \ell^2 m \equiv 1 \pmod{9}}}  \chi_{\ell^2 m}(\mathfrak{b} \mathfrak{n}) F_{\mathfrak{n}} \Big( \frac{N(\ell^2 m)}{X} \Big)
\end{align*}
and
\begin{align*}
    \mathcal{S}_M \big( \widetilde{g}_3(q) \overline{\chi_q(\mathfrak{b}^2 \mathfrak{n})}; F_{\mathfrak{n}} \big) = \sum_{\substack{ m \in \Z[\omega]  \\ m \equiv 1 \pmod{9}}}  \widetilde{g}_3(m) \overline{\chi_m(\mathfrak{b}^2 \mathfrak{n})} F_{\mathfrak{n}} \Big( \frac{N(m)}{X} \Big),
\end{align*}
where for the latter we used the fact that $\widetilde{g}_3(\ell^2 m) = 0$ unless $\ell = 1$.

Write $\mathfrak{b} = b \mathbb{Z}[\omega]$ and $\mathfrak{n} =\lambda^{g} n \mathbb{Z}[\omega]$ for some $g \in \mathbb{Z}_{\geq 0}$ and $b, n \equiv 1 \pmod{3}$, where $\mu^2(b)=1$ since we are assuming $\mathfrak{b}$ is squarefree. By \eqref{cubesupp} and cubic reciprocity,
\begin{equation*}
\chi_{\ell^2 m}(\mathfrak{b} \mathfrak{n}) = \chi_{\ell^2 m}(b \lambda^{g} n)=\chi_{\ell^2 m}(b n)=\chi_{b n}(\ell^2 m) = \overline{\chi_{b n}(\ell)} \chi_{b n}(m),
\end{equation*}
and similarly
\begin{equation*}
    \overline{\chi_{m}(\mathfrak{b}^2 \mathfrak{n})} = \overline{\chi_{m}(b^2 \lambda^{g} n)} = \overline{\chi_{m}(b^2 n)}.
\end{equation*} 
Thus 
\begin{align}\label{first_mom_sum}
    & \mathcal{S}_M \big( \chi_q(b \lambda^g n); F_{\lambda^g n} \big) = \sum_{\substack{\ell \in \Z[\omega] \\ \ell \equiv 1 \pmod 3 \\ N(\ell) \leq Y}} \mu(\ell) \overline{\chi_{b n}(\ell)} \sum_{\substack{ m \in \Z[\omega] \\  \ell^2 m \equiv 1 \pmod{9}}} \chi_{b n}(m) F_{\lambda^{g} n} \Big( \frac{N(\ell^2 m)}{X}  \Big)
\end{align}
and 
\begin{equation}\label{first_mom_dual}
    \mathcal{S}_M \big( \widetilde{g}_3(q) \overline{\chi_q(b^2 \lambda^g n)}; F_{\lambda^g n} \big) = \sum_{\substack{ m \in \Z[\omega]  \\ m \equiv 1 \pmod{9}}}  \widetilde{g}_3(m) \overline{\chi_m(b^2 n)} F_{\lambda^g n} \Big( \frac{N(m)}{X} \Big).
\end{equation}

To deal with \eqref{first_mom_sum}, we first apply Poisson summation (\cref{radialpois}) to observe that the sum over $m$ is equal to
\begin{align*}
    \frac{4 \pi X}{3^{9/2} N(bn\ell^2)} \sum_{\substack{k \in \Z[\omega]}} \ddot{\chi}_{bn}(k) \check{e}\Big({ -\frac{k \ell b^2 n^2}{9\lambda}} \Big) \ddot{F}_{\lambda^g n} \Big( \frac{k \sqrt{X}}{bn \ell^2} \Big),
\end{align*}
where we used $\overline{\ell^2} \equiv \ell \pmod 9$ for $\ell \equiv 1 \pmod 3$. Since $9\lambda = \lambda^5$, we have
\begin{align*}
    \ddot{\chi}_{bn}(k) := \sum_{a \pmod{bn}} \chi_{bn}(9 \lambda a) \check{e} \Big({-\frac{ka}{bn}}\Big) = \overline{\chi_{bn}(\lambda)} g_3(- k, bn) = \overline{\chi_{bn}(\lambda)}  g_3(k, bn).
\end{align*}

Denote the expression corresponding to the first term in \eqref{SMopen_first_mom} by $\mathcal{S}(b)$, and that corresponding to the second term by $\widetilde{\mathcal{S}}(b)$. We conclude from replacing the above in \eqref{first_mom_sum} that
\begin{align}
    \mathcal{S}(b) &= \frac{4 \pi X}{3^{9/2} N(b)} \sum_{g=0}^\infty \frac{1}{3^{g/2}} \sum_{\substack{c \pmod 9 \\ c \equiv 1 \pmod 3}} \overline{\chi_{bc}(\lambda)} \sum_{k\in \Z[\omega]} \sum_{\substack{\ell \in \Z[\omega] \\ \ell \equiv 1 \pmod 3 \\ N(\ell) \leq Y}} \frac{\mu(\ell) \overline{\chi_{b}(\ell)}}{N(\ell)^2} \check{e}\Big({ -\frac{k \ell b^2 c^2}{9\lambda}} \Big) \nonumber \\
    & \qquad \qquad \qquad \times \sum_{\substack{n \in \Z[\omega] \\ n\equiv c \pmod 9}} \frac{\overline{\chi_{n}(\ell)} g_3(k, bn)}{N(n)^{3/2}} \ddot{F}_{\lambda^g n} \Big( \frac{k \sqrt{X}}{bn \ell^2} \Big). \label{S_first_mom_exp}
\end{align}
Let $\mathcal{M}(b)$ denote the terms corresponding to $k=0$ in \eqref{S_first_mom_exp}, and let $\mathcal{S}'(b)$ denote the rest of the terms. Finally, denote $\mathcal{R}(b) := \mathcal{S}'(b) + \widetilde{\mathcal{S}}(b)$.


\subsection{The main term $\mathcal{M}(b)$}

We have
\begin{align}
    \mathcal{M}(b) = \frac{4 \pi X}{3^{9/2} N(b)} \sum_{g=0}^\infty \frac{1}{3^{g/2}} \sum_{\substack{\ell \in \Z[\omega] \\ \ell \equiv 1 \pmod 3 \\ N(\ell) \leq Y}} \frac{\mu(\ell) \overline{\chi_{b}(\lambda \ell)}}{N(\ell)^2} \sum_{\substack{n \in \Z[\omega] \\ n\equiv 1 \pmod 3}} \frac{\overline{\chi_{n}(\lambda \ell)} g_3(0, bn)}{N(n)^{3/2}} \ddot{F}_{\lambda^g n} (0). \nonumber
\end{align}
Observe that $g_3(0, bn) = 0$ unless $bn$ is a cube, which (since $b \equiv 1 \pmod 3$ is squarefree) is equivalent to $n = b^2 m^3$ for $m\equiv 1 \pmod{3}$. In that case, $g_3(0, bn) = g_3(0, b^3m^3) = \varphi(b^3m^3) = \varphi(bm) N(bm)^2$. Further (uniquely) decomposing $m = b'c$ for $b'\mid b^\infty$ and $(b, c)=1$, we obtain
\begin{align}
    \mathcal{M}(b) & = \frac{4 \pi X}{3^{9/2} N(b)^2} \sum_{g=0}^\infty \frac{1}{3^{g/2}} \sum_{\substack{m \in \Z[\omega] \\ m\equiv 1 \pmod 3}} \frac{\varphi(b m)}{ N(m)^{5/2}} \ddot{F}_{\lambda^g b^2 m^3} (0) \sum_{\substack{\ell \in \Z[\omega] \\ \ell \equiv 1 \pmod 3 \\ N(\ell) \leq Y}} \frac{\mu(\ell) \mathbf{1}_{bm}(\ell)}{N(\ell)^2} \nonumber \\
    & = \frac{4 \pi X}{3^{9/2} N(b)^2} \sum_{g=0}^\infty \frac{1}{3^{g/2}} \sumtwo_{\substack{b', c \in \Z[\omega] \\ b', c \equiv 1 \pmod 3 \\ b' \mid b^\infty,\ (b, c) = 1}} \frac{\varphi(b) \varphi(c) \ddot{F}_{\lambda^g b^2 b'^3 c^3} (0)}{N(b')^{3/2} N(c)^{5/2}} \sum_{\substack{\ell \in \Z[\omega] \\ \ell \equiv 1 \pmod 3 \\ N(\ell) \leq Y}} \frac{\mu(\ell) \mathbf{1}_{bc}(\ell)}{N(\ell)^2}. \label{main_term_first_mom}
\end{align}

Recall that $F_{\mathfrak{n}}(r) = F(r) \Phi_1 \big(\frac{N(\mathfrak{n})}{\sqrt{3Xr}} \big)$, where $\Phi_1$ is defined in \eqref{def-Phi} and  $F$ has support in $(1, 2)$. Since $J_0(0)=1$, replacing in \eqref{Vddot}, we have 
\begin{align}
    \ddot{F}_{\mathfrak{n}}(0)  = \int_1^{\sqrt{2}}  r F(r^2) \Phi_1 \Big(\frac{N(\mathfrak{n})}{\sqrt{3X} r} \Big) dr =  \frac{1}{2\pi i} \int_{2 -i\infty}^{2 + i\infty} \Big(\frac{2 \pi N(\mathfrak{n})}{\sqrt{3X}} \Big)^{-w} \frac{\Gamma (\frac{1}{2}+w )}{\Gamma (\frac{1}{2})} \check{F}\Big(\frac{w}{2}\Big) \frac{dw}{2w}, \label{F_dot_dot_zero_first_mom}
\end{align}
for $\check{F}(w) := \int_0^\infty t^w F(t) dt = 2 \int_1^{\sqrt{2}} r^{2w+1} F(r^2) dr$. In particular, by the rapid decay of $\Phi_1$ as in \eqref{Phi_bound}, we have (using $0 \leq F(t) \leq 1$ for the uniformity of the implied constant) the coarse bound
\begin{align}
    \ddot{F}_{\mathfrak{n}}(0) \ll_{A} \Big(1 +  \frac{N(\mathfrak{n})}{\sqrt{X}}\Big)^{-A}. \label{F_dot_dot_zero_bound_first_mom}
\end{align}

We apply \eqref{ell_sum_eval} to evaluate the sum over $\ell$. By \eqref{F_dot_dot_zero_first_mom}, the contribution to \eqref{main_term_first_mom} of the error term $O(\frac{1}{Y})$ is
\begin{align*}
    \ll \frac{X}{Y N(b)} \sum_{\substack{m \in \Z[\omega] \\ m\equiv 1 \pmod 3}} \frac{1}{ N(m)^{3/2}} \ll \frac{X}{Y N(b)}.
\end{align*}
Inserting \eqref{F_dot_dot_zero_first_mom} into \eqref{main_term_first_mom} then gives
\begin{align} 
    \mathcal{M}(b) = \frac{2 \pi X}{3^{9/2} \zeta_\lambda(2)} \frac{1}{2\pi i} \int_{2 - i \infty}^{2 + i \infty} & \Big(\frac{2 \pi}{\sqrt{3X}} \Big)^{-w} \frac{\Gamma (\frac{1}{2}+w )}{\Gamma (\frac{1}{2} )} \check{F}\Big(\frac{w}{2}\Big) \mathcal{G}_{b}(w) \frac{dw}{w} + O\Big(\frac{X}{Y N(b)}\Big), \label{after_mellin_first_mom}
\end{align}
where
\begin{align}
    & \mathcal{G}_{b}(w) := \frac{\varphi(b)}{N(b)^{2+2w}} \prod_{\substack{\pi \text{ prime} \\ \pi \equiv 1 \pmod 3 \\ \pi \mid b}} \Big(1 - \frac{1}{N(\pi)^2}\Big)^{-1} \Bigg( \sum_{g \in \mathbb{Z}_{\geq 0}} \frac{1}{3^{g(1/2+w)}} \Bigg) \nonumber \\
    & \times \Bigg( \sum_{\substack{b' \in \Z[\omega] \\ b' \equiv 1 \pmod 3 \\ b' \mid b^\infty}} \frac{1}{N(b')^{3/2+3w}} \Bigg) \sum_{\substack{c \in \Z[\omega] \\ c \equiv 1 \pmod 3 \\ (b, c) = 1}} \frac{\varphi(c)}{N(c)^{5/2+3w}} \prod_{\substack{\pi \text{ prime} \\ \pi \equiv 1 \pmod 3 \\ \pi \mid c}} \Big(1 - \frac{1}{N(\pi)^2}\Big)^{-1}. \label{generating-series_first_mom}
\end{align}

We evaluate the sum over $c$ to get
\begin{align*}
    & \sum_{\substack{c \in \Z[\omega] \\ c \equiv 1 \pmod 3 \\ (b, c) = 1}} \frac{\varphi(c)}{N(c)^{5/2+3w}} \prod_{\substack{\pi \text{ prime} \\ \pi \equiv 1 \pmod 3 \\ \pi \mid c}} \Big(1 - \frac{1}{N(\pi)^2}\Big)^{-1}  \\
    = &\ \prod_{\substack{\pi \text{ prime} \\ \pi \equiv 1 \pmod 3 \\ \pi \nmid b}} K(\pi, w) = \mathcal{K}(w) \prod_{\substack{\pi \text{ prime} \\ \pi \equiv 1 \pmod 3 \\ \pi \mid b}} K(\pi, w)^{-1}
\end{align*}
for
\begin{align}
    K(\pi, w) := &\ 1 + \Big(1 - \frac{1}{N(\pi)^2}\Big)^{-1} \sum_{k=1}^\infty \frac{\varphi(\pi^{k})}{N(\pi^{k})} \frac{1}{N(\pi)^{k(3/2 + 3w)}} = 1 + \frac{\big(1 + \frac{1}{N(\pi)}\big)^{-1}}{N(\pi)^{3/2 + 3w}-1} \nonumber \\
    = &\ \Big( 1 - \frac{1}{N(\pi)^{3/2+3w}} \Big)^{-1} \Big(1 - \frac{1}{N(\pi)^{3/2+3w}(N(\pi)+1)}\Big) \label{K_prod}
\end{align}
and
\begin{equation}
    \mathcal{K}(w) := \prod_{\substack{\pi \text{ prime} \\ \pi \equiv 1 \pmod 3}} K(\pi, w) =: \zeta_\lambda(3/2+3w) \cdot \mathcal{Q}(w). \label{mathcal_K_def}
\end{equation}

Further evaluating the sums over $g$ and $b'$ in \eqref{generating-series_first_mom}, we conclude that
\begin{equation}
    \mathcal{G}_b(w) = \Big(\frac{3^{1/2+w}}{3^{1/2+w}-1}\Big) \cdot \zeta_\lambda(3/2+3w) \cdot \mathcal{Q}(w) \cdot \mathcal{H}_b(w), \label{G_first_mom_def}
\end{equation}
where recalling that $b$ is squarefree and assuming from now on that $\Re(w) \geq -\frac{1}{2}+\varepsilon$, we have
\begin{align}
    \mathcal{H}_b(w) := &\ \frac{1}{N(b)^{1+2w}}\frac{\varphi(b)}{N(b)} \prod_{\substack{\pi \text{ prime} \\ \pi \equiv 1 \pmod 3 \\ \pi \mid b}} K(\pi, w)^{-1} \Big(1-\frac{1}{N(\pi)^2}\Big)^{-1} \frac{N(\pi)^{3/2+3w}}{N(\pi)^{3/2+3w}-1} \nonumber \\
    = &\ N(b)^{1/2+w} \prod_{\substack{\pi \text{ prime} \\ \pi \equiv 1 \pmod 3 \\ \pi \mid b}} \Big(1+\frac{1}{N(\pi)} + \frac{1}{N(\pi)^{3/2+3w}-1} \Big)^{-1} \frac{1}{N(\pi)^{3/2+3w}-1} \label{H_Euler_prod_exp_first_mom} \\
    \ll &\ N(b)^{-1-2 \Re(w)+\varepsilon}. \label{H_uniform_bound_first_mom}
\end{align}

Note by \eqref{K_prod} and \eqref{mathcal_K_def} that $\mathcal{Q}(w)$ is holomorphic and uniformly bounded if $\Re(w) \geq -\frac{1}{2}+\varepsilon$. Thus $\mathcal{G}_b$ is meromorphic in the same region, with a simple pole at $w=-\frac{1}{6}$ (since one can check that $\mathcal{Q}(-1/6)\neq 0$), and no other poles.

We shift the line of integration in \eqref{after_mellin_first_mom} to $\Re(w) = -\frac{1}{6} + \varepsilon$, picking up the simple pole of the integrand at $w=0$, and conclude -- using the convexity bound for $\zeta_{\lambda}(3/2+3w)$, the absolute convergence of $\mathcal{Q}(w)$, and \eqref{H_uniform_bound_first_mom} -- that the remaining integral is
\begin{equation*}
    \ll \frac{X^{11/12+\varepsilon}}{N(b)^{2/3}} \int_{-1/6+\varepsilon -i\infty}^{-1/6 + \varepsilon +i\infty} |\Gamma(1/2 + w)| \Big|\check{F}\Big(\frac{w}{2}\Big)\Big| |w|^{100}  |dw| \ll \frac{X^{11/12+\varepsilon}}{N(b)^{2/3}}.
\end{equation*}

\begin{remark}
    We could shift all the way to $\Re(w) = -\frac{1}{2}+\varepsilon$, collecting the simple pole of the integrand at $w=-\frac{1}{6}$ to show that the remaining integral is equal to
    \begin{equation*}
        C_b X^{11/12} + O_{b, \varepsilon}(X^{3/4+\varepsilon})
    \end{equation*}
    for some explicit constant $C_{b}$. We also foreshadow that our error term estimates below will contain a term of size $O_{b, \varepsilon}(X^{11/12+\varepsilon})$, and this a feature of using a balanced approximate functional equation.
 Using an unbalanced approximate functional equation,  Hamdar's work \cite{Ham} shows this particular second order main term is illusory. A similar cancellation was observed in \cite{DFL2}, and can be used to compute the first moment of the thin family $\mathcal{F}'_3$ at $s=1/3$ (see \cite{DM23} over function fields).
 Improving the error term in the first moment does not directly improve our final results about non-vanishing, so we refrain from doing so for simplicity.
   \end{remark}

Thus \eqref{after_mellin_first_mom} implies
\begin{align} 
    \mathcal{M}(b) = \frac{2 \pi X}{3^{9/2} \zeta_\lambda(2)} \check{F}(0) \mathcal{G}_{b}(0) + O\Big( \frac{X}{Y N(b)} \Big) + O_\varepsilon \Big(\frac{X^{11/12+\varepsilon}}{N(b)^{2/3}} \Big). \nonumber
\end{align}
Observe by \eqref{K_prod}, \eqref{mathcal_K_def}, \eqref{G_first_mom_def}, and \eqref{H_Euler_prod_exp_first_mom} that
\begin{equation*}
    \mathcal{G}_{b}(0) = \frac{\sqrt{3}}{\sqrt{3} - 1} \zeta_{\lambda}(3/2) \mathcal{Q}(0) \frac{r(b)}{N(b)},
\end{equation*}
where 
\begin{align*}
    \zeta_{\lambda}(3/2) \mathcal{Q}(0) = \prod_{\substack{\pi \text{ prime} \\ \pi \equiv 1 \pmod 3 \\ q := N(\pi)}} \Big( 1 + \frac{q}{(q+1) (q^{3/2}-1)} \Big) > 0
\end{align*}
and $r$ is the multiplicative function given, for $\pi$ prime and $k\in \Z_{\geq 1}$ (where we still denote $q := N(\pi)$) by 
\begin{equation*}
    r(\pi^k) := \frac{q^{5/2}}{q^{5/2} + q^{3/2} - 1} = 1 + O\Big(\frac{1}{N(\pi)}\Big).
\end{equation*}

Therefore we obtain the main term claimed in \cref{SMestimate1}.


\subsection{The error term $\mathcal{R}(b)$}

It remains to show the bound given in \cref{SMestimate1} for the error terms, which will follow from
\begin{align}\label{real_R_bound_first_mom}
    \mathcal{R}(b) \ll_{F, \nu, \varepsilon} X^\varepsilon \Big( X^{3/4} N(b)^{1/2} + \frac{X^{5/6}}{N(b)^{1/2}} + \frac{X^{11/12}}{N(b)^{2/3}} \Big)
\end{align}
for squarefree $b \equiv 1 \pmod 3$ satisfying $N(b) Y^2 \leq X^{1/2-\nu}$ for a fixed $\nu > 0$. Our aim for the rest of this section is to prove \eqref{real_R_bound_first_mom}.

We observe for future reference that the proof of \cref{F_dot_dot_bound_lemma} applies directly to give, for any $\mathfrak{n} \unlhd \Z[\omega]$, $u\in \C$, and $A \in \Z_{\geq 0}$,
\begin{equation}\label{f_dot_dot_first_mom_bound}
    \ddot{F}_{\mathfrak{n}}(u) \ll_{F, A} \Big(1 + |u| + \frac{N(\mathfrak{n})}{\sqrt{X}}\Big)^{-A}.
\end{equation}
Furthermore, for $u\neq 0$, the exact argument via the Mellin--Barnes integral leading to \eqref{F_transform_bulk} also gives
\begin{equation}
    \ddot{F}_{\mathfrak{n}}(u) = \int_{\varepsilon -iX^\varepsilon}^{\varepsilon + iX^\varepsilon}\int_1^{\sqrt{2}} \int_{-\varepsilon -iX^{\varepsilon}}^{-\varepsilon + iX^{\varepsilon}} \mathcal{G}_1(w, r, s)  \frac{|u|^{2s}}{N(\mathfrak{n})^{w}} ds\: dr\: dw + O_{F, \varepsilon}((1 + |u|^{-2\varepsilon})X^{-2000}), \label{F_dot_dot_first_mom_integral}
\end{equation}
where this time we set $G_w(y) := y^{w/2} F(y)$. Note we still have the uniform bound
\begin{equation*}
    G_w^{(j)}(y) \ll_{F, j} (1+|w|)^j,
\end{equation*}
and for $\Re(w) = \varepsilon$  and $\Re(s) = -\varepsilon$ we have
\begin{align}
    \mathcal{G}_1(w, r, s) :=&\ \frac{(-1)^{j_0}}{(2\pi i)^2} \Big(\frac{4\pi^2}{3X}\Big)^{-w/2} \frac{\Gamma (\frac{1}{2}+w )}{\Gamma (\frac{1}{2} )} \frac{G_w^{({j_0})}(r^2) r^{2{j_0}+1}}{w} \frac{\Gamma(-s)}{\Gamma({j_0}+s+1)} \Big(\frac{2\pi r}{9\sqrt{3}}\Big)^{2s} \nonumber \\
    \ll_{F, {j_0}} &\ \frac{X^\varepsilon \Gamma(\frac{1}{2}+w) (1+|w|)^{j_0} \Gamma(-s)}{({j_0}+s)({j_0}-1+s) \cdots (1+s) \Gamma(1+s)} \ll_{j_0} \frac{X^\varepsilon}{(1+ |\Im(s)|)^{j_0}},\label{G_first_mom_bound}
\end{align}
for a certain ${j_0}$ which is chosen to be sufficiently large in terms of $\varepsilon$ (but fixed).


\subsubsection{The standard sum $|\mathcal{S}'(b)|$}

We now use the assumption $N(b) Y^2 \leq X^{1/2-\nu}$ to quickly verify that $|\mathcal{S}'(b)|$ is negligible. Using cubic characters as in \eqref{congruence_with_chars} to detect the congruence condition on the sum over $n$ in \eqref{S_first_mom_exp}, we obtain
\begin{align}
    \mathcal{S}'(b) \ll \frac{X}{N(b)} \sum_{g=0}^\infty \frac{1}{3^{g/2}} \sum_{\substack{\eta \mid 3}} \sum_{0 \neq k\in \Z[\omega]} \sum_{\substack{\ell \in \Z[\omega] \\ \ell \equiv 1 \pmod 3 \\ N(\ell) \leq Y}} \frac{\mu^2(\ell)}{N(\ell)^2}  \Bigg| \sum_{\substack{n \in \Z[\omega] \\ n\equiv 1 \pmod 3}} \frac{\overline{\chi_{n}(\eta \ell)} g_3(k, bn)}{N(n)^{3/2}} \ddot{F}_{\lambda^g n} \Big( \frac{k \sqrt{X}}{bn \ell^2} \Big) \Bigg|.\label{S_prime_bound}
\end{align}

We add a partition of unity $H\big(\frac{N(n)}{R}\big)$ as in \eqref{partition_of_unity} to the sum over $n$ in \eqref{S_prime_bound}, and by \eqref{f_dot_dot_first_mom_bound} we can restrict to dyadic $R\ll X^{1/2+\varepsilon}$ up to a negligible error term $O_{F, \varepsilon}(X^{-1000})$. Denoting by $\mathcal{S}'^L_R(b)$ the right side of \eqref{S_prime_bound} with $N(n)$ localized around $R$ (using $H$) and $N(\ell)$ localized around $L \leq Y$ (using a dyadic decomposition), we conclude that
\begin{equation*}
    \mathcal{S}'(b) \ll \sum_{\substack{i \in \Z \\ R = 2^{i} \gg 1 \\ R \ll X^{1/2+\varepsilon}}} \sum_{\substack{j \in \Z \\ L = 2^{j} \gg 1 \\ L \leq Y}} \mathcal{S}_{R}^{\prime L}(b) + O_{F, \varepsilon}(X^{-1000}).
\end{equation*}

By \eqref{f_dot_dot_first_mom_bound} we can also truncate the sum over $k$ in $\mathcal{S}_{R}^{\prime L}(b)$ at $N(k) \ll \frac{N(b) R L^2}{X^{1-\varepsilon}}$, up to error $O_{F, \varepsilon}(X^{-1000})$. Then opening up $\ddot{F}_{\lambda^g n}$ using \eqref{F_dot_dot_first_mom_integral} and applying the triangle inequality combined with the bound \eqref{G_first_mom_bound} yields
\begin{align*}
    \mathcal{S}_{R}^{\prime L}(b) \ll_F &\ \frac{X^{1+\varepsilon}}{N(b) L^2} \sup_{\substack{t \in \R \\ |t| \ll X^\varepsilon}} \sum_{\substack{\eta \mid 3}} \sum_{\substack{0 \neq k\in \Z[\omega] \\ N(k) \ll \frac{N(b) R L^2}{X^{1-\varepsilon}}}} \sum_{\substack{\ell \in \Z[\omega] \\ \ell \equiv 1 \pmod 3 \\ N(\ell) \sim L}} \mu^2(\ell) \nonumber \\
    & \times  \Bigg| \sum_{\substack{n \in \Z[\omega] \\ n\equiv 1 \pmod 3}} \frac{\overline{\chi_{n}(\eta \ell)} g_3(k, bn)}{N(n)^{3/2 + it}} H \Big( \frac{N(n)}{R} \Big) \Bigg| + X^{-1000}.
\end{align*}

Since we are assuming $N(b) Y^2 \leq X^{1/2-\nu}$ for a fixed $\nu > 0$, from $R \ll X^{1/2+\varepsilon}$ we obtain $\frac{N(b) RL^2}{X^{1-\varepsilon}} \ll \frac{N(b) Y^2}{X^{1/2-2\varepsilon}} \leq X^{-\nu + 2\varepsilon}$, hence the sum over $k \neq 0$ is empty (for $X$ sufficiently large in terms of $\nu$). We conclude that $\mathcal{S}_{R}^{\prime L}(b) \ll_{F, \nu} X^{-1000}$, and therefore
\begin{equation}
    \mathcal{S}^{\prime}(b) \ll_{F, \nu} X^{-500}. \label{S_prime_final_bound}
\end{equation}


\subsubsection{The dual sum $|\widetilde{\mathcal{S}}(b)|$}

Using cubic characters as in \eqref{congruence_with_chars} to detect the congruence condition on the sum over $m$ in \eqref{first_mom_dual}, and adding a partition of unity $H\big(\frac{N(n)}{R}\big)$ to the sum over $n$ via \eqref{partition_of_unity}, we obtain 
\begin{align}
    \widetilde{\mathcal{S}}(b) \ll &\ \sum_{\substack{i \in \Z \\ R = 2^i \gg 1}} \sum_{g=0}^\infty \frac{1}{3^{g/2}} \sum_{\eta \mid 3} \Bigg| \sum_{\substack{n \in \Z[\omega] \\ n\equiv 1 \pmod 3}} \frac{H \big(\frac{N(n)}{R}\big)}{N(n)^{1/2}}\sum_{\substack{ m \in \Z[\omega]  \\ m \equiv 1 \pmod{3}}}  \widetilde{g}_3(m) \overline{\chi_m(\eta b^2 n)} F_{\lambda^g n} \Big( \frac{N(m)}{X} \Big) \Bigg|. \label{dual_sum_dyadic_decomp}
\end{align}
For each $R = 2^i$, denote the corresponding summand (with $N(n)$ localized around $R$) in \eqref{dual_sum_dyadic_decomp} by $\widetilde{\mathcal{S}}_R(b)$. Recalling that $F_{\lambda^g n} \big( \frac{N(m)}{X} \big) = F \big( \frac{N(m)}{X} \big) \Phi_1 \Big( \frac{3^g N(n)}{\sqrt{3 N(m)}} \Big)$, it follows from \eqref{Phi_bound} that the contribution of $R \gg X^{1/2+\varepsilon}$ to \eqref{dual_sum_dyadic_decomp} is $O(X^{-1000})$. Thus
\begin{equation*}
    \widetilde{\mathcal{S}}(b) \ll \sum_{\substack{i \in \Z \\ R = 2^{i} \gg 1 \\ R \ll X^{1/2+\varepsilon}}} \widetilde{\mathcal{S}}_{R}(b)+ O_{F, \varepsilon}(X^{-1000}).
\end{equation*}

For the remaining dyadic ranges $R \ll X^{1/2+\varepsilon}$, we expand $\Phi_1$ using \eqref{def-Phi} and then shift the line of integration to $\Re(w) = \varepsilon$ to conclude that the corresponding sum over $n$ in \eqref{dual_sum_dyadic_decomp} is
\begin{align}
    & = \frac{1}{2\pi i} \int_{\varepsilon - i \infty}^{\varepsilon + i \infty}  \sum_{\substack{n \in \Z[\omega] \\ n\equiv 1 \pmod 3}} \frac{H \big(\frac{N(n)}{R}\big)}{N(n)^{1/2+w}} \sum_{\substack{ m \in \Z[\omega]  \\ m \equiv 1 \pmod{3}}}  \frac{\widetilde{g}_3(m) \overline{\chi_m(\eta b^2 n)}}{N(m)^{-w/2}} F \Big( \frac{N(m)}{X} \Big) \nonumber \\
    &  \qquad \qquad \qquad \qquad \qquad \qquad \qquad \qquad \qquad \qquad \qquad \qquad \times (2\pi \cdot 3^{g-1/2})^{-w}   \frac{\Gamma(1/2+w)}{\Gamma(1/2)}  \frac{dw}{w} \nonumber \\
    & \ll \int_{-X^\varepsilon}^{X^\varepsilon} \Bigg| \sum_{\substack{n \in \Z[\omega] \\ n\equiv 1 \pmod 3}} \frac{H \big(\frac{N(n)}{R}\big)}{N(n)^{1/2+\varepsilon + it}} \sum_{\substack{ m \in \Z[\omega]  \\ m \equiv 1 \pmod{3}}}  \frac{g_3(m) \overline{\chi_m(\eta b^2 n)}}{N(m)^{1/2-\varepsilon/2 - it/2}} F \Big( \frac{N(m)}{X} \Big) \Bigg|  |dt| + X^{-1000}, \label{dual_ready_for_eval}
\end{align}
where the truncation in the second step is justified by Stirling's formula. Writing (uniquely) 
\begin{equation}\label{b_n_decomp}
    b^2 n = \alpha \beta^2 \gamma^3 \delta^3
\end{equation}
with $\alpha, \beta, \gamma, \delta \equiv 1 \pmod 3$, $\mu^2(\alpha \beta \gamma)=1$, and $\delta \mid (\alpha \beta \gamma)^\infty$, we can apply \cref{truncated_Gauss_sums_lemma} to obtain
\begin{align}
    \sum_{\substack{ m \in \Z[\omega]  \\ m \equiv 1 \pmod{3}}}  \frac{g_3(m) \overline{\chi_m(\eta b^2 n)}}{N(m)^{1/2-\varepsilon/2 - it/2}} F \Big( \frac{N(m)}{X} \Big) & = \sum_{\substack{ m \in \Z[\omega]  \\ m \equiv 1 \pmod{3}}}  \frac{g_3(m) \overline{\chi_m(\eta \alpha \beta^2 \gamma^3)}}{N(m)^{1/2-\varepsilon/2 - it/2}} F \Big( \frac{N(m)}{X} \Big) \nonumber \\
    &= \mathcal{P}_{\text{prim}} + O(\mathcal{I}_{\text{prim}}),  \nonumber
\end{align}
where 
\begin{align}
    \mathcal{P}_{\text{prim}} := \bbone_{\beta = 1}\cdot C_\eta \cdot \widetilde{F}(5/6 + \varepsilon/2 + it/2) X^{5/6 + \varepsilon/2 + it/2} \frac{\overline{\widetilde{g}_3(\eta, \alpha)}}{N(\alpha)^{1/6}} \frac{\Delta_{\alpha\gamma}(1)}{\Delta_{\alpha\gamma}(4/3)} \label{P_prim_def}
\end{align}
for a constant $C_\eta$ depending only on $\eta$, and for any $A \in \Z_{\geq 0}$ we have
\begin{align}
    \mathcal{I}_{\text{prim}} &\ll_{A, F, \varepsilon} X^{1/2 + \varepsilon} \mathop{\sum \sum}_{\substack{d, e \in \Z[\omega] \\ d, e\equiv 1\pmod{3} \\ d\mid \alpha ,\ e\mid \gamma}} \frac{1}{N(de)^{1/2+\varepsilon}} \int_{-\infty}^\infty \frac{|\psi(\frac{\eta \alpha \beta^2 e}{d}, 1+\varepsilon -it/2 + iy)|}{(1+|y|)^A} dy \nonumber \\
    & \ll_A X^{1/2 + \varepsilon} \sum_{\substack{f \in \Z[\omega] \\ f \equiv 1\pmod{3} \\ f \mid b^2 n}} \int_{-\infty}^\infty \frac{|\psi(\eta f, 1+\varepsilon + iy)|}{(1+|y|)^A} dy, \label{I_prim_bound}
\end{align}
since $|t| \leq X^\varepsilon$. 


\subsubsection{The primary integral term $\mathcal{I}_{\text{prim}}$}

Inserting the bound \eqref{I_prim_bound} into \eqref{dual_ready_for_eval} and \eqref{dual_sum_dyadic_decomp}, the contribution of the integral term $\mathcal{I}_{\text{prim}}$ to $\widetilde{\mathcal{S}}_R(b)$ is
\begin{align}
    & \ll_{F} X^{1/2+\varepsilon} \sup_{\eta\mid 3} \sup_{\substack{y \in \R}} \sum_{\substack{n \in \Z[\omega] \\ n\equiv 1 \pmod 3}} \frac{H \big(\frac{N(n)}{R}\big)}{N(n)^{1/2+\varepsilon}} \sum_{\substack{f \in \Z[\omega] \\ f \equiv 1\pmod{3} \\ f \mid b^2 n}} \frac{|\psi(\eta f, 1+\varepsilon + iy)|}{(1+|y|)^{100}} \nonumber \\
    & \ll X^{1/2+\varepsilon} \sup_{\eta\mid 3} \sup_{\substack{y \in \R}} \sum_{\substack{f \in \Z[\omega] \\ f \equiv 1\pmod{3}}} \frac{|\psi(\eta f, 1+\varepsilon + iy)|}{(1+|y|)^{100}} \cdot \frac{R^{1/2}}{N\big(\frac{f}{(f, b^2)}\big)} \cdot \bbone_{N(\frac{f}{(f, b^2)}) \ll R}. \nonumber
\end{align}
Denote $r := (f, b^2)$ and $h := \frac{f}{(f, b^2)}$, so $r \mid b^2$ and $N(h) \ll R$. By positivity and  Cauchy--Schwarz, the sum over $f$ in the display above is
\begin{align*}
    & \ll R^{1/2}  \sum_{\substack{r \in \Z[\omega] \\ r \equiv 1 \pmod 3 \\ r \mid b^2}} \sum_{\substack{h \in \Z[\omega] \\ h \equiv 1\pmod{3} \\ N(h) \ll R}} \frac{|\psi(\eta r h, 1+\varepsilon + iy)|}{N(h) (1+|y|)^{100}} \\
    &\ll R^{1/2} \sum_{\substack{r \in \Z[\omega] \\ r \equiv 1 \pmod 3 \\ r \mid b^2}} \Bigg(\sum_{\substack{h \in \Z[\omega] \\ h \equiv 1\pmod{3} \\ N(h) \ll R}} \frac{|\psi(\eta r h, 1+\varepsilon + iy)|^2}{N(h) (1+|y|)^{200}} \Bigg)^{1/2} \Bigg( \sum_{\substack{h \in \Z[\omega] \\ h \equiv 1\pmod{3} \\ N(h) \ll R}} \frac{1}{N(h)} \Bigg)^{1/2} \\
    &\ll R^{1/2+\varepsilon} \sum_{\substack{r \in \Z[\omega] \\ r \equiv 1 \pmod 3 \\ r \mid b^2}} \Big(\frac{N(r)^{1/2} R^\varepsilon}{(1+|y|)^{100}}\Big)^{1/2} \ll \frac{X^\varepsilon R^{1/2} N(b)^{1/2}}{(1+|y|)^{50}},
\end{align*}
where we applied \cref{second_moment_lindelof_lemma}, summation by parts, and a divisor bound in the last two steps. Thus the contribution of $\mathcal{I}_{\text{prim}}$ to $\widetilde{\mathcal{S}}(b)$ is
\begin{equation}
    \ll \sum_{\substack{i \in \Z \\ 1 \ll R = 2^i \ll X^{1/2+\varepsilon}}} X^{1/2+\varepsilon} R^{1/2} N(b)^{1/2} \ll X^{3/4+\varepsilon} N(b)^{1/2}. \label{I_prim_final_bound}
\end{equation}


\subsubsection{The primary polar term $\mathcal{P}_{\text{prim}}$}

We now evaluate the contribution of $\mathcal{P}_{\text{prim}}$, given in \eqref{P_prim_def}, to $\widetilde{\mathcal{S}}(b)$. From \eqref{b_n_decomp} and the fact that $b\equiv 1 \pmod 3$ is squarefree, we see that the condition $\beta = 1$ is equivalent to $n = b \cdot ac^3$ for some $a, c\equiv 1 \pmod 3$ with $\mu^2(a)=1$. Since such a decomposition is unique, by the triangle inequality and rapid decay of the Mellin transform $\widetilde{F}$, the contribution of $\mathcal{P}_{\text{prim}}$ to $\widetilde{\mathcal{S}}_R(b)$ is
\begin{align}
    & \ll_{F} X^{5/6+\varepsilon} \sup_{\eta\mid 3} \sup_{\substack{t \in \R \\ |t|\leq X^\varepsilon}} \Bigg| \sumtwo_{\substack{a, c \in \Z[\omega] \\ a, c \equiv 1 \pmod 3}} \frac{\mu^2(a) H \big(\frac{N(abc^3)}{R}\big)}{N(abc^3)^{1/2+\varepsilon + it}} \frac{\overline{\widetilde{g}_3(\eta, \alpha)}}{N(\alpha)^{1/6}} \frac{\Delta_{\alpha\gamma}(1)}{\Delta_{\alpha\gamma}(4/3)} \Bigg| \label{P_prim_initial} 
\end{align}
plus a negligible term $O(X^{-1000})$. Observe that $\alpha = a$, $\gamma \delta = bc$, and since $\delta \mid (\alpha \gamma)^\infty$,
\begin{align*}
    \Delta_{\alpha \gamma}(s) := \prod_{\substack{\pi \text{ prime} \\ \pi\equiv 1 \pmod{3} \\ \pi \mid \alpha \gamma}} (1 - N(\pi)^{2 - 3s} ) = \Delta_{\alpha \gamma \delta}(s) = \Delta_{a b c}(s).
\end{align*}

We now express $\frac{\Delta_{abc}(1)}{\Delta_{abc}(4/3)}$ via M{\"o}bius inversion. For $v \equiv 1 \pmod 3$, denote 
\begin{align*}
    \Omega(v) :=  \frac{\Delta_{v}(1)}{\Delta_{v}(4/3)} = \prod_{\substack{\pi \text{ prime} \\ \pi\equiv 1 \pmod{3} \\ \pi \mid v}} (1 + N(\pi)^{-1} )^{-1},
\end{align*}
so $\Omega$ is a multiplicative function satisfying $0 \leq \Omega(v) \leq 1$. Consider the multiplicative function $\Theta$ supported on $u \equiv 1 \pmod 3$ and given by
\begin{equation*}
    \Theta(u) := \sum_{\substack{v \in \Z[\omega] \\ v\equiv 1 \pmod 3 \\ v \mid u}} \mu\Big(\frac{u}{v}\Big) \Omega(v).
\end{equation*}
For $\pi\equiv 1 \pmod 3$ prime and $k \in \Z_{\geq 1}$, we have
\begin{align*}
    \Theta(\pi^k) = \Omega(\pi^k) - \Omega(\pi^{k-1}) = 
    \begin{cases}
        0 &\text{ if } k\geq 2, \\
        \frac{-1}{N(\pi)+1} &\text{ if } k=1,
    \end{cases}
\end{align*}
so $\Theta$ is supported on squarefree $u \equiv 1 \pmod 3$ and satisfies $|\Theta(u)| \leq \frac{1}{N(u)}$. By M{\"o}bius inversion, for any $v\equiv 1 \pmod{3}$ we have
\begin{equation*}
    \Omega(v) = \sum_{\substack{u \in \Z[\omega] \\ u\equiv 1 \pmod 3 \\ u \mid v}} \Theta(u).
\end{equation*}
Inserting this into \eqref{P_prim_initial} and applying the triangle inequality to the sums over $c$ and $u$ we conclude that the quantity inside absolute values in that display is 
\begin{align}
    \ll \frac{X^\varepsilon}{N(b)^{1/2}}\sum_{\substack{c \in \Z[\omega] \\ c \equiv 1 \pmod 3}} \frac{1}{N(c)^{3/2}} \sum_{\substack{u \in \Z[\omega] \\ u\equiv 1 \pmod 3}} \frac{\mu^2(u)}{N(u)} \Bigg| \sum_{\substack{a \in \Z[\omega] \\ a \equiv 1 \pmod 3 \\ u \mid abc}} \frac{\mu^2(a) H \big(\frac{N(abc^3)}{R}\big)}{N(a)^{1/2+\varepsilon - it}} \frac{\overline{\chi_{a}(\eta)} g_3(a)}{N(a)^{2/3}} \Bigg|, \label{p_prim_after_mobius}
\end{align}
where we applied complex conjugation to the sum over $a$ and removed the Gauss sum normalization. 

Write $u = d r$ for $d := (u, bc)$ and $r := \frac{u}{(u, bc)}$, so that $d \mid bc$, $\mu^2(r)=1$, and the condition $u \mid abc$ is equivalent to $r \mid a$. By positivity, we can complete the sums to conclude that \eqref{p_prim_after_mobius} is
\begin{align}
    & \ll \frac{X^\varepsilon}{N(b)^{1/2}}\sum_{\substack{c \in \Z[\omega] \\ c \equiv 1 \pmod 3}} \frac{1}{N(c)^{3/2}} \sumtwo_{\substack{d, r \in \Z[\omega] \\ d, r\equiv 1 \pmod 3 \\ d \mid bc}} \frac{\mu^2(r)}{N(r)} \Bigg| \sum_{\substack{a \in \Z[\omega] \\ a \equiv 1 \pmod 3 \\ r \mid a}} \frac{\mu^2(a) H \big(\frac{N(abc^3)}{R}\big)}{N(a)^{1/2+\varepsilon - it}} \frac{\overline{\chi_{a}(\eta)} g_3(a)}{N(a)^{2/3}} \Bigg| \nonumber\\
    & \ll \frac{X^\varepsilon}{N(b)^{1/2}}\sum_{\substack{c \in \Z[\omega] \\ c \equiv 1 \pmod 3 \\ N(c) \ll R^{1/3}}} \frac{1}{N(c)^{3/2}} \sum_{\substack{r \in \Z[\omega] \\ r\equiv 1 \pmod 3 \\ N(r) \ll \frac{R}{N(bc^3)}}} \frac{\mu^2(r)}{N(r)^{5/3}} \Bigg| \sum_{\substack{h \in \Z[\omega] \\ h \equiv 1 \pmod 3}} \frac{\overline{\chi_{h}(\eta r)} g_3(h)}{N(h)^{7/6 + \varepsilon - it}} H \Big(\frac{N(r h b c^3)}{R}\Big) \Bigg|, \label{P_prim_ready_for_Pat}
\end{align}
where we wrote $a = rh$ with $\mu^2(rh)=1$, and used the relations $g_3(rh) = \overline{\chi_{h}(r)} g_3(r) g_3(h)$ and $|g_3(r)| = \mu^2(r) N(r)^{1/2}$. 

Applying \cref{truncated_Gauss_sums_lemma}, since $\mu^2(r) = 1$ we obtain
\begin{align*}
     \sum_{\substack{h \in \Z[\omega] \\ h \equiv 1 \pmod 3}} \frac{\overline{\chi_{h}(\eta r)} g_3(h)}{N(h)^{7/6 + \varepsilon - it}} H \Big(\frac{N(r h b c^3)}{R}\Big) = \mathcal{P}_{\text{sec}} + O(\mathcal{I}_{\text{sec}}),
\end{align*}
where as before 
\begin{align}
    \mathcal{P}_{\text{sec}} := C_\eta \cdot \widetilde{H}(1/6 - \varepsilon + it) \Big(\frac{R}{N(r b c^3)}\Big)^{1/6 - \varepsilon + it} \frac{\overline{\widetilde{g}_3(\eta, r)}}{N(r)^{1/6}} \frac{\Delta_{r}(1)}{\Delta_{r}(4/3)} \ll_\eta \frac{X^\varepsilon R^{1/6}}{N(b)^{1/6} N(r)^{1/3} N(c)^{1/2}} \label{P_sec_bound}
\end{align}
and 
\begin{align}
    \mathcal{I}_{\text{sec}} \ll_{F, \varepsilon} X^\varepsilon \Big(\frac{R}{N(b r c^3)}\Big)^{-1/6} \sum_{\substack{f \in \Z[\omega] \\ f \equiv 1 \pmod{3} \\ f \mid r}} \int_{-\infty}^\infty \frac{|\psi(\eta f, 1+\varepsilon + iy)|}{(1+|y|)^{100}} dy, \label{I_sec_bound}
\end{align}
since $|t| \leq X^\varepsilon$ and $N(r) \ll \frac{R}{N(bc^3)}$. 


\subsubsection{The secondary integral term $\mathcal{I}_{\text{sec}}$}

By \eqref{I_sec_bound}, the contribution of $\mathcal{I}_{\text{sec}}$ to \eqref{P_prim_ready_for_Pat} is 
\begin{align*}
    &\ll_F \frac{X^\varepsilon}{N(b)^{1/3} R^{1/6}} \sup_{y\in\R} \sum_{\substack{c \in \Z[\omega] \\ c \equiv 1 \pmod 3 \\ N(c) \ll R^{1/3}}} \frac{1}{N(c)} \sum_{\substack{r \in \Z[\omega] \\ r\equiv 1 \pmod 3 \\ N(r) \ll \frac{R}{N(bc^3)}}} \frac{1}{N(r)^{3/2}} \sum_{\substack{f \in \Z[\omega] \\ f \equiv 1\pmod{3} \\ f \mid r}} \frac{|\psi(\eta f, 1+\varepsilon + iy)|}{(1+|y|)^{100}} \\
    & \ll \frac{X^\varepsilon}{N(b)^{1/3} R^{1/6}} \sup_{y\in\R} \sum_{\substack{c \in \Z[\omega] \\ c \equiv 1 \pmod 3 \\ N(c) \ll R^{1/3}}} \frac{1}{N(c)} \sum_{\substack{f \in \Z[\omega] \\ f \equiv 1\pmod{3} \\ N(f) \ll \frac{R}{N(bc^3)}}} \frac{|\psi(\eta f, 1+\varepsilon + iy)|}{N(f)^{3/2} (1+|y|)^{100}} \\
    & \ll \frac{X^\varepsilon}{N(b)^{1/3} R^{1/6}} \sup_{y\in\R} \sum_{\substack{c \in \Z[\omega] \\ c \equiv 1 \pmod 3 \\ N(c) \ll R^{1/3}}} \frac{1}{N(c)} \Bigg(\sum_{\substack{f \in \Z[\omega] \\ f \equiv 1\pmod{3} \\ N(f) \ll \frac{R}{N(bc^3)}}} \frac{|\psi(\eta f, 1+\varepsilon + iy)|^2}{N(f)^{2} (1+|y|)^{200}} \Bigg)^{1/2} \ll \frac{X^\varepsilon \cdot \bbone_{R \gg N(b)}}{N(b)^{1/3} R^{1/6}},
\end{align*}
where we applied Cauchy--Schwarz, \cref{second_moment_lindelof_lemma}, and partial summation. Plugging this back into \eqref{P_prim_initial} and then \eqref{dual_sum_dyadic_decomp}, we conclude that the contribution of $\mathcal{I}_{\text{sec}}$ to $\widetilde{\mathcal{S}}(b)$ is
\begin{align}
    \ll_F \sum_{\substack{i \in \Z \\ 1 \ll R = 2^i \ll X^{1/2+\varepsilon}}} X^{5/6 + \varepsilon}  \frac{X^\varepsilon \cdot \bbone_{R \gg N(b)}}{N(b)^{1/3} R^{1/6}} \ll \frac{X^{5/6+\varepsilon}}{N(b)^{1/2}}. \label{I_sec_final_bound}
\end{align}


\subsubsection{The secondary polar term $\mathcal{P}_{\text{sec}}$}

Using \eqref{P_sec_bound}, observe that the contribution of $\mathcal{P}_{\text{sec}}$ to \eqref{P_prim_ready_for_Pat} is 
\begin{align*}
    &\ll_{F, \eta} \frac{X^\varepsilon R^{1/6}}{N(b)^{2/3}} \sum_{\substack{c \in \Z[\omega] \\ c \equiv 1 \pmod 3}} \frac{1}{N(c)^2} \sum_{\substack{r \in \Z[\omega] \\ r\equiv 1 \pmod 3 \\ N(r) \ll \frac{R}{N(bc^3)}}} \frac{1}{N(r)^{2}} \ll \frac{X^\varepsilon R^{1/6}}{N(b)^{2/3}}.
\end{align*}
Inserting this into \eqref{P_prim_initial} and \eqref{dual_sum_dyadic_decomp}, we conclude that the contribution of $\mathcal{P}_{\text{sec}}$ to $\widetilde{\mathcal{S}}(b)$ is
\begin{align}
    \ll_F \sum_{\substack{i \in \Z \\ 1 \ll R = 2^i \ll X^{1/2+\varepsilon}}} X^{5/6 + \varepsilon} \frac{X^\varepsilon R^{1/6}}{N(b)^{2/3}} \ll \frac{X^{11/12 + \varepsilon}}{N(b)^{2/3}}. \label{P_sec_final_bound}
\end{align}


\subsubsection{Final bound for the dual sum $\widetilde{\mathcal{S}}(b)$}

Combining \eqref{I_prim_final_bound}, \eqref{I_sec_final_bound}, and \eqref{P_sec_final_bound}, we conclude that
\begin{equation*}
    \widetilde{\mathcal{S}}(b) \ll_{F, \varepsilon} X^\varepsilon \Big( X^{3/4} N(b)^{1/2} + \frac{X^{5/6}}{N(b)^{1/2}} + \frac{X^{11/12}}{N(b)^{2/3}} \Big).
\end{equation*}
Together with \eqref{S_prime_final_bound}, this implies the bound \eqref{real_R_bound_first_mom}, as desired.
 
\end{proof}


\section{Choosing the mollifier} \label{mollifier_section}

In this section we choose the mollifier \eqref{mollifier} to maximize our rate of non-vanishing and prove \cref{thm:main}. 

We suppose throughout that $\lambda(\mathfrak{b})$ is real-valued, supported on squarefree $0 \neq \mathfrak{b} \unlhd \Z[\omega]$ coprime with $3$ such that $N(\mathfrak{b}) \leq M$, and satisfies $\lambda(\mathfrak{b}) \ll_\varepsilon N(\mathfrak{b})^{-1+\varepsilon}$. We fix a Schwartz function $F$ supported in $(1, 2)$ and satisfying $0 \leq F(t) \leq 1$ for every $t \in \R$ (so implied constants may now depend on $F$). Finally, we choose $Y = M X^\delta$ and $M = X^{1/6 - 3 \delta}$ for a small fixed $\delta > 0$. 

It will be convenient to make the change of variables 
\begin{equation*}
    \xi(\mathfrak{l}) := \sum_{\substack{0 \neq \mathfrak{a} \unlhd \Z[\omega]}} \lambda(\mathfrak{l a}) h(\mathfrak{a}),
\end{equation*}
where $h$ is the multiplicative function defined in \eqref{h_prime_def}. It is possible to recover $\lambda$ directly from $\xi$ by the formula
\begin{equation}\label{lambda_chi_mobius_inv}
    \lambda(\mathfrak{l}) = \sum_{\substack{0 \neq \mathfrak{a} \unlhd \Z[\omega]}} \mu(\mathfrak{a}) h(\mathfrak{a}) \xi(\mathfrak{l a}).
\end{equation}
This also shows that $\lambda$ being supported on squarefree $\mathfrak{l}$ coprime with $3$ and such that $N(\mathfrak{l}) \leq M$ is equivalent to the same property for $\xi$, which we assume from now on.

We will choose $\xi$ satisfying the bound
\begin{align}\label{xi_bound}
    |\xi(\mathfrak{d})| \ll \frac{1}{N(\mathfrak{d}) \log{M}} \prod_{\substack{\mathfrak{p} \text{ prime} \\ \mathfrak{p} \mid \mathfrak{d}}} \Big(1 + O\Big(\frac{1}{N(\mathfrak{p})}\Big) \Big),
\end{align}
which by \eqref{lambda_chi_mobius_inv} and the bound $h(\mathfrak{p}) = 1 + O(N(\mathfrak{p})^{-1/2})$ directly implies $\lambda(\mathfrak{d}) \ll_{\varepsilon} N(\mathfrak{d})^{-1+\varepsilon}$.


\subsection{The first mollified moment}

By \cref{afecenter}, \eqref{nonnegproperty}, \cref{SRestimate1}, \eqref{SMexpand}, and \cref{SMestimate1} (where observe for the latter that our choices satisfy $M Y^2 = X^{1/2 - \delta}$), the first mollified moment is
\begin{align}
    \mathcal{S}(L(1/2, \chi_q) \mathcal{M}(q); F) = C X \check{F}(0) Q_1(M) + O_{\varepsilon}(X^{1-\varepsilon}) \label{first_moment_Q}
\end{align}
for 
\begin{align}
    Q_1(M) := \sum_{\substack{0 \neq \mathfrak{b} \unlhd \Z[\omega] \\ N(\mathfrak{b}) \leq M}} \frac{\lambda(\mathfrak{b}) r(\mathfrak{b})}{\sqrt{N(\mathfrak{b})}} = \sum_{\substack{0 \neq \mathfrak{b} \unlhd \Z[\omega] \\ N(\mathfrak{b}) \leq M}} \frac{r(\mathfrak{b})}{\sqrt{N(\mathfrak{b})}} \sum_{\substack{0 \neq \mathfrak{a} \unlhd \Z[\omega]}} \mu(\mathfrak{a}) h(\mathfrak{a}) \xi(\mathfrak{a b}) =  \sum_{\substack{0 \neq \mathfrak{d} \unlhd \Z[\omega] \\ N(\mathfrak{d}) \leq M}} \xi(\mathfrak{d}) G(\mathfrak{d}), \label{Q_1_formula}
\end{align}
where $G$ is the multiplicative function given on primes $\mathfrak{p}$ by
\begin{equation}
    G(\mathfrak{p}) := \frac{r(\mathfrak{p})}{\sqrt{N(\mathfrak{p})}} - h(\mathfrak{p}) = -1 + O\Big(\frac{1}{N(\mathfrak{p})}\Big). \label{G_primes}
\end{equation}
This follows from the exact definitions of $r$ and $h$ in \eqref{r_prime_def} and \eqref{h_prime_def}, respectively.


\subsection{The second mollified moment}

We write $\mathfrak{b} = (\mathfrak{b}_1,\mathfrak{b}_2)$ and $\mathfrak{b}_1\mathfrak{b}_2 = \mathfrak{a} \mathfrak{b}^2$. The coefficients $\lambda$ will be chosen to be real, so by \cref{afecenter}, \eqref{nonnegproperty}, \cref{SRestimate2}, \eqref{SMexpand2}, and \cref{SMestimate2}, the second mollified moment is
\begin{align*}
    &\mathcal{S}(|L(1/2,\chi_q) \mathcal{M}(q)|^2; F) =  2 \sumtwo_{\substack{0 \neq \mathfrak{b}_1, \mathfrak{b}_2 \unlhd \Z[\omega] \\ N(\mathfrak{b}_1), N(\mathfrak{b}_2)\leq M}} \lambda(\mathfrak{b}_1) \lambda(\mathfrak{b}_2) \sqrt{N(\mathfrak{b}_1\mathfrak{b}_2)} \\
    &\times \Big( D \check{F}(0) X \frac{h(\mathfrak{a})g(\mathfrak{b})}{\sqrt{N(\mathfrak{a})}} \Big[\log \Big( \frac{X}{N(\mathfrak{a})} \Big) +  \mathcal{O}(\mathfrak{b}_1,\mathfrak{b}_2)\Big] \nonumber \\
    & + O_\varepsilon\Big(\frac{X^{1+\varepsilon}}{Y} + \frac{X^{5/6+\varepsilon}}{N(\mathfrak{a})^{1/3}}\Big) + \mathcal{R}(\mathfrak{b}_1,\mathfrak{b}_2)\Big) 
    + O_\varepsilon \Big( \frac{X^{1+\varepsilon} M^{2/3}}{Y^{2/3}} + X^{5/6 + \varepsilon} M \Big). 
\end{align*}

Since $\lambda(\mathfrak{b}) \ll_\varepsilon N(\mathfrak{b})^{-1+\varepsilon}$, the error terms $O_\varepsilon\big(\frac{X^{1+\varepsilon}}{Y} + \frac{X^{5/6+\varepsilon}}{N(\mathfrak{a})^{1/3}}\big)$ contribute
\begin{equation*}
    \ll \frac{X^{1+\varepsilon} M^{1+\varepsilon}}{Y} + X^{5/6+\varepsilon} M^{1/3+\varepsilon},
\end{equation*}
as
\begin{equation*}
    \sumtwo_{\substack{0 \neq \mathfrak{b}_1, \mathfrak{b}_2 \unlhd \Z[\omega] \\ N(\mathfrak{b}_1), N(\mathfrak{b}_2)\leq M}} \frac{1}{N(\mathfrak{b}_1\mathfrak{b}_2)^{1/2} N(\mathfrak{a})^{1/3}} =  \sum_{\substack{ 0\neq \mathfrak{b} \unlhd \Z[\omega] \\ N(\mathfrak{b}) \leq M}} \frac{1}{N(\mathfrak{b})} \sumtwo_{\substack{ 0\neq \mathfrak{b}'_1, \mathfrak{b}'_2 \unlhd \Z[\omega] \\ N(\mathfrak{b}'_1),N(\mathfrak{b}'_2)\leq \frac{M}{N(\mathfrak{b})} \\ (\mathfrak{b}'_1,\mathfrak{b}'_2)=1}} \frac{1}{N(\mathfrak{b}'_1\mathfrak{b}'_2)^{5/6}} \ll M^{1/3}.
\end{equation*}
The support of $\lambda$ is squarefree and coprime with $3$, so the contribution of $\mathcal{R}(\mathfrak{b}_1, \mathfrak{b}_2)$ is
\begin{align*}
    & \ll M^\varepsilon \sumtwo_{\substack{0 \neq \mathfrak{b}_1, \mathfrak{b}_2 \unlhd \Z[\omega] \\ N(\mathfrak{b}_1),N(\mathfrak{b}_2)\leq M\\ (\mathfrak{b}_1\mathfrak{b}_2,3)=1}} \frac{\mu^2(\mathfrak{b}_1)\mu^2(\mathfrak{b}_2)\left|\mathcal{R}(\mathfrak{b}_1,\mathfrak{b}_2)\right|}{\sqrt{N(\mathfrak{b}_1\mathfrak{b}_2)}} \\
    &\ll M^\varepsilon \max_{B_1,B_2\ll M} \frac{1}{\sqrt{B_1B_2}} \sumtwo_{\substack{0 \neq \mathfrak{b}_1, \mathfrak{b}_2 \unlhd \Z[\omega] \\ N(\mathfrak{b}_1)\sim B_1,\ N(\mathfrak{b}_2)\sim B_2\\ (\mathfrak{b}_1\mathfrak{b}_2,3)=1}} \mu^2(\mathfrak{b}_1)\mu^2(\mathfrak{b}_2)\left|\mathcal{R}(\mathfrak{b}_1,\mathfrak{b}_2)\right| \\
    & \ll X^{1/2+\varepsilon} \left( X^{1/6}Y^{1/2}M^{3/2} + YM^2 + X^{1/3}M \right)
\end{align*}
upon using the estimate \eqref{R_average_bound}. Thus, with our choices of $Y$ and $M$,
\begin{align}
    &\ \mathcal{S}(|L(1/2,\chi_q) \mathcal{M}(q)|^2; F) \nonumber \\
    = &\ 2 D \check{F}(0) X Q_2(M) + O_\varepsilon\Big(X^\varepsilon \Big(\frac{XM}{Y} + X^{2/3}Y^{1/2}M^{3/2}+X^{1/2}YM^2 + X^{5/6} M\Big)\Big) \nonumber \\
    = &\ 2 D \check{F}(0) X Q_2(M) + O_\varepsilon(X^{1-\varepsilon}), \label{second_moment_Q}
\end{align}
where 
\begin{equation*}
    Q_2(M) := \sumtwo_{\substack{0 \neq \mathfrak{b}_1, \mathfrak{b}_2 \unlhd \Z[\omega] \\ N(\mathfrak{b}_1), N(\mathfrak{b}_2)\leq M}} \lambda(\mathfrak{b}_1) \lambda(\mathfrak{b}_2) \sqrt{N(\mathfrak{b}_1\mathfrak{b}_2)}  \frac{h(\mathfrak{a})g(\mathfrak{b})}{\sqrt{N(\mathfrak{a})}} \left[\log\pfrac{X}{N(\mathfrak{a})} + \mathcal{O}(\mathfrak{b}_1,\mathfrak{b}_2)\right]
\end{equation*}
and we recall that
\begin{align}
    \mathcal{O}(\mathfrak{b}_1,\mathfrak{b}_2) = C_0 + \sum_{\substack{\mathfrak{p} \text{ prime} \\  \mathfrak{p} \mid (\mathfrak{b}_1, \mathfrak{b}_2)}} D_{1}(\mathfrak{p}) \frac{\log{N(\mathfrak{p})}}{N(\mathfrak{p})} + \sum_{\substack{\mathfrak{p} \text{ prime} \\ \mathfrak{p} \mid \frac{\mathfrak{b}_1 \mathfrak{b}_2}{(\mathfrak{b}_1, \mathfrak{b}_2)^2}}} D_{2}(\mathfrak{p}) \frac{\log{N(\mathfrak{p})}}{\sqrt{N(\mathfrak{p})}} \label{O_sum_def}
\end{align}
for certain $D_i(\mathfrak{p}) \ll 1$ and $C_0$ depending only on $F$ (hence fixed for our purposes).

Let us further evaluate $Q_2(M)$. We split the sum depending on $\mathfrak{b} = (\mathfrak{b}_1,\mathfrak{b}_2)$, making the change of variable $\mathfrak{b}_i \mapsto \mathfrak{b}_i \mathfrak{b}$, to obtain
\begin{align*}
    Q_2(M) = &\ \sum_{\substack{ 0\neq \mathfrak{b} \unlhd \Z[\omega] \\ N(\mathfrak{b}) \leq M}} g(\mathfrak{b})N(\mathfrak{b}) \sumtwo_{\substack{0 \neq \mathfrak{b}_1, \mathfrak{b}_2 \unlhd \Z[\omega] \\ N(\mathfrak{b}_1),N(\mathfrak{b}_2)\leq \frac{M}{N(\mathfrak{b})}\\ (\mathfrak{b}_1,\mathfrak{b}_2)=1}} \lambda(\mathfrak{b}_1\mathfrak{b}) \lambda(\mathfrak{b}_2\mathfrak{b}) \\
    & \times h(\mathfrak{b}_1)h(\mathfrak{b}_2)\left[\log\pfrac{X}{N(\mathfrak{b}_1\mathfrak{b}_2)} + \mathcal{O}(\mathfrak{b}_1\mathfrak{b},\mathfrak{b}_2\mathfrak{b})\right].
\end{align*}
We now remove the condition $(\mathfrak{b}_1,\mathfrak{b}_2)=1$ via M\"{o}bius inversion, setting $\mathfrak{b}_i \mapsto \mathfrak{b}_i \mathfrak{c}$, so that
\begin{align}
    Q_2(M) &= \sum_{\substack{ 0\neq \mathfrak{b} \unlhd \Z[\omega] \\ N(\mathfrak{b}) \leq M}} g(\mathfrak{b})N(\mathfrak{b}) \sum_{\substack{ 0\neq \mathfrak{c} \unlhd \Z[\omega] \\ N(\mathfrak{c}) \leq \frac{M}{N(\mathfrak{b})} }} \mu(\mathfrak{c}) \sumtwo_{\substack{0 \neq \mathfrak{b}_1, \mathfrak{b}_2 \unlhd \Z[\omega] \\ N(\mathfrak{b}_1), N(\mathfrak{b}_2)\leq \frac{M}{N(\mathfrak{bc})}}} \lambda(\mathfrak{b}_1\mathfrak{b}\mathfrak{c}) \lambda(\mathfrak{b}_2\mathfrak{b}\mathfrak{c}) h(\mathfrak{b}_1\mathfrak{c}) h(\mathfrak{b}_2\mathfrak{c}) \nonumber \\
    &  \qquad \qquad \qquad \qquad \qquad \qquad \times  \left(\log\pfrac{X}{N(\mathfrak{b}_1\mathfrak{b}_2\mathfrak{c}^2)} + \mathcal{O}(\mathfrak{b}_1\mathfrak{b}\mathfrak{c},\mathfrak{b}_2\mathfrak{b} \mathfrak{c}) \right) \nonumber \\
    & = \sum_{\substack{ 0\neq \mathfrak{b} \unlhd \Z[\omega] \\ N(\mathfrak{b}) \leq M}} g(\mathfrak{b})N(\mathfrak{b}) \sum_{\substack{ 0\neq \mathfrak{c} \unlhd \Z[\omega] \\ N(\mathfrak{c}) \leq \frac{M}{N(\mathfrak{b})} }} \mu(\mathfrak{c}) h(\mathfrak{c})^2 \sumtwo_{\substack{0 \neq \mathfrak{b}_1, \mathfrak{b}_2 \unlhd \Z[\omega] \\ N(\mathfrak{b}_1), N(\mathfrak{b}_2)\leq \frac{M}{N(\mathfrak{bc})}}} \lambda(\mathfrak{b}_1\mathfrak{b}\mathfrak{c}) \lambda(\mathfrak{b}_2\mathfrak{b}\mathfrak{c}) h(\mathfrak{b}_1) h(\mathfrak{b}_2) \nonumber \\
    &  \qquad \qquad \qquad \qquad \qquad \qquad \times \left(\log\pfrac{X}{N(\mathfrak{b}_1\mathfrak{b}_2\mathfrak{c}^2)} + \mathcal{O}(\mathfrak{b}_1\mathfrak{b}\mathfrak{c},\mathfrak{b}_2\mathfrak{b} \mathfrak{c}) \right) \nonumber \\
    & = \sum_{\substack{ 0\neq \mathfrak{d} \unlhd \Z[\omega] \\ N(\mathfrak{d}) \leq M}} N(\mathfrak{d}) H(\mathfrak{d}) \sumtwo_{\substack{0 \neq \mathfrak{b}_1, \mathfrak{b}_2 \unlhd \Z[\omega] \\ N(\mathfrak{b}_1),N(\mathfrak{b}_2)\leq \frac{M}{N(\mathfrak{d})}}} \lambda(\mathfrak{b}_1\mathfrak{d}) \lambda(\mathfrak{b}_2\mathfrak{d}) h(\mathfrak{b}_1) h(\mathfrak{b}_2) \nonumber \\
    &  \qquad \qquad \qquad \qquad \qquad \qquad \times \left(\log\pfrac{X}{N(\mathfrak{b}_1\mathfrak{b}_2)} + 2\eta(\mathfrak{d}) + \mathcal{O}(\mathfrak{b}_1\mathfrak{d},\mathfrak{b}_2\mathfrak{d}) \right), \label{Q_2_with_error_terms}
\end{align}
where for squarefree $\mathfrak{d}$ we define
\begin{equation} \label{Hd}
    H(\mathfrak{d}) := \sumtwo_{\substack{0 \neq \mathfrak{b}, \mathfrak{c} \unlhd \Z[\omega] \\ \mathfrak{b}\mathfrak{c} = \mathfrak{d}}} g(\mathfrak{b}) \frac{\mu(\mathfrak{c})h(\mathfrak{c})^2}{N(\mathfrak{c})} = \prod_{\substack{\mathfrak{p} \text{ prime} \\ \mathfrak{p} \mid \mathfrak{d}}} \Big(g(\mathfrak{p}) - \frac{h(\mathfrak{p})^2}{N(\mathfrak{p})}\Big), 
\end{equation}
and
\begin{equation*}
    \eta(\mathfrak{d}) := \frac{-1}{H(\mathfrak{d})} \sumtwo_{\substack{0 \neq \mathfrak{b}, \mathfrak{c} \unlhd \Z[\omega] \\ \mathfrak{b}\mathfrak{c} = \mathfrak{d}}} g(\mathfrak{b}) \frac{\mu(\mathfrak{c})h(\mathfrak{c})^2}{N(\mathfrak{c})} \log N(\mathfrak{c}) = \sum_{\substack{\mathfrak{p} \text{ prime} \\ \mathfrak{p} \mid \mathfrak{d}}} \frac{h(\mathfrak{p})^2 \log N(\mathfrak{p})}{N(\mathfrak{p}) H(\mathfrak{p})}
\end{equation*}
by a standard calculation. For $\mathfrak{p}$ prime we have
\begin{equation}
    0 < H(\mathfrak{p}) = 1 + O\Big(\frac{1}{N(\mathfrak{p})}\Big) \label{H_primes}
\end{equation}
from the definitions of $g$ and $h$ (where the first inequality requires some checking for small primes). Comparing with the corresponding sums for $N(\mathfrak{p}) \ll \log{N(\mathfrak{d})}$, observe (for $N(\mathfrak{d})$ large) that
\begin{align*}
    \sum_{\substack{\mathfrak{p} \text{ prime} \\ \mathfrak{p} \mid \mathfrak{d} }} \frac{1}{N(\mathfrak{p})} \leq \log\log\log{N(\mathfrak{d})} + O(1) \quad \text{ and } \quad \sum_{\substack{\mathfrak{p} \text{ prime} \\ \mathfrak{p} \mid \mathfrak{d} }} \frac{\log{N(\mathfrak{p})}}{N(\mathfrak{p})} \leq \log\log{N(\mathfrak{d})} + O(1),
\end{align*}
so \eqref{H_primes} implies
\begin{equation*}
    H(\mathfrak{d}) \ll (\log\log{N(\mathfrak{d})})^{O(1)} \qquad \text{ and } \qquad \eta(\mathfrak{d}) \ll \log\log N(\mathfrak{d}).
\end{equation*}

From \eqref{xi_bound} and the restriction $N(\mathfrak{d})\leq M$ on the support of $\xi(\mathfrak{d})$, we also have the crude bound
\begin{equation*}
    |\xi(\mathfrak{d})| \ll \frac{(\log \log{M})^{O(1)}}{N(\mathfrak{d}) \log{M}}. 
\end{equation*}
Combining these bounds with \eqref{xi_bound}, the contribution of the term $2 \eta(\mathfrak{d})$ to \eqref{Q_2_with_error_terms} is
\begin{align*}
    = 2 \sum_{\substack{ 0\neq \mathfrak{d} \unlhd \Z[\omega] \\ N(\mathfrak{d}) \leq M}} N(\mathfrak{d}) H(\mathfrak{d}) \eta(\mathfrak{d}) \xi(\mathfrak{d})^2 \ll \sum_{\substack{ 0\neq \mathfrak{d} \unlhd \Z[\omega] \\ N(\mathfrak{d}) \leq M}} \frac{(\log\log{N(\mathfrak{d})})^{O(1)}}{(\log{M})^2 N(\mathfrak{d})} \ll \frac{(\log\log{X})^{O(1)}}{\log{X}} = o(1).
\end{align*}

Similarly, the contribution of the term $\mathcal{O}(\mathfrak{b}_1\mathfrak{d},\mathfrak{b}_2\mathfrak{d})$ to \eqref{Q_2_with_error_terms} is
\begin{align}
    = \sum_{\substack{ 0\neq \mathfrak{d} \unlhd \Z[\omega] \\ N(\mathfrak{d}) \leq M}} N(\mathfrak{d}) H(\mathfrak{d}) \Big( C_0 \: \xi(\mathfrak{d})^2 & + \sum_{\substack{\mathfrak{p} \text{ prime} \\  \mathfrak{p} \mid \mathfrak{d}}} D_{1}(\mathfrak{p}) \frac{\log{N(\mathfrak{p})}}{N(\mathfrak{p})} \xi(\mathfrak{d})^2 + \sum_{\substack{\mathfrak{p} \text{ prime} \\  \mathfrak{p} \nmid \mathfrak{d}}} D_{1}(\mathfrak{p}) \frac{\log{N(\mathfrak{p})}}{N(\mathfrak{p})} \xi(\mathfrak{pd})^2 \nonumber \\
    &  + 2 \sum_{\substack{\mathfrak{p} \text{ prime}}} D_{2}(\mathfrak{p}) \frac{\log{N(\mathfrak{p})}}{\sqrt{N(\mathfrak{p})}} \big(\xi(\mathfrak{pd}) \xi(\mathfrak{d}) - \xi(\mathfrak{pd})^2\big)\Big), \nonumber
\end{align}
which we readily bound by
\begin{align}
    &\ll \sum_{\substack{ 0\neq \mathfrak{d} \unlhd \Z[\omega] \\ N(\mathfrak{d}) \leq M}} \frac{(\log \log{M})^{O(1)}}{N(\mathfrak{d}) (\log{M})^2} \Big( 1 + \sum_{\substack{\mathfrak{p} \text{ prime} \\  \mathfrak{p} \mid \mathfrak{d}}} \frac{\log{N(\mathfrak{p})}}{N(\mathfrak{p})} + \sum_{\substack{\mathfrak{p} \text{ prime} \\  N(\mathfrak{p}) \leq M}} \frac{\log{N(\mathfrak{p})}}{N(\mathfrak{p})^3} + \sum_{\substack{\mathfrak{p} \text{ prime} \\ N(\mathfrak{p}) \leq M}} \frac{\log{N(\mathfrak{p})}}{N(\mathfrak{p})^{3/2}} \Big) \nonumber \\
    & \ll \frac{(\log \log{X})^{O(1)}}{\log{X}} = o(1). \nonumber
\end{align}

Therefore, we conclude that
\begin{align}
    Q_2(M) = \sum_{\substack{ 0\neq \mathfrak{d} \unlhd \Z[\omega] \\ N(\mathfrak{d}) \leq M}} N(\mathfrak{d}) H(\mathfrak{d}) \sumtwo_{\substack{0 \neq \mathfrak{b}_1, \mathfrak{b}_2 \unlhd \Z[\omega] \\ N(\mathfrak{b}_1),N(\mathfrak{b}_2)\leq \frac{M}{N(\mathfrak{d})}}} \lambda(\mathfrak{b}_1\mathfrak{d}) \lambda(\mathfrak{b}_2\mathfrak{d}) h(\mathfrak{b}_1) h(\mathfrak{b}_2) \log\pfrac{X}{N(\mathfrak{b}_1\mathfrak{b}_2)} + o(1). \label{Q_2_formula}
\end{align}


\subsection{The optimal mollifier}

Observe that $Q_2(M)$ is essentially equal to the diagonal quadratic form
\begin{align*}
    \log{X} \sum_{\substack{ 0\neq \mathfrak{d} \unlhd \Z[\omega] \\ N(\mathfrak{d}) \leq M}} N(\mathfrak{d}) H(\mathfrak{d}) \xi(\mathfrak{d})^2.
\end{align*}
We wish to minimize it, while maintaining the linear constraint \eqref{Q_1_formula} corresponding to $Q_1(M)$ constant. This is achieved (by Cauchy--Schwarz or Lagrange multipliers) for $\xi(\mathfrak{d})$ proportional to $\frac{G(\mathfrak{d})}{H(\mathfrak{d}) N(\mathfrak{d})}$. More precisely, for squarefree $\mathfrak{d}$ coprime with $3$ and satisfying $N(\mathfrak{d}) \leq M$ we choose 
\begin{align}
    \xi(\mathfrak{d}) = \frac{C}{D \log{M}} \cdot \frac{G(\mathfrak{d})}{N(\mathfrak{d}) H(\mathfrak{d})}, \label{xi_choice}
\end{align}
which by \eqref{G_primes} and \eqref{H_primes} satisfies the constraint \eqref{xi_bound}.


\subsubsection{Endgame for first moment}

Inserting this into \eqref{Q_1_formula} and \eqref{first_moment_Q}, we obtain
\begin{align} \label{firstset}
    \mathcal{S}(L(1/2, \chi_q) \mathcal{M}(q); F) = \check{F}(0) X \frac{C^2}{D \log{M}} \sum_{\substack{0 \neq \mathfrak{d} \unlhd \Z[\omega] \\  N(\mathfrak{d}) \leq M \\ (\mathfrak{d}, 3) = 1}} \mu(\mathfrak{d})^2 \frac{G(\mathfrak{d})^2}{N(\mathfrak{d}) H(\mathfrak{d})} + o(X).
\end{align}
Observe, since $\frac{G(\mathfrak{p})^2}{H(\mathfrak{p})} = 1 + O\big(\frac{1}{N(\mathfrak{p})}\big)$, that a standard argument (e.g.\ via Perron's formula) and \eqref{zeta_residue} imply
\begin{align}
    \sum_{\substack{0 \neq \mathfrak{d} \unlhd \Z[\omega] \\  N(\mathfrak{d}) \leq M \\ (\mathfrak{d}, 3) = 1}} \mu(\mathfrak{d})^2 \frac{G(\mathfrak{d})^2}{N(\mathfrak{d}) H(\mathfrak{d})} & = \mathop{\mathrm{Res}}_{s=0} \zeta_\lambda(1+s) \prod_{\substack{\mathfrak{p} \text{ prime} \\ (\mathfrak{p}, 3) = 1}} \Big(1 - \frac{1}{N(\mathfrak{p})} \Big) \Big(1 + \frac{G(\mathfrak{p})^2}{N(\mathfrak{p}) H(\mathfrak{p})} \Big) (\log{M} + O(1)) \nonumber \\
    & =\frac{2 \pi}{9 \sqrt{3}} \log{M} \prod_{\substack{\mathfrak{p} \text{ prime} \\ (\mathfrak{p}, 3) = 1}} \Big(1 - \frac{1}{N(\mathfrak{p})} \Big) \Big(1 + \frac{G(\mathfrak{p})^2}{N(\mathfrak{p}) H(\mathfrak{p})} \Big) + O(1) \nonumber\\
    &= \frac{2 \pi}{9 \sqrt{3}} \mathcal{P} \log{M} +O(1) \label{first_mom_asymptotic_comp}
\end{align}
for 
\begin{equation*}
    \mathcal{P} := \prod_{\substack{\mathfrak{p} \text{ prime} \\ (\mathfrak{p}, 3) = 1 \\ q\; :=\; N(\mathfrak{p})}} \frac{(q-1) (q+1) (q^4 + 2 q^3 + q^2 - 2q^{3/2} + 1)}{q (q^{5/2} + q^{3/2} - 1)^2},
\end{equation*}
where the expression for $\mathcal{P}$ follows from a direct computation using \eqref{G_primes} and \eqref{Hd}, after recalling the definitions of $r$, $g$, and $h$ in \eqref{r_prime_def}, \eqref{g_prime_def}, and \eqref{h_prime_def}, respectively. From the definitions of $C$ and $D$ in \eqref{Cdef} and \eqref{Ddef}, we conclude that the first mollified moment is equal to $\mathcal{C}_1 \mathcal{P}_1 \check{F}(0) X + o(X)$, where
\begin{align*}
    \mathcal{C}_1 := \frac{2 \pi}{9 \sqrt{3}} \Big(\frac{\pi}{36(\sqrt{3}-1) \cdot \zeta_{\Q(\omega)}(2)} \Big)^2 \Big( \frac{\pi^2}{648(2-\sqrt{3}) \cdot \zeta_{\Q(\omega)}(2)} \Big)^{-1} = \frac{\pi \sqrt{3}}{54 \cdot \zeta_{\Q(\omega)}(2)}
\end{align*}
and we have the remarkable identity
\begin{align*}
     \mathcal{P}_1 := \mathcal{P} \cdot \prod_{\substack{\mathfrak{p} \text{ prime} \\ (\mathfrak{p}, 3) = 1 \\ q\; :=\; N(\mathfrak{p})}}  \Big( 1 + \frac{q}{(q+1)(q^{3/2}-1)} \Big)^2 \Big( 1 - \frac{1}{q(q+1)} + \frac{2q}{(q+1) (q^{3/2}-1)} \Big)^{-1} = 1.
\end{align*}

Thus 
\begin{equation}
    \mathcal{S}(L(1/2, \chi_q) \mathcal{M}(q); F) = \frac{\pi \sqrt{3}}{54 \cdot \zeta_{\Q(\omega)}(2)} \check{F}(0) X + o(X). \label{first_mom_asymp}
\end{equation}


\subsubsection{Endgame for second moment}

By \eqref{Q_2_formula}, 
\begin{align}
    Q_2(M) = \sum_{\substack{ 0\neq \mathfrak{d} \unlhd \Z[\omega] \\ N(\mathfrak{d}) \leq M}} N(\mathfrak{d}) H(\mathfrak{d}) \Big( \xi(\mathfrak{d})^2 \log{X} - 2 \xi(\mathfrak{d}) \sum_{\substack{0 \neq \mathfrak{b} \unlhd \Z[\omega] \\ N(\mathfrak{b}) \leq \frac{M}{N(\mathfrak{d})}}} \lambda(\mathfrak{b} \mathfrak{d}) h(\mathfrak{b}) \log{N(\mathfrak{b})} \Big) + o(1). \nonumber
\end{align}
From our choice of $\xi$ in \eqref{xi_choice} and properties of the relevant multiplicative functions in \eqref{G_primes} and \eqref{H_primes} combined with $h(\mathfrak{p}) = 1 + O(N(\mathfrak{p})^{-1/2})$, we compute
\begin{align*}
    \sum_{\substack{0 \neq \mathfrak{b} \unlhd \Z[\omega] \\ N(\mathfrak{b}) \leq \frac{M}{N(\mathfrak{d})}}} \lambda(\mathfrak{b} \mathfrak{d}) h(\mathfrak{b}) \log{N(\mathfrak{b})} &= \sum_{\substack{\mathfrak{p} \text{ prime} \\ N(\mathfrak{p}) \leq \frac{M}{N(\mathfrak{d})}}} \log{N(\mathfrak{p})} h(\mathfrak{p}) \xi(\mathfrak{pd}) \\
    &= \xi(\mathfrak{d})\sum_{\substack{\mathfrak{p} \text{ prime} \\ N(\mathfrak{p}) \leq \frac{M}{N(\mathfrak{d})} \\ \mathfrak{p} \nmid 3 \mathfrak{d}}} \log{N(\mathfrak{p})} h(\mathfrak{p}) \frac{G(\mathfrak{p})}{N(\mathfrak{p}) H(\mathfrak{p})} \\
    & =  - \xi(\mathfrak{d}) \sum_{\substack{\mathfrak{p} \text{ prime} \\ N(\mathfrak{p}) \leq \frac{M}{N(\mathfrak{d})} \\ \mathfrak{p} \nmid 3\mathfrak{d}}} \frac{\log{N(\mathfrak{p})}}{N(\mathfrak{p})} \Big(1 + O\Big(\frac{1}{N(\mathfrak{p})^{1/2}}\Big)\Big).
\end{align*}
Observe that a standard computation gives
\begin{align*}
    \sum_{\substack{\mathfrak{p} \text{ prime} \\ N(\mathfrak{p}) \leq \frac{M}{N(\mathfrak{d})} \\ \mathfrak{p} \nmid 3 \mathfrak{d}}} \frac{\log{N(\mathfrak{p})}}{N(\mathfrak{p})} \Big(1 + O\Big(\frac{1}{N(\mathfrak{p})^{1/2}}\Big)\Big) & = \sum_{\substack{\mathfrak{p} \text{ prime} \\ N(\mathfrak{p}) \leq \frac{M}{N(\mathfrak{d})}}} \frac{\log{N(\mathfrak{p})}}{N(\mathfrak{p})}  + O\big(1 + \log\log{N(\mathfrak{d})}\big) \\
    & = \log{\Big( \frac{M}{N(\mathfrak{d})} \Big)} + O(\log\log{X}).
\end{align*}
Therefore,
\begin{align}
    Q_2(M) & = \sum_{\substack{ 0\neq \mathfrak{d} \unlhd \Z[\omega] \\ N(\mathfrak{d}) \leq M}} N(\mathfrak{d}) H(\mathfrak{d}) \xi(\mathfrak{d})^2 \big(\log{(X M^2)} - 2 \log{N(\mathfrak{d})} + O(\log\log{X}) \big) + o(1). \nonumber
\end{align}

By \eqref{second_moment_Q} we conclude that up to $o(X)$, $\mathcal{S}(|L(1/2,\chi_q) \mathcal{M}(q)|^2; F)$ is equal to
\begin{align}
    \check{F}(0) X \frac{2 C^2}{D (\log{M})^2} \sum_{\substack{ 0\neq \mathfrak{d} \unlhd \Z[\omega] \\ N(\mathfrak{d}) \leq M \\ (\mathfrak{d}, 3) = 1}} \mu(\mathfrak{d})^2 \frac{G(\mathfrak{d})^2}{H(\mathfrak{d}) N(\mathfrak{d})} \big(\log{(X M^2)} - 2 \log{N(\mathfrak{d})} + O(\log\log{X}) \big). \label{second_mom_prelim_asymp}
\end{align}
The first term in \eqref{second_mom_prelim_asymp} is, apart from the logarithmic factors, identical to twice our expression for the first moment in
\eqref{firstset}. Hence we see that the first term in \eqref{second_mom_prelim_asymp} is equal to
\begin{align*}
    \frac{\log{(XM^2)}}{\log{M}}  \frac{\pi \sqrt{3}}{27 \cdot \zeta_{\Q(\omega)}(2)} \check{F}(0) X  + o(X).
\end{align*}
Since $H(\mathfrak{d})$ is non-negative, the same argument shows that the third term in \eqref{second_mom_prelim_asymp} is $\ll X \frac{\log\log{X}}{\log{X}} = o(X)$, which is subsumed to the error term. Finally, \eqref{first_mom_asymptotic_comp} and partial summation imply
\begin{align*}
    \sum_{\substack{ 0\neq \mathfrak{d} \unlhd \Z[\omega] \\ N(\mathfrak{d}) \leq M \\ (\mathfrak{d}, 3) = 1}} \mu(\mathfrak{d})^2 \frac{G(\mathfrak{d})^2}{H(\mathfrak{d}) N(\mathfrak{d})} \log{N(\mathfrak{d})} = \frac{2 \pi}{9 \sqrt{3}} \mathcal{P} \cdot \frac{(\log{M})^2}{2} + O(\log{M}),
\end{align*}
so the second term in \eqref{second_mom_prelim_asymp} is equal to
\begin{align*}
   - \frac{\pi \sqrt{3}}{27 \cdot \zeta_{\Q(\omega)}(2)} \check{F}(0) X  + o(X).
\end{align*}
Thus denoting $M = X^\theta$ (where we chose $\theta = \frac{1}{6} - 3 \delta$), we have shown that
\begin{align}
    \mathcal{S}(|L(1/2,\chi_q) \mathcal{M}(q)|^2; F) = \frac{\pi \sqrt{3}}{27 \cdot \zeta_{\Q(\omega)}(2)} \Big(1 + \frac{1}{\theta} \Big) \check{F}(0) X  + o(X). \label{second_mom_asymp}
\end{align}


\subsection{Final density computation}

We are now ready to prove our main result.

\begin{proof}[Proof of \cref{thm:main}]

We recall that $F$ is supported in $(1, 2)$ and satisfies $0 \leq F(t) \leq 1$ for all $t \in \R$. By Cauchy--Schwarz and the asymptotic expressions \eqref{first_mom_asymp} and \eqref{second_mom_asymp},
\begin{align*}
    \sum_{\substack{q \in \mathbb{Z}[\omega] \\ q \equiv 1 \pmod{9} \\ X < N(q) \leq 2X  \\ L(1/2,\chi_q) \neq 0}} \mu^2(q) & \geq \sum_{\substack{q \in \mathbb{Z}[\omega] \\ q \equiv 1 \pmod{9} \\ L(1/2,\chi_q) \neq 0}} \mu^2(q) F \Big( \frac{N(q)}{X} \Big) \\
    & \geq \frac{|\mathcal{S}(L(1/2,\chi_q) \mathcal{M}(q); F)|^2}{\mathcal{S}(|L(1/2,\chi_q) \mathcal{M}(q)|^2; F)} = \frac{\pi \sqrt{3}}{108 \cdot \zeta_{\Q(\omega)}(2)} \frac{\theta}{\theta + 1} \check{F}(0) X  + o(X).
\end{align*}

We may choose $\theta$ arbitrarily close to $\frac{1}{6}$, and $F$ arbitrarily close to the indicator function of the interval $(1, 2)$, so that $\check{F}(0)$ approaches $1$. A short computation using Perron's formula and the standard zero-free region for Hecke $L$-functions gives that
\begin{align*}
    \sum_{\substack{q \in \mathbb{Z}[\omega] \\ q \equiv 1 \pmod{9} \\ X < N(q) \leq 2X}} \mu^2(q) = \frac{1}{9} \sum_{\substack{q \in \mathbb{Z}[\omega] \\ q \equiv 1 \pmod{3} \\ X < N(q) \leq 2X}} \mu^2(q) + o(X) = \frac{X}{9} \mathop{\mathrm{Res}}_{s=1} \frac{\zeta_\lambda(s)}{\zeta_\lambda(2s)} + o(X).
\end{align*}
From \eqref{zeta_lambda_identity} and \eqref{zeta_residue} we compute 
\begin{align*}
    \frac{X}{9} \mathop{\mathrm{Res}}_{s=1} \frac{\zeta_\lambda(s)}{\zeta_\lambda(2s)} = \frac{X \cdot \mathop{\mathrm{Res}}_{s=1} \zeta_\lambda(s)}{8 \cdot \zeta_{\Q(\omega)}(2)} = \frac{\pi \sqrt{3}}{108 \cdot \zeta_{\Q(\omega)}(2)}.
\end{align*}

Therefore
\begin{equation*}
    \sum_{\substack{q \in \mathbb{Z}[\omega] \\ q \equiv 1 \pmod{9} \\ X < N(q) \leq 2X  \\ L(1/2,\chi_q) \neq 0}} \mu^2(q) \geq \Big(\frac{\theta}{\theta + 1} - \varepsilon \Big) \sum_{\substack{q \in \mathbb{Z}[\omega] \\ q \equiv 1 \pmod{9} \\ X < N(q) \leq 2X}} \mu^2(q) \geq \Big(\frac{1}{7} - 2\varepsilon\Big) \sum_{\substack{q \in \mathbb{Z}[\omega] \\ q \equiv 1 \pmod{9} \\ X < N(q) \leq 2X}} \mu^2(q)
\end{equation*}
for all $X$ sufficiently large in terms of $\varepsilon$. Summing over dyadic $X$ finishes the proof of \cref{thm:main}.
    
\end{proof}


\bibliography{weylcubic} 

\end{document}